\documentclass[a4paper,12pt]{amsart}
\usepackage[utf8]{inputenc}
\usepackage[T1]{fontenc}
\usepackage[UKenglish]{babel}
\usepackage[margin=20mm]{geometry}
\usepackage{color}
\usepackage{bbm,bm}
\usepackage{subfigure}
\usepackage{float}

\usepackage{graphicx}

\usepackage{amsmath,amssymb,amsfonts,amsthm}
\usepackage{mathrsfs,eucal,dsfont}
\usepackage{verbatim,enumitem}

\usepackage{hyperref,url}
\usepackage{amssymb, amsmath, amsthm, amscd}

\usepackage{eucal,mathrsfs,dsfont}
\usepackage{color}

\renewcommand{\leq}{\le}
\renewcommand{\geq}{\ge}

\newcommand{\e}{\text{e}}

\newcommand{\I}{\mathds 1}

\def\d{{\rm d}}
\def\<{\langle}
\def\>{\rangle}

\newtheorem{claim}{Claim}
\newtheorem{theorem}{Theorem}[section]
\newtheorem{lemma}[theorem]{Lemma}
\newtheorem{proposition}[theorem]{Proposition}
\newtheorem{corollary}[theorem]{Corollary}
\numberwithin{equation}{section}

\theoremstyle{definition}
\newtheorem{definition}[theorem]{Definition}

\newtheorem{remark}[theorem]{Remark}

\begin{document}
\allowdisplaybreaks
\title[Uniqueness of the critical long-range percolation metrics]
{\bfseries  Uniqueness of the critical long-range percolation metrics}

\author{Jian Ding \qquad  Zherui Fan  \qquad  Lu-Jing Huang}

\thanks{\emph{J. Ding:}
School of Mathematical Sciences, Peking University, Beijing, China.
  \texttt{dingjian@math.pku.edu.cn}}
\thanks{\emph{Z. Fan:}
School of Mathematical Sciences, Peking University, Beijing, China.
  \texttt{2301110064@pku.edu.cn}}

\thanks{\emph{L.-J. Huang:}
School of Mathematics and Statistics\&Key Laboratory of Analytical Mathematics and Applications (Fujian Normal University), Ministry of Education,  Fuzhou, China.
  \texttt{huanglj@fjnu.edu.cn}}


\date{}
\maketitle

\begin{abstract}

    In this work, we study the random metric for the critical long-range percolation on $\mathds{Z}^d$. A recent work by B\"aumler \cite{Baumler22} implies the subsequential scaling limit, and our main contribution is to prove that the subsequential limit is uniquely characterized by a natural list of axioms. Our proof method is hugely inspired by recent works of Gwynne and Miller \cite{GM21}, and Ding and Gwynne \cite{DG23} on the uniqueness of Liouville quantum gravity metrics.


\medskip


\end{abstract}
\allowdisplaybreaks

\tableofcontents

\section{Introduction}\label{section1}

The critical long-range percolation model (LRP) is a model of percolation on $\mathds{Z}^d$, where all edges $\langle {\bm i},{\bm j}\rangle $ with $\|{\bm i}-{\bm j}\|_{1}=1$ (i.e. ${\bm i}$ and ${\bm j}$ are nearest neighbors) occur independently with probability $p_1$ and an edge $\langle {\bm i},{\bm j}\rangle$ with $\|{\bm i}-{\bm j}\|_{1}>1$ (which will be referred to as a long edge in what follows) occurs independently with probability
\begin{equation}\label{connectprob}
    1-\exp\left\{-\beta \int_{V_1({\bm i})}\int_{V_1({\bm j})}\frac{1}{|{\bm u}-{\bm v}|^{2d}}\d {\bm u}\d {\bm v}\right\}.
\end{equation}
Here $V_1({\bm i})$ is a cube in $\mathds{R}^d$ with center ${\bm i}$ and side length 1, $\| \cdot\|_1$ is the $\ell^1$-norm and $|\cdot|$ is the $\ell^2$-norm. We also call this model a discrete $\beta$-LRP model or a critical long-range bond percolation model. In this paper, we consider the particular case for $p_1= 1$, where the connectivity is trivial. Instead, we will focus on the metric properties of this percolation model and the trivial connectivity coming from the assumption of $p_1=1$ provides a level of simplification on challenging questions for metrics. In addition, our particular choice for the form of the connecting probability follows that in \cite{Baumler22} with the purpose of offering scaling invariant properties since we will study the scaling limit of the metrics. 

Denote by $\widehat{d}(\cdot,\cdot)$ the graph distance (also known as the chemical distance) on this discrete model, where each of the nearest neighbor edges and the long edges is counted as weight 1. It has been shown in \cite{Baumler22} (see also \cite{Ding-Sly13} for the  one-dimensional case) that $d(\bm 0,{\bm n})\asymp |{\bm n}|^\theta$ for some $\theta=\theta(\beta,d)$ as $|{\bm n}|\to\infty$. Our main result shows the existence and uniqueness of the scaling limit of the chemical distance $\widehat{d}(\lfloor n\cdot\rfloor, \lfloor n\cdot\rfloor)$ (here $\lfloor n{\bm x}\rfloor=(\lfloor n{\bm x}^1\rfloor,\cdots,\lfloor n{\bm x}^d\rfloor)$ for $n\in\mathds{N}$ and ${\bm x}=({\bm x}^1,\cdots,{\bm x}^d)\in\mathds{R}^d$).

\begin{theorem}\label{realmaintheorem}
    Let $\widehat{d}$ be the chemical distance on the discrete $\beta$-LRP model. Let $\widehat{a}_n$ be the median of $\widehat{d}({\bm 0},n{\bm 1})$ {\rm(}here ${\bm 1}=(1,1,\cdots,1)\in\mathds{R}^d${\rm )} and $\widehat{D}_n=\widehat{a}_n^{-1}\widehat{d}(\lfloor n\cdot\rfloor, \lfloor n\cdot\rfloor)$. Then there exists a unique random metric $D$ on $\mathds{R}^d$ such that $\widehat{D}_n$ converges to $D$ in law with respect to the topology of local uniform convergence on $\mathds{R}^{2d}$.
\end{theorem}

In fact, we will show that the scaling limit satisfies a list of axioms which uniquely characterize the metric (see Definition \ref{strongLRPmetric} and Theorem \ref{dis-maintheorem}).

\begin{figure}[htbp]
\centering
\subfigure{\includegraphics[scale=0.3]{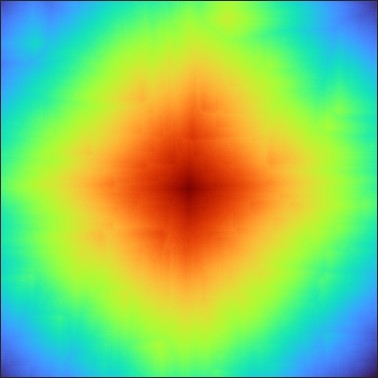}}
\quad
\subfigure{\includegraphics[scale=0.3]{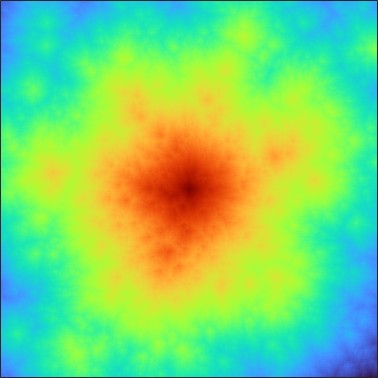}}
\quad
\subfigure{\includegraphics[scale=0.3]{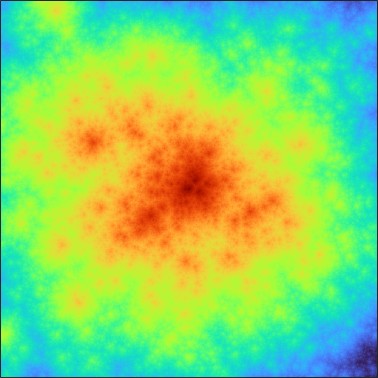}}
\quad
\subfigure{\includegraphics[scale=0.3]{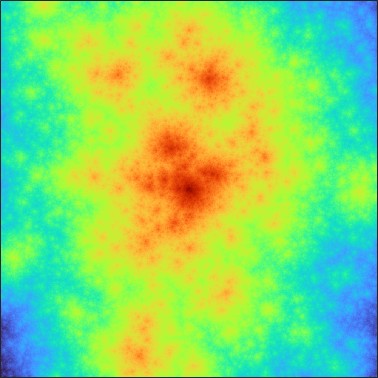}}
\quad
\subfigure{\includegraphics[scale=0.3]{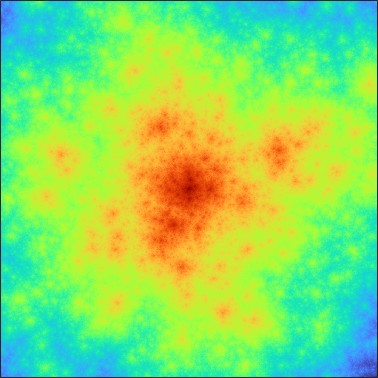}}
\quad
\subfigure{\includegraphics[scale=0.3]{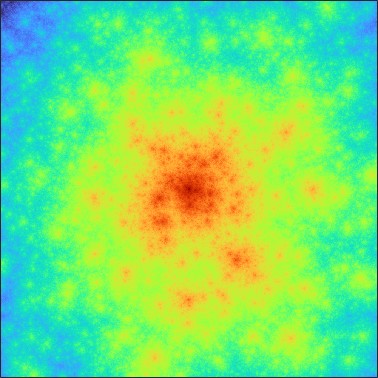}}
\caption{Simulation of $\beta$-LRP metric balls for $\beta=0.01$ (top left), $\beta=0.1$ (top middle), $\beta=0.5$ (top right), $\beta=1$ (bottom left), $\beta=2$ (bottom middle) and $\beta=5$ (bottom right). 
The centers of the above cubes are ${\bm 0}$.
The colors on the graph, ranging from warm (red) to cool (blue), represent the $\beta$-LRP distance from ${\bm 0}$, from small to large.
The simulation was produced using discrete LRP on a $1000\times 1000$ subset of $\mathds{Z}^2$.
}
\label{MN1}
\end{figure}

\begin{figure}[htbp]
\centering
\subfigure{\includegraphics[scale=0.3]{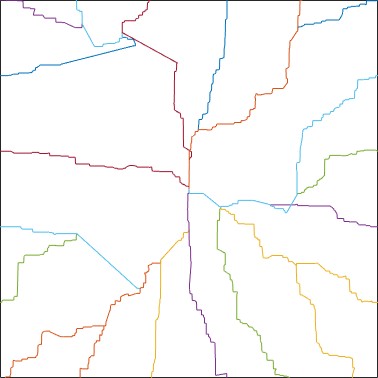}}
\quad
\subfigure{\includegraphics[scale=0.3]{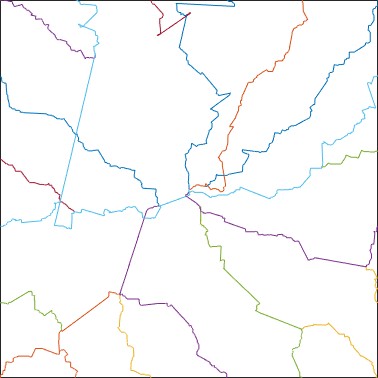}}
\quad
\subfigure{\includegraphics[scale=0.3]{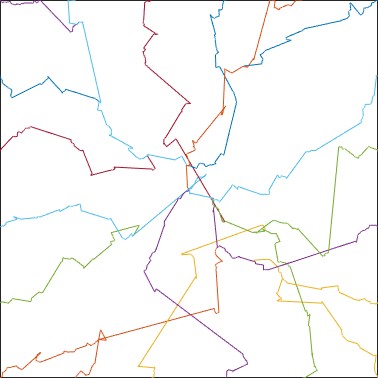}}
\quad
\subfigure{\includegraphics[scale=0.3]{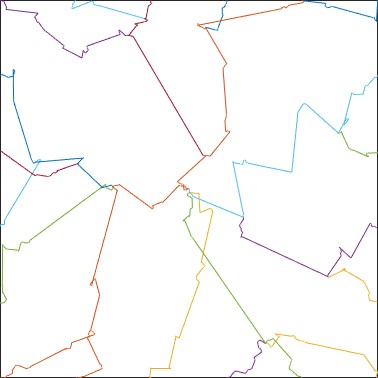}}
\quad
\subfigure{\includegraphics[scale=0.3]{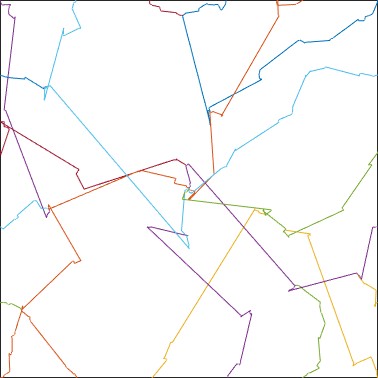}}
\quad
\subfigure{\includegraphics[scale=0.3]{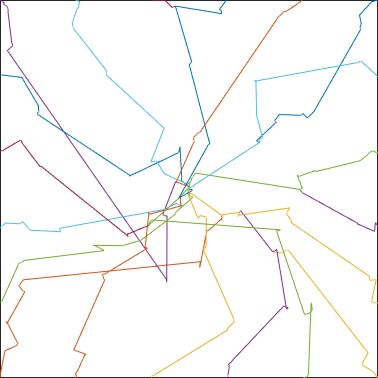}}
\caption{Simulation of $\beta$-LRP geodesics starting at ${\bm 0}$ (the center of the cube) in dimension 2 for $\beta=0.01$ (top left), $\beta=0.1$ (top middle), $\beta=0.5$ (top right), $\beta=1$ (bottom left), $\beta=2$ (bottom middle) and $\beta=5$ (bottom right). In the drawing, straight lines typically represent jumps via long edges. The simulation was produced using discrete LRP on a $1000\times 1000$ subset of $\mathds{Z}^2$.
}
\label{MN2}
\end{figure}

\begin{figure}[htbp]
\centering
\subfigure{\includegraphics[scale=0.2]{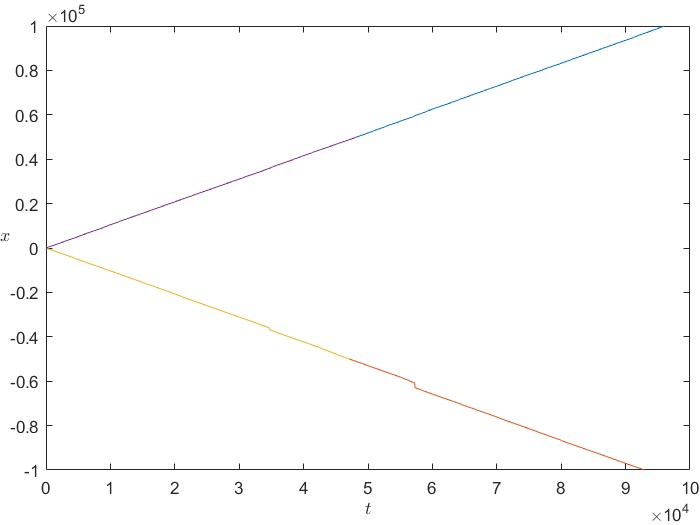}}
\quad
\subfigure{\includegraphics[scale=0.2]{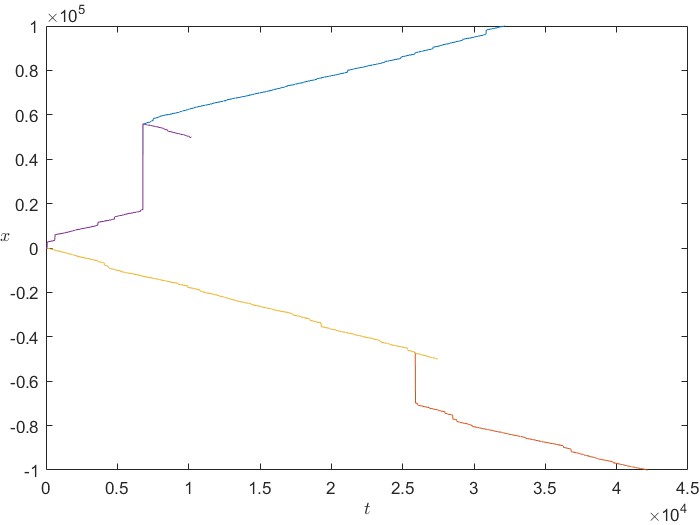}}
\quad
\subfigure{\includegraphics[scale=0.2]{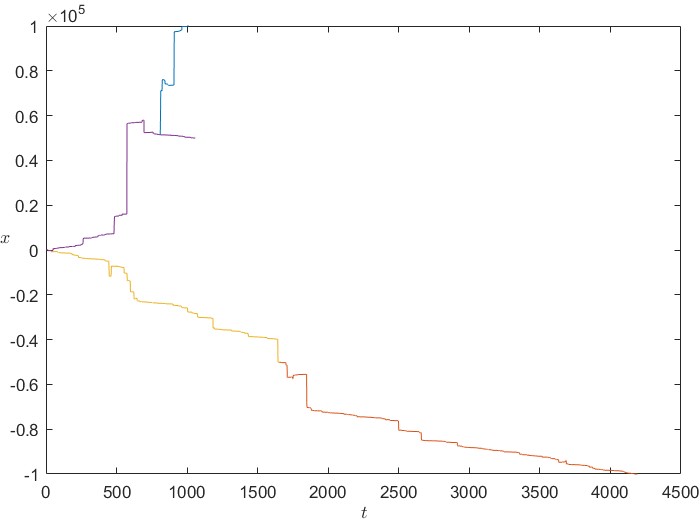}}
\quad
\subfigure{\includegraphics[scale=0.2]{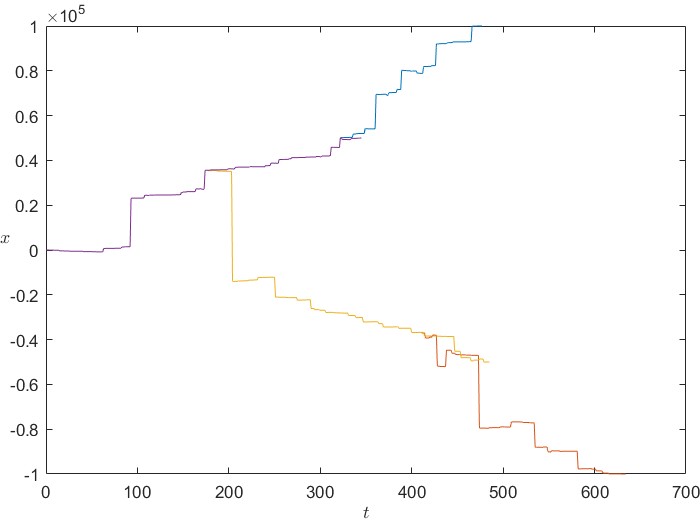}}
\quad
\subfigure{\includegraphics[scale=0.2]{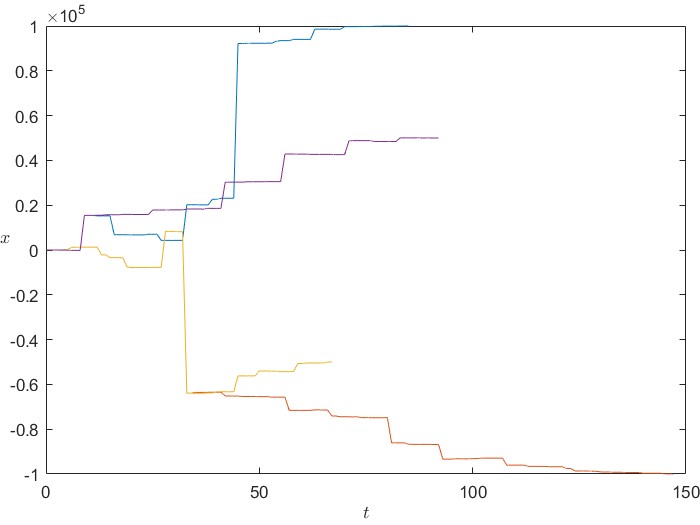}}
\quad
\subfigure{\includegraphics[scale=0.2]{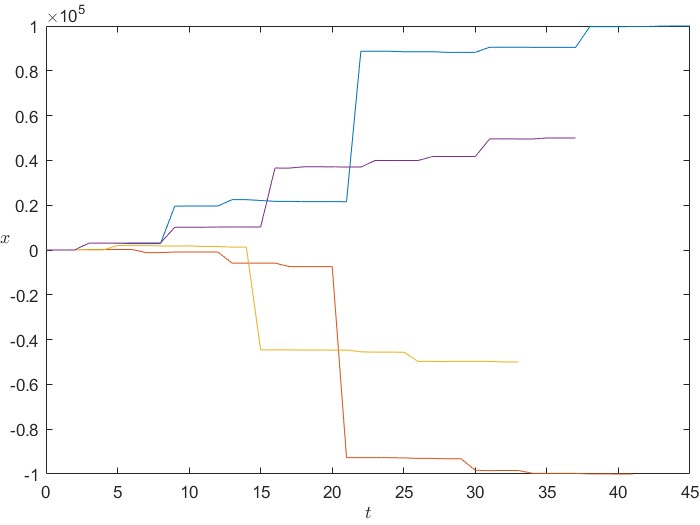}}
\caption{Simulation of $\beta$-LRP geodesics starting at 0 in dimension 1 for $\beta=0.01$ (top left), $\beta=0.1$ (top middle), $\beta=0.5$ (top right), $\beta=1$ (bottom left), $\beta=2$ (bottom middle) and $\beta=5$ (bottom right). In these plots, we denote the time $t$ of the geodesic in horizontal coordinates and the position $x$ of the geodesic at time $t$ in vertical coordinates. The simulation was produced using discrete LRP on $[-10^5,10^5]\cap\mathds{Z}$.
}
\label{MN3}
\end{figure}

\begin{figure}[htbp]
    \centering
    \subfigure{\includegraphics[scale=0.2]{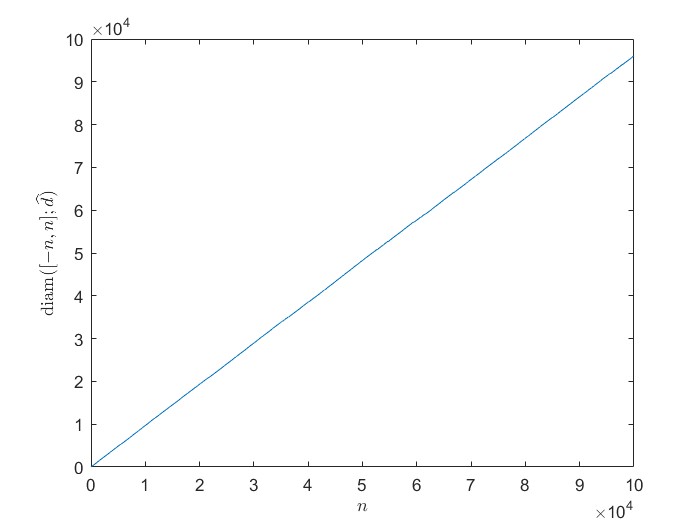}}
    \quad
    \subfigure{\includegraphics[scale=0.2]{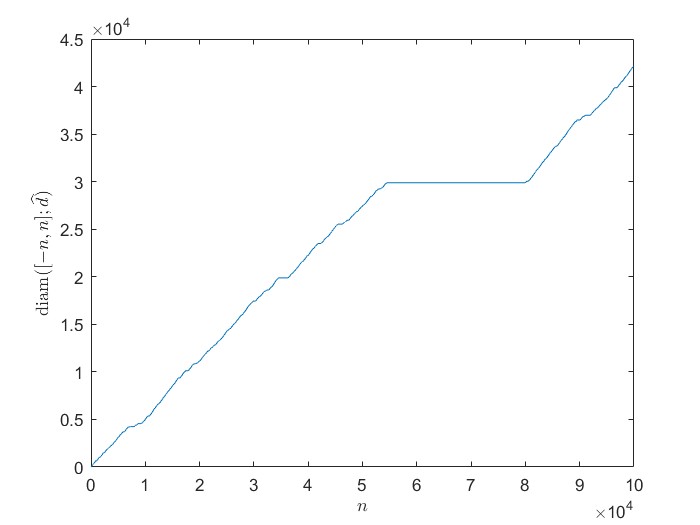}}
    \quad
    \subfigure{\includegraphics[scale=0.2]{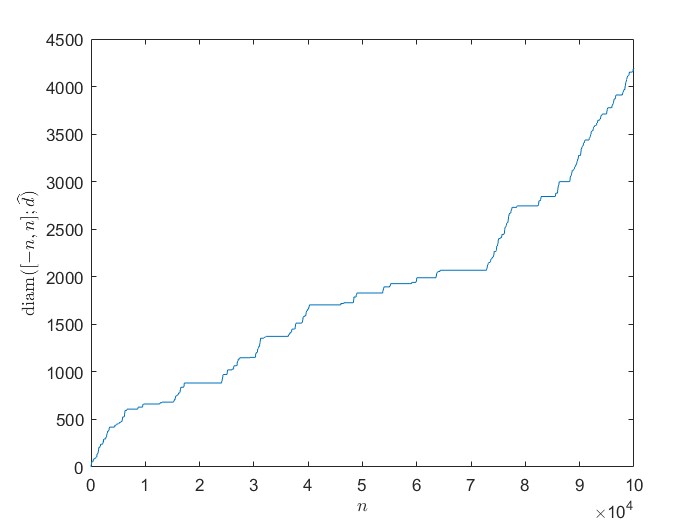}}
    \quad
    \subfigure{\includegraphics[scale=0.2]{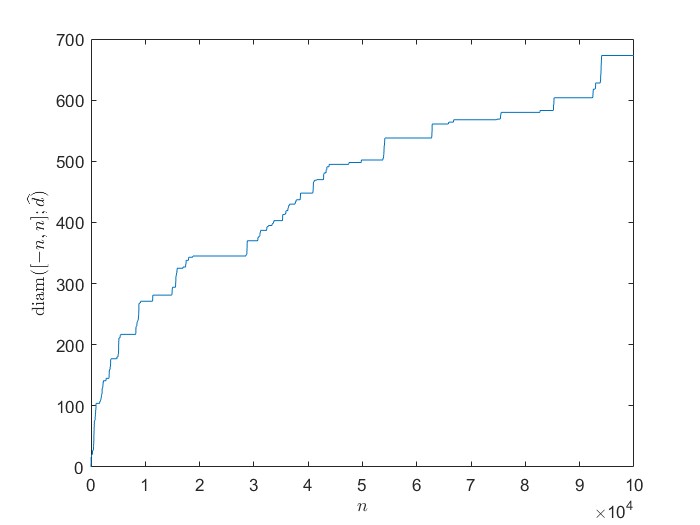}}
    \quad
    \subfigure{\includegraphics[scale=0.2]{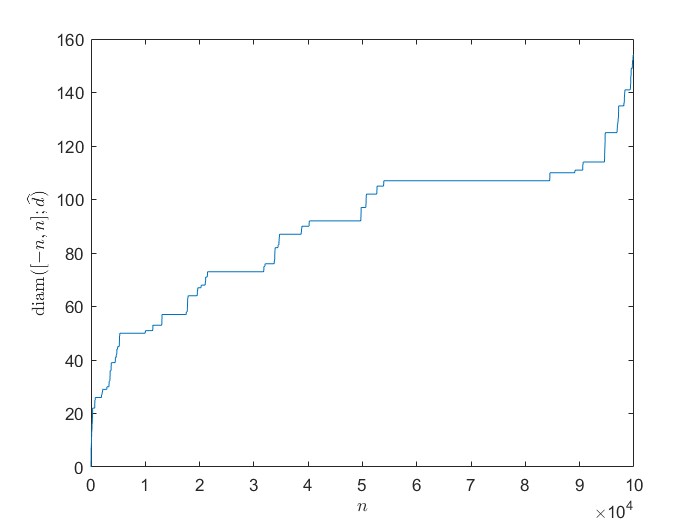}}
    \quad
    \subfigure{\includegraphics[scale=0.2]{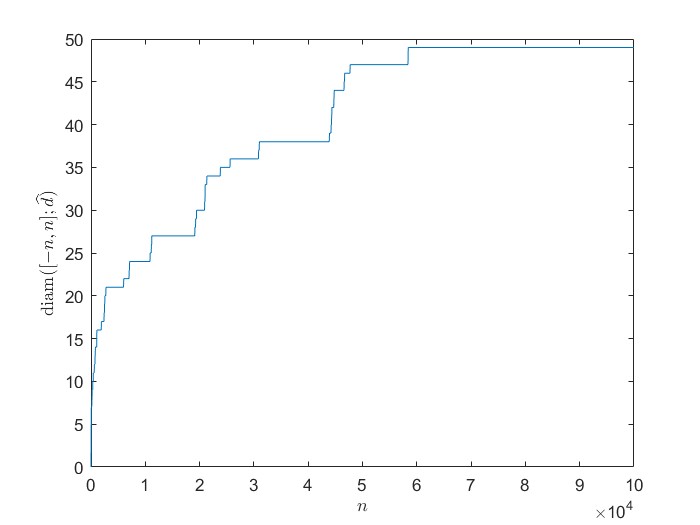}}
    \caption{Simulation of $\beta$-LRP distances in dimension 1 for $\beta=0.01$ (top left), $\beta=0.1$ (top middle), $\beta=0.5$ (top right), $\beta=1$ (bottom left), $\beta=2$ (bottom middle) and $\beta=5$ (bottom right).
In these figures, we show the growth of the diameter $[-n,n]$ with respect to $\widehat{d}$ as $n$ increases.  The simulation was produced using discrete LRP on $[-10^5,10^5]\cap\mathds{Z}$.
    }
    \label{MN4}
    \end{figure}

\subsection{Related work}\label{relatedwork}

Long-range percolation, introduced in \cite{S83,Zhangpuli1983}, is a percolation model on $\mathds{Z}^d$ where each pair of vertices can be connected with a bond.
To be more precise, consider the sequence $\{p_{\bm k}\}_{\bm k\in\mathds{Z}^d}$, where $p_{\bm k}\in [0,1]$ and $p_{\bm k}=p_{\bm k'}$ for all $\bm k,\bm k'\in\mathds{Z}^d$ such that $|\bm k|=|\bm k'|$. We also assume that
$$
0<\beta:=\lim_{|\bm k|\to \infty}\frac{p_{\bm k}}{|\bm k|^{-s}}<\infty
$$
for some $s>0$.
The long-range percolation model on $\mathds{Z}^d$ is defined by edges $\langle {\bm i},{\bm j}\rangle$ occuring independently with probability $p_{\bm i-\bm j}$. In particular, $p_{1}$ is the probability that two nearest neighbors are connected.
A natural and well-known question concerns the limiting properties of the (aforementioned) graph distance metric $\widehat{d}$.

It is known that the renormalization structure implies five different regimes, depending on whether $s$ is smaller, greater or equal to $d$ or $2d$.
Each regime exhibits distinct behavior as follows.
\begin{itemize}
\item For the case $s<d$, \cite{BKPS04} showed that the graph distance between two points is at most $\lceil d/(d-s) \rceil$.  Recently, \cite{H22} proved the precise estimates on the moments of all orders of the volume of the cluster of the origin inside a box, and then it applies these estimates to prove up-to-constants estimates on the tail of the volume of the cluster of the origin.

\item When $s=d$, it was shown in \cite{CGS02} that, with probability tending to 1, the chemical distance scales as $\log|{\bm x}-{\bm y}|/\log\log|{\bm x}- {\bm y}|$ as $|{\bm x}-{\bm y}|\to\infty$. Recently, \cite{W22} further showed that
    $$\text{diam}([0,N]_\mathds{Z}^d)\frac{\log\log N}{\log N}\to d$$
    as $N\to \infty$ in probability.

\item When $d<s<2d$, \cite{B04,B11} proved that $\widehat{d}({\bm x},{\bm y})$ grows like $(\log|{\bm x}- {\bm y}|)^{1/\log_2(2d/s)+o(1)}$. Lately, this result was improved to $\widehat{d}({\bm x},{\bm y})\asymp_P(\log|{\bm x}- {\bm y}|)^{1/\log_2(2d/s)}$ for the continuous model in \cite{BL19}, and for the discrete model in \cite{BK21}. Here $A_n\asymp_P B_n$ means that for all $\varepsilon>0$ there is a constant $C>0$ such that
$$
\mathds{P}[C^{-1}B_n\leq A_n\leq CB_n]\geq 1-\varepsilon\quad \text{for all }n\geq 1.
$$

\item For $s>2d$, \cite{BB01} showed that for $d=1$, $\widehat{d}({\bm x},{\bm y})$ grows linearly with Euclidean distance, while for $d\geq 2$, \cite{CGS02} established a lower bound  that $\widehat{d}({\bm x},{\bm y})\geq |{\bm x}-{\bm y}|^\eta$ for some $0<\eta<1$ depending on $s$ and $d$. Additionally, \cite{B18} further provided a similar lower bound with $\eta=1$.

\item In the critical case $s=2d$, \cite{CGS02} also proved that the typical diameter of a box grows at a polynomial rate (of the side length of the box) with some power strictly smaller than 1. Recently, \cite{Baumler22} proved that (see also \cite{Ding-Sly13} for the one-dimensional case) the typical distance between two points $\bm 0$ and $n\bm 1$ grows as $n^\theta$ for some $\theta\in(0,1)$, and the same holds for the diameter of $[0, n]^d_\mathds{Z}$.  In other words, they showed the existence of an exponent $\theta=\theta(\beta,d)\in(0,1)$ such that
$$
\widehat{d}(\bm 0, n\bm1)\asymp_P \text{diam}([0,n]_\mathds{Z}^d)\asymp_P n^\theta.
$$
Moreover, it was shown by \cite{B22b} that the exponent $\theta(\beta,d)$ is continuous and strictly decreasing as a function in $\beta$, and $\theta(\beta,1)=1-\beta+o(\beta)$ for small $\beta$ in dimension one.
\end{itemize}

There are also many other works on long-range percolation. For example, in the case of $p_1<1$  (recall that $p_1$ is the probability that two nearest neighbors are connected), regarding the important question of the existence and uniqueness of the infinite component in one dimension, \cite{S83} showed that percolation does not occur if $s>2$, while \cite{NS86} established the existence of oriented percolation if $s<2$.  Moreover, it was found that  non-oriented percolation occurs in the critical case $s=2$ if $\beta$ is sufficiently large and $p_1$ is close to 1. This result for $s=2$ was further improved by \cite{AN86} (see also \cite{DGT20}), which demonstrated the criticality of $\beta$. Specifically, it showed that percolation does not occur if $\beta\leq 1$, but non-oriented percolation does occur if $\beta>1$ and $p_1$ is close to 1. Afterwards, the oriented case for $s=2$ was solved by \cite{MSV10}, also proving that $\beta_c=1$.
In the high-dimensional case, the insertion of long-range connections is not essential for the existence of percolation. Instead, the main problem of interest lies in quantifying the effect of such connections on the critical behavior. For further details, refer to e.g. \cite{MS96,SSV99,Fd06}.
Another interesting direction in the study of long-range percolation is the behavior of the random walk on its infinite cluster. In \cite{GKZ93}, the authors showed that supercritical percolation in $\mathds{Z}^d$ is transient for all $d\geq 3$. For $d=1,2$, \cite{B02} proved the following: if $d<s<2d$, then the random walk is transient; if $s\geq 2d$, then the random walk is recurrent.
There are also numerous related results regarding the heat kernel (see e.g. \cite{CS12,KM08}) and the scaling limits (see e.g. \cite{CS13,CCK22}) of random walks on long-range percolation clusters.

\medskip

In this paper, our main focus lies in constructing the long-range percolation metric at criticality. While it may not be completely obvious, there is a striking similarity between the critical long-range percolation model and the Liouville quantum gravity (LQG) through their scaling invariant properties and their hierarchical structures. In fact, our proof method on the uniqueness is hugely inspired by recent work on the uniqueness of the LQG metrics (and perhaps drawing such a connection can be regarded as an interesting conceptual contribution that may shed light on future study on other aspects for the LRP model). In what follows, we briefly review the progress on the LQG metric.

Liouville quantum gravity is a one-parameter family of random fractal surfaces that was first studied by physicists in the 1980s as a class of canonical models of random two-dimensional Riemannian manifolds \cite{Pol81, Dav88, distler1989conformal}. 	
{\em Formally} speaking, a $\gamma$-Liouville quantum gravity ($\gamma$-LQG) surface is a random ``Riemannian manifold'' with Riemannian metric tensor $\e^{\gamma h} ds^2$ where $h$ is a variant of the {\em Gaussian free field} (GFF) on some domain $U \subset \mathbb C$ and $\gamma \in (0, 2]$ is the underlying parameter.
This is of course not well-defined since the GFF is a random Schwarz distribution (see, e.g.,~\cite{She07, berestycki2021gaussian, werner2020lecture} for a comprehensive introduction to the GFF). Rigorous mathematical investigation into this surface as a random metric measure space was set off with the construction of the associated volume measure in \cite{Kahane85} (in the more general setting of Gaussian multiplicative chaos) and in \cite{DupShe11} (which also established important fractal properties of this random volume measure); see~\cite{RhodesVargas14} for an excellent account on this subject.
	
On the metric side,  the $\sqrt{8/3}$-LQG metric (here $\sqrt{8/3}$ corresponds to the case of pure gravity) was constructed through a series of papers \cite{MSbrownianmapI,MSbrownianmapII,MSbrownianmapIII} which also established a deep connection with the {\em Brownian map} (which is the scaling limit of random planar maps, as shown by \cite{LeGall13,Miermont13}). For a general $\gamma \in (0, 2)$ (the so-called {\em subcritical phase}), the metric was recently constructed  as a culmination of several works \cite{DingDubDunFalc20, DubFalcGwynPfeSun20, GwynneMillerconcon2021, GM21, DD19, DF20, GM19c}. Their construction starts by producing candidate distance functions which are obtained as subsequential limits \cite{DingDubDunFalc20} of a family of random metrics
known as the {\em Liouville first passage percolation} (LFPP). The limiting metric is then shown to be unique
in two steps. Firstly, every possible subsequential limit is shown to be a measurable function of the GFF $h$ satisfying a list of axioms motivated by the natural properties and scaling behavior of the LFPP metrics \cite{DubFalcGwynPfeSun20}. Subsequently
these axioms are shown to uniquely characterize a random metric \cite{GwynneMillerconcon2021, GM21} on the plane. See \cite{ding2021introduction} for a detailed overview of the construction of LQG using LFPP.

Associated with the LFPP metric is a parameter $\xi = \xi(\gamma)$ which gives a reparametrization of the LQG metric although its explicit dependence on $\gamma$ is currently unknown (see~\cite{DingGwynne20, ding2021introduction}).
There is yet another parametrization of LQG which is perhaps more popular in the physics community, namely the {\em matter central charge} ${\bf c}_{{\rm M}}$. The subcritical phase (i.e., $\gamma \in (0, 2)$) corresponds to ${\bf c}_{{\rm M}} \in (-\infty, 1)$ whereas the critical ($\gamma = 2$) and the supercritical phases ($\gamma$ complex, $|\gamma| = 2$) correspond to ${\bf c}_{{\rm M}} = 1$ and ${\bf c}_{{\rm M}} \in (1, 25)$ respectively. See \cite{ding2021introduction} and \cite[Section~1]{DG23} for the interplay between these different parameters. Of these three phases, the supercritical phase is the most mysterious, not least because $\gamma$ is complex. However, one can still assign a LFPP parameter $\xi > 0$ to such $\gamma$ \cite{ding2020tightness} and consider subsequential limits of LFPP metrics as in the subcritical case. This program was carried out in a series of works \cite{ding2020tightness, pfeffer2021weak,DG23} and hence the LQG metric is now defined for all values of $\xi \in (0,\infty)$.

As mentioned earlier, there are similarities between the critical LRP model and LQG. In particular, the parameters associated with metrics of these two models seem to have some kind of correspondence.
Intuitively, the parameter $\beta$ associated with the critical LRP metric corresponds to the parameter $\xi=\xi(\gamma)$ (as defined in the previous paragraph) associated with the LFPP metric, while the LRP exponent $\theta=\theta(\beta)$ corresponds to the LFPP exponent $Q=Q(\xi)$ which is introduced by \cite{ding2020tightness}. It is known that $Q\in (0,\infty)$ for all $\xi>0$, and $Q$ is a continuous, non-increasing function of $\xi$, see e.g. \cite{DG19,GHS19,GP19,ding2020tightness,DGS21}. However, the exact relationship between $Q$ and $\xi$ is still  unknown, except for the special case of pure gravity. Regarding to the LRP model, \cite{B22b} proves that the function $\beta\to \theta(\beta)$ is continuous and strictly decreasing. In particular, as we presented before, it shows that $\theta(\beta)=1-\beta+o(\beta)$ for small $\beta$ in dimension one. It is important to note that the exact relationship between $\theta(\beta)$ and $\beta$ is still an open question. We are under the feeling that, computing the exact value of $\theta(\beta)$ is a question of major challenge, as the question for computing $Q(\xi)$ in the LQG model.

Our approach in proving uniqueness for the critical LRP metrics is hugely inspired by the method developed for the LQG metric. We will formulate a natural list of axioms (see Definition \ref{strongLRPmetric}), which as one can verify are satisfied by any subsequential limit of the LRP metric. The main challenge of this work is then to prove that this list of axioms fully characterizes the LRP metric, i.e., to prove the uniqueness. To this end, we follow in the overview level the framework initiated in \cite{GM21} and further developed in \cite{DG23}. In the implementation, we encounter substantial challenges for the LRP model since for instance a LRP geodesic is not a continuous path with respect to the Euclidean topology and the LRP metric does not enjoy the Weyl scaling as the LQG metric.
That being said, we also have some advantages on the LRP model since for instance we have more user-friendly independence for LRP and also in LRP the long-range edges can help us attracting geodesics to desirable locations. In addition, regardless of $\beta$, the LRP model behaves more or less like the subcritical LQG metric, in the sense that one does not see ``singularities'' as observed in the supercritical LQG metric \cite{ding2020tightness} which is a major technical challenge in \cite{DG23}.
Furthermore, some a priori estimates for the LRP metric are already available. For instance, the power law growth was proved in \cite{Baumler22} while for the LQG metric an up to poly-log estimate needed to be  proved separately \cite{DingGwynne2023} before establishing the uniqueness in the supercritical regime.




\subsection{Continuous LRP}\label{overview}

A uesful tool for studying long-range percolation is a continuous analog of LRP on $\mathds{R}^d$, which is also of interest on its own right.
For $\beta > 0$, the continuous critical long-range percolation model with parameter $\beta$ (continuous $\beta$-LRP) is a percolation model on $\mathds{R}^d$, where the set of edges $\mathcal{E}$ is given by a Poisson point process where edges between ${\bm x}$ and ${\bm y}$ occur with intensity $\beta/|{\bm x}-{\bm y}|^{2d}$ with respect to the Lebesgue measure on $\mathds{R}^{2d}$.
As above, we use the notation $\langle \cdot,\cdot\rangle$ to denote a long edge, and for a long edge $\langle {\bm x}, {\bm y}\rangle $, we say its scope is $|{\bm x}-{\bm y}|$.

We can similarly consider the metric structure for the model.
For $\delta'>\delta>0$, we define $d_{(\delta,\delta')}(\cdot,\cdot)$ as the graph distance which only uses long edges whose scopes are in the interval $(\delta,\delta')$, denoted by $\mathcal{E}_{(\delta,\delta')}$. That is,
\begin{equation}\label{dist_delta}
\begin{split}
&d_{(\delta,\delta')}({\bm x},{\bm y})\\
&=\min_{m\geq 0}\left\{|{\bm x}-{\bm u}_1|+\sum_{i=1}^{m-1}|{\bm v}_i-{\bm u}_{i+1}|+|{\bm v}_m-{\bm y}|:\ \{\langle {\bm u}_i,{\bm v}_i\rangle\}_{1\leq i\leq m}\subseteq \mathcal{E}_{(\delta,\delta')}\right\},
\end{split}
\end{equation}
where the case of $m=0$ evaluates to $|{\bm x}-{\bm y}|$. From the above definition, we see that long edges with scopes in $(\delta, \delta')$ serve as ``express ways'' and have length 0.
It is clear that the resulting graph is connected, and for any ${\bm x},{\bm y}\in\mathds{R}^d$, the graph distance $d_{(\delta,\delta')}({\bm x},{\bm y})$ between them satisfies $d_{(\delta,\delta')}({\bm x},{\bm y})\leq |{\bm x}-{\bm y}|$.

We will also study the scaling limit of the graph distance metric for this model. 


\subsection{Axiomatic characterization of strong $\beta$-LRP metric}\label{strongmetric}

The existence of subsequential limits more or less follows from \cite{Baumler22}, and the key challenge in this paper is the uniqueness of the limit. To this end, we will follow the framework on the random metrics for Liouville quantum gravity \cite{GM21, DG23}. That is, our proof proceeds by formulating a set of axioms for the LRP metric: on the one hand, we will show that all subsequential scaling limits of the discrete metrics satisfy these axioms; on the other hand, we will show that these axioms together uniquely determines the law of the metric.

To state our axioms precisely, we will need some preliminary definitions concerning metric spaces.

\begin{definition}
Let $(X,d)$ be a metric space.
\begin{itemize}

\item Let $P:[a,b]\to (X,d)$ be a curve that is continuous with respect to the topology generated by $d$. The \textit{$d$-length} of $P$ is defined by
$$
{\rm len}(P;d)=\sup_{T}\sum_{i=1}^{\#T}d(P(t_{i-1}),P(t_i)),
$$
where the supremum is over all partitions $T:a=t_0<\cdots<t_{\#T}=b$ of $[a,b]$ (here $\# T$ denotes the cardinality of $T$).

\item We say $(X,d)$ is a \textit{length space} if for each $x,y\in X$ and each $\varepsilon>0$, there exists a curve of $d$-length at most $d(x,y)+\varepsilon$ from $x$ to $y$. A curve from $x$ to $y$ of $d$-length exactly $d(x,y)$ is called a geodesic.

\item For $Y\subset X$, the internal metric of $d$ restricted on $Y$ is defined by
$$
d(x,y;Y)=\inf_{P\subset Y}{\rm len}(P;d)\quad \forall x,y\in Y,
$$
where the infimum is over all paths $P$ in $Y$ from $x$ to $y$. Note that $d(\cdot,\cdot;Y)$ is a metric on $Y$.

\end{itemize}
\end{definition}

The axioms which characterize our metric are given in the following definition.

\begin{definition}\label{strongLRPmetric} (strong $\beta$-LRP metric).
Let $\mathcal{D}'$ be the space of measures
$$\left\{\sum_{n=1}^N\delta_{({\bm x}_n,{\bm y}_n)}:\ \{({\bm x}_n,{\bm y}_n)\}_{n=1}^N\subset \mathds{R}^{2d}\ \text{contains no repeated elements and } 0\le N\le\infty\right\}
$$
with the usual weak topology. It is worth emphasizing that here we allow $N$ to take infinity value.
For $\beta>0$, a (strong) $\beta$-LRP metric is a measurable function $\mathcal{E}\to D$ from $\mathcal{D}'$ to the space of continuous pseudometrics on $\mathds{R}^d$ with local uniform topology satisfying the following properties. Let $\mathcal{E}$ be an edge set given by the set of atoms of a Poisson point process where edges between ${\bm x}$ and ${\bm y}$ occur with intensity $\beta/|{\bm x}-{\bm y}|^{2d}$ with respect to the Lebesgue measure on $\mathds{R}^{2d}$ for the parameter $\beta$.
Then the associated metric $D$ satisfies the following axioms.

    \begin{enumerate}
        \item[I.] {\bf Length space.} Consider the equivalence relationship $\sim$ defined by  ${\bm x}\sim {\bm y}$ if and only if $D({\bm x},{\bm y})=0$. Then the pair $(\mathds{R}^d/\sim,D)$ forms a length space. Additionally, there exists a geodesic connecting any two elements ${\bm x}$ and ${\bm y}$ in $\mathds{R}^d$.

        \item[II.] {\bf Locality.} Let $U$ be an open and deterministic subset of $\mathds{R}^d$. The $D$-internal metric $D(\cdot,\cdot;U)$ is determined by $\mathcal{E}|_{U\times U}$. It is  independent of all edges with at least one end point outside of $U$.

        \item[III.] {\bf Regularity.} For any ${\bm x},{\bm y}\in\mathds{R}^d$, $D({\bm x},{\bm y})=0$ if $\langle {\bm x},{\bm y}\rangle \in\mathcal{E}$.

        \item[IV.] {\bf Translation and scale covariance.} For any $r>0,{\bm z}\in\mathds{R}^d$,
        $$
        D(r\cdot+{\bm z},r\cdot+{\bm z})\stackrel{\rm law}{=}r^\theta D(\cdot,\cdot).
        $$

        \item[V.] {\bf Tightness.} $D({\bm 0},([-1,1]^d)^c)>0$ almost surely and $\mathds{E}\left[\exp\left\{({\rm diam}([0,1]^d;D))^\eta\right\}\right]<\infty$ for any $\eta\in(0,1/(1-\theta))$.

    \end{enumerate}
\end{definition}

Our main results establish the existence of subsequential limiting metrics and demonstrate that any such metric is a (unique) strong $\beta$-LRP metric.

\begin{theorem}\label{dis-maintheorem}
    Let $\widehat{d}$ be the chemical distance on the discrete $\beta$-LRP model. Let $\widehat{a}_n$ be the median of $ \widehat{d}({\bm 0},n{\bm 1}) $ and $\widehat{D}_n(\cdot,\cdot)=\widehat{a}_n^{-1}\widehat{d}(\lfloor n\cdot\rfloor,\lfloor n\cdot\rfloor)$.
    Then $\widehat{D}_n$ converges to a strong $\beta$-LRP metric with respect to the topology of local uniform convergence on $\mathds{R}^{2d}$.
\end{theorem}

Note that Theorem \ref{dis-maintheorem} implies Theorem \ref{realmaintheorem} and also gives a description of the scaling limit. We emphasize that the scaling limit is a strong $\beta$-LRP metric on $\mathds{R}^{2d}$ obtained by constructing a coupling between the discrete model and the continuous model (see the first paragraph in Section \ref{dis-existence}).
We also have an analog of Theorem \ref{dis-maintheorem} for the continuous model.

\begin{theorem}\label{maintheorem}
    Let $a_n$ be the median of $d_{(1/n,\infty)}({\bm 0},{\bm 1})$ and $D_n=a_n^{-1}d_{(1/n,+\infty)}$. Then $D_n$ converges to a strong $\beta$-LRP metric with respect to the local uniform topology of $C(\mathds{R}^{2d})$.
\end{theorem}

Furthermore, we will show the typical growth speed of $\widehat{a}_n$ and $a_n$ as $n\to\infty$.

\begin{lemma}\label{an-bounded}
    $\left\{n^{1-\theta}a_n \right\}_{n\in\mathds{N}}$, $\left\{n^{\theta-1}a_n^{-1}\right\}_{n\in\mathds{N}}$,  $\left\{n^{-\theta}\widehat{a}_n \right\}_{n\in\mathds{N}}$ and $\left\{n^{\theta}\widehat{a}_n^{-1}\right\}_{n\in\mathds{N}}$ are all uniformly bounded.
\end{lemma}

Note that part of the lemma has been proved in \cite{Baumler22} (also see \cite{Ding-Sly13}). Lemma \ref{an-bounded} will be completely proved in Section \ref{section2} for general $d$, with crucial input from \cite{Baumler22}.

\subsection{Weak $\beta$-LRP metrics}\label{sectionWeakLRP}
Every possible subsequential limiting metric of $\{\widehat{D}_n\}$(resp. $\{D_n\}$) satisfies Axioms I, II and III. This is intuitively clear from the definition, and not too hard to check rigorously (see Theorems \ref{dis-subsequence-exist}-\ref{subsequence-exist}). It is also easy to see that every possible subsequential limit of $\{\widehat{D}_n\}$(resp. $\{D_n\}$) satisfies Axiom IV for $r = 1$ (i.e., it satisfies the coordinate change formula for translations). However, it is worth emphasizing that even if we have the scaling invariance of the Poisson point process in the continuous model, it is far from obvious that the subsequential limits satisfy Axiom IV when $r \neq 1$. The reason is that in the continuous model re-scaling the space changes the value of $\delta=n^{-1}$ in \eqref{dist_delta} (taking $\delta'=\infty$): for $n\in\mathds{N}, r > 0$, one has
\begin{equation*}
    D_n(r\bm z,r\bm w)\stackrel{\rm law}{=}(nr)^{1-\theta}d_{(1/(nr),\infty)}(\bm z,\bm w)
\end{equation*}
from the scaling invariance of the Poisson 
point process 
for edges. So, since we only have subsequential limits of $\{D_n\}$, we cannot directly deduce that the subsequential limit satisfies an exact spatial scaling property.

Because of the above issue, we do not know how to check Axiom IV for subsequential limits of $\{\widehat{D}_n\}$(resp. $\{D_n\}$) directly. Instead, we will prove a stronger uniqueness statement than the ones in Theorems \ref{dis-maintheorem} and \ref{maintheorem}, under a weaker list of axioms which can be checked for subsequential limits of $\{\widehat{D}_n\}$(resp. $\{D_n\}$) (Theorems \ref{dis-subsequence-exist} and \ref{subsequence-exist}). We will then deduce from this stronger uniqueness statement that the weaker list of axioms implies the axioms in Definition \ref{strongmetric} (Lemma \ref{weak=>strong}).

We now present the definition of a weaker version of $\beta$-LRP metric.
\begin{definition}\label{weakmetric} (weak $\beta$-LRP metric).
Let $\mathcal{D}'$ and $\mathcal{E}$ be as in Definition \ref{strongLRPmetric}. For $\beta>0$, a weak $\beta$-LRP metric with parameter $\beta$ is a measurable function $\mathcal{E}\to D$ from $\mathcal{D}'$ to the space of continuous pseudometrics on $\mathds{R}^d$ with local uniform topology satisfying Axioms I, II, III from Definition \ref{strongLRPmetric} and the following additional axioms for some constant $\theta=\theta(\beta,d)>0$.
    \begin{enumerate}
        \item[IV'] {\bf Translation invariance.}  For any ${\bm z}\in\mathds{R}^d$,
        $$
        D(\cdot+{\bm z},\cdot+{\bm z})\stackrel{\rm law}{=} D(\cdot,\cdot).
        $$
        \item[V1'] {\bf Tightness across different scales (lower bound).} $\{\frac{r^\theta}{D({\bm 0},([-r,r]^d)^c)}\}_{r>0}$ is tight.
        \item[V2'] {\bf Tightness across different scales (upper bound).} $$\sup_{r>0}\mathds{E}\left[\exp\left\{\left(\frac{{\rm diam}([0,r]^d;D)}{r^\theta}\right)^\eta\right\}\right]<\infty$$ for any $\eta\in(0,1/(1-\theta))$.
    \end{enumerate}
\end{definition}

It is worth emphasizing that the biggest difference from the axiomatization of the LQG metric is that here we do not have the Weyl scaling property (which describes exactly how the LQG metric will change if we add a smooth function to the underlying Gaussian free field).
However, in the long-range percolation model, Axioms I and III above serve as a replacement for the Weyl scaling property: from these axioms, it can be inferred that if we add more edges, the distance will be reduced. In addition, our underlying logic for attracting geodesics is also a bit different from that used in \cite{GM21,DG23} for the LQG metric which takes advantage of the Weyl scaling property; see the paragraph before Proposition \ref{simpleprop4.3} for more discussions.

To prove Theorem \ref{dis-maintheorem}, we first demonstrate that $\{\widehat{D}_n\}$ is tight and that any subsequential limiting metric is a weak $\beta$-LRP metric as follows. See Section \ref{dis-existence} for its proof.

\begin{theorem}\label{dis-subsequence-exist}
    $\{\widehat{D}_n\}$ is tight with respect to the local uniform convergence topology. Furthermore, for every sequence of $n$'s tending to infinity, there is a weak $\beta$-LRP metric $D$ and a subsequence $\{n_k\}$ along which $\widehat{D}_{n_k}$ converges in law to $D$.
\end{theorem}

Similarly in the continuous model, to prove Theorem \ref{maintheorem}, we first demonstrate that $\{D_n\}$ is tight and that any subsequential limiting metric is a weak $\beta$-LRP metric as follows. See Section \ref{existence} for its proof.

\begin{theorem}\label{subsequence-exist}
    $\{D_n\}$ is tight with respect to the local uniform topology of $C(\mathds{R}^{2d})$. Furthermore, for every sequence of $n$'s tending to infinity, there is a weak $\beta$-LRP metric $D$ and a subsequence $\{n_k\}$ along which $D_{n_k}$ converges in probability to $D$.
\end{theorem}

Additionally, most of this paper (Sections \ref{bi-lipschitz}, \ref{section3} and \ref{section4}) is devoted to the proof of the uniqueness of the weak $\beta$-LRP metric, which is also the major challenge addressed in this work.

\begin{theorem}\label{uniqueness}
    For any two weak $\beta$-LRP metrics $D$ and $\widetilde{D}$, there exists a deterministic constant $C$ such that a.s.
    $$
    \widetilde{D}=CD.
    $$
\end{theorem}

Let us now explain why Theorems \ref{dis-subsequence-exist}, \ref{subsequence-exist} and \ref{uniqueness} are sufficient to establish our main result, Theorems \ref{dis-maintheorem} and \ref{maintheorem}. We first observe that every strong $\beta$-LRP metric is a weak $\beta$-LRP metric.
\begin{lemma}\label{strong=>weak}
    Every strong $\beta$-LRP metric {\rm(}Definition \ref{strongmetric}{\rm)} is a weak $\beta$-LRP metric (Definition \ref{weakmetric}).
\end{lemma}
\begin{proof}
    Let $D$ be a strong $\beta$-LRP metric, it is clear that $D$ satisfies Axioms I, II, III and IV'. To demonstrate that $D$ is a weak $\beta$-LRP metric, we need to check Axioms  V1' and V2' of Definition \ref{weakmetric}. Specifically, from Axiom V's assertion that $0<D({\bm0},([-1,1]^d)^c)<\infty$ a.s., together with $\frac{D({\bm 0},([-r,r]^d)^c)}{r^\theta}\stackrel{\rm law}{=}D({\bm 0},([-1,1]^d)^c)$, we readily obtain Axiom V1'. Moreover, Axiom IV (with ${\bm z} = \bm 0$) and Axiom V together directly imply Axiom V2'.
\end{proof}

Theorem \ref{uniqueness} implies that one also has the converse to Lemma \ref{strong=>weak}.
\begin{lemma}\label{weak=>strong}
    Every weak $\beta$-LRP metric $D$ is also a strong $\beta$-LRP metric.
\end{lemma}
\begin{proof}[Proof assuming Theorem \ref{uniqueness}]
    This proof is similar to that of \cite[Lemma 1.10]{GM21}. Let $D$ be a weak $\beta$-LRP metric. It is clear that $D$ satisfies Axioms I, II, III and V of Definition \ref{strongmetric}. To show that $D$ is a strong $\beta$-LRP metric, we need to check Axiom IV of Definition \ref{strongmetric} in the case when ${\bm z} = {\bm0}$ (note that we already have translation invariance from Definition \ref{weakmetric}), i.e. for any $r>0$,
    $$
    r^{-\theta}D(r\cdot,r\cdot)\stackrel{\rm law}{=}D(\cdot,\cdot).
    $$
    To this end, we define a metric $D^{(r)}$ for $r>0$ using the same law as $D$ from the edge set $r\mathcal{E}:=\{\langle r\bm x,r\bm y\rangle:\langle\bm x,\bm y\rangle\in\mathcal{E}\}$. It is clear that $D^{(r)}(r\cdot,r\cdot)$ is also a weak $\beta$-LRP metric, since $D^{(r)}$ has the same law as $D$ thanks to the scaling invariance of the Poisson point process for edges as well as the fact that $D^{(r)}(r\bm x,r\bm y)=0$ for any $\langle\bm x,\bm y\rangle\in\mathcal{E}$. Let $d^{(r)}(\cdot,\cdot)=D^{(r)}(r\cdot,r\cdot)$. Hence, by Theorem \ref{uniqueness}, there exists $C(r)>0$ such that
    \begin{equation*}
        d^{(r)}(\cdot,\cdot)=C(r) D(\cdot,\cdot).
    \end{equation*}
    It suffices to show that $C(r)=r^\theta$ for $r>0$. First, note that $d^{(r_1r_2)}=(d^{(r_1)})^{(r_2)}$, and hence $$C(r_1r_2)=C(r_1)C(r_2).$$
    In addition, because $d^{(r)}$ has the same law as $D(r\cdot,r\cdot)$, it follows that $C(r)^{-1}D({\bm 0},([-r,r]^d)^c)$ has the same law as $D({\bm0},([-1,1]^d)^c)$. Thus the families
    $$\{C(r)^{-1}D({\bm0},([-r,r]^d)^c)\}_{r>0}\text{ and }\{C(r)D({\bm0},([-r,r]^d)^c)^{-1}\}_{r>0}$$
     are both tight (noting that $0<D({\bm0},([-1,1]^d)^c)<\infty$ almost surely). By comparing these families to Axioms V1' and V2', we conclude that $\frac{C(r)}{r^\theta}$ are bounded uniformly (both from above and from below). Combining this with the fact that $C(r_1r_2)=C(r_1)C(r_2)$, we obtain $C(r)=r^\theta$ for each value of $r>0$.
\end{proof}

We now present the
\begin{proof}[Proof of Theorems \ref{realmaintheorem} and \ref{dis-maintheorem} assuming Theorems \ref{dis-subsequence-exist} and \ref{uniqueness}]
    Our first step is to show that for any two subsequential limiting metrics $D$ and $\widetilde{D}$, they are equal to each other. By Theorem \ref{dis-subsequence-exist}, these metrics are both weak $\beta$-LRP metrics. Consequently, from Theorem \ref{uniqueness}, there exists a deterministic constant $C>0$ such that $\widetilde{D}=CD$. Moreover, from the definition of $\widehat{a}_n$, the median of $\widehat{D}_n({\bm0},n{\bm 1})=\widehat{a}_n^{-1}\widehat{d}({\bm0},n{\bm 1})$ is always equal to 1 for all $n\in\mathds{N}$. This property holds for the subsequential limit as well, which means that the medians of $D({\bm0},{\bm1})$ and $\widetilde{D}({\bm0},{\bm1})$ are both 1. Thus $C=1$ and $\widetilde{D}=D$.

    Henceforth, combining this result with Theorem \ref{dis-subsequence-exist}, we conclude that $\widehat{D}_n$ converges to a weak $\beta$-LRP metric with respect to the local uniform convergence on $\mathds{R}^{2d}$. Additionally, Lemmas \ref{strong=>weak} and \ref{weak=>strong} show that this limit is a strong $\beta$-LRP metric.
\end{proof}

We can also prove Theorem \ref{maintheorem} assuming Theorems \ref{subsequence-exist} and \ref{uniqueness} under the same proof as Theorem \ref{dis-maintheorem} by replacing Theorem \ref{dis-subsequence-exist} with Theorem \ref{subsequence-exist}.


Notably, we can describe the continuity of the weak $\beta$-LRP metric $D$ using Axioms IV' and V2' along with Lemma \ref{HolderLemma}.
\begin{proposition}\label{LimitContinuousTail}
    For any weak $\beta$-LRP metric $D$, there exists a positive constant $C_D$ depending only on $d$ and the law of $D$ such that for any $R>0$ and  $M>2^{\theta+d+1}$,
    \begin{equation*}
        \mathds{P}\left[\sup_{{\bm x},{\bm x}'\in[-R,R]^d}\frac{D({\bm x},{\bm x}';[-R,R]^d)}{\|{\bm x}-{\bm x}'\|_\infty^\theta(\log\frac{4R}{\|{\bm x}-{\bm x}'\|_\infty})}>M\right]\le 2^{d+1}C_D \exp\{-2^{-\theta-1}M\}.
    \end{equation*}
\end{proposition}
\begin{proof}
    The statement can be immediately inferred from Axioms IV', V2' along with Lemma \ref{HolderLemma}, which is discussed in Section \ref{dis-existence} below.
\end{proof}

Additionally, we can further strengthen Axiom I by proving the following proposition.
\begin{proposition}\label{Dxy=0}
    Let $D$ be a {\rm(}weak or strong{\rm)} $\beta$-LRP metric $D$. Then almost surely, for any ${\bm x},{\bm y}\in\mathds{R}^d$, $D({\bm x},{\bm y})=0$ if and only if $\langle {\bm x},{\bm y}\rangle \in\mathcal{E}$. Moreover, Axiom I can be strengthened by the following property. Let $\sim$ be the equivalence relation that ${\bm x}\sim {\bm y}$ if and only if $\langle{\bm x},{\bm y}\rangle\in\mathcal{E}$. Then $(\mathds{R}^d/\sim,D)$ is a length space. Furthermore, for any ${\bm x},{\bm y}\in\mathds{R}^d$, there exists a geodesic from ${\bm x}$ to ${\bm y}$.
\end{proposition}

See Section \ref{proof111} for the proof of Proposition \ref{Dxy=0}. Furthermore, we can show that the parameter $\theta$ in Definitions \ref{strongLRPmetric} of strong $\beta$-LRP metric and \ref{weakmetric} of weak $\beta$-LRP metric is uniquely determined by $d$ and $\beta$.

\begin{theorem}\label{uniquetheta-1}
Fix $\beta>0$. Let $\theta_1>0$ be such that there exists a weak $\beta$-LRP metric with $\theta=\theta_1$. Then we have that $\theta_1=\theta_0$, where $\theta_0$ is chosen in Theorem \ref{discrete-dist}  and depends only on $\beta$ and $d$.
\end{theorem}

See Section \ref{uniquetheta} for the proof of Theorem \ref{uniquetheta-1}. In the rest of this subsection, we consider the Hausdorff dimension of $\mathds{R}^d$ endowed with the strong (weak) $\beta$-LRP metric. To this end, we will first introduce some notations. Let $(Y,D)$ be a  metric space. For $\Delta > 0$, the $\Delta$-Hausdorff content of $(Y, D)$ is the number
\begin{equation}\label{Hausdorff}
    C_\Delta(Y, D) := \inf\left\{ \sum_{j=1}^\infty r_j^\Delta:\text{ there is a cover of $Y$ by $D$-balls of radii }\{r_j\}_{j\geq 1} \right\}
\end{equation}
and the \textit{Hausdorff dimension} of $(Y, D)$ is defined by $\inf\{\Delta > 0 : C_\Delta(Y, D) = 0\}$.

In the following, we take $Y=\mathds{R}^d$. For a set $X\subset \mathds{R}^d$, let ${\rm dim}_{\mathcal{H}}^0 X$ and ${\rm dim}_{\mathcal{H}}^\beta X$ denote the
Hausdorff dimensions of the metric spaces $(X,|\cdot|)$ and $(X,D)$ respectively, where $D$ is a strong (weak) $\beta$-LRP metric. We will refer to these two quantities as the \textit{Euclidean dimension} and \textit{$\beta$-LRP dimension} of the set $X$, respectively. Note that the $\beta$-LRP dimension depends on the choice of parameters $\beta$ and $d$.

We provide an upper bound of ${\rm dim}_{\mathcal{H}}^\beta X$ for any $X\subset \mathds{R}^d$ as follows.
\begin{theorem}\label{theorem-Hausdorff}
    Let $X$ be a deterministic Borel subset of $\mathds{R}^d$. Then  we have a.s.
    \begin{equation}\label{uppbhasd}
        {\rm dim}_{\mathcal{H}}^\beta X\leq \theta^{-1}{\rm dim}_{\mathcal{H}}^0 X.
    \end{equation}
    Here $\theta$ is the constant chosen in Theorem \ref{discrete-dist}.
\end{theorem}

The proof of Theorem \ref{theorem-Hausdorff} will be presented in Section \ref{Hasd}.
\begin{remark}
We expect the equality in \eqref{uppbhasd} to hold, but the difficulty lies in the fact that we do not have a good lower tail for $D(\bm 0,\bm 1)$. This is the reason why we can not adapt in the straightforward manner the techniques in \cite[Section 2.3]{GP22} for the lower bound on the Hausdorff dimension for the LQG metric.
\end{remark}


\subsection{Notational conventions}
We write $\mathds{N}=\{1,2,3,\cdots\}$. For $a<b$, we define $[a,b]_\mathds{Z}=[a,b]\cap\mathds{Z}$. We write ${\bm 1}=(1,1,\cdots,1)\in\mathds{R}^d$. In the rest of the paper, we use the boldface font (e.g. ${\bm x,\bm y,\bm z,\bm u,\bm v,\bm w,\bm i,\bm j,\bm k,\bm l}$) to represent vectors.

If $f:(0,\infty)\to \mathds{R}$ and $g:(0,\infty)\to (0,\infty)$, we say that $f(\varepsilon)=O_\varepsilon(g(\varepsilon))$ (resp. $f(\varepsilon)=o_\varepsilon(g(\varepsilon))$) as $\varepsilon\to0$ if $f(\varepsilon)/g(\varepsilon)$ remains bounded (resp. tends to 0) as $\varepsilon\to 0$. We similarly define $O(\cdot)$ and $o(\cdot)$ errors as a parameter goes to infinity.

Let $\{E_\varepsilon\}_{\varepsilon>0}$ be a non-parameter family of events. We say that $E_\varepsilon$ occurs with
\begin{itemize}
\item \textit{polynomially high probability} as $\varepsilon\to 0$ if there is a $\mu$ (independent from $\varepsilon$) such that $\mathds{P}[E_\varepsilon]\geq 1-O_\varepsilon(\varepsilon^\mu)$.

\item \textit{superpolynomially high probability} as $\varepsilon\to 0$ if  $\mathds{P}[E_\varepsilon]\geq 1-O_\varepsilon(\varepsilon^\mu)$ for all $\mu>0$.
\end{itemize}

For two sequences of random variables $\{A_n\}_{n\geq 1}$ and $\{B_n\}_{n\geq 1}$, we write $A_n\asymp_P B_n$ if for all $\varepsilon>0$, there exist $c,C>0$ such that $\mathds{P}[cB_n\leq A_n\leq CB_n]\geq 1-\varepsilon$ for all $n\geq 1$.

For any ${\bm x}\in\mathds{R}^d$, we write $|{\bm x}|$ as the Euclidean norm of ${\bm x}$, write $\|{\bm x}\|_{1}$ as the $\ell^1$-norm of ${\bm x}$ and write $\|{\bm x}\|_\infty$ as the $\ell^\infty$-norm of ${\bm x}$. For any $r>0$ and ${\bm x}\in\mathds{R}^d$, we define $V_r({\bm x})$ as the closed cube in $\mathds{R}^d$ with center at ${\bm x}$ and side length $r$
(note that in the renormalization arguments in the subsequent paper, to avoid intersection between different cubes, when not specifically stated, the symbol $V_r(\bm x)$ is considered to represent the cube (with center at $\bm x$ and side length $r$) that is closed on the left and bottom sides, and open on the right and top sides in the renormalization arguments).
For any $A\subset \mathds{R}^d$, we write $V_r(A)=\{{\bm z}\in \mathds{R}^d:\ \text{dist}({\bm z},A;\|\cdot\|_\infty)\leq r/2\}$.

For two sets $A,B\subset \mathds{R}^d$, we write ${\rm dist}(A,B)$ as the Euclidean distance between $A$ and $B$, i.e., ${\rm dist}(A,B)=\inf_{{\bm x}\in A,{\bm y}\in B}|{\bm x}-{\bm y}|$.

For $A\subset\mathds{R}^d$, we write ${\rm diam}(A;D)$ as the internal diameter of $A$ with respect to the metric $D$, i.e., ${\rm diam}(A;D)=\sup_{{\bm x},{\bm y}\in A}D({\bm x},{\bm y};A)$. For ${\bm x},{\bm y}\in\mathds{R}^d$, we write $\lfloor {\bm x}\rfloor:=(\lfloor {\bm x}^1\rfloor, \cdots, \lfloor {\bm x}^d\rfloor)$.

 We use ${\rm Bin}(n,p)$ to denote the distribution of a binomial variable with parameters $n\in\mathds{N}$ and $p\in[0,1]$. We use ${\rm Poi}(\lambda)$ to denote a Poisson variable with parameter $\lambda>0$.

Throughout the paper, we use $c_1, c_2, \dots$ to denote positive constants that are numbered by sections. Note that these constants may vary from section to section.

\subsection{Outline}\label{outline}
As detailed in Sections \ref{overview}-\ref{sectionWeakLRP}, our goal is to establish Theorems \ref{dis-subsequence-exist}, \ref{subsequence-exist} and \ref{uniqueness}. In this subsection, we give an overview of the proof of Theorem \ref{uniqueness}, which encapsulates the main challenge. To begin, suppose that $\beta>0$ and that $D$ and $\widetilde{D}$ are two weak $\beta$-LRP metrics on $\mathds{R}^d$. Our objective is to show that there exists a deterministic constant $C>0$ such that a.s.\ $\widetilde{D}=CD$.

\subsubsection{Optimal bi-Lipschitz constants}

Consider two weak $\beta$-LRP metrics on $\mathds{R}^d$, denoted as $D(\cdot,\cdot)$ and $\widetilde{D}(\cdot,\cdot)$. Proposition \ref{bilip} (presented below) indicates that these metrics are bi-Lipschitz equivalent, that is, there exists a deterministic constant $C>0$ such that a.s.
$$
C^{-1}D({\bm x},{\bm y})\leq \widetilde{D}({\bm x},{\bm y})\leq CD({\bm x},{\bm y})\quad \text{for all }{\bm x},{\bm y}\in\mathds{R}^d.
$$
Consequently, we can define the optimal upper and lower bi-Lipschitz constants between $D$ and $\widetilde{D}$ as follows:
\begin{equation}\label{optimalconst}
\begin{split}
c_*&=\sup\left\{c'>0:\ c'D({\bm x},{\bm y})\leq \widetilde{D}({\bm x},{\bm y})\ \text{for all}\ {\bm x},{\bm y}\in\mathds{R}^d\right\},\\
C_*&=\inf\left\{C'>0:\ \widetilde{D}({\bm x},{\bm y})\leq C'D({\bm x},{\bm y})\ \text{for all}\ {\bm x},{\bm y}\in\mathds{R}^d\right\}.
\end{split}
\end{equation}
Proposition \ref{cCdtm} (which we will present below) asserts that each of $c_*$ and $C_*$ is a.s.\ equal to a deterministic constant. As a result,
we can replace $c_*$ and $C_*$ by their a.s.\ values in Proposition \ref{cCdtm}, so that each of $c_*$ and $C_*$ is a deterministic constant depending only on the laws of $D$ and $\widetilde{D}$ and a.s.
\begin{equation}\label{DDtilde}
c_*D({\bm x},{\bm y})\leq \widetilde{D}({\bm x},{\bm y})\leq C_* D({\bm x},{\bm y})\quad \text{for all }{\bm x},{\bm y}\in \mathds{R}^d.
\end{equation}

\subsubsection{Main idea of the proof of Theorem \ref{uniqueness}}

Our objective is to establish Theorem \ref{uniqueness} by demonstrating that $c_*=C_*$. In the rest of this subsection, we provide an overview of how we will prove this fact. To make the main concepts as clear as possible, we will gloss over certain technical details, so some of the statements in this subsection may not be entirely accurate without additional caveats. More precise outlines can be found at the beginning of each individual section and subsection.

At a high level, our strategy for proving the uniqueness of weak $\beta$-LRP metrics shares similarities with the approach used to establish the uniqueness of LQG metrics in \cite{GM21,DG23}.  However, major challenges arise in our setting, owing to the fact that likely the geodesics in our model are discontinuous curves with respect to the topology of $\mathds{R}^d$ and the fact that we do not have the Weyl scaling property for $\beta$-LRP metrics.

In what follows, we provide a general overview on our proof approach. We begin by assuming, for the sake of contradiction, that $c_*<C_*$.
In Section \ref{section3} we will show that for any $c'\in (c_*,C_*)$, there are many ``good'' pairs of points ${\bm x,\bm y}\in\mathds{R}^d$  such that $\widetilde{D}({\bm x,\bm y})< c'D({\bm x,\bm y})$.
In fact,
in Section \ref{section4} we will show that the set of such points is dense enough that every $D$-geodesic $P$ has to get close to each of ${\bm x}$ and ${\bm y}$ for many ``good'' pairs of points ${\bm x,\bm y}$.
We then replace segments of $P$ with the concatenation of a $\widetilde{D}$-geodesic from a point on $P$ to ${\bm x}$, a $\widetilde{D}$-geodesic from ${\bm x}$ to ${\bm y}$, and a $\widetilde{D}$-geodesic from ${\bm y}$ to a point on $P$. This results in a new path with the same end points as $P$.

Our choice of good pair of points ${\bm x,\bm y}$ ensures that the $\widetilde{D}$-length of each of the replacement segments is at most slightly larger than $c'$ times its $D$-length. In addition, due to the definition of $C_*$, we know that the  $\widetilde{D}$-length of any segment of $P$ that 
is 
not replaced is at most $C_*$ times its $D$-length. Morally, we would like to say that this implies that there exists $c''\in(c',C_*)$ such that a.s.
\begin{equation}\label{c''}
\widetilde{D}({\bm x,\bm y})\leq c''D({\bm x,\bm y}),\quad \forall {\bm x,\bm y}\in \mathds{R}^d.
\end{equation}
The bound \eqref{c''} contradicts the fact that $C_*$ is the optimal upper bi-Lipschitz constant defined in \eqref{optimalconst}.

In the remainder of this section, we provide a more comprehensive outline of our proof strategy by describing each section in turn. This will provide a more detailed understanding of the overall structure and the flow of our argument.

\subsubsection{Section \ref{section2}: Some preparations}

Since renormalization is a repeatedly applied tool in this paper which relies heavily on the discrete model, we will prove some estimates on the discrete $\beta$-LRP in Section \ref{section2}. Additionally, we will show the proof of Lemma \ref{an-bounded} by proving a high-dimensional version of \cite[Theorem 1]{Ding-Sly13}.

\subsubsection{Section \ref{bi-lipschitz}: Bi-Lipschitz equivalence of two weak $\beta$-LRP metrics}
Our aim is to establish the bi-Lipschitz equivalence of any two arbitrary weak $\beta$-LRP metrics, using renormalization as our main tool. By renormalization and coupling with the discrete model, we can show that, in essence, every weak $\beta$-LRP metric is comparable with the discrete metric. In particular, there exists a constant $C_0$ such that for each sufficiently small $\varepsilon>0$, with high probability (vaguely speaking) the following estimate holds for any $R>0$. For any ${\bm x,\bm y}\in[-R,R]^d$ with $D({\bm x,\bm y})\ge 2\varepsilon^{\theta/2}R^\theta$,
$$
C_0^{-1}(\varepsilon R)^\theta  \widehat{d}({\bm k}_{\bm x},{\bm k}_{\bm y})\leq D({\bm x,\bm y})\leq C_0(\varepsilon R)^\theta \widehat{d}({\bm k}_{\bm x},{\bm k}_{\bm y}).
$$
Here $ \widehat{d}$ is the metric on the discrete model, and ${\bm k}_{\bm z}:=\lfloor\frac{{\bm z}}{\varepsilon R}\rfloor$ for ${\bm z}\in\mathds{R}^d$. 
Applying the same renormalization and the corresponding discrete metric $\widehat{d}$  
to two different weak $\beta$-LRP metrics $D$ and $\widetilde{D}$,
we can compare the two metrics using $ \widehat{d}$ as an intermediary. Moreover, we prove Proposition \ref{Dxy=0} and Theorem \ref{uniquetheta-1}  using the same renormalization. It is important to stress that in fact we prove the bi-Lipschitz equivalence for two local $\beta$-LRP metrics (defined in Section \ref{bi-lipschitz} and satisfying even weaker conditions than the weak $\beta$-LRP   
metric in Definition \ref{weakmetric}) because we need such bi-Lipschitz equivalence to complete the proof of tightness in Sections \ref{dis-existence} and \ref{existence}.
Finally, we will prove Theorem \ref{theorem-Hausdorff} in Section \ref{Hasd}.

\subsubsection{Section \ref{dis-existence}: Subsequential limiting metrics {\rm(}weak $\beta$-LRP metrics{\rm)} in the discrete model}

To establish the tightness of discrete $\beta$-LRP distances $\widehat{D}_n:=\widehat{a}_n^{-1}\widehat{d}(\lfloor n\cdot\rfloor,\lfloor n\cdot\rfloor)$ (recall $\widehat{a}_n$ in Lemma \ref{an-bounded} above), we use the tail estimate in \cite[Theorem 6.1]{Baumler22} for the discrete model to derive the existence of a subsequential limit of the $\beta$-LRP distance $D$. We will then construct a coupling of $\widehat{D}_n$ and the continuous $\beta$-LRP model $\mathcal{E}$ to demonstrate that every subsequential limiting metric $D$ is a weak $\beta$-LRP metric by verifying each axiom in Definition \ref{weakmetric}, ultimately concluding the proof of Theorem \ref{dis-subsequence-exist}.

\subsubsection{Section \ref{existence}: Subsequential limiting metrics {\rm(}weak $\beta$-LRP metrics{\rm)} in the continuous model}

Similar to what we do in Section \ref{dis-existence}, we establish the tightness of $\beta$-LRP distances $D_n:=a_n^{-1}d_{(1/n,\infty)}$ (recall $a_n$ in Lemma \ref{an-bounded} above) in Section \ref{existence}. Our main tool for the existence of a subsequential limiting metric is \cite[Theorem 6.1]{Baumler22}, which is an analogous version of the tail estimate for the discrete model.
After that, we will then demonstrate that every subsequential limiting metric is a weak $\beta$-LRP metric by verifying each axiom in Definition \ref{weakmetric} using a similar proof as in Section \ref{dis-existence}, ultimately concluding the proof of Theorem \ref{subsequence-exist}.

\subsubsection{Section \ref{section3}: Quantifying the optimality of the optimal bi-Lipschitz constants}
Let $C'\in(c_*,C_*)$. By the definition \eqref{optimalconst} of $c_*$ and $C_*$, it holds with positive probability that there exist points ${\bm x,\bm y}\in \mathds{R}^d$ such that $\widetilde{D}({\bm x,\bm y})>C'D({\bm x,\bm y})$. The purpose of Section \ref{section3} is to prove a quantitative version of this statement  which draws strong inspiration from \cite[Proposition 3.3]{DG23} for the LQG metric.

The following is a simplified version of the main result of Section \ref{section3} (see Proposition \ref{mrinS3}).

\begin{proposition}\label{simplemrinS3}
There exists $p\in(0,1)$, depending only on $d, \beta$ and the laws of $D$ and $\widetilde{D}$, such that for each $C'\in (0,C_*)$, there exists $\varepsilon_0>0$ depending on $d, \beta$ and the laws of $\widetilde{D}$ and $D$ with the following property. For each $\varepsilon\in(0,\varepsilon_0]$,
\begin{equation}\label{prob-good-1}
\mathds{P}[\exists \text{ a ``regular'' pair of points } {\bm u},{\bm v}\in [0,\varepsilon)^d \text{ such that }\widetilde{D}({\bm u},{\bm v})>C'D({\bm u},{\bm v})]\geq p.
\end{equation}
\end{proposition}

The statement that ${\bm u}$ and ${\bm v}$ are ``regular'' in \eqref{prob-good-1} means that these points satisfy several regularity conditions which are presented precisely in Definition \ref{def-H}. These conditions include lower and upper bounds on $D({\bm u},{\bm v})$. We emphasize that the parameter $p$ in Proposition \ref{simplemrinS3} does not depend on $C'$. This will be crucial for our proof later.

We will prove Proposition \ref{simplemrinS3} by contradiction. To this end, we will assume that there are arbitrarily small values $\varepsilon>0$  such that
\begin{equation}\label{Dtilde<C'D}
\mathds{P}[\widetilde{D}({\bm u},{\bm v})\leq C'D({\bm u},{\bm v}), \forall \text{ ``regular'' pairs of points }{\bm u},{\bm v}\in[0,\varepsilon)^d]\geq 1-p,
\end{equation}
and we will derive a contradiction if $p$ is sufficiently small. If $p$ is small enough (depending only on $d, \beta$ and the laws of $D$ and $\widetilde{D}$), then under the assumption of \eqref{Dtilde<C'D}, along with the independence of edges and a union bound via renormalization technique, we can derive the following. For any bounded cube $[-r,r]^d\subset \mathds{R}^d$, it holds with high probability that there are numerous small cubes of side length $\varepsilon$ within $[-r,r]^d$ such that the event in \eqref{Dtilde<C'D} occurs.
This in turn implies that a long path will pass through a significant number of these ``good'' cubes.

We will then work on the high-probability event that we have  many such ``good'' cubes in $[-r,r]^d$. 
Consider the points ${\bm x,\bm y}\in[-r,r]^d$ such that there exists a $D$-geodesic $P$ from ${\bm x}$ to ${\bm y}$ that lies in $[-r,r]^d$. We will replace several segments of $P$ between pairs of ``regular'' points ${\bm u},{\bm v}$ as in \eqref{Dtilde<C'D} with $\widetilde{D}$-geodesics from ${\bm u}$ to ${\bm v}$.
The $\widetilde{D}$-length of each of these geodesics is at most $C'D({\bm u},{\bm v})$. Furthermore, by \eqref{optimalconst}, the $\widetilde{D}$-length of each segment of $P$ which we did not replace is at most $C_*$ times its $D$-length. Hence this produces a path from ${\bm x}$ to ${\bm y}$ with $\widetilde{D}$-length of at most $C''D({\bm x,\bm y})$, where $C''\in (C',C_*)$ is a constant depending only on $C',d, \beta$ and the laws of $D$ and $\widetilde{D}$.

We can show that with high probability, this works for all $D$-geodesics contained in $[-r,r]^d$. Therefore, by taking $[-r,r]^d$ to be arbitrarily large, we reach a contradiction with the definition of $C_*$. As a result, we prove Proposition \ref{simplemrinS3}.

By the symmetry in our hypotheses for $D$ and $\widetilde{D}$, we also get the following analog of Proposition \ref{simplemrinS3} with the roles of $D$ and $\widetilde{D}$ interchanged.

\begin{proposition}\label{simplemrinS3-2}
There exists $p\in(0,1)$, depending only on $d, \beta$ and the laws of $D$ and $\widetilde{D}$, such that for each $c'>c_*$, there exists $\varepsilon_0>0$ depending on $d, \beta$ and the laws of $\widetilde{D}$ and $D$ with the following property. For each $\varepsilon\in(0,\varepsilon_0]$,
\begin{equation}\label{Dtilde<c'D}
\mathds{P}[\exists \text{ a ``regular'' pair of points } {\bm u,\bm v}\in [0,\varepsilon)^d \text{ such that }\widetilde{D}({\bm u,\bm v})<c'D({\bm u,\bm v})]\geq p.
\end{equation}
\end{proposition}

\medskip

\subsubsection{Section \ref{section4}: The core argument}

The idea for the rest of the proof of Theorem \ref{uniqueness} is to show that if $c_*<C_*$, then Proposition \ref{simplemrinS3-2} implies a contradiction to Proposition \ref{simplemrinS3}. We remark that in the overview level Section \ref{section4} follows the framework of \cite[Section 4]{DG23} for the LQG metric, but it is worth emphasizing that substantial additional challenges need to be addressed due to the discontinuity of paths and the lack of Weyl scaling in the LRP model. In this proof, we will employ the multi-scale analysis.

We set
$$
c'=\frac{c_*+C_*}{2},\quad \text{so that }c'\in(c_*,C_*)\  \text{if }c_*<C_*,
$$
and modify the event in \eqref{Dtilde<c'D} by considering its cube version (as defined in Definition \ref{interval-good}). As presented in Proposition \ref{simplemrinS3-2}, the probability of this event has a positive lower bound $p$. However, $p$ may be exceedingly small, which calls for applications of multi-scale analysis to increase it value.

To achieve this, we fix three parameters $\alpha$ (small enough), $K$ ($K\gg 1/\alpha$) and $\varepsilon$ (small enough).
We begin by dividing $\mathds{R}^d$ into small cubes of side length $\varepsilon/K$, and this is our first scale. We refer to such a cube as \textit{nice} if the event in \eqref{Dtilde<c'D}  (with $\varepsilon/K$ instead of $\varepsilon$ and small cubes instead of ${\bm u,\bm v}$) occurs (see Definition \ref{interval-good}). According to the independence of the Poisson point process, the probability of such a cube being nice has a lower bound $p$, which however can be exceedingly small.
To lift this probability up, we next merge every collection of $(\alpha^{2.5}K)^d$ cubes (of side length $\varepsilon/K$) into a cube of side length $\alpha^{2.5}\varepsilon$, forming our second scale.
A cube of this scale is considered \textit{very nice} if it contains at least one nice cube and satisfies certain regularity conditions.
These conditions ensure that if a path encounters a nice cube whose side length is $\varepsilon/K$ within a very nice cube with side length $\alpha^{2.5}\varepsilon$, it will spend a certain amount of time near the cube other than immediately jumping out of the cube (see Definition \ref{very nice} for more details).
After forming the second scale, we obtain larger cubes of side length $\alpha^2\varepsilon$ by merging every collection of $\alpha^{-0.5d}$ cubes (of side length $\alpha^{2.5}\varepsilon$). These cubes constitute our third scale and each such cube is referred to as \textit{very very nice} if it contains at least one very nice cube (see Definition \ref{very very nice}).
With the definitions described above and the independence of Poisson point process, we can observe that the probability of a cube being very very nice is close to 1 since $\alpha$ is sufficiently small.
Lastly, we refer to a cube of side length $\varepsilon$ (the fourth scale) as \textit{super good} if each of its sub-cubes (of side length $\alpha^2\varepsilon$) is very very nice and it satisfies certain regularity conditions (see Definition \ref{super good}). In other words, a cube $V_\varepsilon({\bm z}_k)\subset \mathds{R}^d$ (of side length $\varepsilon$) is considered super good
if it in a sense ``evenly'' contains enough ``regular'' pairs of small cubes (in the smallest scale), denoted by $(J_{k,i}^{(1)},J_{k,i}^{(2)})$, which satisfy
\begin{equation}\label{Dtilde<c'D-2}
\widetilde{D}(J_{k,i}^{(1)}, J_{k,i}^{(2)})<c'D(J_{k,i}^{(1)}, J_{k,i}^{(2)}),
\end{equation}
and several regularity conditions, such as
\begin{itemize}
\item[(1)] ${\rm dist}(J_{k,i}^{(2)}, J_{k,i+1}^{(1)})\geq \alpha^{2.5}\varepsilon/2$;
\item[(2)] there is a $D$-geodesic $P$ from $J_{k,i}^{(1)}$ to  $J_{k,i}^{(2)}$ such that $P$ is contained in a small neighborhood of $J_{k,i}^{(1)}\cup J_{k,i}^{(2)}$;
\item[(3)] every path must spend a certain amount of time before escaping a neighborhood of $J_{k,i}^{(1)}\cup J_{k,i}^{(2)}$.
\end{itemize}
Because of the independence of Poisson point process, the probability that a cube $V_\varepsilon({\bm z}_k)$ being super good is very close to 1. This implies that any path with sufficient length contained in a compact set must encounter a positive fraction of super good cubes and spend a significant amount of time within them.


It is worth emphasizing that we used multi-scale analysis to lift up the probability of a cube being super good mainly because the paths in our model are non-continuous with respect to the Euclidean topology of $\mathds{R}^d$, which is fundamentally different from the continuous paths in the LQG model.
In fact, in the LQG model, \cite{GM21,DG23} mainly used a series of concentric annuli to increase the probability for the existence of pairs of ``good'' points. Then based on the continuity of the paths, once a path hits some ball of small radius, it must pass through a ``good'' annulus which contains many pairs of ``good'' points. 
However, likely the geodesics in our model are discontinuous.
Even if a path hits a small cube, it may not necessarily hit a neighborhood of that cube, so the aforementioned method in \cite{GM21,DG23} is not applicable in our setting. Therefore, here we use multi-scale analysis to lay out super good cubes with a sufficiently high density, making it essentially impossible for paths to avoid all the super good cubes.

After a geodesic hits a super good cube $V_\varepsilon({\bm z}_k)$,  we want to attract it further so that it passes through some ``regular'' pairs of small cubes $(J_{k,i}^{(1)}, J_{k,i}^{(2)})$ in $V_\varepsilon({\bm z}_k)$. This will allow us to use \eqref{Dtilde<c'D-2}. To accomplish this,
 we concatenate these ``regular'' pairs of small cubes by (randomly) adding a number of edges between $J_{k,i}^{(2)}$ and $J_{k,i+1}^{(1)}$ (This is analogous to \cite{GM21,DG23} by subtracting a large bump function and applying the Weyl scaling valid for the LQG metric;  see more detailed discussions in the next paragraph). This results in a new metric denoted as $D_k^+$.
Note that the regularity condition (1) implies that $D(J_{k,i}^{(2)},J_{k,{i+1}}^{(1)})$ is not small (with high probability), while in the new metric $D_k^+$, we reduce this distance to 0 with high probability.
The regularity condition (3) suggests that when a $D_k^+$-geodesic exits $J_{k,i}^{(2)}$, it is more likely to go through the long edges we added, so it will hit $J_{k,{i+1}}^{(1)}$ and then $J_{k,{i+1}}^{(2)}$.
Additionally, the regularity conditions (2) and (3) ensure that $D$-geodesics between $J^{(1)}_{k,i}$ and $J_{k,i}^{(2)}$ are the same as $D_k^+$-geodesics. This is because at least one end point of the long edges we add is far from the small neighborhood of $J^{(1)}_{k,i}\cup J^{(2)}_{k,i}$, while $D$-geodesics from $J^{(1)}_{k,i}$ to $J^{(2)}_{k,i}$ are only within their small neighborhood and thus are not affected by the added edges (with high probability). Hence, the relation \eqref{Dtilde<c'D-2} still holds even after adding edges.
Consequently, the above regularity conditions guarantee that when a $D_k^+$-geodesic hits $V_\varepsilon({\bm z}_k)$, it is likely to pass through some ``regular'' pairs of small cubes which satisfy \eqref{Dtilde<c'D-2}. This, combined with the triangle inequality, will help us obtain some estimates, such as  Proposition \ref{simpleprop4.3} below.

 Let us further illustrate the comparison between the methods used to attract geodesics to pass through ``regular'' pairs of small cubes in our model and those used in \cite{GM21,DG23} for the LQG model. In the LQG model, when a geodesic hits some good annulus,
the method of subtracting a large and essentially constant function $f$ from the Gaussian free field within the annulus is used to attract geodesics to the pairs of good points inside and increase the probability of hitting them.
This results in a new metric which still satisfies \eqref{Dtilde<c'D-2} for those pairs of good points by the Weyl scaling property (i.e. $D_{h-f}=\e^{-f}D_{h}$). However, in our model, the Weyl scaling property is no longer applicable, and \eqref{Dtilde<c'D-2} may no longer hold after we add edges inside or around the good cubes. To address this challenge, we add long edges determined by a new independent Poisson point process, which does not change $D$-geodesics between  $J_{k,i}^{(1)}$ and  $J_{k,i}^{(2)}$, but links these good cubes together to attract geodesics to going along the desired path. Note that the edges we add are random and thus we need a more careful treatment (such as estimates on tail probabilities) to obtain the desired estimates.

Most of Section \ref{section4} is devoted to proving the estimate which roughly speaking says the following.

\begin{proposition}\label{simpleprop4.3}
Assume that $c_*<C_*$. As $\delta\to0$, it holds uniformly over all ${\bm x,\bm y}\in\mathds{R}^d$ that
$$
\mathds{P}\left[\widetilde{D}({\bm x,\bm y})\geq (C_*-\delta)D({\bm x,\bm y}),\ \text{regularity conditions}\right]=O_\delta(\delta^\mu),\quad \forall \mu>0.
$$
\end{proposition}

To prove Proposition \ref{simpleprop4.3}, as we mentioned before, we need to show that with high probability, there are numerous super good cubes $V_\varepsilon({\bm z}_k)$ such that a $D$-geodesic $P$ from ${\bm x}$ to ${\bm y}$ passes through some ``regular'' pairs of small cubes $(J_{k,i}^{(1)}, J_{k,i}^{(2)})$ in it. This with \eqref{Dtilde<c'D-2} in turn will show that $\widetilde{D}({\bm x,\bm y})$ is bounded away from $C_*D({\bm x,\bm y})$ (see \eqref{Dtilde<(C*-)D}).

Inspired by the idea in \cite[Section 4]{DG23}, we will show $P$ hits many super good cubes using an argument based on counting the number of events of a certain type which occur. More precisely, for sufficiently small $\varepsilon>0$, we divide $[-R,R]^d$ uniformly into $1/(2\varepsilon)^d$ small cubes of side length $\varepsilon R$, denoted by $V_{\varepsilon R}({\bm z}_k)$.
For a set $Z\subset (\varepsilon R)[-1/\varepsilon,1/\varepsilon]_\mathds{Z}^d$, we will let $\mathsf{F}_{Z,\varepsilon}(\mathcal{E},\mathcal{E}_Z^+)$ be (roughly speaking) the event that the following holds.
\begin{itemize}
\item[(1)] For ${\bm z}_k\in Z$, $V_{3\varepsilon R}({\bm z}_k)$ is super good with respect to $\mathcal{E}$.

\item[(2)] A $D$-geodesic from ${\bm x}$ to ${\bm y}$ hits $V_{\varepsilon R}({\bm z}_k)$ for each ${\bm z}_k\in Z$.

\item [(3)] For all ${\bm z}_k\in Z$, a $D_Z^+$-geodesic from ${\bm x}$ to ${\bm y}$ passes through some ``regular'' pair of cubes $(J_{k,i}^{(1)}, J_{k,i}^{(2)})$ in $V_{\varepsilon R}({\bm z}_k)$. Here $D_Z^+$ is the metric after adding edges to concatenate ``regular'' pairs of small cubes in all $V_{\varepsilon R}({\bm z}_k) $ for $ {\bm z}_k\in Z$.
\end{itemize}

In addition, we define a new edge set $\mathcal{E}_Z^-$ and a corresponding metric $D_Z^-$ after eliminating  some long edges between pairs of nice cubes which is in some sense an inverse procedure of adding edges mentioned above. We define $\mathsf{G}_{Z,\varepsilon}^-(\mathcal{E}_Z^-,\mathcal{E})$ similar to  $\mathsf{F}_{Z,\varepsilon}(\mathcal{E},\mathcal{E}_Z^+)$,  but using  $D_Z^-$ instead of $D$ and using $D$ instead of $D_Z^+$, i.e., $\mathsf{G}_{Z,\varepsilon}^-(\mathcal{E}_Z^-,\mathcal{E})$ is the event that the following hold.
 \begin{itemize}
\item[(1)] For ${\bm z}_k\in Z$, $V_{3\varepsilon R}({\bm z}_k)$ is super good with respect to $\mathcal{E}_Z^- $.

\item[(2)] A $D_Z^-$-geodesic from ${\bm x}$ to ${\bm y}$ hits $V_{\varepsilon R}({\bm z}_k)$ for each ${\bm z}_k\in Z$.

\item [(3)] For all ${\bm z}_k\in Z$, a $D$-geodesic from ${\bm x}$ to ${\bm y}$ passes through some ``regular'' pair of cubes $(J_{k,i}^{(1)}, J_{k,i}^{(2)})$ in $V_{\varepsilon R}({\bm z}_k)$.
\end{itemize}

Using basic properties of Poisson distribution, one could see that there is a constant $M_1>0$ depending only on the laws of $D$ and $\widetilde{D}$ such that
      \begin{equation}\label{RN}
        M_1^{-\# Z}\le\frac{\mathds{P}[\mathsf{G}_{Z,\varepsilon}^-(\mathcal{E}_Z^-,\mathcal{E})]}
        {\mathds{P}[\mathsf{G}_{Z,\varepsilon}^-(\mathcal{E},\mathcal{E}_Z^+)]}\le M_1^{\# Z},
    \end{equation}
(see Lemma \ref{Lemma4.4DG21}). We will eventually take a sufficiently large number $m$, which does not depend on ${\bm x,\bm y}$ or $\varepsilon$.
From \eqref{RN}, we can infer that the number of sets $Z$ with $\# Z\leq m$ for which $\mathsf{F}_{Z,\varepsilon}(\mathcal{E},\mathcal{E}_Z^+)$ occurs should be comparable to the number of such sets for which $\mathsf{G}_{Z,\varepsilon}^-(\mathcal{E}_Z^-,\mathcal{E})$ occurs.

Additionally, if $\varepsilon $ is small enough, we can show that the number of sets $Z$ with $\# Z\leq m$ such that $\mathsf{F}_{Z,\varepsilon}(\mathcal{E},\mathcal{E}_Z^+)$ occurs grows like a positive power of $\varepsilon^{-m}$ (refer to Proposition \ref{Prop4.5DG21}). As previously mentioned, there are numerous sets $Z_0$ with $V_{\varepsilon R}({\bm z}_k)$ being supper good and hit by $P$ for every ${\bm z}_k\in Z_0$. The objective is to produce many sets $Z\subset Z_0$ for which these criteria hold, along with the additional condition that a $D_Z^+$-geodesic, denoted by $P_Z^+$, passes through at least one ``regular'' pair of small cubes in $V_{\varepsilon R}({\bm z}_k)$.
For this, we start with a set $Z_0$ satisfying the aforementioned properties and iteratively remove the ``bad'' points ${\bm z}_k\in Z_0$ where such $D_Z^+$-geodesic $P_Z^+$ from ${\bm x}$ to ${\bm y}$ does not pass through any ``regular'' pair of small cubes in it. This procedure results in a set $Z\subset Z_0$ where $\mathsf{F}_{Z,\varepsilon}(\mathcal{E},\mathcal{E}_Z^+)$ occurs, and $\#Z$ is not too much smaller than $\#Z_0$. Refer to the proof of Lemma \ref{Lemma4.13DG21}.

By combining the previous two paragraphs with an elementary calculation, we can conclude that with high probability, there are lots of sets $Z$ with $\#Z\leq m$ such that $\mathsf{G}_{Z,\varepsilon}^-(\mathcal{E}_Z^-,\mathcal{E})$ occurs. In particular, there must be lots of cubes $V_{\varepsilon R}({\bm z}_k)$ for which $P$ passes through at least one pair of nice cubes in it. As mentioned earlier, this gives Proposition \ref{simpleprop4.3}.

Once Proposition \ref{simpleprop4.3} is proven, a union bound over many pairs of points ${\bm x,\bm y}\in[-1,1]^d$ leads to the following approximation (refer to Lemma \ref{delta to 0}).

\begin{proposition}\label{{delta to 0}-2}
Assume that $c_*<C_*$. We have
$$
\lim_{\delta\to0}\mathds{P}\left[\exists\ \text{a ``regular'' pair of points ${\bm x,\bm y}\in [-1,1]^d$ s.t. }\widetilde{D}({\bm x,\bm y})\geq (C_*-\delta)D({\bm x,\bm y})\right]=0.
$$
\end{proposition}

Proposition \ref{{delta to 0}-2} is incompatible with Proposition  \ref{simplemrinS3} since the parameter $p$ in Proposition  \ref{simplemrinS3} does not depend on $C'$. We thus obtain a contradiction to the assumption that $c_*<C_*$, so we conclude that $c_*=C_*$ and hence Theorem \ref{uniqueness} holds.

\section{Some preparations}\label{section2}

Given that renormalization will be applied frequently in proving the main results, it would be advantageous to have a foundational understanding of some estimates for the discrete long-range percolation model. Therefore, we study that in this section. The other goal of this section is to spell out an explicit proof of Lemma 1.6, where the main work was carried out in \cite{Baumler22}.

To begin with, let us recall the discrete $\beta$-LRP model and recall that $\widehat{d}$ is the chemical distance for the discrete $\beta$-LRP model.
As stated in \cite[Theorem 1.1]{Baumler22}, the following theorem shows that the distance between the two sites ${\bm 0}$ and $\bm u$ grows like $|\bm u|^\theta$ for some $\theta\in (0,1)$.

\begin{theorem}\label{discrete-dist}
For $d\geq 1$ and all $\beta>0$, there exists $\theta=\theta(d, \beta)\in(0,1)$ such that for any $\bm u\in\mathds{Z}^d$,
$$
\widehat{d}({\bm 0},\bm u)\asymp_P |\bm u|^\theta.
$$
\end{theorem}

In addition, we can also show an analogous version for the continuous model (see \cite{Ding-Sly13} for the one-dimensional case).
\begin{theorem}\label{continuous-dist}
For $d\geq 1$ and all $\beta>0$, there exists $\theta=\theta(d, \beta)\in(0,1)$ {\rm(}same as that in Theorem \ref{discrete-dist}{\rm)} such that for any $\bm u\in\mathds{R}^d$,
    $$
    d_{(1,\infty)}({\bm 0},\bm u)\asymp_P |\bm u|^\theta,
    $$
    where the metric $d_{(1,\infty)}(\cdot,\cdot)$ is defined in \eqref{dist_delta}.
\end{theorem}

Theorem \ref{continuous-dist} will be proved in the next subsection. Provided with Theorems \ref{discrete-dist} and \ref{continuous-dist}, we can present the
\begin{proof}[Proof of Lemma \ref{an-bounded}]
    By the scaling invariance property for the Poisson point process, we have that
    $$d_{(1/n,\infty)}(\cdot/n,\cdot/n)\stackrel{\rm law}{=}n^{-1}d_{(1,\infty)}(\cdot,\cdot).$$
    Combined with Theorems \ref{discrete-dist} and \ref{continuous-dist}, it completes the proof of the lemma.
\end{proof}

In order to obtain union bounds in our renormalization approach, we require some estimates on the number of proper paths starting from ${\bm 0}$ in the discrete model.  Let $\mathcal{P}_{\leq m}$ denote the collection of  self-avoiding paths starting from ${\bm 0}$ with lengths at most $m$ under the metric $\widehat{d}$, and let $|\mathcal{P}_{\leq m}|$ be the number of paths in  $\mathcal{P}_{\leq m}$ for $m\ge1$.

The following estimate for $\mathds{E}[|\mathcal{P}_{\leq m}|]$ is a similar version of \cite[Lemma 3.2]{Baumler22}.

\begin{lemma}\label{number-path-k}
For all $d\geq 1$ and all $\beta>0$, there exists a constant $C_{dis}=C_{dis}(\beta,d)>0$, depending only on $\beta$ and $d$, such that for all $m\geq 1$,
$$
\mathds{E}[|\mathcal{P}_{\leq m}|]\leq C_{dis}^m.
$$
\end{lemma}

\begin{proof}
    The underlying intuition behind the proof is that these paths are dominated by paths in a branching process where the offspring distribution has mean given by \eqref{sum} below.

    For any $k\in\mathds{N}$, let $\mathcal{P}_k$ be the collection of self-avoiding paths starting from ${\bm 0}$ with length $k$ and in particular we have $\mathcal{P}_0=\{{\bm 0}\}$. Note that for any $k+1$ distinct integers $\bm i_1,\bm i_2,\cdots,\bm i_{k+1}\in\mathds{Z}^d\setminus\{{\bm 0}\}$, a path $P:{\bm i_0}={\bm 0}\to {\bm i_1}\to\cdots\to {\bm i_k}\to{\bm i_{k+1}}$ with length $k+1$ is contained in $\mathcal{P}_{k+1}$ if and only if its subpath $P':{\bm i_0}={\bm 0}\to {\bm i_1}\to \cdots\to {\bm i_k}$ is contained in $\mathcal{P}_k$ and $\langle {\bm i_k},{\bm i_{k+1}}\rangle\in\mathcal{E}$. Thus, combining this with the independence between different edges,
    \begin{equation}\label{lem21step1}
        \mathds{P}[P\in\mathcal{P}_{k+1}]= \mathds{P}[P'\in\mathcal{P}_k]p_{{\bm i_k},{\bm i_{k+1}}},
    \end{equation}
    where
    \begin{equation*}
    p_{{\bm i},{\bm j}}:=
    \begin{cases}
    1, &\quad \|\bm i-\bm j\|_1=1,\\
    1-\exp\left\{-\beta \int_{V_1({\bm i})}\int_{V_1({\bm j})}\frac{1}{|{\bm u}-{\bm v}|^{2d}}\d {\bm u}\d {\bm v}\right\},&\quad\|\bm i-\bm j\|_1>1.
    \end{cases}
    \end{equation*}
     Now we first sum \eqref{lem21step1} over all possible $\bm i_{k+1}$ with a fixed $P'$ and then sum over all possible $P'$, and from this procedure we get that
    \begin{equation}\label{lem21step2}
        \begin{split}
            \mathds{E}[|\mathcal{P}_{k+1}|]&=\sum_{P'=({\bm 0}, {\bm i_1},\cdots, {\bm i_k})}\sum_{{\bm i_{k+1}}\notin P'}\mathds{P}[P'\in\mathcal{P}_k]p_{{\bm i_k},{\bm i_{k+1}}}\\
            &\leq \sup_{{\bm i}\in\mathds{Z}^d}\left[\sum_{ {{\bm j}\in\mathds{Z}^d: \bm j\neq \bm i}}p_{{\bm i},{\bm j}}\right]\sum_{P'=({\bm 0}, {\bm i_1},\cdots, {\bm i_k})}\mathds{P}[P'\in\mathcal{P}_k]
            = \sup_{{\bm i}\in\mathds{Z}^d}\left[\sum_{ {\bm j}\in\mathds{Z}^d:{\bm j}\neq {\bm i}}p_{{\bm i},{\bm j}}\right]\mathds{E}[|\mathcal{P}_k|].
        \end{split}
    \end{equation}
    Here $({\bm i_1},\cdots, {\bm i_k})$ takes value over all ordered subsets of $\mathds{Z}^d\setminus \{{\bm 0}\}$ with $k$ elements. Note that for ${\bm i}\in\mathds{Z}^d$, from the translation invariance of $p_{{\bm i},{\bm j}}$, we obtain that
    \begin{equation}\label{sum}
        \sum_{{\bm j}\in \mathds{Z}^d:{\bm j}\neq {\bm i}}p_{{\bm i},{\bm j}}=2d+\sum_{{\bm j}\in \mathds{Z}^d:  \|\bm j\|_1>1}\left[1-\exp\left\{-\beta \int_{V_1({\bm 0})}\int_{V_1({\bm j})}\frac{1}{|{\bm u}-{\bm v}|^{2d}}\d {\bm u}\d {\bm v}\right\}\right].
    \end{equation}
    Note that the series on the right hand side of \eqref{sum} converges since
    $$
    1-\exp\left\{-\beta \int_{V_1({\bm 0})}\int_{V_1({\bm j})}\frac{1}{|{\bm u}-{\bm v}|^{2d}}\d {\bm u}\d {\bm v}\right\}\leq  \frac{2^{2d}\beta}{|{\bm j}|^{2d}}\quad \text{for all }|{\bm j}|>2.
    $$
    Thus combining this with \eqref{lem21step2}, we get that there exists a constant $c_{dis}>0$ depending only on $\beta$ and $d$ such that
    \begin{equation*}
        \mathds{E}[|\mathcal{P}_{k+1}|]\leq c_{dis}\mathds{E}[|\mathcal{P}_k|]\quad\text{for all }k\in\mathds{N},
    \end{equation*}
 which implies that
    $$
    \mathds{E}[|\mathcal{P}_k|]\leq c_{dis}^k
    $$
    for all $k\in\mathds{N}$ (note that $\mathcal{P}_0$ only contains one element ${\bm 0}$). As a result, it comes from $\mathcal{P}_{\leq m}=\cup_{k=0}^m \mathcal{P}_k$ that the lemma is true with an appropriate choice of $C_{dis}> 0$  depending only on $d$ and $\beta$.
\end{proof}

In the aforementioned discrete critical long-range bond percolation model, we may introduce another layer of randomness such that each site ${\bm i}\in\mathds{Z}^d$ is ``active'' (resp. ``non-active'') with probability $p_{\rm site}\ ({\rm resp. }\ 1-p_{\rm site})$, and all sites and bonds are mutually independent. This results in a new percolation model, known as long-range site-bond percolation. Notably, unlike classical site-bond percolation, ``non-active'' sites in our model can be passed through just like active sites.

\subsection{Proof of Theorem \ref{continuous-dist}}
In this subsection, we will use Theorem \ref{discrete-dist} and a coupling between the discrete model and the continuous model to prove Theorem \ref{continuous-dist}.

First, we will show a stronger upper bound about the diameter under $d_{(1,\infty)}$, which is an analog of \cite[Theorem 6.1]{Baumler22}.
\begin{proposition}\label{apriori-cont-1}
    For any $\eta\in(0,1/(1-\theta))$, we have the following uniform upper bound about the moment generating function:
    \begin{equation}\label{SimpleMGF-0}
        \sup_{R>1}\mathds{E}\left[\exp\left\{\left(\frac{{\rm diam}([0,R]^d);d_{(1,\infty)}}{R^\theta}\right)^\eta\right\}\right]<\infty.
    \end{equation}
\end{proposition}

\begin{proof}
    We first consider the case when $R\in\mathds{Z}$. We can construct a coupling of the continuous model and a corresponding discrete model defined on $\mathds{Z}^d$. For any ${\bm k},{\bm l}\in\mathds{Z}^d$, in the corresponding discrete model we have that ${\bm k}$ is connected with ${\bm l}$ whenever $\|{\bm k}-{\bm l}\|_{1}=1$ or when $V_1({\bm k})$ and $V_1({\bm l})$ are connected directly by a long edge in the continuous model. 
    It is straightforward to show that this discrete model is also a critical long-range bond percolation model (see the paragraph before \eqref{connectprob}). 
    Thus it is natural for us to inherite the notation and denote by $\widehat{d}(\bm k, \bm l)$ the chemical distance between $\bm k$ and $\bm l$ in this discrete model.

    For fixed ${\bm x,\bm y}\in[0,R]^d$, let $\bm i_{\bm x}=\lfloor {\bm x}\rfloor$ and $\bm i_{\bm y}=\lfloor {\bm y}\rfloor$.
    Here $\lfloor {\bm x}\rfloor=(\lfloor {\bm x}^1\rfloor,\cdots,\lfloor {\bm x}^d\rfloor)\in\mathds{Z}^d$.
    Let $P_{{\bm x,\bm y}}^{(dis)}=(\bm k_0=\bm i_{\bm x}, \bm k_1,\cdots,\bm k_N=\bm i_{\bm y})$ be a geodesic from $\bm i_{\bm x}$ to $\bm i_{\bm y}$ in the discrete model. To construct a path connecting ${\bm x}$ and ${\bm y}$ in the continuous model when $\bm i_{\bm x}\neq \bm i_{\bm y}$, we will define ${\bm x}_j,{\bm y}_j\in V_1(\bm k_j)$ for $ j\in [0,N]_\mathds{Z}$ such that either $|{\bm y}_j-{\bm x}_{j+1}|\leq 2d$ or ${\bm y}_j$ and ${\bm x}_{j+1}$ are connected by a long edge in the continuous model, in accordance to the geodesic $P_{{\bm x,\bm y}}^{(dis)}$.
    The iterative definition of ${\bm x}_j,{\bm y}_j$ is as follows.

    Assuming that ${\bm x}_0={\bm x},{\bm y}_0,\cdots, {\bm y}_{j-1},{\bm x}_j$ for $0\leq j< N$ have been defined such that ${\bm x}_i,{\bm y}_i\in V_1(\bm k_i)$ for all $i\in [0,j-1]_\mathds{Z}$ and ${\bm x}_j\in V_1(\bm k_j)$, we consider the following two cases.
        \begin{itemize}
            \item[(1)] If $\|\bm k_j-\bm k_{j+1}\|_{1}=1$, we let ${\bm y}_j={\bm x}_j$ and ${\bm x}_{j+1}=\bm k_{j+1}$. Clearly, $|{\bm y}_j-{\bm x}_{j+1}|\le |{\bm x}_j-\bm k_j|+|\bm k_j-\bm k_{j+1}|\le 2d$.
            \item[(2)] If $\|\bm k_j-\bm k_{j+1}\|_{1}> 1$,  since $\bm k_j$ and $\bm k_{j+1}$ are connected by a long edge in the discrete model, we can choose ${\bm y}_j\in V_1(\bm k_j)$ and ${\bm x}_{j+1}\in V_1(\bm k_{j+1})$ such that $\langle {\bm y}_j,{\bm x}_{j+1}\rangle\in \mathcal{E}$.
        \end{itemize}
    Finally, we let ${\bm y}_N={\bm y}$. We define $P_{{\bm x},{\bm y}}$ to be the path ${\bm x}_0={\bm x}\to {\bm y}_0\to {\bm x}_1\to {\bm y}_1\to\cdots\to {\bm x}_N\to {\bm y}_N={\bm y}$, where ${\bm x}_j\to {\bm y}_j$ walks on the underlying Euclidean line, and ${\bm y}_j\to {\bm x}_{j+1}$ walks on the long edge $\langle {\bm y}_j, {\bm x}_{j+1}\rangle$ or a Euclidean line if such a long edge does not exist (see the left picture in Figure \ref{sect2construction} for an illustration).
\begin{figure}[htbp]
\centering
\subfigure{\includegraphics[scale=0.5]{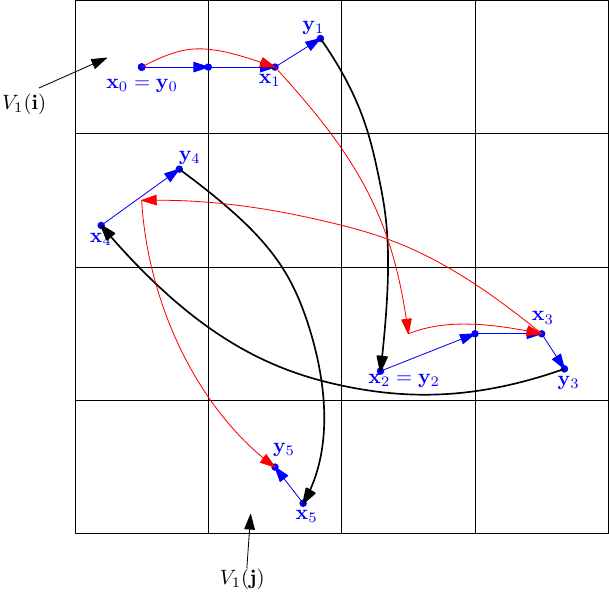}}
\quad
\subfigure{\includegraphics[scale=0.5]{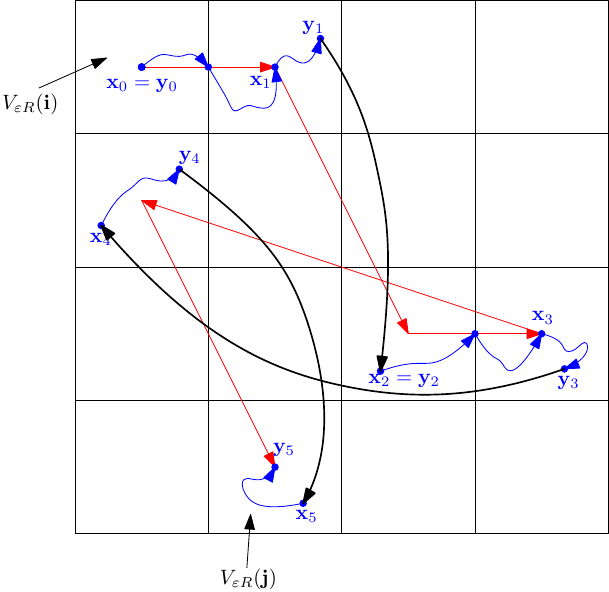}}
\caption{The left picture is the illustration for the construction of $P_{{\bm x},{\bm y}}$ in the two-dimensional case. We use the red curves to represent the discrete path $P_{{\bm x},{\bm y}}^{(dis)}$ and use the blue and black curves to represent the path $P_{{\bm x},{\bm y}}$. Here black curves represent edges $P_{{\bm x},{\bm y}}$ uses and blue lines represent paths restricted in a unit cube. 
The right picture is the illustration for the definition of a simple path in the continuous model in the two-dimensional case (see Definition \ref{pathd-c}). The difference from the left picture is that we replace blue lines by blue curves to present $D$-geodesics restricted in a cube.
Note that the blue curve in the right picture is only a visual representation of a geodesic, and it is not an actual geodesic within each cube since likely an actual geodesic will have jumps arising from long edges of smaller scopes which we skipped in this illustration.}
\label{sect2construction}
\end{figure}

From the construction above, it is easy to see that when $\bm i_{\bm x}\neq \bm i_{\bm y}$,
    \begin{equation}\label{cont-disc}
        d_{(1,\infty)}({\bm x},{\bm y})\le\sum_{j=0}^{{ \widehat{d}}(\bm i_{\bm x},\bm i_{\bm y})}|{\bm x}_j-{\bm y}_j|+\sum_{j=1}^{{ \widehat{d}}(\bm i_{\bm x},\bm i_{\bm y})}|{\bm y}_{j-1}-{\bm x}_j|\I_{\{\langle {\bm y}_{j-1},{\bm x}_j\rangle\notin\mathcal{E}\}}\le 3d{ \widehat{d}}(\bm i_{\bm x},\bm i_{\bm y})+1.
    \end{equation}
Note that \eqref{cont-disc} holds trivially when $\bm i_{\bm x}=\bm i_{\bm y}$. Thus, taking supreme over all $({\bm x},{\bm y})\in[0,R]^{2d}$ yields
    \begin{equation}\label{cont-disc-2}
        {\rm diam}([0,R]^d;d_{(1,\infty)})\le 3d{\rm diam}([0,R]_\mathds{Z}^d;\widehat{d})+1\le 4d{\rm diam}([0,R]_\mathds{Z}^d;\widehat{d}).
    \end{equation}

For general $R>1$, there exists $k\in\mathds{N}$ such that $2^{k}<R\le 2^{k+1}$. Using \eqref{cont-disc-2} we get
    \begin{equation}\label{CompareDtildeC}
        R^{-\theta}{\rm diam}([0,R]^d;d_{(1,\infty)})\le 2^{\theta} \frac{{\rm diam}([0,2^{k+1}]^d;d_{(1,\infty)})}{(2^{k+1})^\theta}\le  2^{\theta+2}d\frac{{\rm diam}([0,2^{k+1}]_\mathds{Z}^d;\widehat{d})}{(2^{k+1})^\theta}.
    \end{equation}
Furthermore, according to \cite[Theorem 6.1]{Baumler22},
    \begin{equation}\label{DiscreteMGF}
        \sup_{k\in\mathds{N}}\mathds{E}\left[\exp\left\{\left(\frac{{\rm diam}([0,2^{k+1}]_{\mathds{Z}}^d;\widehat{d})}{(2^{k+1})^\theta}\right)^\eta\right\}\right]<\infty
    \end{equation}
holds for any $\eta\in(0,1/(1-\theta))$. Thus, for any $\eta\in(0,1/(1-\theta))$, we choose $\eta_1\in(0,1/(1-\theta))$ such that $\eta<\eta_1$. Then applying \eqref{CompareDtildeC}  and \eqref{DiscreteMGF} with $\eta=\eta_1$, we get \eqref{SimpleMGF-0}, which implies the proposition.
\end{proof}

Now we prove the lower bound with the same coupling.
\begin{proposition}\label{apriori-cont-2}
    $\{\frac{r^\theta}{d_{(1,\infty)}({\bm 0},([-r,r]^d)^c)}\}_{r>1}$ is tight.
\end{proposition}

We need some preparations before the proof. We begin by giving an encoding of paths in the continuous long-range percolation model under $d_{(1,\infty)}$. We will denote a path $P$ as a sequence of points $P = {\bm z}_0\to {\bm z}_1\to\cdots\to {\bm z}_m$; this means that the path $P$ traverses those points in that order and in addition for any $j\in[1,m]_\mathds{Z}$, the path $P$ goes from ${\bm z}_j$ to ${\bm z}_{j+1}$ either using the long edge $\langle {\bm z}_j, {\bm z}_{j+1} \rangle$ when  $\langle {\bm z}_j, {\bm z}_{j+1} \rangle \in \mathcal{E}$ or (otherwise) going along the underlying Euclidean line.

For each $\langle {\bm z}_{i-1},{\bm z}_i\rangle\in \mathcal{E}$, we define its scope to be $|{\bm z}_{i-1} - {\bm z}_i|$ and call $\langle{\bm z}_{i-1}, {\bm z}_i\rangle$ as a hop.
Since all $d_{(1,\infty)}$-paths only use long edges with scopes at least 1, we let
$$
I = I_P := \{i \in [1,m]_\mathds{Z} : \langle {\bm z}_{i-1}, {\bm z}_i\rangle \in \mathcal{E}, |{\bm z}_{i-1} - {\bm z}_i| \in (1,\infty)\}
$$
be the set of jump times of the path $P$. Then the length of $P$ is defined to be $\|P\|_1:=\sum_{i\in[1,m]_\mathds{Z}\setminus I} |{\bm z}_{i-1} - {\bm z}_i|$ (note that the scopes in $I$ are not counted in measuring $\|P\|_1$).
For $i \notin I$, we say that $({\bm z}_{i-1}, {\bm z}_i)$ is a gap (with length $|{\bm z}_{i-1}-{\bm z}_i|$). We say a path is proper if it does not contain two consecutive gaps and does not reuse a jump. Since almost surely no two edges share an end point, this means that a proper path alternates between hops and gaps. We say that two proper paths are equivalent if they have the same starting point and make the same set of jumps in the same order (namely, they are equal except for the final end point).
A shortest path (i.e., a path joining two fixed points whose length is the minimal among all such paths) will be called a geodesic.

For $t \geq 0$, we consider proper paths starting from the origin ${\bm 0}$ with lengths at most $t$ in the metric $d_{(1,\infty)}$, and we denote by $\mathscr{P}_t$ the collection of equivalence classes of these proper paths. It will be useful to also count restricted classes of such paths. For $P \in \mathscr{P}_t$, denote by $h(P) = |I|$ the number of hops in the path $P$.
\begin{lemma}\label{DS13lem21}
    For all $d\geq 1$ and all $\beta > 0$, there exists $C_{cont} > 0$ such that for all $t > 0$ and $\alpha \geq C_{cont}$,
    $$
    \mathds{E}[|\mathscr{P}_t|] \leq C_{cont}^t\quad \text{and}\quad \mathds{P}[\exists P \in \mathscr{P}_t : h(P) \geq \alpha t] \leq (C_{cont}/\alpha)^t.
    $$
\end{lemma}
\begin{proof}
    For convenience, for any $k\in\mathds{N}, t>0$ and ${\bm x}\in\mathds{R}^d$, let $\mathscr{P}_t^{\bm x}(k)$ be the collection of equivalence classes of proper paths starting from ${\bm x}$ with lengths at most $t$ in the metric $d_{(1,\infty)}$ and with $k$ hops. If ${\bm x}={\bf 0}$, we will omit ${\bm x}$ then. From translation invariance of the Poisson point process, we get that $\mathds{E}[|\mathscr{P}_t^{{\bm x}}(k)|]$ does not depend on ${\bm x}$ and thus we denote the expectation by $f_k(t)$.

    Now we will calculate $f_k(t)$ inductively. For any $k>0$ and a path $P\in\mathscr{P}_t(k)$, we consider the first hop $\langle{\bm x}, {\bm y}\rangle$ which divides $P$ into two parts before and after this hop. The former part contains a gap with length $|{\bm x}|$ and a hop $\langle{\bm x} , {\bm y}\rangle$, and the latter part is contained in $\mathscr{P}_{t-|{\bm x}|}^{{\bm y}}(k-1)$. As a result, we can get the following recursion by counting the possible hop and the subsequent path respectively:
    \begin{equation}\label{recursion}
        f_k(t)\leq\iint_{|{\bm x}|\leq t, |{\bm x}-{\bm y}|\geq 1}\frac{\beta f_{k-1}(t-|{\bm x}|)}{|{\bm x}-{\bm y}|^{2d}}\d {\bm x}\d {\bm y}=c_d\beta\int_0^t r^{d-1}f_{k-1}(t-r)\d r,
    \end{equation}
    where $c_d>0$ is a constant depending only on $d$.

    From now on, we will use \eqref{recursion} and the fact that $f_0(t)=1$ to give upper bounds on all $f_k$ inductively. To be more precise, we will prove by induction on $k$ that
    \begin{equation}\label{fk}
        f_k(t)\leq\frac{( c_d\beta (d-1)!t^d)^k}{(kd)!}.
    \end{equation}
    It is obvious that \eqref{fk} holds when $k=0$. Assume that \eqref{fk} holds for $k\geq 0$. Then from \eqref{recursion}, we get that
    \begin{equation*}
        \begin{split}
            f_{k+1}(t)&\leq((kd)!)^{-1}(c_d\beta)^{k+1}((d-1)!)^k\int_0^t r^{d-1}(t-r)^{kd}\d r\\
            &=\frac{(c_d \beta (d-1)!t^d)^{k+1}}{((k+1)d)!}\frac{((k+1)d)!}{(kd)!(d-1)!}\int_0^1 r^{d-1}(1-r)^{kd}\d r\\
            &=\frac{(c_d \beta (d-1)!t^d)^{k+1}}{((k+1)d)!},
        \end{split}
    \end{equation*}
    where the last equality is due to $\int_0^1 r^{q_1}(1-r)^{q_2}\d r=\Gamma(q_1+1)\Gamma(q_2+1)/\Gamma(q_1+q_2+2)$ for all $q_1,q_2>-1$.
    Consequently, \eqref{fk} holds for all $k\in\mathds{N}$.

    Now let $\widehat{c}=(\beta c_d(d-1)!)^{1/d}$. Then $f_k(t)\leq\frac{(\widehat{c}t)^{kd}}{(kd)!}$. We can obtain that
    \begin{equation}\label{hd-lem21step1}
        \mathds{E}[|\mathscr{P}_t|]=\sum_{k\geq 0}f_k(t)\leq\sum_{k\geq 0}\frac{(\widehat{c}t)^{kd}}{(kd)!}\leq \e^{\widehat{c}t}.
    \end{equation}
    In addition, by Markov's inequality we get that
    \begin{equation}\label{hd-lem21step2}
        \mathds{P}[\exists P \in \mathscr{P}_t : h(P) \geq \alpha t]\leq \sum_{k\geq \alpha t}f_k(t)\leq\sum_{k\geq \alpha t}\frac{(\widehat{c}t)^{kd}}{(kd)!}\leq \e^{\widehat{c}t}\mathds{P}[{\rm Poi}(\widehat{c}t)\geq \alpha t].
    \end{equation}
    From Markov's inequality again, the right hand side of \eqref{hd-lem21step2} is bounded by $\e^{-\alpha t+\widehat{c}\e t}\leq (\e^{\widehat{c}\e}/\alpha)^t$. As a result, combining \eqref{hd-lem21step1} and \eqref{hd-lem21step2}, we get the lemma with $C_{cont}=\e^{\widehat{c}\e}$.
\end{proof}

We also record the following tightness result for the discrete model, which is an immediate consequence of \cite[Lemmas 2.3 and 4.10]{Baumler22}.
\begin{lemma}\label{Baumler4-10}
    $\left\{\frac{r^\theta}{\widehat{d}({\bm 0},([-r,r]_\mathds{Z}^d)^c)}\right\}_{r>1}$ is tight.
\end{lemma}

Now we present the
\begin{proof}[Proof of Proposition \ref{apriori-cont-2}]
    Let $P$ be a geodesic under $d_{(1,\infty)}$ from ${\bm 0}$ to $([-r,r]^d)^c$ (in the case of multiple geodesics, we choose an arbitrary one in a prefixed manner). Assume that $P:{\bm 0}={\bm x}_1\to {\bm y}_1\to {\bm x}_2\to\cdots\to {\bm x}_m\to {\bm y}_m$ uses $m-1$ edges in $\mathcal{E}$ (i.e. $\langle {\bm y}_i,{\bm x}_{i+1} \rangle$ for $i\in[1,m-1]_{\mathds{Z}}$) and goes along the Euclidean line from ${\bm x}_j$ to ${\bm y}_j$ for $j\in[1,m]_{\mathds{Z}}$.
    For $ j\in[1,m]_{\mathds{Z}}$, we choose $\bm k_j,\bm l_j\in\mathds{Z}^d$  such that ${\bm x}_j\in V_1(\bm k_j)$ and ${\bm y}_j\in V_1(\bm l_j)$. Thus there exists a constant $c(d)>0$ depending only on $d$ such that
    $$c(d)|{\bm x}_j-{\bm y}_j|+2d\geq  \|\bm k_j-\bm l_j\|_{1}\geq \widehat{d}(\bm k_j,\bm l_j).$$
    Additionally, from the definition of the coupling, we have $\widehat{d}(\bm l_j,\bm k_{j+1})=1$ since there is a long edge connecting $\bm l_j$ and $\bm k_{j+1}$ directly. Thus summing over $j$ yields that
    \begin{equation*}
        \widehat{d}(\bm 0,([-r,r]_\mathds{Z}^d)^c)\leq c(d)d_{(1,\infty)}(\bm 0,([-r,r]^d)^c)+(2d+1)(h(P)+1).
    \end{equation*}

    For any fixed $\varepsilon>0$, by Lemma \ref{Baumler4-10}, we can first take $c_1=c_1(\varepsilon)>0$ such that for any $r>1$,
    \begin{equation}\label{E(1)}
        \mathds{P}[\widehat{d}(\bm 0,([-r,r]_\mathds{Z}^d)^c)\leq c_1r^\theta]<\varepsilon/2.
    \end{equation}
    Moreover, set $c_2=\frac{c_1}{4(2d+1)C_{cont}+2c(d)}>0$. From Lemma \ref{DS13lem21} with $t=c_2 r^\theta$ and $\alpha=2C_{cont}$, we get
    \begin{equation}\label{E(2)}
        \mathds{P}[d_{(1,\infty)}(\bm 0,([-r,r]^d)^c)<c_2 r^\theta, h(P)>2C_{cont}c_2 r^\theta]\leq 2^{-c_2 r^\theta}.
    \end{equation}

In the following, we define some events. Let $E_{r,1}$ be the event that $\widehat{d}(\bm 0,([-r,r]_\mathds{Z}^d)^c)\geq c_1r^\theta$ and $E_{r,2}$ be the event that $h(P)\leq 2C_{cont}c_2 r^\theta$. Set $E_{r}=E_{r,1}\cap E_{r,2}$.
Additionally, let $F_r$ be the event that $d_{(1,\infty)}(\bm 0,([-r,r]^d)^c)<c_2 r^\theta$. Then we obtain
    \begin{equation*}
    \begin{split}
        \mathds{P}[F_r]&\leq \mathds{P}[E_{r,1}^c]+\mathds{P}[F_r\cap E_{r,2}^c]+\mathds{P}[F_r\cap E_r]\\
        &\leq \varepsilon/2+2^{-c_2 r^\theta}+\I_{\{c_1 r^\theta<4d+2\}}
        \end{split}
    \end{equation*}
    from \eqref{E(1)} and \eqref{E(2)} and the fact that on the event $E_r\cap F_r$,
    $$
    c_1 r^\theta\leq\widehat{d}(\bm 0,(([-r,r]_\mathds{Z})^d)^c)\leq c(d)d_{(1,\infty)}(\bm 0,([-r,r]^d)^c)+(2d+1)(h(P)+1)\leq c_1 r^\theta/2+(2d+1).
    $$
    Hence, we can choose an $r_\varepsilon>1$ depending only on $\beta,d$ and $\varepsilon$ such that $\mathds{P}[F_r]<\varepsilon$ when $r>r_\varepsilon$. Furthermore, when $1<r\leq r_\varepsilon$, we can choose $c_3=c_3(\varepsilon)>0$ such that $$
    \mathds{P}[d_{(1,\infty)}(\bm 0,([-r,r]^d)^c)<c_3r^\theta]\leq \mathds{P}[d_{(1,\infty)}(\bm 0,([-1,1]^d)^c)<c_3r_\varepsilon^\theta]<\varepsilon
    $$
     from the fact that $d_{(1,\infty)}(\bm 0,([-1,1]^d)^c)>0$ almost surely. As a result, taking $c=c_2\wedge c_3$, we get$$\mathds{P}[d_{(1,\infty)}(\bm 0,([-r,r]^d)^c)<cr^\theta]<\varepsilon$$ for any $r>1$, which implies the proposition.
\end{proof}

\begin{proof}[Proof of Theorem \ref{continuous-dist}]
    With the fact that $$d_{(1,\infty)}(\bm 0,\bm u)\in[d_{(1,\infty)}(\bm 0,([|{\bm u}|/d,|{\bm u}|/d]^d)^c),{\rm diam}([-|{\bm u}|,|{\bm u}|]^d,d_{(1,\infty)})],$$
    it is immediately implied by Propositions \ref{apriori-cont-1} and \ref{apriori-cont-2}.
\end{proof}

\section{Bi-Lipschitz equivalence}\label{bi-lipschitz}
In this section, we will prove the bi-Lipschitz equivalence of two arbitrary ``local'' metrics, which is an even slightly weaker metric than the weak $\beta$-LRP metric and is defined as  follows.

\begin{definition}\label{localmetric}
We say a pseudometric $D$ on $\mathds{R}^d$ is a \textit{local $\beta$-LRP metric} if it satisfies Axioms I, III, IV', V1', V2' and the following axiom.
\begin{itemize}
\item[II''] {\bf weak locality.} For any finite sequence of disjoint open sets $V_1,V_2,\cdots,V_N\subset \mathds{R}^d$, the pairs $(\mathcal{E}|_{V_i\times V_i};D(\cdot,\cdot;V_i))$, $i=1,2,\cdots, N$ and
$\mathcal{E}|_{(\cup_{i=1}^N(V_i\times V_i))^c}$ are all independent.
\end{itemize}

\end{definition}

Due to the independence of the Poisson point process, it is clear that Axiom II (locality) implies Axiom II'' (weak locality). That is, every weak $\beta$-LRP metric is also a local $\beta$-LRP metric. As a result, all the findings in this section hold true for weak $\beta$-LRP metrics.

In this section, when not specifically stated, we assume that $D$ and $\widetilde{D}$ are two local $\beta$-LRP metrics. The following proposition draws strong inspiration from \cite[Theorem 1.6]{GM19c}.

\begin{proposition}\label{bilip}
    Let $D$ and $\widetilde{D}$ be two local $\beta$-LRP metrics. Then there is a finite deterministic constant $C>0$ which depends only on $\beta,d$ and the laws of $D$ and $\widetilde{D}$ such that a.s.
$$
C^{-1}D({\bm x},{\bm y})\leq \widetilde{D}({\bm x},{\bm y})\leq CD({\bm x},{\bm y}),\quad \forall {\bm x},{\bm y}\in\mathds{R}^d.
$$
\end{proposition}

Thus we can define the optimal upper and lower bi-Lipschitz constants between $D$ and $\widetilde{D}$ as
$$
c_*=\sup\left\{c'>0:\ c'D({\bm x},{\bm y})\leq \widetilde{D}({\bm x},{\bm y})\ \text{for all}\ {\bm x},{\bm y}\in\mathds{R}^d\right\}
$$
and
$$
 C_*=\inf\left\{C'>0:\ \widetilde{D}({\bm x},{\bm y})\leq C'D({\bm x},{\bm y})\ \text{for all}\ {\bm x},{\bm y}\in\mathds{R}^d\right\}.
$$
Furthermore, if $D$ and $\widetilde{D}$ are also weak $\beta$-LRP metrics, we can show that $c_*$ and $C_*$ are both deterministic constants as follows.
\begin{proposition}\label{cCdtm}
Let $D$ and $\widetilde{D}$ be two weak $\beta$-LRP metrics. Each of $c_*$ and $C_*$, as defined above is a.s.\ equal to a deterministic constant.
\end{proposition}

The primary method for proving Propositions \ref{bilip} and \ref{cCdtm} is through renormalization. To streamline its use in the proof, we will perform preliminary work in Section \ref{goodint}.
Then, in Section \ref{proof111},  we will present Proposition \ref{Dxy=0} in a version that pertains to local $\beta$-LRP metrics (which implies Proposition \ref{Dxy=0}) since it forms part of the input required for the proof of Proposition \ref{bilip}. From there, we can use similar techniques to complete the proof of Proposition \ref{bilip}. Additionally, Subsections \ref{uniquetheta} and \ref{Hasd}
are devoted to the proofs of Theorems \ref{uniquetheta-1} and \ref{theorem-Hausdorff}, respectively.

\subsection{Good cubes and associated estimates}\label{goodint}

We first introduce the definition of ``good'' cubes as follows.

\begin{definition}\label{h-good}
For $s>0$, ${\bm z}\in\mathds{R}^d$ and $\alpha\in(0,1)$, we say that a cube $V_{3s}({\bm z})$ is $(3s,\alpha)$-good if the following condition holds. For any two different edges $\langle {\bm u}_1,{\bm v}_1\rangle, \langle {\bm u}_2,{\bm v}_2\rangle\in\mathcal{E}$, with ${\bm u}_1\in V_s({\bm z})^c$, ${\bm v}_1\in V_s({\bm z})$, ${\bm u}_2\in V_{3s}({\bm z})$ and ${\bm v}_2\in V_{3s}({\bm z})^c$, we have $|{\bm v}_1-{\bm u}_2|\geq \alpha s$. Additionally, there is a constant $b>0$ (which does not depend on $\bm z$ or $s$ and will be chosen in Lemma \ref{h-regular-low} below) such that
$$
D({\bm v}_1,{\bm u}_2;V_{3s}({\bm z}))\geq (b\alpha s)^\theta.
$$
\end{definition}

\begin{figure}[htbp]
    \centering
    \includegraphics[scale=0.6]{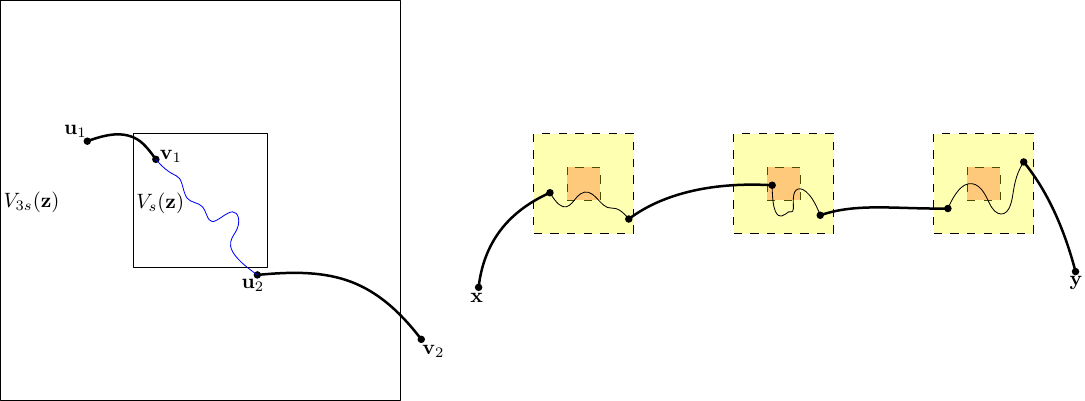}
    \caption{The left figure provides an explanation for Definition \ref{h-good}. The right figure illustrates how we apply the definition of good cubes: for any sufficiently long path, we can with overwhelming probability find a set of $(3s,\alpha)$-good cubes of side length $3s$ (yellow cubes) along the path. These cubes satisfy that the path hits the smaller cubes of side length $s$ (orange cubes) located at the centers of them, and also spends enough time within the yellow cubes. Note that the thick black curves represent the long edges, and the thin black curves represent the paths in cubes.}
    \label{Def34}
\end{figure}

We can show that, for suitable choices of $\alpha$ and $b$, it holds with arbitrarily high probability that $V_{3s}({\bm z})$ is $(3s,\alpha)$-good as follows.

\begin{lemma}\label{h-probgood}
For fixed ${\bm z}\in\mathds{R}^d$, $s>0$ and sufficiently small $\alpha\in(0,1)$, there exist constants $b=b(\alpha)>0$ {\rm (}depending only on $d,\beta,\alpha$ and the law of $D${\rm)} and $c_1>0$ {\rm(}depending only on $d,\beta$ and the law of $D${\rm)} such that $V_{3s}({\bm z})$ is $(3s,\alpha)$-good with probability at least $1-c_1\alpha \log(1/\alpha)$.
\end{lemma}

We can obtain the desired statement in Lemma \ref{h-probgood} by using Axiom IV' (translation invariance) and Lemmas \ref{h-longedge} and \ref{h-regular-low} below.

\begin{lemma}\label{h-longedge}
For fixed $s>0$ and sufficiently small $\alpha\in(0,1)$, there exists a constant $c_2>0$ {\rm(}depending only on $\beta$ and $d${\rm)} such that with probability at least $1-c_2\alpha \log(1/\alpha)$, for any two edges $\langle {\bm u}_1,{\bm v}_1\rangle, \langle {\bm u}_2,{\bm v}_2\rangle\in\mathcal{E}$, with ${\bm u}_1\in V_s({\bm 0})^c$, ${\bm v}_1\in V_s({\bm 0})$, ${\bm u}_2\in V_{3s}({\bm 0})$ and ${\bm v}_2\in V_{3s}({\bm 0})^c$, we have $|{\bm v}_1-{\bm u}_2|\geq \alpha s$.
\end{lemma}

\begin{proof}
For convenience, we denote by $A$ the event in Lemma \ref{h-longedge}, and assume that $1/(2\alpha)\in\mathds{Z}$. Otherwise, we can replace $1/(2\alpha)$ with $\lfloor 1/(2\alpha)\rfloor$ in the following proof.

Note that if ${\bm u}_2\in V_{3s}(\bm 0)\cap V_{s+\alpha s}(\bm 0)^c$, then  $|{\bm v}_1-{\bm u}_2|\geq \alpha s$ obviously. Thus we only need to upper-bound the probability for the case that ${\bm u}_2\in V_{s+\alpha s}(\bm 0)$.
Now we divide the cube $V_{s}({\bm 0})$ into $(1/\alpha)^d$ small cubes of equal side length, denoted by $J_{i,j}$ for $i\in [0,1/(2\alpha)]_{\mathds{Z}}$ and $j\in [1,2d (1/\alpha-2i)^{d-1}]_{\mathds{Z}}$.
Here the subscript $i$ in $J_{i,j}$ represents that there are $i$ layers of $J_{i',\cdot}$ between $J_{i,j}$ and the boundary of $V_s({\bm 0})$, and the subscript $j$ represents the $j$-th cube among those so that $J_{i, j}$ and $J_{i, j+1}$ are neighboring each other for all $j$. We also divide $V_{s+\alpha s}(\bm 0)\setminus V_s(\bm 0)$ into small cubes of side length $\alpha s$, denoted by $J_{-1,j}$ for $j\in [1,2d (1+1/\alpha)]_\mathds{Z}$. Here the subscript $j$ also represents the $j$-th cube among those so that $J_{-1, j}$ and $J_{-1, j+1}$ are neighboring each other for all $j$ (see the left picture in Figure \ref{Jij} below). We also denote by $\widetilde{J}_{i,j}$ the union of $J_{i,j}$ and all cubes adjacent to $J_{i,j}$.

\begin{figure}[htbp]
\centering
\subfigure{\includegraphics[scale=0.5]{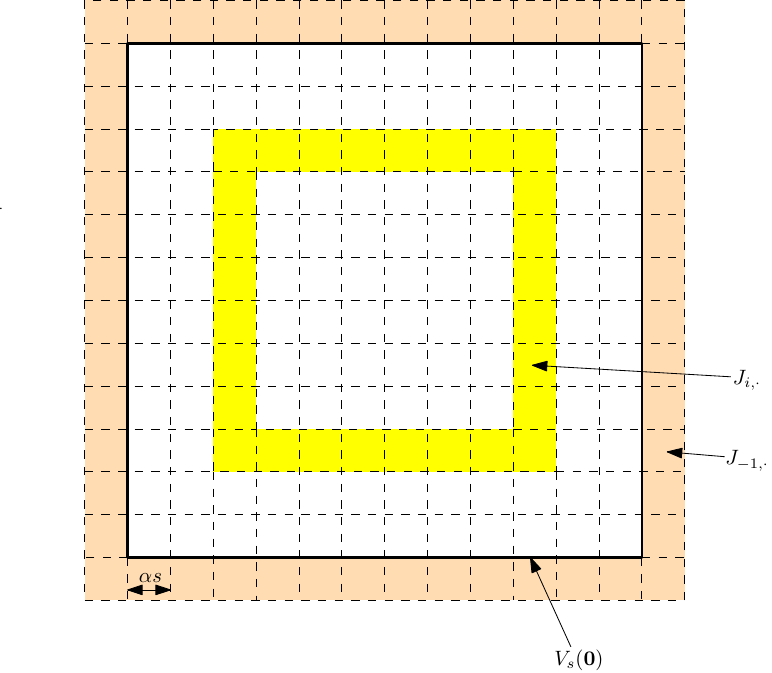}}
\quad
\subfigure{\includegraphics[scale=0.5]{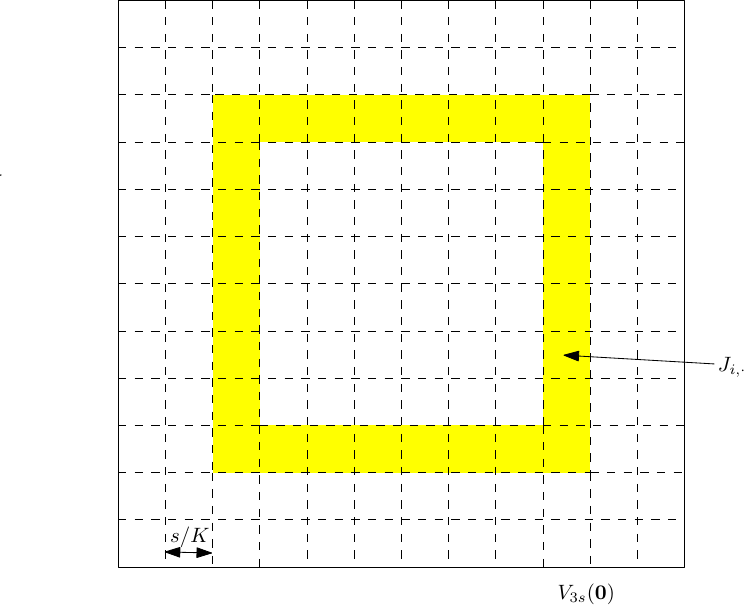}}
\caption{The left picture is an illustration for the proof of Lemma \ref{h-longedge}. We divide the cube $V_{s}(\bm 0)$ into smaller cubes of side length $\alpha s$, denoted as $J_{i,j}$. The subscript $i$ indicates the number of layers between the boundary and $J_{i, \cdot}$. For example, the small cubes in the yellow layer  are denoted as $J_{2,\cdot}$. In addition, the small cubes in the orange layer are denoted as $J_{-1,\cdot}$.
The right picture is an illustration for the proof of Lemma \ref{h-regular-low}. We divide the cube $V_{3s}(\bm 0)$ into smaller cubes of side length $s/K$, denoted as $J_{i,j}$. The subscript $i$ indicates the number of layers between the boundary and $J_{i, \cdot}$.  For example, the small cubes in the yellow layer are denoted as $J_{2,\cdot}$.}
\label{Jij}
\end{figure}

For $i\in [-1,1/(2\alpha)]_{\mathds{Z}}$ and $j\in [1,2d (1/\alpha-2i)^{d-1}]_{\mathds{Z}}$, we let $A_{i,j}$ be the event that there exist two edges $\langle {\bm u}^{i,j}_{1},{\bm v}^{i,j}_{1}\rangle, \langle {\bm u}^{i,j}_{2},{\bm v}^{i,j}_{2}\rangle\in\mathcal{E}$ such that ${\bm u}^{i,j}_{1}\in V_s({\bm 0})^c$, ${\bm v}^{i,j}_{1}\in J_{i,j}$, ${\bm u}^{i,j}_{2}\in \widetilde{J}_{i,j}$ and ${\bm v}^{i,j}_{2}\in V_{3s}({\bm 0})^c$. Then it is clear that
\begin{equation}\label{Acupper}
A^c\subset \bigcup_{i=0}^{1/(2\alpha)}\bigcup_{j=1}^{2d (1/\alpha-2i)^{d-1}}A_{i,j}.
\end{equation}
Thus it suffices to upper-bound $\mathds{P}[A_{i,j}]$ for each $i$ and $j$.

Indeed, by the definition of $J_{i,j}$, one has that  $\text{dist}(J_{i,j},V_s({\bm 0})^c)\geq i\alpha s$ and $\text{dist}(\widetilde{J}_{i,j},V_{3s}({\bm 0})^c)\geq s-\alpha s$ for all $i\in[0,1/(2\alpha)]_\mathds{Z}$.
Combining this with the independence of edges, we get that for $i\in [-1,1/(2\alpha)]_\mathds{Z}$ and $j\in [1,2d (1/\alpha-2i)^{d-1}]_{\mathds{Z}}$,
\begin{equation*}
\begin{split}
\mathds{P}[A_{i,j}]&=\left[1-\exp\left\{-\int_{V_s(\bm 0)^c}\int_{J_{i,j}}\frac{\beta}{|{\bm u}-{\bm v}|^{2d}}\d {\bm u}\d {\bm v}\right\}\right]
\cdot\left[1-\exp\left\{-\int_{V_{3s}(\bm 0)^c}\int_{\widetilde{J}_{i,j}}\frac{\beta}{|{\bm u}-{\bm v}|^{2d}}\d {\bm u}\d {\bm v}\right\}\right]\\
&\leq \left[1-\exp\left\{-c_{d,1} \beta(\alpha s)^d\int_{i\alpha s}^\infty\frac{1}{r^{d+1}}\d r\right\}\right]
\cdot\left[1-\exp\left\{-c_{d,2}\beta (\alpha s)^d\int_{s-\alpha s}^\infty\frac{1}{r^{d+1}}\d r\right\}\right]\\
&= \left[1-\exp\left\{-c_{d,3} \beta i^{-d}\right\}\right]
\cdot\left[1-\exp\left\{-c_{d,4}\beta \alpha^d\right\}\right]\\
&\leq c_{d,5}\beta^2 \alpha^d i^{-d},
\end{split}
\end{equation*}
where $c_{d,l},l=1,\cdots,5$ are positive finite constants depending only on $d$. Combining this with \eqref{Acupper}, we obtain
\begin{equation*}
\begin{split}
\mathds{P}[A^c]&\leq \sum_{i=0}^{1/(2\alpha)}\sum_{j=1}^{2d (1/\alpha-2i)^{d-1}}\mathds{P}[A_{i,j}]
\leq c_{d,6} \beta^2\alpha^d\sum_{i=0}^{1/(2\alpha)}(1/\alpha-2i)^{d-1} i^{-d}\\
&\leq c_{d,7}\beta^2\alpha^d\int_1^{1/(2\alpha)}(1/(\alpha u)-2)^{d-1}u^{-1}\d u
\leq c_{d,\beta}\alpha\log(1/\alpha)
\end{split}
\end{equation*}
where $c_{d,6}$ and $c_{d,7}$ are positive finite constants depending only on $d$, and $c_{d,\beta}$ is a positive finite constant depending only on $d$ and $\beta$.
\end{proof}

In the following, we say $({\bm u}_1,{\bm v}_1,{\bm u}_2,{\bm v}_2)$ is a \textit{$V_{3s}({\bm z})$-special pair of edges} if $\langle {\bm u}_1,{\bm v}_1\rangle, \langle {\bm u}_2,{\bm v}_2\rangle\in\mathcal{E}$ with ${\bm u}_1\in V_s({\bm z})^c$, ${\bm v}_1\in V_s({\bm z})$, ${\bm u}_2\in V_{3s}({\bm z})$, ${\bm v}_2\in V_{3s}({\bm z})^c$.

\begin{lemma}\label{h-regular-low}
For fixed $s>0$ and sufficiently small $\alpha\in(0,1)$, there exist constants $b=b(\alpha)>0$ {\rm(}depending only on $\beta,d,\alpha$ and the law of $D${\rm)} and $c_3>0$ (depending only on $\beta$ and $d$) such that with probability at least $1- c_3\alpha $ the following is true.
For any  $V_{3s}(\bm 0)$-special pair of edges $({\bm u}_1,{\bm v}_1,{\bm u}_2,{\bm v}_2)$ with $|{\bm v}_1-{\bm u}_2|\geq \alpha s$, we have
$$
D({\bm v}_1,{\bm u}_2;V_{3s}({\bm 0}))\geq (b\alpha s)^\theta.
$$
\end{lemma}

\begin{proof}
For notational convenience, denote the event in the lemma by $B$. Let $K$ be a sufficiently large number chosen in \eqref{takeK} below.
We also denote by $\mathcal{C}$ the set of paths in $V_{3s}({\bm 0})$,
and by $\mathcal{C}_{\geq s/K}$ the collection of paths in $\mathcal{C}$ only using long edges with scopes at least $s/K$ and walking on the underlying Euclidean line between long edges.
Note that for every path $P\in\mathcal{C}$, we can construct a path $P_K\in \mathcal{C}_{\geq s/K}$ as follows. We first keep its long edges with scopes at least $s/K$ and then use Euclidean lines to connect those long edges in order. For ease, we say $P_K$ is the $s/K$-backbone  path of $P$.

We now define three events as follows.
\begin{itemize}
\item Let $B_1$ be the event that there exist two edges $\langle {\bm w}_1,{\bm z}_1\rangle, \langle {\bm w}_2,{\bm z}_2\rangle\in\mathcal{E}$ in $V_{3s}({\bm 0})$ with scopes at least $s/K$ such that their Euclidean distance is smaller than $s/K^3$.

\item Let $B_2$ be the event that there exists a $V_{3s}(\bm 0)$-special pair of edges $({\bm u}_1,{\bm v}_1,{\bm u}_2,{\bm v}_2)$ with $|{\bm v}_1-{\bm u}_2|\geq \alpha s$
 such that there exists an $s/K$-backbone path $P_K$ from ${\bm v}_1$ to ${\bm u}_2$ where the Euclidean length of each gap in $P_K$ is smaller than $s/K^3$.

\item Let $B_3$ be the event that for some $V_{3s}(\bm 0)$-special pair of edges $({\bm u}_1,{\bm v}_1,{\bm u}_2,{\bm v}_2)$ with $|{\bm v}_1-{\bm u}_2|\geq \alpha s$, and for  all $s/K$-backbone  paths from ${\bm v}_1$ to ${\bm u}_2$ with at least one gap whose Euclidean length is larger than $s/K^3$, their each gap $({\bm x},{\bm y})$ (i.e. the line between ${\bm x}$ and ${\bm y}$) satisfies
    $$D({\bm x},{\bm y};V_{3s}({\bm 0}))\leq c_K s^\theta. $$
     Here $c_K$ is a positive constant chosen below \eqref{cp}, which depends only on $\beta,d,K$ and the law of $D$.
\end{itemize}
Then it is clear that $B^c\subset B_1\cup B_2\cup B_3$ if we choose $K$ and $b$ such that $(b\alpha)^\theta=c_K$ (see \eqref{takeK}).

In the following, we estimate the probability of events  $B_1, B_2$ and  $B_3$. To do that, we divide the cube $V_{3s}({\bm 0})$ into $(\frac{ 3s}{K^{-3}s})^d=3^d K^{3d}$ small cubes of equal side length, denoted by $J_{i,j}$ for $i\in [0,3K^3/2]_\mathds{Z}$ and $j\in[1,2d(3K^3-2i)^{d-1}]_\mathds{Z}$. Here the subscript $i$ in $J_{i,j}$ represents that there are $i$ layers of $J_{i', \cdot}$ between $J_{i, \cdot}$ and the boundary of $V_{3s}({\bm 0})$, while the subscript $j$ represents the $j$-th cube among those so that $J_{i, j}$ is adjacent to $J_{i, j+1}$ (see the right picture in Figure \ref{Jij}).

For each $i\in [0,3K^3/2]_\mathds{Z}$ and $j\in[1,2d(3K^3-2i)^{d-1}]_\mathds{Z}$, we let $A_{i,j}$ be the event that there exist two edges with scopes at least $s/K$ connecting $J_{i,j}$ and $J_{i,j}^c$. Then by the independence of edges,  using the similar calculation for $\mathds{P}[A^c]$ in the proof of Lemma \ref{h-longedge} we have
\begin{equation}\label{B1}
\begin{split}
\mathds{P}[B_1]&\leq \sum_{i=0}^{3K^3/2}\sum_{j=1}^{2d(3K^3-2i)^{d-1}}\mathds{P}[A_{i,j}]\\
&\leq \sum_{i=0}^{3K^3/2}\sum_{j=1}^{2d(3K^3-2i)^{d-1}}\left[1-\exp\left\{-\int_{J_{i,j}}\int_{|{\bm v}-{\bm u}|\geq s/K}\frac{\beta}{|{\bm u}-{\bm v}|^{2d}}\d {\bm u}\d {\bm v}\right\}\right]^2\\
&\leq \sum_{i=0}^{3K^3/2}\sum_{j=1}^{2d(3K^3-2i)^{d-1}}\left[1-\exp\left\{-c_{d,1}\beta K^{-3d}s^d\int_{s/K}^\infty\frac{1}{r^{d+1}} \d r\right\}\right]^2\leq c_{d,\beta,1} K^{-d}.
\end{split}
\end{equation}
Moreover, recall that $V_{r}(A)=\{{\bm z}\in\mathds{R}^d:\ \text{dist}({\bm z},A;\|\cdot\|_\infty)\leq r/2\}$ for $r>0$ and $A\subset \mathds{R}^d$. On the event $B_1^c\cap B_2$, we can see that there is at most one long edge (with scope larger than $s/K$) in $P_K$, since (on $B_2$) $P_K$  satisfies that its each gap's Euclidean length is smaller than $s/K^3$. Thus $P_K$ uses only one edge $\langle {\bm x},{\bm y}\rangle$ with scope large than $s/K$, where ${\bm x}\in V_{sK^{-3}}({\bm v}_1)$ and ${\bm y}\in V_{sK^{-3}}({\bm u}_2)$.  This implies

\begin{equation}\label{B1cB2}
\begin{split}
&\mathds{P}[B_1^c\cap B_2]\\
&\leq \sum_{i,j:J_{i,j}\in V_s({\bm 0})}\sum_{i',j':J_{i',j'}\in V_{(K^{-1}\vee \alpha)s}(J_{i,j})^c}\left[1-\exp\left\{-\int_{\widetilde{J}_{i,j}}\int_{\widetilde{J}_{i',j'}}\frac{\beta}{|{\bm u}-{\bm v}|^{2d}}\d {\bm u}\d {\bm v}\right\}\right]\\
&\quad\quad\quad \quad\quad\quad \times\left[1-\exp\left\{-\int_{\widetilde{J}_{i',j'}}\int_{V_{3s}({\bm 0})^c}\frac{\beta}{|{\bm u}-{\bm v}|^{2d}}\d {\bm u}\d {\bm v}\right\}\right]\\
&\leq \sum_{i,j:J_{i,j}\in V_s({\bm 0})}\sum_{i',j':J_{i',j'}\in V_{2s}({\bm 0})\setminus V_{s/K}(J_{i,j})}\left[1-\exp\left\{-c_{d,3}\beta K^{-4d}\right\}\right]\cdot\left[1-\exp\left\{-c_{d,4}\beta K^{-3d}\right\}\right]\\
& \quad +\sum_{i,j:J_{i,j}\in V_s({\bm 0})}\sum_{i',j':J_{i',j'}\in V_{2s}({\bm 0})^c}\left[1-\exp\left\{-\frac{c_{d,5}\beta }{|i-i'-2|^{2d}}\right\}\right]\cdot \left[1-\exp\left\{-c\frac{_{d,6}\beta}{(i')^d}\right\}\right]\\
&=: (I)+(II).
\end{split}
\end{equation}
In the following, we estimate terms $(I)$ and $(II)$ separately. Firstly, note that $\#\{(i,j):J_{i,j}\in V_s({\bm 0})\}=K^{3d}$ and $\#\{(i',j'):J_{i',j'}\subset V_{2s}({\bm 0})\ \text{and}\ \text{dist}(J_{i',j'}, J_{i,j})\geq s/K \}\leq 2^d K^{3d}$.
Combining this with the inequality $1-\e^{-x}\leq x$ for $x>0$, we obtain
\begin{equation}\label{I}
(I)\leq  c_{d,7}\beta^2K^{6d}K^{-7d}= c_{d,7}\beta^2K^{-d}.
\end{equation}
For the term $(II)$, from $1-\e^{-x}\leq x$ for all $x>0$ we have that

\begin{equation}\label{II}
\begin{split}
(II)&\leq c_{d,8}\beta^2\sum_{i,j:J_{i,j}\in V_s({\bm 0})}\sum_{i'\geq 1,j':J_{i',j'}\in V_{2s}({\bm 0})^c}\frac{1}{|i-i'-2|^{2d}(i')^d}\\
&\quad +c_{d,8}\beta^2(3K^3)^{d-1}\sum_{i,j:J_{i,j}\in V_s({\bm 0})}\frac{1}{|i-2|^{2d}}\\
&\leq c_{d,9}\beta^2\int_{K^3}^{3K^3/2}\d y\int_{1}^{K^3/2}\frac{(3K^3-2x)^{d-1}(3K^3-2y)^{d-1}}{(y-x-2)^{2d}x^d}\d x+c_{d,10}\beta^2 K^{3(d-1)}K^{3d} K^{-6d}\\
&\leq c_{d,11}\beta^2K^{3(d-1)-6d}\int_{K^3}^{3K^3/2}(3K^3-2y)^{d-1}\d y\int_{1}^{K^3/2}\frac{1}{x^d}\d x+c_{d,10}K^{-3}\\
&\leq c_{d,12}\beta^2 K^{-3}\log K.
\end{split}
\end{equation}
Applying \eqref{I} and \eqref{II} to \eqref{B1cB2}, we have
\begin{equation}\label{B1cB2-2}
\mathds{P}[B_1^c\cap B_2]\leq c_{d,\beta,2}K^{-(d\wedge 3)}\log K
\end{equation}
for some constant $c_{d,\beta,2}>0$, which depends only on $d$ and $\beta$.

Additionally, denote by $\xi$ the number of long edges with scopes at least $s/K$ in $V_{3s}({\bm 0})$. By the property of the Poisson point process,
$$
\mathds{E}[\xi]=\int_{V_{3s}({\bm 0})}\d {\bm u}\int_{|{\bm v}-{\bm u}|\geq s/K}\frac{1}{|{\bm u}-{\bm v}|^{2d}}\d {\bm v}\leq c_{d,6} K^{d}.
$$
Therefore, from Markov's inequality we get
$$
\mathds{P}[\xi\geq K^{2d}]\leq K^{-2d}\mathds{E}[\xi]\leq c_{d,6}K^{-d}.
$$
Now let $\mathtt{p}>0$ be a sufficiently small number, which will be chosen below. From Axioms IV' (translation invariance) and V1'(tightness across different scales (lower bound)), we can choose a constant $c_{\mathtt{p}}>0$ such that for any ${\bm x},{\bm y}\in\mathds{R}^d$ with $|{\bm x}-{\bm y}|\geq s/K^3$,
\begin{equation}\label{cp}
\mathds{P}\left[D({\bm x},{\bm y};V_{3s}({\bm 0}))\leq c_{\mathtt{p}}s^\theta\right]\leq \mathds{P}[D(\bm 0,{\bm x}-{\bm y})\leq c_{\mathtt{p}}K^{3\theta}|{\bm x}-{\bm y}|^\theta]\leq \mathtt{p}.
\end{equation}
In the following, we take $\mathtt{p}=K^{-6d}$ and $c_K=c_{\mathtt{p}}$. Combining \eqref{cp}  with the fact that the end points of any gap must belong to one of the $\xi$ long edges, we obtain
\begin{equation*}
\begin{split}
\mathds{P}[B_1^c\cap B_2^c\cap B_3]&\leq \mathds{P}[\xi\geq K^{2d}]+\mathds{P}[B_1^c\cap B_2^c\cap B_3\cap \{\xi\leq K^{2d}\}]\\
&\leq c_{d,13} K^{-d}+ \mathtt{p}\sum_{r=1}^{K^{2d}}4r^2\leq c_{d,13} K^{-d}+c_{d,14}\mathtt{p} K^{5d}\leq c_{d,15}K^{-d}.
\end{split}
\end{equation*}
Then it follows from this, \eqref{B1}, \eqref{B1cB2-2} and the fact $B^c\subset B_1\cup B_2\cup B_3$ that
$$
\mathds{P}[B^c]\leq c_{d,\beta,3}K^{-(d\wedge 3)}\log K.
$$

We finally take sufficiently large $K$ such that
\begin{equation}\label{takeK}
K^{-(d\wedge 3)}\log K=\alpha,
\end{equation}
and $b=b(\alpha)>0$ such that $(b\alpha)^\theta=c_K$. Then from the above analysis we obtain that $\mathds{P}[B^c]\leq c_{d,\beta} \alpha$ for some constant $c_{d,\beta}>0$ which depends only on $d$ and $\beta$.
\end{proof}

Now we present the
\begin{proof}[Proof of Lemma \ref{h-probgood}]
If we denote the events in Lemmas \ref{h-longedge} and \ref{h-regular-low} by $A$ and $B$, respectively, then it is clear that
$$
\mathds{P}[V_{3s}({\bm z})\ \text{is $(3s,\alpha)$-good}]\geq 1-(\mathds{P}[A^c]+\mathds{P}[A\cap B^c])\geq 1-c\alpha\log(1/\alpha),
$$
where $c$ is a constant depending only on $d,\beta$ and the law of $D$.
\end{proof}

\subsection{Proof of Propositions \ref{Dxy=0}, \ref{bilip} and \ref{cCdtm}}\label{proof111}

In this subsection, we will use a renormalization argument to prove these propositions. For this purpose, fix $\gamma>0$ and $R>0$ (we will eventually send $R$ to infinity and $\gamma$ to 0).

We now introduce the renormalization. For fixed $n\in\mathds{N}$ with $\varepsilon:=1/n \ll \gamma$ (we will eventually send $n$ to infinity), we divide $\mathds{R}^d$ into small cubes of side length $\varepsilon R$, denoted by $V_{\varepsilon R}(\bm k)$ for ${\bm k}=({\bm k}^1,\cdots,{\bm k}^d)\in(\varepsilon R)\mathds{Z}^d$.
Then we identify the cubes $V_{\varepsilon R}(\bm k)$ as vertices ${\bm k}$ and call the resulting graph $G$. By the self-similarity of the model, it is obvious that $G$ is the critical long-range bond percolation model on $(\varepsilon R)\mathds{Z}^d$.  We say that ${\bm k}$ is good if $V_{3\varepsilon R}(\bm k)$ is $(3\varepsilon R, \alpha)$-good.

As a corollary of Lemma \ref{h-probgood}, we have the following result: a vertex in $G$ is good with arbitrarily high probability if we choose the parameters $\alpha$ and $b$ suitably.

\begin{corollary}\label{h-good-prob-2}
For $\alpha\in(0,1)$, let $b_0=b_0(\alpha)$ be chosen in Lemma \ref{h-regular-low}. Then for each $\delta_0\in(0,1)$, there exists sufficiently small $\alpha=\alpha(\delta_0)>0$ {\rm(}depending only on $d,\beta,\delta_0$ and the law of $D${\rm)} such that
$$
\mathds{P}[{\bm k}\text{ is good}]=\mathds{P}[V_{3\varepsilon R}(\bm k)\text{ is $(3\varepsilon R, \alpha)$-good}]>1-\delta_0\quad \text{for all }{\bm k}\in(\varepsilon R)\mathds{Z}^d.
$$
\end{corollary}

For the rest of this section,  we select a small $\delta_0>0$ such that $2C_{dis}\delta_0<1$, where $C_{dis}$ is the constant defined in Lemma \ref{number-path-k} (which depends only on $\beta$ and $d$). We will also refer to the constants $\alpha$ and  $b_0$ as those chosen in Corollary \ref{h-good-prob-2}.

To distinguish the paths in different models, we write $P$ and $P^G$ for the paths in the continuous model and in $G$, respectively. We also use $d^G$ to denote the chemical distance of $G$.
We will see in the following definition that for each path $P$ in the continuous model, we can construct a self-avoiding path in $G$, which is the ``skeleton'' of $P$.

\begin{definition}\label{pathc-d}
Let $P$ be a path in the continuous model. We construct \textit{the skeleton path $P^G$} in $G$ (which is self-avoiding) of $P$ as follows. We start by defining a sequence $\widetilde{P}^{G}=(\widetilde{\bm k}_1,\widetilde{\bm k}_2,\cdots,\widetilde{\bm k}_N)$ that represents the centers of the $V_{\varepsilon R}(\bm k)$ that the path $P$ hits in order.
Since $P$ can hit the same cube multiple times, it is possible that this sequence contains loops. Therefore, $\widetilde{P}^{G}$ can be viewed as a path with loops in $G$. We next apply the following loop erasing procedure to $\widetilde{P}^{G}$:
\begin{itemize}
\item if $\widetilde{P}^{G}$ is self-avoiding (i.e., visiting each point at most once), set $P^G=\widetilde{P}^{G}$;

\item otherwise, recursively define $s_1=\max\{j\leq N:\ \widetilde{\bm k}_j=\widetilde{\bm k}_1\}$ and let the start point of $P^G$ be ${\bm k}_1=\widetilde{\bm k}_{s_1}$;

\item for $i\geq 1$, if $s_i<N$, define $s_{i+1}=\max\{j\leq N: \widetilde{\bm k}_j=\widetilde{\bm k}_{s_i+1}\}$ and set the $(i+1)$th point of $P^G$ to be ${\bm k}_{i+1}=\widetilde{\bm k}_{s_{i+1}}$.
\end{itemize}
From the above construction, it is clear that $P^G$ is a self-avoiding path in $G$ as desired.
\end{definition}

Conversely, for any path $P^G$ between two different points ${\bm i}$ and ${\bm j}$, as we define next, there exists a ``simple'' path $P$ from $V_{\varepsilon R}(\bm i)$ to $V_{\varepsilon R}(\bm j)$ in the continuous model (here simple means without any loop in the sequence of the centers of the cubes $V_{\varepsilon R}(\bm k)$ that $P$ hits in order) whose skeleton path is $P^G$. The construction is similar to that in the proof of Proposition \ref{apriori-cont-1}.
\begin{definition}\label{pathd-c}
Let ${\bm i},{\bm j}\in (\varepsilon R)\mathds{Z}^d$ and let $P^{G}=({\bm k}_0=\bm i,{\bm k}_1,\cdots,{\bm k}_N={\bm j})$ be a path from ${\bm i}$ to ${\bm j}$ in $G$.
We will first define ${\bm x}_l,{\bm y}_l\in V_{\varepsilon R}({\bm k}_l)$ for $l\in [0,N]_\mathds{Z}$ such that either $\|{\bm k}_l-{\bm k}_{l+1}\|_{1}=\varepsilon R$ or ${\bm y}_l$ and ${\bm x}_{l+1}$ are connected by a long edge in the continuous model, in accordance to the path $P^{G}$.
    The iterative definition of ${\bm x}_l,{\bm y}_l$ is as follows:

    Let ${\bm x}_0$ be the center of $V_{\varepsilon R}({\bm i})$. Assuming that ${\bm x}_0,{\bm y}_0,\cdots, {\bm y}_{l-1},{\bm x}_l$ $(l\geq 0)$ have been defined such that ${\bm x}_i,{\bm y}_i\in V_{\varepsilon R}({\bm k}_i)$ for all $i\in [0,l-1]_\mathds{Z}$ and ${\bm x}_l\in V_{\varepsilon R}({\bm k}_l)$, we consider the following two cases.
        \begin{itemize}
            \item[(1)] If $\|{\bm k}_l-{\bm k}_{l+1}\|_{1}=\varepsilon R$, we let ${\bm y}_l={\bm x}_l$ and let ${\bm x}_{l+1}$ be the center of $V_{\varepsilon R}({\bm k}_{l+1})$.
            \item[(2)] If $\|{\bm k}_l-{\bm k}_{l+1}\|_{1}>\varepsilon R$,  then this implies that ${\bm k}_l$ and ${\bm k}_{l+1}$ are connected by a long edge in the discrete model $G$ from the renormalization. Thus, we can choose ${\bm y}_l\in V_{\varepsilon R}({\bm k}_l)$ and ${\bm x}_{l+1}\in V_{\varepsilon R}({\bm k}_{l+1})$ such that $\langle {\bm y}_l,{\bm x}_{l+1}\rangle\in \mathcal{E}$.
        \end{itemize}
 Finally, let ${\bm y}_N$ be the center of $V_{\varepsilon R}(\bm j)$. Then we define  $P$  to be a path ${\bm x}_0\to {\bm y}_0\to {\bm x}_1\to {\bm y}_1\to\cdots\to {\bm x}_N\to {\bm y}_N$ in the continuous model, where ${\bm x}_l\to {\bm y}_l$ walks on a $D(\cdot,\cdot;V_{\varepsilon R}({\bm k}_l))$-geodesic and ${\bm y}_l\to {\bm x}_{l+1}$ walks on the long edge $\langle {\bm y}_l, {\bm x}_{l+1}\rangle$ (if $\|{\bm k}_l-{\bm k}_{l+1}\|_{1}>\varepsilon R$)
or on a $D(\cdot,\cdot;V_{\varepsilon R}({\bm k}_l))$-geodesic from ${\bm y}_l$ to a fixed point ${\bm z}_l\in \overline{V_{\varepsilon R}({\bm k}_l)}\cap \overline{V_{\varepsilon R}({\bm k}_{l+1})}$ (chosen arbitrarily in a prefixed manner) and then a $D(\cdot,\cdot;V_{\varepsilon R}({\bm k}_{l+1}))$-geodesic from  ${\bm z}_l$ to ${\bm x}_{l+1}$ if $\|{\bm k}_l-{\bm k}_{l+1}\|_{1}=\varepsilon R$.
It is obvious that the sequence of the centers of the cubes $V_{\varepsilon R}(\bm k)$ that $P$ hits in order is the sequence of points in $P^G$, which does not have any loop. Thus we call $P$ is a \textit{simple path} in the continuous model associated with $P^G$ (see the right picture in Figure \ref{sect2construction} above).
\end{definition}

In the following, we denote by $\mathcal{C}(P^G)$ the set of simple paths in the continuous model associated with $P^G$, and unless otherwise specified, we always assume that the path $P^G$ in $G$ is self-avoiding.

For any ${\bm i},{\bm j}\in(\varepsilon R)\mathds{Z}^d$ and $m\in\mathds{N}$, we denote by $\mathcal{P}_m({\bm i},{\bm j})$ the collection of self-avoiding paths $P^G$ from ${\bm i}$ to ${\bm j}$ in $G$ with length $m$. We let $\mathcal{P}_{\geq m}({\bm i},{\bm j})=\bigcup_{l\geq m}\mathcal{P}_l({\bm i},{\bm j})$.

Let $F_{\varepsilon,1}$ be the event that every path $P^G$ in $\bigcup_{{\bm i},{\bm j}\in (\varepsilon R)[-1/\varepsilon,1/\varepsilon)^d_\mathds{Z}}\mathcal{P}_{\geq \varepsilon^{-\theta/4}}({\bm i},{\bm j})$ passes through at least $|P^G|/(4\cdot 3^d)$ of the good sites. The following lemma shows that $F_{\varepsilon,1}$ occurs with very high probability.


\begin{lemma}\label{Fr4lemma-2}
Let $\delta_0\in(0,1)$ with $2C_{dis}\delta_0<1$. Let $\alpha$ and $b_0$ be chosen as in Corollary \ref{h-good-prob-2}.
Then $F_{\varepsilon,1}$ occurs with probability at least $1-O_\varepsilon(\varepsilon^{\mu})$ for all $\mu>0$, where the implicit constant in $O_\varepsilon(\cdot)$ depends only on $\beta, d, \mu, \delta_0$ and the law of $D$.
\end{lemma}




To prove Lemma \ref{Fr4lemma-2}, we first consider some estimates for a fixed path. To achieve this, we start with a notion for a ``good'' path, which we define in general since such notions will be repeatedly used later.

    \begin{definition}\label{hd-PGgood}
    Let $\alpha\in(0,1)$ and  $b_0=b_0(\alpha)$ be the constants chosen in Corollary \ref{h-good-prob-2}. Additionally, let $\{E_{\bm k}\}_{{\bm k}\in(\varepsilon R)\mathds{Z}^d}$ be a collection of events such that $E_{\bm k}$ is determined by edges in $V_{3\varepsilon R}(\bm k)$ and $D(\cdot,\cdot; V_{3\varepsilon R}(\bm k))$.
    We say a self-avoiding path $P^G=({\bm k}_1,{\bm k}_2,\cdots,{\bm k}_L)$ on $G$ is \textit{$(\{E_{\bm k}\}_{{\bm k}\in(\varepsilon R)\mathds{Z}^d},\alpha)$-good} if at least $\frac{1}{4\cdot 3^d}L$ of its indices $j\in [1,L]_\mathds{Z}$ satisfy the following conditions.
        \begin{enumerate}
            \item $E_{{\bm k}_j}$ occurs.
            \item For each $P\in \mathcal{C}(P^G)$, let $({\bm u}_{{\bm k}_j,1},{\bm v}_{{\bm k}_j,1},{\bm u}_{{\bm k}_j,2},{\bm v}_{{\bm k}_j,2})$ be a $V_{3\varepsilon R}(\bm k_j)$-special pair of edges such that $P$ first uses $\langle {\bm u}_{{\bm k}_j,1},{\bm v}_{{\bm k}_j,1}\rangle$ to enter $V_{\varepsilon R}(\bm k_j)$ and then does not leave $V_{3\varepsilon R}(\bm k_j)$ until using $\langle {\bm u}_{{\bm k}_j,2},{\bm v}_{{\bm k}_j,2}\rangle$. 
            Then
            $$
            |{\bm v}_{{\bm k}_j,1}-{\bm u}_{{\bm k}_j,2}|\geq \alpha \varepsilon R\quad\text{and}\quad D({\bm v}_{{\bm k}_j,1},{\bm u}_{{\bm k}_j,2};V_{3\varepsilon R}(\bm k_j))\geq (b_0\alpha \varepsilon R)^\theta.
            $$
            \item For any two different indices $j_1,j_2$, we have $\bm k_{j_1}-\bm k_{j_2}\in(3\varepsilon R)\mathds{Z}^d$.
        \end{enumerate}

    \end{definition}

    If $P^G$ is not $(\{E_{\bm k}\}_{{\bm k}\in(\varepsilon R)\mathds{Z}^d},\alpha)$-good, we say $P^G$ is $(\{E_{\bm k}\}_{{\bm k}\in(\varepsilon R)\mathds{Z}^d},\alpha)$-bad.

    \begin{lemma}\label{hd-BK}
        Let $\delta\in(0,1)$. Assume that $\{E_{\bm k}\}_{{\bm k}\in(\varepsilon R)\mathds{Z}^d}$ is a collection of events such that $E_{\bm k}$ is determined by edges in $V_{3\varepsilon R}(\bm k)$ and $D(\cdot,\cdot;V_{3\varepsilon R}(\bm k))$.
        Then there exist constants $p_c>0,\alpha_0>0$ {\rm(}depending only on $d,\beta,\delta$ and the law of $D${\rm)} such that if $\mathds{P}[E_{\bm k}]\geq p_c$ for all ${\bm k}\in(\varepsilon R)\mathds{Z}^d$ and $\alpha\in(0,\alpha_0)$, then for any fixed self-avoiding path $P^G$ with $d^G$-length $L$,
        We have
        $$
        \mathds{P}\left[\text{$P^G$ is $(\{E_{\bm k}\}_{{\bm k}\in(\varepsilon R)\mathds{Z}^d},\alpha)$-good}\right]\geq 1-\delta^L.
        $$
    \end{lemma}

    \begin{proof}
        Fix any possible self-avoiding path $P^G$ with $d^G$-length $L$. For any $\bm{i}\in(\varepsilon R)\{0,1,2\}^d$, let $P^{({\bm i})}=\{{\bm k}\in P^G: \bm{k}-\bm{i}\in(3\varepsilon R)\mathds{Z}^d\}$. Then $\{P^{({\bm i})}:{\bm{i}}\in(\varepsilon R)\{0,1,2\}^d\}$ is a partition of $P^G$.
        From Pigeonhole's Principle, there exists an ${\bm i}_*\in(\varepsilon R)\{0,1,2\}^d$ such that $|P^{({\bm i}_*)}|\geq 3^{-d}L$. Note that $V_{3\varepsilon R}(\bm k)$ for all ${\bm k}\in P^{({\bm i}_*)}$ are disjoint.

        Let $A$ be the event that there exist at most $L/(4\cdot3^{d} )$ many ${\bm k}$'s in $P^{({\bm i}_*)}$ such that $E_{\bm k}$ occurs.  Then from the independence between different $E_{\bm k}$,
        \begin{equation}\label{hd-BKLemma-1}
            \mathds{P}[A]\leq 2^L(1-p_c)^{3L/(4\cdot 3^d)}.
        \end{equation}

        On the event $A^c$, if the path $P^G$ is $(\{E_{\bm k}\}_{{\bm k}\in(\varepsilon R)\mathds{Z}^d},\alpha)$-bad, then for any $\Lambda\subset P^{({\bm i}_*)}$ with $|\Lambda|\geq L/(2\cdot 3^d)$ such that $E_{\bm k}$ occurs for any ${\bm k}\in\Lambda$, there exists $P\in\mathcal{C}(P^G)$ and at least $L/(4\cdot 3^d)$ ${\bm k}$'s in $\Lambda$ such that for these corresponding $V_{3\varepsilon R}(\bm k)$-special pairs of edges $({\bm u}_{{\bm k},1},{\bm v}_{{\bm k},1},{\bm u}_{{\bm k},2},{\bm v}_{{\bm k},2})$, 
        either
        $$
        |{\bm v}_{{\bm k},1}-{\bm u}_{{\bm k},2}|<\alpha\varepsilon R\quad \text{or} \quad D({\bm v}_{{\bm k},1},{\bm u}_{{\bm k},2};V_{3\varepsilon R}({\bm k}))\leq (b_0\alpha\varepsilon R)^\theta.
        $$
         As a result, there are at least $L/(4\cdot 3^d)$ ${\bm k}$'s in $\Lambda$ such that $V_{3\varepsilon R}(\bm k)$ is not $(3\varepsilon R,\alpha)$-good with disjoint witness (recall Definition \ref{h-good}).
         Then from Lemma \ref{h-probgood},
           Axiom II'' (weak locality) and BK inequality \cite{BK85}, we get
        \begin{equation}\label{hd-BKLemma-2}
            \mathds{P}\left[A^c\cap\{P^G\text{ is }(\{E_{\bm k}\}_{{\bm k}\in(\varepsilon R)\mathds{Z}^d},\alpha)\text{-bad}\}\right]\leq 2^L( c\alpha\log(1/\alpha))^{L/(4\cdot 3^{d})}.
        \end{equation}

        Combining \eqref{hd-BKLemma-1} and \eqref{hd-BKLemma-2}, we get the lemma by taking sufficiently small $p_c$ and $\alpha$ such that  $2(1-p_c)^{3/(4\cdot 3^d)}<\delta/2$ and $2( c\alpha\log(1/\alpha))^{1/(4\cdot 3^d)}<\delta/2$. 
    \end{proof}

We now can present the

\begin{proof}[Proof of Lemma \ref{Fr4lemma-2}]
Let ${\bm i},{\bm j}\in (\varepsilon R)[-1/\varepsilon,1/\varepsilon)_{\mathds{Z}}^d$ and let $P_{{\bm i\bm j}}^G$ be a path from ${\bm i}$ to ${\bm j}$ in $G$ with $|P_{\bm i\bm j}^G|\geq \varepsilon^{-\theta/4}$.
Applying Lemma \ref{hd-BK} with $\delta=\delta_0$ and $E_{\bm k}$ being the whole probability space for all ${\bm k}\in (\varepsilon R)\mathds{Z}^d$, we see that for each $m\geq \varepsilon^{-\theta/4}$ and each path $P_{\bm i\bm j}^G\in \mathcal{P}_m(\bm i,\bm j)$,
\begin{equation}\label{ijpath-active-2}
\begin{split}
&\mathds{P}\left[\text{the number of good sites in $P_{\bm i\bm j}^G$ is at most }|P_{\bm i\bm j}^G|/(4\cdot 3^d) \Big| G\right]\\
&=\mathds{P}\left[P^G_{\bm i\bm j} \text{ is }(\{E_{\bm k}\}_{{\bm k}\in (\varepsilon R)\mathds{Z}^d},\alpha)\text{-bad}\Big|G\right]
\leq \delta_0^m.
\end{split}
\end{equation}

In addition, by Lemma \ref{number-path-k} and Markov's inequality, we get that
\begin{equation}\label{number-path-m-2}
\mathds{P}\left[|\mathcal{P}_m(\bm i,\bm j)|\geq (2C_{dis})^m\right]\leq (2C_{dis})^{-m}\mathds{E}[|\mathcal{P}_m(\bm i,\bm j)|]\leq 2^{-m},
\end{equation}
where $C_{dis}$ is the constant in Lemma \ref{number-path-k}, depending only on $\beta$ and $d$. This implies
\begin{equation}\label{number-path>=-2}
\begin{split}
&\mathds{P}\left[ \exists m\geq \varepsilon^{-\theta/4}\ \text{such that }|\mathcal{P}_m(\bm i ,\bm j)|\geq (2C_{dis})^m\right]\\
&\leq \sum_{ m\geq  \varepsilon^{-\theta/4}} 2^{-m}=O_\varepsilon(\varepsilon^{\mu})\quad \forall\mu>0.
\end{split}
\end{equation}
Here the implicit constant in $O_\varepsilon(\cdot)$ depends only on $\beta, d, \mu$ and the law of $D$.

For brevity, we denote the events in \eqref{ijpath-active-2} and \eqref{number-path>=-2} by $M(P^G_{\bm i\bm j})$ and $N_{\bm i\bm j}$, respectively.
Combining \eqref{ijpath-active-2} and \eqref{number-path>=-2} yields that
\begin{align*}
\mathds{P}[F_{\varepsilon,1}^c]&\leq \sum_{\bm i,\bm j\in(\varepsilon R)[-1/\varepsilon,1/\varepsilon)^d_{\mathds{Z}}}\mathds{P}[N_{\bm i\bm j}]+
\sum_{\bm i,\bm j\in(\varepsilon R)[-1/\varepsilon,1/\varepsilon)^d_{\mathds{Z}}}\sum_{m\geq \varepsilon^{-\theta/4}}\mathds{E}\left[\I_{(N_{\bm i\bm j})^c}\sum_{P_{\bm i\bm j}^G\in\mathcal{P}_m(\bm i,\bm j)}\mathds{P}\left[ M(P_{\bm i\bm j}^G)\big| G\right]\right]\\
&\leq O_\varepsilon(\varepsilon^{\mu})
+ \sum_{\bm i,\bm j\in(\varepsilon R)[-1/\varepsilon,1/\varepsilon)^d_{\mathds{Z}}}
\sum_{m\geq \varepsilon^{-\theta/4}}(2C_{dis})^m\cdot \delta_0^{m}\\
&\leq O_\varepsilon(\varepsilon^{\mu})\quad \forall \mu>0,
\end{align*}
where the last inequality holds
since $2C_{dis} \delta_0 < 1$ (in light of choices of  $\alpha, b_0$ in  Corollary \ref{h-good-prob-2}).
Hence the proof is complete.
\end{proof}

Since $\varepsilon\ll \gamma$, from Lemma \ref{Fr4lemma-2} we get that
\begin{equation*}\label{P[F_n-2]}
\mathds{P}[F_{\varepsilon,1}]\geq 1-\gamma.
\end{equation*}

In what follows, to obtain a lower bound on the number of cubes $V_{3\varepsilon R}(\bm k)$ traversed by a $D$-geodesic from ${\bm x}$ to ${\bm y}$ for all ${\bm x},{\bm y}\in[-R,R]^d$, we require a uniform upper bound on the $D$-diameter of such cubes $V_{3\varepsilon R}(\bm k)$ (see Lemma \ref{diamIk-2} below). This is because we hope that these geodesics lie entirely within a compact set.
To achieve this, we need the following lemma, which  shows that a geodesic cannot stray far from the two end points.

\begin{lemma}\label{geodesic-go-not-far}
    For any fixed $\gamma>0$, there exists $u=u(\gamma)>1$, depending only on $\beta,\gamma$ and the law of $D$, such that for any $R>0$,
    \begin{equation*}
        \mathds{P}[D([-R,R]^d,([-uR,uR]^d)^c)< \mathrm{diam}([-R,R]^d;D)]<\gamma,
    \end{equation*}
    which implies with probability at least $1-\gamma$, for any ${\bm x},{\bm y}\in[-R,R]^d$, each $D$-geodesic from ${\bm x}$ to ${\bm y}$ is contained in $[-uR,uR]^d$.
\end{lemma}
\begin{proof}
    From Axiom V2' (tightness across different scales (upper bound)), we know that there exists $C=C(\gamma)>0$, which depends only on $\beta,d,\gamma$ and the law of $D$, such that for any $R>0$,
    \begin{equation*}
        \mathds{P}[ \mathrm{diam}([-R,R]^d;D)\le C R^\theta]>1-\gamma/2.
    \end{equation*}
    Using the triangle inequality
    $$
    D([-R,R]^d,([-uR,uR]^d)^c)\ge D({\bm0},([-uR,uR]^d)^c)-\mathrm{diam}([-R,R]^d;D),
    $$
    it suffices to show that there exists $u>0$ (which does not depend on $R$) such that for any $R>0$,
    \begin{equation}\label{target-2-10}
        \mathds{P}[D(\bm 0,([-uR,uR]^d)^c)\geq  2CR^\theta]>1-\gamma/2.
    \end{equation}
Indeed, assuming \eqref{target-2-10} we have that with probability at least $1-\gamma$,
$$
D([-R,R]^d,([-uR,uR]^d)^c)\geq CR^\theta\geq \mathrm{diam}([-R,R]^d;D),
$$
which implies the desired statement.

We next prove \eqref{target-2-10}.  By Axiom V1' (tightness across different scales (lower bound)), there exists a constant $c_\gamma>0$ such that for every $r>0$,
    \begin{equation}\label{exitlowerbound-2}
        \mathds{P}\left[D(\bm 0,([-r,r]^d)^c)\ge c_\gamma r^\theta\right]>1-\frac{\gamma}{2}.
    \end{equation}
    Thus, if we take $u>0$ with $u^\theta c_\gamma >2C$ (which does not depend on $r$), then we can obtain \eqref{target-2-10} by taking $r=uR$ in \eqref{exitlowerbound-2}, which implies the lemma.
\end{proof}

Hence, we see from Lemma \ref{geodesic-go-not-far} that for any fixed $\gamma>0$, there is a sufficiently large $u=u(\gamma)\in\mathds{N}$ (which depends only on $\beta,d,\gamma$ and the law of $D$)
such that with probability at least $1-\gamma$, each $D$-geodesic between any two points in $[-R,R]^d$ is contained in $[-uR, uR]^d$.
We will refer to this event as $F_{\gamma,2}$. Thus, we have
\begin{equation*}\label{P[F_gamma-2]}
\mathds{P}[F_{\gamma,2}]\geq 1-\gamma.
\end{equation*}

For any  ${\bm x},{\bm y}\in[-R,R]^d$, let ${\bm k}_{\bm x},{\bm k}_{\bm y}\in (\varepsilon R)[-u/\varepsilon,u/\varepsilon)^d_\mathds{Z}$ be such that  ${\bm x}\in V_{\varepsilon R}({\bm k}_{\bm x}), {\bm y}\in V_{\varepsilon R}({\bm k}_{\bm y})$.
We next let $F_{\varepsilon,3}$ be the event that for any ${\bm x},{\bm y}\in[-R,R]^d$ with $D({\bm x},{\bm y})\ge 2\varepsilon^{\theta/2}R^\theta$,
    \begin{equation*}
        d^G({\bm k}_{\bm x},{\bm k}_{\bm y}) +1\ge \max\left\{ c_4(\varepsilon R)^{-\theta} D({\bm x},{\bm y}),\varepsilon^{-\theta/4}\right\}
    \end{equation*}
for some constant $c_4>0$. Here, 2 and 4 are arbitrarily chosen from $(1,\infty)$. We just select them for the convenience of later use.

The following lemma is the main input of Proposition \ref{bilip}, which shows that the probability of $F_{\varepsilon,3}$ is very high for a suitable $c_4>0$.
\begin{lemma}\label{d-longer-than-D}
    Assuming that $F_{\gamma,2}$ occurs, there exists a constant $c_4>0$ {\rm(}depending only on $\beta,d$ and the law of $D${\rm)} such that $F_{\varepsilon,3}$ occurs with probability at least $1-O_\varepsilon(\varepsilon^{\mu})$ for any $\mu>0$, where the implicit constant depends only on $\beta, d, \gamma, \mu$  and the law of $D$.
\end{lemma}

We first show that the lower bound $\varepsilon^{-\theta/4}$ in Lemma \ref{d-longer-than-D} is sensible by checking that the diameter of $V_{\varepsilon R}(\bm k)$ with respect to $D$ has a uniform upper bound with high probability in the following lemma. For that, we let $F_{\varepsilon,4}$ be the event that
    $$
    \sup_{\bm k\in (\varepsilon R)[-u/\varepsilon,u/\varepsilon)^d_{\mathds{Z}}}{\rm diam}(V_{\varepsilon R}(\bm k);D)\leq \varepsilon^{3\theta/4} R^\theta.
    $$

\begin{lemma}\label{diamIk-2}
    $F_{\varepsilon,4}$ occurs with probability at least $1-O_\varepsilon(\varepsilon^{\mu})$ for all $\mu>0$, where the implicit constant in the $O_\varepsilon(\cdot)$ depends only on $\beta, d, \gamma$  and the law of $D$ {\rm(}since $u$ in the event $F_{\gamma,2}$ depends only on $\beta,d,\gamma$ and the law of $D${\rm)}.
\end{lemma}

\begin{proof}
    Applying Markov's inequality and  Axiom V2' (tightness across different scales (upper bound)) with $\eta=1$, we get:
    \begin{equation*}
        \begin{split}
        &\mathds{P}\left[\sup_{\bm k\in(\varepsilon R)[-u/\varepsilon,u/\varepsilon)^d_{\mathds{Z}}}{\rm diam}(V_{\varepsilon R}(\bm k);D)\geq \varepsilon^{3\theta/4} R^\theta \right]
        \leq \sum_{\bm k\in(\varepsilon R)[-u/\varepsilon,u/\varepsilon)^d_{\mathds{Z}}}\mathds{P}\left[\frac{{\rm diam}(V_{\varepsilon R}(\bm k);D)}{(\varepsilon R)^\theta}\geq \varepsilon^{-\theta/4} \right]\\
        &\leq (2u\varepsilon^{-1})^{d}\e^{-\varepsilon^{-\theta/4}}\mathds{E}\left[\exp\left\{\frac{{\rm diam}([0,\varepsilon R)^d;D)}{(\varepsilon R)^\theta}\right\}\right]=1-O_\varepsilon (\varepsilon ^{\mu}) \quad \forall \mu>0,
        \end{split}
        \end{equation*}
    which implies the desired statement.
\end{proof}

In particular, when the event $F_{\gamma,2}\cap F_{\varepsilon,4}$ occurs, we have the following estimate.
\begin{lemma}\label{Dleqd}
Assuming that $F_{\gamma,2}\cap F_{\varepsilon,4}$ occurs, we then have that
\begin{equation*}\label{numbL-2}
   D(V_{\varepsilon R}(\bm k_1),V_{\varepsilon R}(\bm k_2))\leq \left(d^G(\bm k_1,\bm k_2)+1\right)\varepsilon^{3\theta/4} R^\theta\quad \text{for all }\bm k_1,\bm k_2\in(\varepsilon R)[-1/\varepsilon,1/\varepsilon)^d_{\mathds{Z}}.
\end{equation*}
\end{lemma}

\begin{proof}
For any fixed $\bm k_1,\bm k_2\in (\varepsilon R)[-1/\varepsilon,1/\varepsilon)^d_{\mathds{Z}}$, let $P^G_{\bm k_1,\bm k_2}$ be a $d^G$-geodesic in $G$. Since $G$ is the renormalization from the continuous model, from Definition \ref{pathd-c} we see that  there exists a simple path $P$ from $V_{\varepsilon R}(\bm k_1)$ to $V_{\varepsilon R}(\bm k_2)$
(as mentioned earlier here simple means that there is no loop in the sequence of the centers of the cubes $V_{\varepsilon R}(\bm k)$ that it hits in order)
whose skeleton path is $P^G_{\bm k_1,\bm k_2}$. Thus, on the event $F_{\gamma,2}\cap F_{\varepsilon,4}$,
$$
 D(V_{\varepsilon R}(\bm k_1),V_{\varepsilon R}(\bm k_2))\leq {\rm len}(P;D)\le \left(d^G(\bm k_1,\bm k_2)+1\right)\varepsilon^{3\theta/4} R^\theta.
$$
The proof is complete.
\end{proof}

We also let $F_{\varepsilon,5}$ be the event that for any $\bm i,\bm j\in (\varepsilon R)[-u/\varepsilon,u/\varepsilon)^d_{\mathds{Z}}$ and every path $P^G_{\bm i\bm j}$ from $\bm i$ to $\bm j$ with $|P^G_{\bm i\bm j}|\geq \varepsilon^{-\theta/4}$, we have
 $$|P^G_{\bm i\bm j}|\geq c_5(\varepsilon R)^{-\theta}\sum_{\bm k\in P^G_{\bm i\bm j}}{\rm diam}(V_{\varepsilon R}(\bm k);D)$$
for some constant $ c_5>0$.

 It is clear that ${\rm diam}(V_{\varepsilon R}(\bm k);D)$ is determined by the internal metric $D(\cdot,\cdot;V_{\varepsilon R}(\bm k))$ for each $\bm k\in (\varepsilon R)\mathds{Z}^d$. Hence, ${\rm diam}(V_{\varepsilon R}(\bm k);D)$ are i.i.d.\ random variables by Axiom II'' (weak locality) and the translation invariance.
Furthermore, they are also independent of all edges in $G$ by Axiom II'' (weak locality) again.
Thus, we can get the following estimate.

\begin{lemma}\label{Fr3lemma-2}
    There exists a constant $c_5>0$ {\rm(}depending only on $\beta,d$ and the law of $D${\rm)} such that $F_{\varepsilon,5}$ occurs with probability at least  $1-O_\varepsilon(\varepsilon^{\mu})$ for all $\mu>0$.
\end{lemma}

\begin{proof}
Denote
$$M_D:=\sup_{\varepsilon R>0}\mathds{E}\left[\exp\left(\frac{{\rm diam}([0,\varepsilon R)^d;D)}{(\varepsilon R)^{\theta}}\right)\right].$$
By Axiom V2' (tightness across different scales (upper bound)), it is obvious that $ M_D<\infty$ and it does not depend on $\varepsilon$ or $R$. Furthermore, ${\rm diam}(V_{\varepsilon R}(\bm k);D)$ for $\bm k\in (\varepsilon R)[-u/\varepsilon,u/\varepsilon)^d_\mathds{Z}$ are i.i.d. and independent of all edges in $G$. Thus, applying the exponential Markov's inequality (i.e., $\mathds{P}[X\geq x]\leq \e^{-x}\mathds{E}[\e^X]$),
 we obtain that for any $\bm i,\bm j\in (\varepsilon R)[-u/\varepsilon,u/\varepsilon)^d_{\mathds{Z}}$, $m\geq \varepsilon^{-\theta/4}$, $\widetilde{c}>0$, and any $P_{\bm i\bm j}^G\in\mathcal{P}_m(\bm i,\bm j)$,
\begin{equation}\label{diamsum-2}
\mathds{P}\left[ (\varepsilon R)^{-\theta}\sum_{\bm k\in P^G_{\bm i\bm j}}{\rm diam}(V_{\varepsilon R}(\bm k);D)\geq  \frac{m}{\widetilde{c}} \Big| G\right]
\leq \e^{-\frac{m}{\widetilde{c}}}M_D^m=\e^{-(\frac{1}{\widetilde{c}}-\log  M_D)m}.
\end{equation}

Moreover, using arguments similar to those in \eqref{number-path-m-2} and \eqref{number-path>=-2}, we can show that
\begin{equation}\label{number-path>=1-2}
\begin{split}
&\mathds{P}\left[ \exists m\geq \varepsilon^{-\theta/4}\ \text{such that }|\mathcal{P}_m(\bm i,\bm j)|\geq (2C_{dis})^m\right]\\
&\leq \sum_{ m\geq  \varepsilon^{-\theta/4}} 2^{-m}=O_\varepsilon(\varepsilon^{\mu})\quad \forall\mu>0,
\end{split}
\end{equation}
where the implicit constant in the $O_\varepsilon(\cdot)$ depends only on $\beta,d,\mu$  and the law of $D$.
We denote the events in \eqref{diamsum-2} and \eqref{number-path>=1-2} by $M'(P_{\bm i\bm j}^G)$ and $N'_{\bm i\bm j}$, respectively. Combining \eqref{diamsum-2} and \eqref{number-path>=1-2}, we obtain that
\begin{equation*}
\begin{split}
&\mathds{P}[F_{\varepsilon,5}^c]\leq \sum_{\bm i,\bm j\in (\varepsilon R)[-u/\varepsilon,u/\varepsilon)^d_{\mathds{Z}}}\mathds{P}[N'_{\bm i\bm j}]
+\sum_{\bm i,\bm j\in (\varepsilon R)[-u/\varepsilon, u/\varepsilon)^d_{\mathds{Z}}}\sum_{m\geq \varepsilon^{-\theta/4}}\mathds{E}\left[\I_{(N'_{\bm i\bm j})^c}\sum_{P_{\bm i\bm j}^G\in\mathcal{P}_m(\bm i,\bm j)}\mathds{P}\left[ M'(P_{\bm i\bm j}^G)\big| G\right]\right]\\
&\leq O_\varepsilon(\varepsilon^{\mu})
+(2u\varepsilon^{-1})^{2d}\sum_{m\geq \varepsilon^{-\theta/4}}(2C_{dis})^m\e^{-(\frac{1}{\widetilde{c}}-\log M_D)m}\\
&=O_\varepsilon(\varepsilon^{\mu})+(2u\varepsilon^{-1})^{2d}\sum_{m\geq  \varepsilon^{-\theta/4}}\e^{-(\frac{1}{\widetilde{c}}-\log M_D-\log(2C_{dis}))m}\quad \quad \forall \mu>0.
\end{split}
\end{equation*}
Thus,  if we choose the constant $\widetilde{c}>0$ small enough that $1/\widetilde{c}>2(\log  M_D+\log(2C_{dis}))$, then we obtain $\mathds{P}[F_{\varepsilon,5}^c]=O_\varepsilon(\varepsilon^{\mu})$ (with $c_5=\widetilde{c}$) for all $\mu>0$. This implies the lemma.
\end{proof}

Now we finish the
\begin{proof}[Proof of Lemma \ref{d-longer-than-D}]
    To prove Lemma \ref{d-longer-than-D}, it suffices to show that for $c_5$ in Lemma \ref{Fr3lemma-2}, the following holds on the event $F_{\gamma,2}\cap F_{\varepsilon,4}\cap F_{\varepsilon,5}$. For any ${\bm x},{\bm y}\in[-R,R]^d$ with $D({\bm x},{\bm y})\ge 2\varepsilon^{\theta/2}R^\theta$,
    take $\bm k_{\bm x},\bm k_{\bm y}\in (\varepsilon R)[-1/\varepsilon,1/\varepsilon)_\mathds{Z}^d$ satisfying $\bm x\in V_{\varepsilon R}(\bm k_{\bm x}),\bm y\in V_{\varepsilon R}(\bm k_{\bm y})$. Then
    \begin{equation*}
        d^G(\bm k_{\bm x},\bm k_{\bm y})+1\ge \max\left\{ c_4(\varepsilon R)^{-\theta} D({\bm x},{\bm y}),\varepsilon^{-\theta/4}\right\}.
    \end{equation*}

    From Lemma \ref{Dleqd}, we obtain that  on $F_{\gamma,2}\cap F_{\varepsilon,4}\cap F_{\varepsilon,5}$,
    $$
    D(V_{\varepsilon R}(\bm k_{\bm x}),V_{\varepsilon R}(\bm k_{\bm y}))\leq  \left(d^G(\bm k_{\bm x},\bm k_{\bm y})+1\right)\varepsilon^{3\theta/4}R^\theta.
    $$
    Combining this with the fact that
    \begin{equation*}
    \begin{split}
    D(V_{\varepsilon R}(\bm k_{\bm x}),V_{\varepsilon R}(\bm k_{\bm y}))&\geq D({\bm x},{\bm y})-({\rm diam}(V_{\varepsilon R}(\bm k_{\bm x});D)+{\rm diam}(V_{\varepsilon R}(\bm k_{\bm y});D))\\
    &\geq 2\varepsilon^{\theta/2}R^\theta-2\varepsilon^{3\theta/4}R^\theta\geq \varepsilon^{\theta/2}R^\theta
    \end{split}
    \end{equation*}
    for sufficiently small $\varepsilon$ (i.e. sufficiently large $n$), we get that  on $F_{\gamma,2}\cap F_{\varepsilon,4}\cap F_{\varepsilon,5}$,
    \begin{equation}\label{Lkxky}
        d^G(\bm k_{\bm x},\bm k_{\bm y})+1\geq \varepsilon^{-3\theta/4}R^{-\theta}D(V_{\varepsilon R}(\bm k_{\bm x}),V_{\varepsilon R}(\bm k_{\bm y}))\geq \varepsilon^{-\theta/4}.
    \end{equation}

Additionally, let $P^G_{\bm k_{\bm x},\bm k_{\bm y}}$ be a $d^G$-geodesic from $\bm k_{\bm x}$ to $\bm k_{\bm y}$ in $G$. From Definition \ref{pathd-c}, there exists a simple path $P$ from $V_{\varepsilon R}(\bm k_{\bm x})$ to $V_{\varepsilon R}(\bm k_{\bm y})$
whose skeleton path is $P^G_{\bm k_{\bm x},\bm k_{\bm y}}$. We now construct a path $P'$ from ${\bm x}$ to ${\bm y}$ by  connecting $\bm x$ to the start point of $P$ with an internal geodesic in $V_{\varepsilon R}(\bm k_{\bm x})$ and $\bm y$ to the end point of $P$ with an internal geodesic in $V_{\varepsilon R}(\bm k_{\bm y})$. As a result, on $F_{\gamma,2}\cap F_{\varepsilon,5} $  we have
    \begin{equation}\label{Dupperbound-2}
        D({\bm x},{\bm y})\le {\rm len}(P';D)\le \sum_{\bm k\in P^G_{\bm k_{\bm x},\bm k_{\bm y}}}{\rm diam}(V_{\varepsilon R}(\bm k);D)\le \frac{1}{c_5}\left(d^G(\bm k_{\bm x},\bm k_{\bm y})+1\right)(\varepsilon R)^{\theta}.
    \end{equation}
    Combining \eqref{Lkxky} and \eqref{Dupperbound-2}, we obtain the proof of Lemma \ref{d-longer-than-D}.
\end{proof}

We also present the following proposition, which forms part of the input required for the proof of Proposition \ref{bilip}.

\begin{proposition}\label{Dxy=0-local}
    Let $D$ be a local $\beta$-LRP metric $D$. Then almost surely, for any ${\bm x},{\bm y}\in\mathds{R}^d$, $D({\bm x},{\bm y})=0$ if and only if $\langle {\bm x},{\bm y}\rangle \in\mathcal{E}$. Moreover, Axiom I can be reinforced by the following property. Let $\sim$ be the following equivalence relation that $\bm x\sim\bm y$ if and only if $\langle{\bm x},{\bm y}\rangle\in\mathcal{E}$. Then $(\mathds{R}^d/\sim,D)$ is a length space. Furthermore, for any ${\bm x},{\bm y}\in\mathds{R}^d$, there exists a geodesic from ${\bm x}$ to ${\bm y}$.
\end{proposition}

To prove Proposition \ref{Dxy=0-local}, we need the following lemma.

\begin{lemma}\label{Dinterval=0}
    Let $D$ be a local $\beta$-LRP metric. For any $ k\in\mathds{Z}$ and $ \bm l\in2^k\mathds{Z}^d$, let $V^k(\bm l)=V_{2^k}(\bm l)$. Let $E_{\bm i,\bm j}^k$ {\rm(}for $\|\bm i-\bm j\|_{1}>3\cdot 2^k${\rm)} be the event that $D(V^k(\bm i),V^k(\bm j))=0$ and $V^k(\bm i)$ and $V^k(\bm j)$ are not directly connected by $\mathcal{E}$. Then $\mathds{P}[E_{\bm i,\bm j}^k]=0$.
    As a consequence,
    $$\mathds{P}\left[\bigcup_{k\in \mathds{Z}}\bigcup_{\bm i,\bm j\in 2^k\mathds{Z}^d: \|\bm i-\bm j\|_{1}>3\cdot 2^k}E_{\bm i,\bm j}^k\right]=0.$$
\end{lemma}
\begin{proof}
    For fixed $k\in \mathds{Z}$, recalling the renormalization argument, we can define a discrete model as follows. We identify $V^k(\bm l)$ as the vertex $\bm l$ for all $\bm l\in 2^k\mathds{Z}^d$ and call the resulting graph $G$. By the self-similarity of the model, $G$ is the discrete long-range percolation defined before \eqref{connectprob}.

    For any sufficiently small $\alpha>0$, we say $\bm l$ is $\alpha$-good if $V_{3\cdot 2^k}(\bm l)$ is $(3\cdot2^k,\alpha)$-good as in Definition \ref{h-good}.
    We also define $\widetilde{E}_{\bm i,\bm j,\alpha}$ as the event that there exists a path in $G$ from $\bm i$ to $\bm j$ with length $m>1$ such that any vertex on the path (except $\bm i,\bm j$) is not $\alpha$-good. Then by the definition of  $(3\cdot2^k,\alpha)$-good,  for $\bm i,\bm j\in 2^k\mathds{Z}^d$ with  $\|\bm i-\bm j\|_{1}> 3\cdot 2^k$, we see that if $\langle \bm i,\bm j\rangle\notin\mathcal{E}$ and $\widetilde{E}_{\bm i,\bm j,\alpha}^c$ occurs,
    then $D(V^k(\bm i),V^k(\bm j))\ge (b_0\alpha 2^k)^\theta>0$ (here $b_0$ is chosen as in Corollary \ref{h-good-prob-2}).
    As a result,
    \begin{equation}\label{EEtilde}
    \mathds{P}[E_{\bm i,\bm j}^k]\leq  \mathds{P}[\widetilde{E}_{\bm i,\bm j,\alpha}].
    \end{equation}

     Recall that $\mathcal{P}_m(\bm i,\bm j)$ is the set consisting of all the paths from $\bm i$ to $\bm j$ with length $m$.
    Let $\widetilde{E}_{\bm i,\bm j,\alpha}^m(m>1)$ be the event that there exists a path in $\mathcal{P}_m(\bm i,\bm j)$ such that every vertex on the path (here we do not consider $\bm i,\bm j$ either) is not $\alpha$-good. Then applying Lemma \ref{hd-BK} with $E_{\bm k}$ being the whole probability space $\widetilde{\Omega}$ for all $\bm k \in 2^k\mathds{Z}^d$ and $\delta>0$ such that $C_{dis}\delta<1$, we get that for sufficiently small $\alpha$,
    \begin{equation*}
        \begin{split}
            \mathds{P}[E_{\bm i,\bm j}^k]&\leq  \mathds{P}[\widetilde{E}_{\bm i,\bm j,\alpha}]\leq \sum_{m\geq 2}\mathds{P}[\widetilde{E}_{\bm i,\bm j,\alpha}^m]\\
            &\leq \sum_{m\geq 2}\mathds{E}\left[\sum_{P^G\in\mathcal{P}_m(\bm i,\bm j)}\mathds{P}\left[
            P^G\text{ is }(\widetilde{\Omega},\alpha)\text{-bad}\big|G\right]\right]\\
            &\leq \sum_{m\geq 2}\mathds{E}[|\mathcal{P}_m(\bm i,\bm j)|]\delta^{m-1}
            \leq \sum_{m\geq 2}C_{dis}^m \delta^{m-1} = \frac{C^2_{dis}\delta}{1-C_{dis}\delta}.
        \end{split}
    \end{equation*}
    Thus by sending $\delta$ to 0, we obtain $\mathds{P}[E_{\bm i,\bm j}^k]=0$, which implies the desired result.
\end{proof}

We now complete the

\begin{proof}[Proof of Proposition \ref{Dxy=0-local}]
    It suffices to show the conclusion for a local LRP metric $D$.
Suppose ${\bm x},{\bm y}\in \mathds{R}^d$ are fixed with $\langle {\bm x},{\bm y}\rangle \notin\mathcal{E}$. For each $k\in\mathds{Z}$, recall that $V^k(\cdot)$ is defined in Lemma \ref{Dinterval=0}.
Let $\bm i_k({\bm x}),\bm j_k({\bm y})\in 2^k\mathds{Z}^d$ be such that ${\bm x}\in V^k(\bm i_k({\bm x}))$ and ${\bm y}\in V^k(\bm i_k({\bm y}))$.
By the property of the Poisson point process, it is evident that the number of long edges connecting the cubes $ V^k(\bm i_k({\bm x}))$ and $ V^k(\bm i_k({\bm y}))$ is finite for each small $k<0$.
Moreover, since $\langle {\bm x},{\bm y}\rangle\notin\mathcal{E}$, we can select $k'<0$ sufficiently small such that $ V^{k'}(\bm i_{k'}({\bm x}))$ and $ V^{k'}(\bm i_{k'}({\bm y}))$ are not connected directly by any long edge and that $\|\bm i_{k'}({\bm x})-\bm j_{k'}({\bm y})\|_{1}> 3\cdot 2^{k'}$.
In addition, it is clear that
    $$
    D({\bm x},{\bm y})\ge D\left(V^{k'}(\bm i_{k'}({\bm x})),V^{k'}(\bm i_{k'}({\bm y}))\right).
    $$
Thus, if we denote by $E_{{\bm x},{\bm y}}$ the event that $D({\bm x},{\bm y})=0$ and $\langle {\bm x},{\bm y}\rangle\notin \mathcal{E}$ for ${\bm x},{\bm y}\in \mathds{R}^d$, then
$E_{{\bm x},{\bm y}}\subseteq \cup_{k\in\mathds{Z}:\|\bm i_{k}({\bm x})-\bm j_{k}({\bm y})\|_{1}> 3\cdot 2^{k}}E_{\bm i_{k}({\bm x}),\bm j_{k}({\bm y})}^{k}$.
Combining this with Lemma \ref{Dinterval=0}, we get that
$$
\mathds{P}[\cup_{{\bm x},{\bm y}\in\mathds{R}^d}E_{{\bm x},{\bm y}}]\leq \mathds{P}\left[\cup_{k\in \mathds{Z}}\cup_{\|\bm i-\bm j\|_{1}>3\cdot 2^k}E_{\bm i,\bm j}^k\right]=0,
$$
which implies that almost surely for any $\bm x,\bm y\in\mathds{R}^d$, $D(\bm x, \bm y)=0$ if and only if $\langle \bm x,\bm y \rangle\in\mathcal{E}$. Combining this with Axiom I yields the remaining part of the proposition.   
\end{proof}

From Proposition \ref{Dxy=0-local}, we obtain immediately the

\begin{proof}[Proof of Proposition \ref{Dxy=0}]
The statement is easily derived from the fact that every weak or strong $\beta$-LRP metric is also a local $\beta$-LRP metric.
\end{proof}

We turn to the

\begin{proof}[Proof of Proposition \ref{bilip}]

By Lemmas \ref{Fr4lemma-2}, \ref{geodesic-go-not-far} and \ref{d-longer-than-D} and our assumption that $\varepsilon \ll \gamma$, we observe that
$$
\mathds{P}[F_{\varepsilon,1}\cap F_{\gamma,2}\cap F_{\varepsilon,3}]\geq 1-3\gamma.
$$
Then throughout the proof, we assume that $F_{\varepsilon,1}\cap F_{\gamma,2}\cap F_{\varepsilon,3}$ occurs. For any ${\bm x},{\bm y}\in [-R,R]^d$ with $D({\bm x},{\bm y})\geq 2\varepsilon^{\theta/2}R^\theta$, let $\bm k_{\bm x},\bm k_{\bm y}\in (\varepsilon R)[-1/\varepsilon,1/\varepsilon)^d_\mathds{Z}$ be such that ${\bm x}\in V_{\varepsilon R}(\bm k_{\bm x})$ and ${\bm y}\in V_{\varepsilon R}(\bm k_{\bm y})$. Then by the definition of $F_{\varepsilon,3}$ in Lemma \ref{d-longer-than-D}, we see that
\begin{equation}\label{Duper-2}
d^G(\bm k_{\bm x},\bm k_{\bm y})+1\geq \max\left\{ c_4(\varepsilon R)^{-\theta}D({\bm x},{\bm y}),\varepsilon^{-\theta/4}\right\}.
\end{equation}

Denote by $P_{\bm k_{\bm x},\bm k_{\bm y}}^G$ the skeleton path in $G$ derived from a $D$-geodesic between $V_{\varepsilon R}(\bm k_{\bm x})$ and $V_{\varepsilon R}(\bm k_{\bm y})$. Then
 by the definitions of $F_{\varepsilon,1}$ and good cubes and the fact that $|P_{\bm k_{\bm x},\bm k_{\bm y}}^G| \geq d^G(\bm k_{\bm x},\bm k_{\bm y})$, we get
\begin{equation}\label{Dlower-2}
D({\bm x},{\bm y})\geq \frac{1}{ 4\cdot 3^d} |P_{\bm k_{\bm x},\bm k_{\bm y}}^G|(b_0\alpha\varepsilon R)^\theta\ge \frac{1}{ 4\cdot 3^d}d^G(\bm k_{\bm x},\bm k_{\bm y})(b_0\alpha\varepsilon R)^\theta,
\end{equation}
where $\alpha$ and $b_0$ are small constants chosen in Corollary \ref{h-good-prob-2}.

Combining \eqref{Duper-2} and \eqref{Dlower-2}, we obtain that with probability at least $1-3\gamma$, the following holds.  For any ${\bm x},{\bm y}\in [-R,R]^d$ with $D({\bm x},{\bm y})\geq 2\varepsilon^{\theta/2}R^\theta$,
\begin{equation}\label{Dupdown}
    \frac{1}{ 4\cdot 3^d}(b_0\alpha\varepsilon R)^\theta d^G(\bm k_{\bm x},\bm k_{\bm y})\leq D({\bm x},{\bm y})\leq{ 2}  c_4^{-1}(\varepsilon R)^\theta d^G(\bm k_{\bm x},\bm k_{\bm y}).
\end{equation}

A similar result is true for the metric $\widetilde{D}$ with another set of constants $\widetilde{\alpha},\widetilde{b}_0,\widetilde{c}_{4}$ since $\widetilde{D}$ is also a local $\beta$-LRP metric. To be presice, we have that with probability at least $1-3\gamma$, the following holds.  For any ${\bm x},{\bm y}\in [-R,R]^d$ with $D({\bm x},{\bm y})\geq 2\varepsilon^{\theta/2}R^\theta$,
\begin{equation}\label{Dtildeupdown}
    \frac{1}{ 4\cdot 3^d}(\widetilde{b}_0\widetilde{\alpha}\varepsilon R)^\theta d^G(\bm k_{\bm x},\bm k_{\bm y})\leq \widetilde{D}({\bm x},{\bm y})\leq{ 2} \widetilde{c}_{4}^{-1}(\varepsilon R)^\theta d^G(\bm k_{\bm x},\bm k_{\bm y}).
\end{equation}

Combining \eqref{Dupdown} with \eqref{Dtildeupdown}, we get that it holds with probability at least $1-6\gamma$ that
\begin{equation*}
        C^{-1}D({\bm x},{\bm y})\leq \widetilde{D}({\bm x},{\bm y})\leq C D({\bm x},{\bm y}),\quad \forall {\bm x},{\bm y}\in[-R,R]^d\text{ with }D({\bm x},{\bm y})\wedge\widetilde{D}({\bm x},{\bm y})\geq 2\varepsilon^{\theta/2}R^\theta,
\end{equation*}
where $C:= \max\{{ 8\cdot 3^d}\widetilde{c}_{4}^{-1}(b_0\alpha)^{-\theta},{ 8\cdot 3^d}c_{4}^{-1}(\widetilde{b}_0\widetilde{\alpha})^{-\theta}\}$.

Sending $\varepsilon$ to 0 (i.e. $n$ to infinity) and then sending $R$ to infinity, we get
\begin{equation}\label{Dtilde>0}
    \mathds{P}[C^{-1}D({\bm x},{\bm y})\leq \widetilde{D}({\bm x},{\bm y})\leq C D({\bm x},{\bm y}),\ \forall {\bm x},{\bm y}\in\mathds{R}^d\text{ with }D({\bm x},{\bm y})\wedge \widetilde{D}({\bm x},{\bm y})>0]\geq 1-6\gamma.
\end{equation}
In addition, it follows from Proposition \ref{Dxy=0} that a.s.\ $D({\bm x},{\bm y})=0$ if and only if $\widetilde{D}({\bm x},{\bm y})=0$ for all ${\bm x},{\bm y}\in\mathds{R}^d$.  Combining this fact with \eqref{Dtilde>0} we see that
$$
\mathds{P}\left[C^{-1}D({\bm x},{\bm y})\leq \widetilde{D}({\bm x},{\bm y})\leq C D({\bm x},{\bm y}),\ \forall {\bm x},{\bm y}\in\mathds{R}^d\right]\geq 1-6\gamma.
$$
Finally, we complete the proof by sending $\gamma$ to 0 (recall that $C$ does not depend on $\gamma$).
\end{proof}

Finally we present the

\begin{proof}[Proof of Proposition \ref{cCdtm}]
    Let $D$ and $\widetilde{D}$ be two weak $\beta$-LRP metrics. We will only consider the statement for $C_*$ here. Suppose  $C>0$ such that $\mathds{P}[C_*>C]>0$. We will show that $\mathds{P}[C_*>C]=1$.

    From $\mathds{P}[C_*>C]>0$, we can find a large deterministic $R>0$ such that with positive probability, denoted by $q_1$, there are points ${\bm x},{\bm y}\in [-R,R]^d$ such that $\widetilde{D}({\bm x},{\bm y})/D({\bm x},{\bm y})>C$.
    In order to prove the lemma, we wish to build many almost independent trials where each trial certifies $C_*>C$ with positive probability, and it is natural to extract the independence from the locality of the LRP metric. To this end, for any $r>0$ and $\bm k\in r\mathds{Z}^d$, let $E_{\bm k,r}$ be the event that there exist ${\bm x},{\bm y}\in V_r(\bm k)$ such that
    $$\widetilde{D}({\bm x},{\bm y})/D({\bm x},{\bm y})>C,$$
    and each $D$-geodesic and $\widetilde{D}$-geodesic from ${\bm x}$ to ${\bm y}$ are contained within $ V_r(\bm k)$.

    Additionally, from Lemma \ref{geodesic-go-not-far}, we see that there is a sufficiently large $u=u(q_1)>0$ (depending only on $\beta,d$, $q_1$ and the law of $D$) such that with probability at least $1-q_1/2$, each $D$-geodesic and $\widetilde{D}$-geodesic between any two points of $[-R,R]^d$ is contained in $[-uR,uR]^d$.
    Therefore, taking $r=2uR>2R$, we see that $E_{\bm 0,2uR}$ occurs with probability at least $q_1/2$. For convenience, let $q_2=\mathds{P}[E_{\bm 0,2uR}]\geq q_1/2>0$. From Axiom IV', we get that $q_2=\mathds{P}[E_{\bm k,2uR}]$ for any $\bm k\in r\mathds{Z}^d$.

    Now it suffices to show that almost surely there exists a (random) $\bm k\in r\mathds{Z}^d$ such that $E_{\bm k,2uR}$ occurs. Note that for any $\bm k,\bm l\in r\mathds{Z}^d$, we do not have exact independence between $E_{\bm k,2uR}$ and $E_{\bm l,2uR}$ even if $\bm k,\bm l$ are far away from each other, since both $E_{\bm k,2uR}$ and $E_{\bm l,2uR}$ are dependent on edges connecting $V_{2uR}(\bm k)$ and $V_{2uR}(\bm l)$. As a result, we can not use Kolmogorov 0-1 law directly. However, since the correlation is small when $\bm k$ and $\bm l$ are sufficiently far away, we can use the second moment method to find a $\bm k\in r\mathds{Z}^d$ such that $E_{\bm k,2uR}$ occurs.

    To this end, we need to first define an event that will facilitate the control of the correlation. For any $\bm k,\bm l\in r\mathds{Z}^d $ such that $|\bm k-\bm l|>2uR$, let $C_{\bm k,\bm l}$ be the event that $ V_{2uR}(\bm k)$ and $V_{2uR}(\bm l)$ are not connected by $\mathcal{E}$. For any $\delta>0$, we can choose a sufficiently large (deterministic) $R_\delta>2uR$ such that $\mathds{P}[C_{\bm k,\bm l}]>1-\delta$ whenever $|\bm k-\bm l|>R_\delta$.  Now for any sufficiently large $N\in\mathds{N}$, we choose $\bm k_1,\cdots,\bm k_N\in r\mathds{Z}^d$ such that for any $i\neq j\in[1,N]_\mathds{Z}$, $|\bm k_i-\bm k_j|>R_\delta$. Let $X_N=\sum_{i=1}^N\I\{E_{\bm k_i,2uR}\}$. Then we have the first moment
    \begin{equation}\label{firstmoment}
        \mathds{E}[X_N]=\sum_{i=1}^N\mathds{P}[E_{\bm k_i,2uR}]=Nq_2.
    \end{equation}

    We now turn to estimate $\mathds{E}[X_N^2]$. Note that for any $i\neq j\in[1,N]_\mathds{Z}$, conditioned on $C_{\bm k_i,\bm k_j}$, we have that $E_{\bm k_i,2uR}$ (resp. $E_{\bm k_i,2uR}$) is a.s. determined by $\mathcal{E}\cap (V_{2uR}(\bm k_i)\times (\mathds{R}^d\setminus V_{2uR}(\bm k_j)))$ (resp. $\mathcal{E}\cap (V_{2uR}(\bm k_i)\times (\mathds{R}^d\setminus V_{2uR}(\bm k_j)))$). As a result, the independence of the Poisson point process yields the conditional independence between $E_{\bm k_i,2uR}$ and $E_{\bm k_j,2uR}$ conditioned on $C_{\bm k_i,\bm k_j}$. Then we get that
    $$ \mathds{P}[E_{\bm k_i,2uR}\cap E_{\bm k_j,2uR}|C_{\bm k_i,\bm k_j}]=\mathds{P}[E_{\bm k_i,2uR}|C_{\bm k_i,\bm k_j}]\mathds{P}[E_{\bm k_i,2uR}|C_{\bm k_i,\bm k_j}]. $$
    Therefore we have that
    \begin{equation}\label{correlation}
        \begin{aligned}
        \mathds{P}[E_{\bm k_i,2uR}\cap E_{\bm k_j,2uR}]&\leq \mathds{P}[E_{\bm k_i,2uR}\cap E_{\bm k_j,2uR}|C_{\bm k_i,\bm k_j}]\mathds{P}[C_{\bm k_i,\bm k_j}]+\mathds{P}[C_{\bm k_i,\bm k_j}^c] \\
        &= \mathds{P}[E_{\bm k_i,2uR}|C_{\bm k_i,\bm k_j}]\mathds{P}[E_{\bm k_i,2uR}|C_{\bm k_i,\bm k_j}]+(1-\mathds{P}[C_{\bm k_i,\bm k_j}]) \\
        &\leq \mathds{P}[E_{\bm k_i,2uR}]\mathds{P}[E_{\bm k_i,2uR}](\mathds{P}[C_{\bm k_i,\bm k_j}])^{-2}+(1-\mathds{P}[C_{\bm k_i,\bm k_j}]) \\
        &\leq (1-\delta)^{-2}q_2^2+\delta.
        \end{aligned}
    \end{equation}
    Here the last inequality used the fact that $|\bm k_i-\bm k_j|>R_\delta$, which implies that $\mathds{P}[C_{\bm k_i,\bm k_j}]>1-\delta$. Then summing \eqref{correlation} over all possible $i\neq j\in[1,N]_\mathds{Z}$, we get that
    \begin{equation}\label{secondmoment}
        \mathds{E}[X_N^2]=\sum_{i=1}^N\mathds{P}[E_{\bm k_i,2uR}]+\sum_{i\neq j\in[1,N]_\mathds{Z}}\mathds{P}[E_{\bm k_i,2uR}\cap E_{\bm k_j,2uR}]\leq Nq_2+N(N-1)((1-\delta)^{-2}q_2^2+\delta).
    \end{equation}

    Combining Cauchy-Schwarz inequality with \eqref{firstmoment} and \eqref{secondmoment}, we get that
    $$ \mathds{P}[X_N>0]\geq \frac{\mathds{E}[X_N]^2}{\mathds{E}[X_N^2]} \geq \frac{(Nq_2)^2}{Nq_2+N(N-1)((1-\delta)^{-2}q_2^2+\delta)}. $$
    Note that $X_N>0$ means that there exists at least a $j\in[1,N]_\mathds{Z}$ such that $E_{\bm k_j,2uR}$ occurs, which implies $C_*>C$. Letting $N\to\infty$ and then $\delta\to 0$, we get that the right hand side of the inequality above tends to 1, which implies that $\mathds{P}[C_*>C]=1$.
\end{proof}

\subsection{Proof of Theorem \ref{uniquetheta-1}}\label{uniquetheta}

The proof is similar to that of Proposition \ref{bilip}.
    Assume that $D$ is a weak $\beta$-LRP metric with $\theta=\theta_1>0$.
    Note that the proof of Proposition \ref{bilip} is still valid if we replace all $\theta$ by $\theta_1$.
    In particular, from \eqref{Dupdown} (replacing all $\theta$ by $\theta_1$ and taking $R=1$), we get that for any fixed $\gamma>0$, there exist $C>0$ and $\varepsilon_0>0$ (depending on $\beta,d,\gamma$ and the law of $D$) such that for any $\varepsilon<\varepsilon_0$, the following event (denoted by $F_{\gamma,\varepsilon}$) occurs with probability at least $1-3\gamma$. For any $\bm x,\bm y\in[-1,1]^d$ with $D(\bm x,\bm y)\geq 2\varepsilon^{\theta_1/2}$,
    \begin{equation}\label{Fgammaeps}
        C^{-1}d^G(\bm k_{\bm x},\bm k_{\bm y})\varepsilon^{\theta_1}\leq D(\bm x,\bm y)\leq Cd^G(\bm k_{\bm x},\bm k_{\bm y})\varepsilon^{\theta_1}.
    \end{equation}
    Additionally, let $G_{\gamma,\varepsilon}$ be the event that $D(\bm 0,\bm 1)\geq 2\varepsilon^{\theta_1/2}.$ Then from Axiom V1', we can let $\varepsilon>0$ be an arbitrarily small constant depending only on $\beta,d,\gamma$ and the law of $D$ such that $$\mathds{P}[F_{\gamma,\varepsilon}\cap G_{\gamma,\varepsilon}]\geq 1-4\gamma.$$
     Now on the event $F_{\gamma,\varepsilon}\cap G_{\gamma,\varepsilon}$, taking $(\bm x,\bm y)=(\bm 0,\bm 1)$ in \eqref{Fgammaeps} we get that
    \begin{equation}\label{FGgammaeps}
        C^{-1}d^G(\bm k_{\bm 0},\bm k_{\bm 1})\varepsilon^{\theta_1}\leq D(\bm 0,\bm 1)\leq C^{-1}d^G(\bm k_{\bm 0},\bm k_{\bm 1})\varepsilon^{\theta_1}.
    \end{equation}
    Furthermore, from Theorem \ref{discrete-dist} and the translation invariance of the discrete $\beta$-LRP, we get that there exists $C_1>0$ depending on $\beta,d$ and $\gamma$ such that with probability at least $1-\gamma$,
    \begin{equation}\label{dGbound}
        C_1^{-1}|\bm k_{\bm 1}-\bm k_{\bm 0}|^{\theta_0}\leq d^G(\bm k_{\bm 0},\bm k_{\bm 1})\leq C_1|\bm k_{\bm 1}-\bm k_{\bm 0}|^{\theta_0}.
    \end{equation}
    Recall that the renormalization is obtained from dividing $\mathds{R}^d$ into small cubes of side length $\varepsilon$.
    Then it is clear that there exists a constant $C_d>0$ depending only on $d$ such that $|\bm k_{\bm 1}-\bm k_{\bm 0}|\in[C_d^{-1}\varepsilon^{-1},C_d\varepsilon^{-1}].$ Combining this with \eqref{FGgammaeps} and \eqref{dGbound}, we get that the following occurs with probability at least $1-5\gamma$. For $C_0=CC_1C_d^{\theta_0}$,
    $$
    C_0^{-1}\varepsilon^{\theta_1-\theta_0}\leq D(\bm 0,\bm 1)\leq C_0\varepsilon^{\theta_1-\theta_0}.
    $$
    Since $\varepsilon>0$ is arbitrarily small, if $\theta_1\neq \theta_0$, then we get a contradiction when we take $\gamma<1/5$ close to 0 and let $\varepsilon\to 0$. Hence, we must have $\theta_1=\theta_0$. \hfill $\square$

\subsection{Proof of Theorem \ref{theorem-Hausdorff}}\label{Hasd}
Recall that for $\Delta > 0$, the $\Delta$-Hausdorff content and the \textit{Hausdorff dimension} of $(Y, D)$ is defined around \eqref{Hausdorff}.
For a set $X\subset \mathds{R}^d$, recall that  ${\rm dim}_{\mathcal{H}}^0 X$ and ${\rm dim}_{\mathcal{H}}^\beta X$ are the \textit{Euclidean dimension} and \textit{$\beta$-LRP dimension} of the set $X$, respectively.

    Let $D$ be a weak $\beta$-LRP metric and let $X$ be a deterministic Borel subset of $\mathds{R}^d$.  By the countable stability of the  Hausdorff dimension, we can assume that $X$ is contained in a compact subset $K$ of $\mathds{R}^d$ without loss of generality. By \eqref{Hausdorff} and the definition of the Hausdorff dimension, for any $\ell>{\rm dim}_{\mathcal{H}}^0 X$, the following is true: for any $n\in\mathds{N}$, we can choose a (deterministic) covering of $X$ by $\ell^\infty$ balls $\{V_{r^{(n)}_i}(\bm x^{(n)}_i)\}_{i\in\mathds{N}}$ such that
    \begin{equation}\label{XEucD}
        \sum_{i=1}^\infty (r^{(n)}_i)^\ell<n^{-2}.
    \end{equation}

    From Axioms IV' and V2', we get that
    $$
    C_D:=\sup_{\bm z\in\mathds{R}^d,r>0}\mathds{E}\left[\exp\left\{\frac{{\rm diam}(V_r(\bm z);D)}{r^\theta}\right\}\right]=\sup_{r>0}\mathds{E}\left[\exp\left\{\frac{{\rm diam}([0,r]^d;D)}{r^\theta}\right\}\right]<\infty.
    $$
    Let $\widetilde{r}^{(n)}_i={\rm diam}(V_{r^{(n)}_i}(\bm x_i);D)$. Let $A_{n,i}$ be the event that $\{\widetilde{r}^{(n)}_i>(r^{(n)}_i)^{\theta-\zeta}\}$. Here $\zeta$ is a positive constant such that $\ell/(\theta-\zeta)<(2\ell -{\rm dim}_{\mathcal{H}}^0 X)/\theta$. Thus we can use Markov's inequality to get that for $i\in\mathds{N}$ and $\zeta>0$, there exists a constant $C_\ell$ depending only on $C_D,\ell$ and $\zeta$  such that for any $n,i\in\mathds{N}$.
    $$
    \mathds{P}[A_{n,i}]\leq C_D \exp\{-(r^{(n)}_i)^{-\zeta}\}\leq C_\ell (r^{(n)}_i)^\ell.
    $$
    Since from \eqref{XEucD}, we have $\sum_{n,i}(r^{(n)}_i)^\ell<\sum_{n}n^{-2}<\infty$. Combining this with Borel-Cantelli Lemma, there exists a random $N\in\mathds{N}$ such that a.s. for any $n>N$ and $i\in\mathds{N}$, $A_{n,i}$ does not occur, which means $\widetilde{r}^{(n)}_i\leq (r^{(n)}_i)^{\theta-\zeta}$ for any $n>N$ and $i\in\mathds{N}$. Thus we have that a.s. for any $n>N$,
    $$ \sum_{i=1}^\infty(\widetilde{r}^{(n)}_i)^{\ell/(\theta-\zeta)}<\sum_{i=1}^\infty (r^{(n)}_i)^\ell<n^{-2}. $$
     Applying this to \eqref{Hausdorff}, we get that a.s. $C_{\ell/(\theta-\zeta)}(X,D)=0$ and $ {\rm dim}_{\mathcal{H}}^\beta X\leq \ell/(\theta-\zeta)<(2\ell -{\rm dim}_{\mathcal{H}}^0 X)/\theta $ (recall the choice of $\zeta$). Letting $\ell\to {\rm dim}_{\mathcal{H}}^0 X$, we get that a.s. $ {\rm dim}_{\mathcal{H}}^\beta X\leq {\rm dim}_{\mathcal{H}}^0 X/\theta $, which is the desired result.  \hfill $\square$

\section{Proof of Theorem \ref{dis-subsequence-exist}}\label{dis-existence}
In this section, we aim to prove Theorem \ref{dis-subsequence-exist}. We will construct a coupling between the discrete model and the continuous model and show the convergence in probability under the continuous setting. To be precise, for any $n\in\mathds{N}$, let $G_n$ be a random graph on $\mathds{Z}^d$ such that for any ${\bm i},{\bm j}\in\mathds{Z}^d$, ${\bm i}$ and ${\bm j}$ are connected in $G_n$ if and only if $V_{1/n}({\bm i}/n)$ and $V_{1/n}({\bm j}/n)$ are connected by $\mathcal{E}$ in the continuous model or $\bm i$ is a nearest neighbor of $\bm j$.
Then it is easy to see that $G_n$ is a critical long-range bond percolation model. Let $d^{G_n}$ be the chemical distance of the graph $G_n$ and let $\widetilde{d}_n({\bm x}, {\bm y})=d^{G_n}(\lfloor n {\bm x}\rfloor,\lfloor n {\bm y}\rfloor)$ for $ {\bm x}, {\bm y}\in\mathds{R}^d$.
Here $\lfloor n {\bm x}\rfloor=(\lfloor n{\bm x}^1\rfloor,\cdots,\lfloor n{\bm x}^d\rfloor)\in\mathds{Z}^d$ for $ {\bm x}=({\bm x}^1,\cdots,{\bm x}^d)\in\mathds{R}^d$. Note that $d^{G_n}$ has the same law as $\widehat{d}$. We will give the proof of the following proposition in the rest of the section, which directly implies Theorem \ref{dis-subsequence-exist}.
\begin{proposition}\label{dis-prop-subsequence}
    $\{\widehat{a}_n^{-1}\widetilde{d}_n\}$ is tight with respect to the topology of local uniform convergence on $\mathds{R}^{2d}$. Furthermore, for every sequence of $n$'s tending to infinity, there is a weak $\beta$-LRP metric $D$ and a subsequence $\{n_k\}$ along which $\widehat{a}_{n_k}^{-1}\widetilde{d}_{n_k}$ converges in probability to $D$.
\end{proposition}

\subsection{Tightness}
In this subsection we will prove the tightness of $\{\widehat{a}_n^{-1}\widetilde{d}_n\}$ and show that any subsequential limit satisfies Axioms V1' and V2'. In Lemma \ref{an-bounded}, we have shown that $\{\widehat{a}_n^{-1}n^\theta\}$ and $\{\widehat{a}_nn^{-\theta}\}$ are both uniformly bounded. Thus we only need to prove the tightness of $\{ n^{-\theta}\widehat{d}(\lfloor n\cdot\rfloor ,\lfloor n\cdot\rfloor)\}$ combining with the fact that $\widehat{d}$ and $d^{G_n}$ have the same law.

Our primary tool in this proof is the following uniform tail bound about the diameter under $\widehat{d}$, which was shown by \cite[Theorem 6.1]{Baumler22}.
\begin{theorem}\label{dis-upperbound}
    For any $\eta\in(0,1/(1-\theta))$, we have the following uniform upper bound about the moment generating function:
    $$
    \sup_{n\in\mathds{N}}\mathds{E}\left[\exp\left\{\left(\frac{{\rm diam}([0,n]_{\mathds{Z}}^d;\widehat{d})}{n^\theta}\right)^\eta\right\}\right]<\infty.
    $$
\end{theorem}

With the estimate above in hand, we can move on to proving the tightness of $\{ n^{-\theta}\widehat{d}(\lfloor n\cdot\rfloor ,\lfloor n\cdot\rfloor)\}$.
\begin{proposition}\label{dis-tightness}
    $\{ n^{-\theta}\widehat{d}(\lfloor n\cdot\rfloor,\lfloor n\cdot\rfloor)\}$ is tight with respect to the topology of local uniform convergence. Moreover, for any $R>0$, the family of random metrics
    $$\left\{n^{-\theta}\widehat{d}(\lfloor n\cdot\rfloor,\lfloor n\cdot\rfloor;[-nR,nR]_\mathds{Z}^d)\right\}$$
     is tight with respect to the topology of local uniform convergence on $[-R,R]^d$. Furthermore, any subsequential limit can be viewed as a continuous function on $\mathds{R}^{2d}$.
\end{proposition}

Note that $\widehat{d}(\lfloor n\cdot\rfloor,\lfloor n\cdot\rfloor)$ is not a continuous function on $\mathds{R}^{2d}$. As a result, we will review a more general version of Arzela-Ascoli theorem as follows, (see e.g. \cite[Theorem 6.2]{DJERGC14}), to deal with discontinuous functions.
\begin{lemma}\label{newAAtheorem}
{\rm(\cite[Theorem 6.2]{DJERGC14})}
    Let $X$ be a compact metric space and $Y$ be a complete metric space. Denote by $S(X,Y)$ the space of functions from $X$ to $Y$ endowed with the topology of uniform convergence on $X$. Also denote by $C(X,Y)\subset S(X,Y)$ the space of continuous functions from $X$ to $Y$. Let $\{ v_n \}_{n\in \mathds {N} }$ be a sequence in $S(X,Y)$ such that there exists a function $\omega :X\times X\to [0,\infty ]$ and a sequence $ \{\delta _{n}\}_{n\in \mathds {N} }\subset [0,\infty )$ satisfying
    \begin{align*}
        &\lim _{d_{X}(t,t')\to 0}\omega (t,t')=0,\quad \lim _{n\to \infty }\delta _{n}=0, \\
        &\forall (t,t')\in X\times X,\quad \forall n\in \mathds {N} ,\quad d_{Y}(v_{n}(t),v_{n}(t'))\leq \omega (t,t')+\delta _{n}.
    \end{align*}
    Assume also that, for all $t\in X$, $\{v_{n}(t) \}_{n\in \mathds {N}}$ is relatively compact in $Y$. Then $\{v_{n}\}_{n\in \mathds {N} }$ is relatively compact in $S(X,Y)$, and any subsequential limit of $\{v_{n}\}_{n\in \mathds {N} }$ in this space is in $C(X,Y)$.
\end{lemma}

Additionally, we will introduce a useful rule for proving the continuity of a random metric, which will be used multiple times throughout the paper.
\begin{lemma}\label{HolderLemma}
    Suppose $f(\cdot,\cdot)$ is a random pseudometric on $\mathds{R}^d$ satisfying Axioms IV' and V2'. We denote its internal metric in $[-R,R]^d$ by $f(\cdot,\cdot;[-R,R]^d)$ for $R>0$. If
    \begin{equation}\label{Cf}
        C_f:=\sup_{r>0}\mathds{E}\left[\exp\left\{\frac{{\rm diam}([0,r]^d;f)}{r^\theta}\right\}\right]<\infty,
    \end{equation}
    then for any $R>0$ and $ M>2^{\theta+1+d}>0$,
    \begin{equation*}
        \mathds{P}\left[\sup_{{\bm x},{\bm x}'\in[-R,R]^d}\frac{f({\bm x},{\bm x}';[-R,R]^d)}{\|{\bm x}-{\bm x}'\|_\infty^\theta(\log\frac{4R}{\|{\bm x}-{\bm x}'\|_\infty})}>M\right]\le 2^{d+1}C_f\exp\{-2^{-\theta-1}M\}.
    \end{equation*}
\end{lemma}
\begin{proof}
    For any fixed $R>0$ and $ M>0$,
    \begin{equation}\label{tightstep1}
        \begin{split}
            &~\mathds{P}\left[\sup_{{\bm x},{\bm x}'\in[-R,R]^d}\frac{f({\bm x},{\bm x}';[-R,R]^d)}{\|{\bm x}-{\bm x}'\|_\infty^\theta\log\frac{4R}{\|{\bm x}-{\bm x}'\|_\infty}}>M\right]\\
            \le &~\sum_{k\ge0}\mathds{P}\left[\exists {\bm x},{\bm x}'\in[-R,R]^d\ \text{with}\ \|{\bm x}-{\bm x}'\|_\infty\in[2^{-k}R,2^{-k+1}R], \frac{f({\bm x},{\bm x}';[-R,R]^d)}{\|{\bm x}-{\bm x}'\|_\infty^\theta\log\frac{4R}{\|{\bm x}-{\bm x}'\|_\infty}}>M\right]\\
            \le &~\sum_{k\ge0}\mathds{P}\left[\exists {\bm x},{\bm x}'\in[-R,R]^d\ \text{with}\ \|{\bm x}-{\bm x}'\|_\infty\in[2^{-k}R,2^{-k+1}R],\frac{f({\bm x},{\bm x}';[-R,R]^d)}{(2^{-k+1}R)^\theta}>2^{-\theta}M (k+1)\right].\\
        \end{split}
    \end{equation}
    Next, we will estimate the terms on the right hand side of the last inequality in \eqref{tightstep1}. For any fixed $k\geq 0$, we divide $[-R,R]^d$ into $2^{kd}$ small cubes of side length $2^{-k+1}R$ and denote them as $I_j$ for $j\in [0,2^{kd}-1]_\mathds{Z}$. It is easy to see that for any ${\bm x},{\bm x}'$ satisfying $\|{\bm x}-{\bm x}'\|_\infty\le2^{-k+1}R$, there exist $j_1,j_2\in[0,2^{kd}-1]_\mathds{Z}$ such that ${\bm x}\in I_{j_1}$ and ${\bm x}'\in I_{j_1}\cup I_{j_2}$ with $I_{j_1}\cap I_{j_2}\neq\emptyset$.  As a result $f({\bm x},{\bm x}';[-R,R]^d)\le2\sup_{j}{\rm diam}(I_j;f)$, which implies
    \begin{equation}\label{tightstep2}
        \begin{split}
            &~\mathds{P}\left[\exists {\bm x},{\bm x}'\in[-R,R]^d\ \text{with}\ \|{\bm x}-{\bm x}'\|_\infty\in[2^{-k}R,2^{-k+1}R],\frac{f({\bm x},{\bm x}';[-R,R]^d)}{(2^{-k+1}R)^\theta}>2^{-\theta}M (k+1)\right]\\
            \le&~\mathds{P}\left[\exists j\in[0,2^{kd}-1]_{\mathds{Z}},\frac{{\rm diam}(I_j;f)}{(2^{-k+1}R)^\theta }>2^{-\theta-1}M(k+1)\right]\\
            \le&~2^{kd}\sup_{j\in[0,2^{kd}-1]_{\mathds{Z}}}\mathds{P}\left[\frac{{\rm diam}(I_j;f)}{(2^{-k+1}R)^\theta }>2^{-\theta-1}M (k+1)\right]\\
            =&~2^{kd}\mathds{P}\left[\frac{{\rm diam}([0,2^{-k+1}R]^d;f)}{(2^{-k+1}R)^\theta }>2^{-\theta-1}M (k+1)\right].\\
        \end{split}
    \end{equation}
    Here the last equality is due to translation invariance of $f$ (Axiom IV'). By \eqref{Cf} we can use Markov's inequality to obtain
    \begin{equation}\label{tightstep3}
        \mathds{P}\left[\frac{{\rm diam}([0,2^{-k+1}R]^d;f)}{(2^{-k+1}R)^\theta}>2^{-\theta-1}M (k+1)\right]\le C_f \exp\left\{-2^{-\theta-1}M(k+1)\right\}.
    \end{equation}

     Plugging \eqref{tightstep3} to \eqref{tightstep2} and then combining with \eqref{tightstep1}, we get
    for any $M>2^{\theta+1+d}>0$,
    \begin{equation*}
        \begin{split}
        \quad\mathds{P}\left[\sup_{{\bm x},{\bm x}'\in[-R,R]}\frac{f({\bm x},{\bm x}';[-R,R]^d)}{|{\bm x}-{\bm x}'|^\theta(\log\frac{4R}{|{\bm x}-{\bm x}'|})}>M\right]&\le C_f \sum_{k\ge0} \left(2^d\exp\{-2^{-\theta-1}M\}\right)^{k+1}\\
        &\le\frac{2^dC_f \exp\{-2^{-\theta-1}M\}}{1-2e^{-2}}\\
        &\le 2^{d+1}C_f \exp\{-2^{-\theta-1}M\},
        \end{split}
    \end{equation*}
    which implies the lemma.
\end{proof}

Now we present the
\begin{proof}[Proof of Proposition \ref{dis-tightness}]
    Without loss of generality, we can assume that $R$ is a positive integer.

    Note that due to the fact that $\widehat{d}$ only takes discrete values, if ${\bm x},{\bm y}$ are too close to each other, $\|{\bm x}-{\bm y}\|_\infty$ is very small but $n^{-\theta}\widehat{d}(\lfloor n {\bm x}\rfloor,\lfloor n {\bm y}\rfloor)$ may be larger than $n^{-\theta}$. Thus we can not use Arzela-Ascoli Theorem directly. Instead, we use a general version of Arzela-Ascoli Theorem (see Lemma \ref{newAAtheorem}). To be precise, in Lemma \ref{newAAtheorem}, we take $X=[-R,R]^{2d}$ and $Y=\mathds{R}$ (both endowed with Euclidean metric). Let $v_n( {\bm x}, {\bm y})=n^{-\theta}\widehat{d}(\lfloor n {\bm x}\rfloor,\lfloor n {\bm y}\rfloor)$ for ${\bm x},{\bm y}\in\mathds{R}^d$. Note that for any ${\bm x}_1, {\bm x}_2, {\bm y}_1, {\bm y}_2\in[-R,R]^d$, we have $$|v_n( {\bm x}_1, {\bm y}_1)-v_n({\bm x}_2, {\bm y}_2)|\leq v_n( {\bm x}_1, {\bm x}_2)+v_n( {\bm y}_1, {\bm y}_2)$$ from the triangle inequality. Then it suffices to show that the following two collections of random variables are tight:
    $$
    \left\{ n^{-\theta}{\rm diam}([-nR,nR]_\mathds{Z}^d;\widehat{d}) \right\}_{n\in\mathds{N}},\quad \left\{\sup_{{\bm x},{\bm y}\in[-R,R]^d}\frac{n^{-\theta}\widehat{d}(\lfloor n {\bm x}\rfloor,\lfloor n {\bm y}\rfloor;[-nR,nR]_\mathds{Z}^d)}{\max\{\|{\bm x}- {\bm y}\|_\infty^{\theta/2},n^{-\theta}\}}\right\}_{n\in\mathds{N}}.
    $$
    Note that the first tightness is directly implied by Theorem \ref{dis-upperbound}. Thus it suffices to show that the second collection is tight.

    For any $n\in\mathds{N}$ and ${\bm x},{\bm y}\in[-R,R]^d$, if $\|{\bm x}-{\bm y}\|_\infty\leq 3/n$, then $n^{-\theta}\widehat{d}(\lfloor n{\bm x} \rfloor,\lfloor n {\bm y}\rfloor ;[-nR,nR]^d)\leq 3n^{-\theta}$; otherwise, if $\|{\bm x}-{\bm y}\|_\infty >3/n$, we have $\|\lfloor n{\bm x} \rfloor-\lfloor n{\bm y} \rfloor\|_\infty\geq n\|{\bm x}-{\bm y}\|_\infty/3\geq 1$. Combining these two cases, we have
    $$
    \sup_{{\bm x},{\bm y}\in[-R,R]^d}\frac{n^{-\theta}\widehat{d}(\lfloor n{\bm x}\rfloor,\lfloor n{\bm y}\rfloor;[-nR,nR]_\mathds{Z}^d)}{\max\{\|{\bm x}-{\bm y}\|_\infty^{\theta/2},n^{-\theta}\}}\leq 3\max\left\{1, \sup_{{\bm i},{\bm j}\in[-nR,nR]_\mathds{Z}^d}\frac{\widehat{d}({\bm i},{\bm j};[-nR,nR]_\mathds{Z}^d)}{n^{\theta/2}\|{\bm i}-{\bm j}\|_\infty^{\theta/2}} \right\}
    $$
    and it suffices to show the tightness of
    $$
    \left\{\sup_{{\bm i},{\bm j}\in[0,n]_\mathds{Z}^d}\frac{\widehat{d}({\bm i},{\bm j};[0,n]_\mathds{Z}^d)}{n^{\theta/2}\|{\bm i}-{\bm j}\|_\infty^{\theta/2}}\right\}_{n\in\mathds{N}}
    $$
    by the translation invariance of $\widehat{d}$. Without loss of generality, we only consider the case when $n=2^m$ for some $m\in\mathds{N}$. Then we apply Theorem \ref{dis-upperbound}, \eqref{tightstep1} and \eqref{tightstep2}, but replacing $f(\cdot,\cdot;[-R,R]^d)$ with $\widehat{d}(\cdot,\cdot;[0,n]_{\mathds{Z}}^d)$ and replacing $[2^{-k}R,2^{-k+1}R]$ with $[2^k,2^{k+1}]$, to get the following result for any $M>1$:
    \begin{equation}\label{prob-tight}
    \begin{split}
    \mathds{P}\left[\sup_{\bm i,\bm j\in [0,n]^d_\mathds{Z}}\frac{\widehat{d}(\bm i,\bm j;[0,n]^d_\mathds{Z})}{n^{\theta/2}\|\bm i-\bm j\|_\infty^{\theta/2}}>M\right]
    &\leq \widehat{C}\sum_{0\leq k\leq \log_2n}2^{-kd}n^d\exp\left\{-2^{-\theta-1}M(2^{-k}n)^{\theta/2}\right\}\\
    &\leq \widehat{C}\sum_{0\leq l\leq \log_2 n} 2^{-d(\log_2 n-l)}n^d\exp\left\{-2^{-\theta-1}M(2^{-(\log_2 n-l)}n)^{\theta/2}\right\}
    \\
    &\leq \widehat{C}\sum_{l\geq 0}2^{ld}\exp\left\{-2^{-\theta-1}M(2^l)^{\theta/2}\right\},
    \end{split}
    \end{equation}
where
$$
\widehat{C}:=\sup_{n\in \mathds{N}}\mathds{E}\left[\exp\left\{\frac{\text{diam}([0,n]_\mathds{Z}^d)}{n^\theta}\right\}\right]<\infty
$$
from Theorem \ref{dis-upperbound}. Hence, we obtain the desired statement since the last line of \eqref{prob-tight} vanishes as $M\to \infty$.
\end{proof}

By combining Proposition \ref{dis-tightness} with Lemma \ref{an-bounded}, we get the tightness of $\{\widehat{a}_n^{-1}\widetilde{d}_n\}$. Furthermore, we can show that any subsequential limit $D$ satisfies Axioms V1' and V2'.
\begin{lemma}\label{dis-Axiom5}
    Any subsequential limit $D$ of $\{\widehat{a}_n^{-1}\widetilde{d}_n\}$ satisfies Axioms V1' and V2', i.e.
    $$\left\{\frac{r^\theta}{D({\bm 0},([-r,r]^d)^c)}\right\}_{r>0}$$ is tight and
    $$ \sup_{r>0}\mathds{E}\left[\exp\left\{\left(\frac{{\rm diam}([0,r]^d;D)}{r^\theta}\right)^\eta\right\}\right]<\infty $$
    for all $\eta\in (0, 1/(1-\theta))$.
\end{lemma}
\begin{proof}
    Axiom V1' is immediately implied by the tightness of $$\left\{\frac{r^\theta}{n^{-\theta}\widehat{d}({\bm 0},([-nr,nr]_\mathds{Z}^d)^c)}\right\}_{r>0,nr>2}$$
    in Lemma \ref{Baumler4-10}.

    Now we prove Axiom V2'. From Theorem \ref{dis-upperbound}, we get that for all $\eta\in (0, 1/(1-\theta))$,
    \begin{equation}\label{tighteta}
    \sup_{r>0,n\in\mathds{N},nr>2}\mathds{E}\left[\exp\left\{\left(\frac{n^{-\theta}{\rm diam}([0,nr]_\mathds{Z}^d;\widehat{d})}{r^\theta}\right)^\eta\right\}\right]<\infty.
    \end{equation}
Thus, for any $\eta\in (0,1/(1-\theta))$, we choose $\eta_1\in (0,1/(1-\theta))$ such that $\eta<\eta_1$. Then applying \eqref{tighteta} with $\eta=\eta_1$ and Lemma \ref{an-bounded}, we can obtain that
    $$
    \sup_{r>0,n\in\mathds{N},nr>2}\mathds{E}\left[\exp\left\{\left(\frac{{\rm diam}([0,r]_\mathds{Z}^d;\widehat{a}_n^{-1}\widetilde{d}_n)}{r^\theta}\right)^\eta\right\}\right]<\infty.
    $$
    Since $D$ is a subsequential limit of $\{\widehat{a}_n^{-1}\widetilde{d}_n\}$, there is a sequence $\{n_k\}$ such that $\widehat{a}_{n_k}^{-1}\widetilde{d}_{n_k}$ converges to $D$ in law as $n_k\to\infty$. Hence,
    applying Fatou's Lemma to the above inequality, we arrive at
    $$
    \sup_{r>0}\mathds{E}\left[\exp\left\{\left(\frac{{\rm diam}([0,r]^d;D)}{r^\theta}\right)^\eta\right\}\right]
    \leq\sup_{r>0}\liminf_{n_k\to\infty}\mathds{E}\left[\exp\left\{\left(\frac{{\rm diam}([0,r]_\mathds{Z}^d;\widehat{a}_{n_k}^{-1}\widetilde{d}_{n_k})}{r^\theta}\right)^\eta\right\}\right]<\infty,
    $$
    which implies  Axiom V2'.
\end{proof}

\subsection{Translation invariance}
In this subsection, we will show that any subsequential limit $D$ satisfies Axiom IV' (translation invariance).

\begin{proposition}\label{dis-Axiom4}
    Assume that $\widehat{a}_{n_k}^{-1}\widetilde{d}_{n_k}$ converges to $D$ in law as $k\to\infty$ with respect to the topology of local uniform convergence on $\mathds{R}^{2d}$. Then for any ${\bm z}\in\mathds{R}^d$, $D(\cdot+{\bm z},\cdot+{\bm z})$ has the same law as $D(\cdot,\cdot)$.
\end{proposition}
\begin{proof}
    Fix ${\bm z}\in\mathds{R}^d$. For any $n_k$, we can choose $\bm m_k({\bm z})\in\mathds{Z}^d$ such that $|(1/n_k)\bm m_k({\bm z})-{\bm z}|<d/n_k$. Since $$\widetilde{d}_n(\cdot+(1/n)\bm m,\cdot+(1/n)\bm m)\stackrel{\rm law}{=} \widehat{d}(\lfloor n\cdot\rfloor +\bm m,\lfloor n\cdot\rfloor +\bm m)\stackrel{\rm law}=\widehat{d}(\lfloor n\cdot\rfloor ,\lfloor n\cdot\rfloor)\stackrel{\rm law}{=}\widetilde{d}_n(\cdot,\cdot)$$
    for any $n\in\mathds{N}$ and $\bm m\in\mathds{Z}^d$, we get that $\widehat{a}_n^{-1}\widetilde{d}_n(\cdot+(1/n_k)\bm m_k({\bm z}),\cdot+(1/n_k)\bm m_k({\bm z}))$ converges to $D$ in law with respect to the topology of local uniform convergence on $\mathds{R}^{2d}$.

    In addition, since for any ${\bm x},{\bm y}\in\mathds{R}^d$,
    $$\left|\widehat{a}_{n_k}^{-1}\widetilde{d}_{n_k}({\bm x}+(1/n_k)\bm m_k({\bm z}),{\bm y}+(1/n_k)\bm m_k({\bm z}))-\widehat{a}_{n_k}^{-1}\widetilde{d}_{n_k}({\bm x}+{\bm z},{\bm y}+{\bm z})\right|\leq 4d\widehat{a}_{n_k}^{-1}\to 0 $$
    uniformly over ${\bm x},{\bm y}$ as $n_k\to\infty$ from the fact that $|(1/n_k) \bm m_k({\bm z})-{\bm z}|<d/n_k$ and Lemma \ref{an-bounded}, we obtain that $\widehat{a}_{n_k}^{-1}\widetilde{d}_{n_k}(\cdot+{\bm z},\cdot+{\bm z})$ also converges to $D$ in law. However, from the condition in the proposition, we already have $\widehat{a}_{n_k}^{-1}\widetilde{d}_{n_k}(\cdot+{\bm z},\cdot+{\bm z})$ converges to $D(\cdot+{\bm z},\cdot+{\bm z})$ in law. As a result, $D(\cdot+{\bm z},\cdot+{\bm z})$ has the same law as $D(\cdot,\cdot)$.
\end{proof}

\subsection{Length space}
In this subsection, we will present a proposition that implies any subsequential limiting metric is, in an appropriate sense, a length space.
\begin{proposition}\label{dis-lengthspace}
    Assume that $\widehat{a}_{n_k}^{-1}\widetilde{d}_{n_k}$ converges to $D$ in law as $k \to \infty$ with respect to the topology of local uniform convergence on $\mathds{R}^{2d}$. Then $D$ is a length metric {\rm(}viewed as a metric on the quotient metric $\mathds{R}^d/\sim${\rm)}. Here $\sim$ is the equivalence relation that ${\bm x}\sim {\bm y}$ if and only if $D({\bm x}, {\bm y}) = 0$.
\end{proposition}

Before starting our proof, we will first show the fact that any $D$-bounded set is also Euclidean bounded.
\begin{lemma}\label{dis-bounded}
 Let $D$ be the limit in Proposition \ref{dis-lengthspace}. Then a.s.\ for every compact set $K \subset \mathds{R}^d$, we have
    $$
    \lim_{R\to\infty} D(K,([-R,R]^d)^c)=\infty.
    $$
    In particular, every D-bounded subset of R is also Euclidean bounded.
\end{lemma}
\begin{proof}
    Since for any compact $K$, there exists a positive integer $r > 0$ such that $K \subset [-r, r]^d$, it suffices to show that almost surely for any $r \in \mathds{N}$,
    $$
    \lim_{R\to\infty} D([-r,r],([-R,R]^d)^c)=\infty.
    $$
    Hence we only need to consider the case when $r > 0$ is fixed. This is true since we have a.s. ${\rm diam}([-r, r]^d; D) < \infty$ and $\lim_{R\to\infty} D(\bm 0, ([-R, R]^d)^c) = \infty$.
\end{proof}

Now we turn to the
\begin{proof}[Proof of Proposition \ref{dis-lengthspace}]
First, according to Skorohod representation theorem, we see that there exist a probability space and random variables  $\widehat{a}_{n_k}^{-1}\widetilde{d}'_{n_k}$ and $D'$ (equal in distribution to  $\widehat{a}_{n_k}^{-1}\widetilde{d}_{n_k}$ and $D$, respectively) on it such that $\widehat{a}_{n_k}^{-1}\widetilde{d}'_{n_k}$ converges a.s.\ to  $D'$.

Since $D'$ can be viewed as a continuous function on $\mathds{R}^{2d}$ (see Proposition \ref{dis-tightness}), it is a complete metric due to the completeness of $\mathds{R}^d$. Therefore, by e.g. \cite[Theorem 2.4.16]{BBI01}, it suffices to show that for any points ${\bm z}, {\bm w} \in \mathds{R}^d$, there exists a midpoint between ${\bm z}$ and ${\bm w}$, i.e., a point ${\bm x} \in \mathds{R}^d$ such that $D'({\bm z}, {\bm x}) = D'({\bm w}, {\bm x}) = \frac{1}{2} D'({\bm z}, {\bm w})$.

    To this end, for any fixed $r\in \mathds{N}$, let $E_r$ be the event that $\lim_{R\to \infty} D'([-r,r]^d,([-R,R]^d)^c)=\infty$. By Lemma \ref{dis-bounded}, we have $\mathds{P}[E_r]=1$. We also let $F$ be the event that $\widehat{a}_{n_k}^{-1}\widetilde{d}'_{n_k}$ converges to  $D'$. Then from the above analysis, we get $\mathds{P}[F]=1$. This implies that $\mathds{P}[\cap_{r\in\mathds{N}}E_r\cap F]=1$. Now on the event $E_r\cap F$, we can see that there is a random $R>0$ such that
    \begin{equation}\label{rR}
    D'(\bm z,\bm w)+1<D'([-r,r]^d,([-R,R]^d)^c)\quad \text{for all }\bm z,\bm w\in [-r,r]^d.
    \end{equation}
    Additionally, for fixed  ${\bm z},{\bm w}\in[-r,r]^d$, from the definition of chemical distance $\widetilde{d}'_n$, it is easy to show that there exists ${\bm x}_{n_k}$ such that $|\widetilde{d}'_{n_k}({\bm z},{\bm x}_{n_k})-\frac{1}{2}\widetilde{d}'_{n_k}({\bm z},{\bm w})|<1$ and $|\widetilde{d}'_{n_k}({\bm w},{\bm x}_{n_k})-\frac{1}{2}\widetilde{d}'_{n_k}({\bm z},{\bm w})|<1$. We claim that on the event $E_r\cap F$, we have that there are at most finitely many $\bm x\in\{\bm x_{n_k}\}$ such that $\bm x\in ([-R,R]^d)^c$ (here $R$ is chosen in \eqref{rR}). Indeed, we will prove this claim by contradiction.
    Assume (otherwise) that there is a random subsequence $\{n_{k_l}\}$ of $\{n_k\}$ such that
    $\bm x_{n_{k_l}}\in ([-R,R]^d)^c$. Since $\widehat{a}_{n_k}^{-1}\widetilde{d}'_{n_k}$ converges to  $D'$ on the event $F$, we have that
    there is a random $N_1>0$ such that
    \begin{equation}\label{ankl}
    |\widehat{a}_{n_k}^{-1}\widetilde{d}'_{n_k} (\cdot,\cdot)- D'(\cdot,\cdot)|<1/2.
     \end{equation}
     Then we get that on the event $E_r\cap F$, for all $n_{k_l}\geq N_1$,
    \begin{equation*}
    \begin{split}
    \widehat{a}_{n_{k_l}}^{-1}\widetilde{d}'_{n_{k_l}}(\bm z,\bm x_{n_{k_l}})>D'(\bm z,\bm x_{n_{k_l}})-1/2>D'(\bm z,\bm w)+1/2> \widehat{a}_{n_{k_l}}^{-1}\widetilde{d}'_{n_{k_l}}(\bm z,\bm w),
    \end{split}
    \end{equation*}
     where the first and the last inequalities used \eqref{ankl}, and the second inequality used \eqref{rR}. This arrives at a contradiction to the definition of $\bm x_{n_{k_l}}$. Having proved that on the event $E_r\cap F$, there are at most finitely many $\bm x\in\{\bm x_{n_k}\}$ such that $\bm x\in ([-R,R]^d)^c$, we see that there is a random subsequence $\{n_{k_l}\}$ of $\{n_k\}$ and a random point $\bm x$ such that $\bm x_{n_{k_l}}\to \bm x$.
 Thus combining with Lemma \ref{an-bounded} which shows $\widehat{a}_n\to\infty$, we get that on the event $E_r\cap F$,
    $$
    D'({\bm z},{\bm x})=\lim_{l\to\infty}\widehat{a}_{n_{k_l}}^{-1}\widetilde{d}'_{n_{k_l}}({\bm z}, {\bm x}_{n_{k_l}}) = \frac{1}{2}\lim_{l\to\infty}\widehat{a}_{n_{k_l}}^{-1}\widetilde{d}'_{n_{k_l}}({\bm z}, {\bm w}) = \frac{1}{2}D'({\bm z},{\bm w})\quad \text{for all }\bm z,\bm w\in [-r,r]^d.
    $$
    Similarly, we get $D'({\bm z},{\bm x})=D'({\bm w},{\bm x})=\frac{1}{2}D'({\bm z},{\bm w})$. Finally, we obtain the desired statement in the proposition by the arbitrariness of $r\in \mathds{N}$.
\end{proof}

\subsection{Strong regularity of subsequential limiting metrics}

In this subsection, we will establish the strong regularity (introduced in Proposition \ref{Dxy=0}) of the subsequential limiting metrics of  $\{\widehat{a}_n^{-1}\widetilde{d}_n\}$ as follows.

\begin{proposition}\label{SRSL}
Any subsequential limit $D$ of $\{\widehat{a}_n^{-1}\widetilde{d}_n\}$ satisfies the strong regularity, i.e., the following holds almost surely. For any $\bm x,\bm y\in \mathds{R}^d$, $D(\bm x,\bm y)=0$ if and only if $\langle \bm x, \bm y\rangle\in \mathcal{E}$.
\end{proposition}

To prove Proposition \ref{SRSL}, we first do some preparations. For convenience, let $\mathcal{D}_n=\widehat{a}_n^{-1}\widetilde{d}_n$ for all $n\geq 1$, and let
$$
\widehat{\mathcal{E}}_n=\left\{\langle\bm x,\bm y \rangle: \bm x,\bm y\in\mathds{R}^d \text{ such that }V_{1/n}(\lfloor n\bm x\rfloor/n) \text{ and }V_{1/n}(\lfloor n\bm y\rfloor/n)\text{ are connected by }\mathcal{E}\right\}.
$$
We also need the following definition of ``good'' cubes with respect to $\mathcal{D}_n$, which is similar with Definition \ref{h-good}.

\begin{definition} Fix $n\geq 1$.
For $s\gg 1/n$, ${\bm z}\in\mathds{R}^d$ and $\alpha\in(0,1)$, we say that a cube $V_{3s}({\bm z})$ is $(3s,\alpha)$-good with respect to $\mathcal{D}_n$ if the following condition holds. For any two different edges $\langle {\bm u}_1,{\bm v}_1\rangle,\langle {\bm u}_2, {\bm v}_2\rangle\in \widehat{\mathcal{E}}_n$, with ${\bm u}_1\in V_s({\bm z})^c$, ${\bm v}_1\in V_s({\bm z})$, ${\bm u}_2\in V_{3s}({\bm z})$ and ${\bm v}_2\in V_{3s}({\bm z})^c$, we have $|{\bm v}_1-{\bm u}_2|\geq \alpha s$. Additionally, there is a constant $b>0$ (which does not depend on $\bm z$ or $s$ and will be chosen in Lemma \ref{h-regular-low} below) such that
$$
\mathcal{D}_n({\bm v}_1,{\bm u}_2;V_{3s}({\bm z}))\geq (b\alpha s)^\theta.
$$
\end{definition}

Similar to Lemma \ref{h-probgood}, we have the following result.

\begin{lemma}\label{h-probgood-2}
Fix $n\geq 1$. For any ${\bm z}\in\mathds{R}^d$, $s>0$ and sufficiently small $\alpha\in(0,1)$, there exist constants $b=b(\alpha)>0$ {\rm(}depending only on $d,\beta$ and $\alpha${\rm)} and $c_1>0$ {\rm(}depending only on $d$ and $\beta${\rm)} such that with probability at least $1-c_1\alpha \log(1/\alpha)$, $V_{3s}({\bm z})$ is $(3s,\alpha)$-good with respect to $\mathcal{D}_n$.
\end{lemma}

As the proof of Lemma \ref{h-probgood-2} closely resembles that of Lemma \ref{h-probgood}, we will focus on highlighting the key differences in the proof below, while omitting the explicit details of the proof itself.
The proof of Lemma \ref{h-probgood-2} relies on obtaining versions of Lemmas \ref{h-longedge} and \ref{h-regular-low} by replacing $\mathcal{E}$ with edges in $G_n$ and replacing $D$ with $\mathcal{D}_n$. To achieve this, we only need to replace all instances of points $\bm u$ used in the proofs of these two lemmas with the cubes $V_{1/n}(\lfloor n\bm u\rfloor/n)$. Additionally, we substitute Axiom V1' in the proof of Lemma \ref{h-regular-low} with the tightness of $\{\mathcal{D}_n\}$ (see Propositions \ref{discrete-dist} and \ref{dis-prop-subsequence}). Then from the self-similarity of the model, we can deduce the desired versions of Lemmas  \ref{h-longedge} and \ref{h-regular-low}  with respect to $\mathcal{D}_n$. Finally, following a similar approach as the proof of Lemma \ref{h-probgood}, we can complete the proof of Lemma  \ref{h-probgood-2}. It is worth emphasizing that the self-similarity and the tightness of $\{\mathcal{D}_n\}$ ensure that the parameters $\alpha, b(\alpha)$ and $c_1$ in Lemma \ref{h-probgood-2} do not depend on $n$.

We next introduce the renormalization for $(\mathds{R}^d, \widehat{\mathcal{E}}_n)$. For fixed $s\gg 1/n$, we divide $\mathds{R}^d$
into small cubes of side length $s$, denoted by $V_{s}(\bm k)$ for $\bm k\in s\mathds{Z}^d$. Then we identify the cubes $V_{s}(\bm k)$ as vertices $\bm k$ and call the resulting graph $\mathcal{G}_n$. We say $\bm k$ is good in $\mathcal{G}_n$ if $V_{3s}(\bm k)$ is  $(3s,\alpha)$-good with respect to $\mathcal{D}_n$. Similar with Corollary \ref{h-good-prob-2}, from Lemma \ref{h-probgood-2}, we see that for sufficiently small $\delta_0>0$ with $2C_{dis}\delta_0<1$ (where $C_{dis}$ is defined in Lemma \ref{number-path-k}), there exist sufficiently small $\alpha=\alpha(\delta_0)>0$ and $b=b(\alpha)>0$ (depending only on $d,\beta$ and $\delta_0$) such that
\begin{equation}\label{cor3.8-2}
\mathds{P}[\bm k \text{ is good}]>1-\delta_0\quad \text{for all }\bm k\in s\mathds{Z}^d.
\end{equation}

We refer to a path $P^{\mathcal{G}_n}$ as $\alpha$-good in $\mathcal{G}_n$ if, upon replacing $\mathcal{G}_n$ for $G$ in Definition \ref{hd-PGgood} and replacing $\{E_{\bm k}\}$ for the whole probability space, the conditions in the definition hold.
\begin{lemma}\label{hd-BK-2}
 Let $\delta\in(0,1)$.
        Then there exists  a constant $\alpha_0>0$ {\rm(}depending only on $d,\beta$ and $\delta${\rm )} such that for all $\alpha\in(0,\alpha_0)$ and for any fixed self-avoiding path $P^{\mathcal{G}_n}$ with graph length $L$,
        We have
        $$
        \mathds{P}\left[\text{$P^{\mathcal{G}_n}$ is $\alpha$-good in $\mathcal{G}_n$}\right]\geq 1-\delta^L.
        $$
\end{lemma}

The proof of Lemma \ref{hd-BK-2} is similar to that of Lemma \ref{hd-BK} (just replacing Lemma \ref{h-probgood} with Lemma \ref{h-probgood-2}, and replacing Corollary \ref{h-good-prob-2} with \eqref{cor3.8-2}). We thus omit the proof.

With the above lemma at hand, we have the following result.
\begin{lemma}\label{Dinterval=0-2}
    Fix $n\geq 1,k \in\mathds{Z}$ with $1/n\ll 2^k$ and $\delta>0$ such that $C_{dis}\delta<1$. For any $ \bm l\in2^k\mathds{Z}^d$, let $V^k(\bm l)=V_{2^k}(\bm l)$. Let $E_{\bm i,\bm j}^{n,k}$ {\rm(}for $\|\bm i-\bm j\|_{1}>3\cdot 2^k${\rm)} be the event that $\mathcal{D}_n(V^k(\bm i),V^k(\bm j))<(2^kb\alpha)^\theta$ and $V^k(\bm i)$ and $V^k(\bm j)$ are not directly connected by $\widehat{\mathcal{E}}_n$. Then
    $$\mathds{P}[E_{\bm i,\bm j}^{n,k}]\leq \frac{C_{dis}^2\delta}{1-C_{dis}\delta}.$$
\end{lemma}

Although Lemma \ref{Dinterval=0-2} appears to be stronger than Lemma \ref{Dinterval=0} (since we prove an upper bound for the probability of the quantitative event event $\mathcal{D}_n(V^k(\bm i),V^k(\bm j))<(2^kb\alpha)^\theta$ in Lemma \ref{Dinterval=0-2} rather than the qualitative event $D(V^k(\bm i),V^k(\bm j))=0$ in Lemma \ref{Dinterval=0}), the analysis above \eqref{EEtilde} actually leads to a stronger version of Lemma \ref{Dinterval=0} which is similar to Lemma \ref{Dinterval=0-2}.
Thus, replacing Lemma \ref{hd-BK} with Lemma \ref{hd-BK-2} in the proof of  Lemma \ref{Dinterval=0}, we can complete the proof of Lemma \ref{Dinterval=0-2}. We omit the details here.

It is worth emphasizing that we obtain a uniform tail  for all $\mathcal{D}_n$ in Lemma \ref{Dinterval=0-2}. Thus from it we can derive the following result.

\begin{lemma}\label{Dinterval=0-3}
    Let $D$ be a subsequential limit of $\{\mathcal{D}_n\}$. For any $ k\in\mathds{Z}$ and $ \bm l\in2^k\mathds{Z}^d$, let $V^k(\bm l)=V_{2^k}(\bm l)$. Let $E_{\bm i,\bm j}^k$ {\rm(}for $\|\bm i-\bm j\|_{1}>3\cdot 2^k${\rm)} be the event that $D(V^k(\bm i),V^k(\bm j))=0$ and $V^k(\bm i)$ and $V^k(\bm j)$ are not directly connected by $\mathcal{E}$. Then $\mathds{P}[E_{\bm i,\bm j}^k]=0$.
    As a consequence,
    $$\mathds{P}\left[\bigcup_{k\in \mathds{Z}}\bigcup_{\bm i,\bm j\in 2^k\mathds{Z}^d: \|\bm i-\bm j\|_{1}>3\cdot 2^k}E_{\bm i,\bm j}^k\right]=0.$$
\end{lemma}
\begin{proof}
Assume that $D$ is the limit of $\{\mathcal{D}_{n_l}\}$. By Lemma \ref{Dinterval=0-2}, for all $k\in\mathds{Z}$ with $1/n_l\ll 2^k$ we have
\begin{equation}\label{Dnk}
\mathds{P}[E_{\bm i,\bm j}^{n_l,k}]\leq \frac{C_{dis}^2\delta}{1-C_{dis}\delta}.
\end{equation}
Letting $n_l\to \infty$ and then letting $\delta\to 0$, we can obtain $\mathds{P}[E_{\bm i,\bm j}^k]=0$.
\end{proof}

With the preceding preparations, the proof of Proposition \ref{SRSL} is similar to that of Proposition \ref{Dxy=0-local} just replacing  Lemma \ref{Dinterval=0} with Lemma \ref{Dinterval=0-3}. We also omit the details here.

\subsection{Subsequential limits}
In this subsection, we consider the convergence of internal metrics of $\widehat{a}_n^{-1}\widetilde{d}_n$, construct a subsequential limit and show that the internal metric of the limit is the limit of the internal metric, which slightly strengthens the result in Proposition \ref{dis-tightness}.

We start with the definition of dyadic cubes, which will be used to approximate general open sets later.

\begin{definition}\label{dyadic}
    A closed cube $I\subset \mathds{R}^d$ is \textit{dyadic} if $I$ has side length $2^k$ and vertices in $2^k\mathds{Z}^d$ for some $k\in\mathds{Z}$. We say that $W\subset \mathds{R}^d$ is a \textit{dyadic set} if there exists a finite collection of dyadic cubes $\mathcal{I}$ such that $W$ is the interior of $\cup_{I\in \mathcal{I}}I$. Note that a dyadic set is a bounded open set.
\end{definition}

\begin{lemma}\label{dis-limitonset}
    Let $\mathcal{W}$ be the set of all dyadic sets. For any sequence $\{n_k\}_{k\geq 1}\subset \mathds{N}$ tending to infinity, there is a subsequence $\{n'_k\}_{k\geq 1}$ and a random length metric $D$ such that the following is true. We have the convergence of joint laws
    \begin{equation}\label{dis-jointlaw}
        \left(\widehat{a}_{n'_k}^{-1}\widetilde{d}_{n'_k},\ \{\widehat{a}_{n'_k}^{-1}\widetilde{d}_{n'_k}(\cdot,\cdot;\overline{W})\}_{W\in\mathcal{W}}\right)\to \left(D,\{D_W\}_{W\in\mathcal{W}}\right),
    \end{equation}
    where the first coordinate is endowed with the topology of local uniform convergence on $\mathds{R}^d\times \mathds{R}^d$ and each element of the collection in the second coordinate is endowed with the topology of local uniform convergence on $\overline{W}\times \overline{W}$. Furthermore, for each $W\in\mathcal{W}$ we have a.s.\ $D_W(\cdot,\cdot;W)=D(\cdot,\cdot;W)$.
\end{lemma}

We now proceed with the proof of Lemma \ref{dis-limitonset}.
First, the following lemma is from Propositions \ref{dis-tightness} and \ref{dis-lengthspace} immediately, which gives the tightness of the internal metrics associated to $\widehat{a}_n^{-1}\widetilde{d}_n$.

\begin{lemma}\label{dis-internaltight}
    Let $R>0$. The laws of the internal metrics $\{\widehat{a}_n^{-1}\widetilde{d}_n(\cdot,\cdot;[-R,R]^d)\}_{n\geq 1}$ are tight with respect to  the topology of local uniform convergence on $\mathds{R}^{2d}$ and any subsequential limit of these laws is supported on length metrics which can also be viewed as a continuous function.
\end{lemma}

Recall that in this paper, we use the term ``length metric'' to refer to a metric $D$ that becomes a length metric when considered on the quotient space $\mathds{R}^d/\sim$. Here, the equivalence relation $\sim $ is defined such that ${\bm x}\sim {\bm y}$ if and only if $D({\bm x},{\bm y})=0$.

We now generalize the internal metrics on cubes to internal metrics on closures of dyadic sets.

\begin{lemma}\label{dis-tightDW}
    Let $W\subset \mathds{R}^d$ be a dyadic set. The laws of the internal metrics $\{\widehat{a}_n^{-1}\widetilde{d}_n(\cdot,\cdot;\overline{W})\}_{n\geq 1}$ are tight with respect to the topology of local uniform convergence on $\overline{W}\times\overline{W}$ and any subsequential limit of these laws is supported on length metrics.
\end{lemma}

 The proof of the lemma above is similar to that of Lemma  \ref{dis-internaltight} (just replacing $[-R,R]^d$ with $W$). So we omit the proof.

The last lemma we need for the proof of Lemma \ref{dis-limitonset} is the following deterministic compatibility statement for limits of internal metrics.

\begin{lemma}\label{internalmetric}
    Let $V\subset U\subset \mathds{R}^d$ be open. Let $\{D^n\}$ be a sequence of 
    length metrics on $U$ which converges to a 
    length metric $D_U$ {\rm(}with respect to the topology of local uniform convergence on $U\times U${\rm)} satisfying Axiom I and the strong regularity in Proposition \ref{SRSL}. Suppose also that $D^n(\cdot,\cdot;\overline{V})$ converges to a 
    length metric $D_V$ {\rm(}with respect to the topology of local uniform convergence on $V\times V${\rm)}. Then $D_U(\cdot,\cdot;V)=D_V(\cdot,\cdot;V)$.
\end{lemma}

\begin{proof}

Suppose ${\bm x},{\bm y}\in V$ satisfy that $D_U({\bm x},{\bm y})<D_U({\bm x},V^c)$. Since $D_U$ is a length metric, we have that $D_U({\bm x},{\bm y})=D_U({\bm x},{\bm y};V)=D({\bm x},{\bm y};\overline{V})$. Furthermore, for large enough $n\in\mathds{N}$ we have $D^n({\bm x},{\bm y})<D^n({\bm x},V^c)$, which implies that $D^n({\bm x},{\bm y})=D^n({\bm x},{\bm y};V)=D^n({\bm x},{\bm y};\overline{V})$. Therefore, $D^n({\bm x},{\bm y})$ converges to both $D_U({\bm x},{\bm y})=D_U({\bm x},{\bm y};V)$ and $D_V({\bm x},{\bm y})$. Consequently, $D_U({\bm x},{\bm y};V)=D_V({\bm x},{\bm y})$ for each ${\bm x},{\bm y}\in V$ with $D_U({\bm x},{\bm y})<D_U({\bm x},V^c)$.

To prove the general case, we first define some notations.
For a fixed $\delta>0$ and $\varepsilon \ll \delta$ , let $V_\delta=\{\bm x\in V: \text{dist}(\bm x,V^c)\geq \delta\}$ and let $A_\varepsilon$ be the event that there is at most one long edge connecting $V_{\varepsilon}(\bm x)$ to $V^c$ for all $\bm x\in (\varepsilon/2)\mathds{Z}^d\cap V_{\delta}$. By the property of the Poisson point process, one has that
\begin{equation}\label{Avareps}
\begin{split}
\mathds{P}[A_{\varepsilon}^c]&\leq\sum_{\bm x\in (\varepsilon/2)\mathds{Z}^d\cap V_{\delta}} \left[1-\exp\left\{  -\beta\int_{V_{\varepsilon}(x)}\int_{V^c}\frac{1}{|\bm z-\bm y|^{2d}}\d \bm z\d \bm y\right\}\right]^2\\
&\leq \sum_{\bm x\in (\varepsilon/2)\mathds{Z}^d\cap V_{\delta}} \left[1-\exp\left\{  -\beta\int_{V_{\varepsilon}(x)}\left(\int_{V_{\delta-\varepsilon}(\bm y)^c}\frac{1}{|\bm z-\bm y|^{2d}}\d \bm z\right)\d \bm y\right\}\right]^2\\
&\leq \sum_{\bm x\in (\varepsilon/2)\mathds{Z}^d\cap V_{\delta}} (1-\exp\{-c_1\beta\varepsilon^d\delta^{-d}\})^2 \leq c_2 \varepsilon^d/\delta^{2d},
\end{split}
\end{equation}
where $c_1,c_2$ are two constants depending only on $d,\beta$ and the Euclidean diameter of $V_\delta$.

In the following, we assume that the event $A_\varepsilon$ occurs. Then for any $\bm z\in V_\delta$ with $V_{\varepsilon/2}(\bm z)\cap V_\delta\neq \emptyset$, there is at most one long edge connecting $V_{\varepsilon/2}(\bm z)$ and $V^c$. In addition, for a  $D_U$-geodesic  $P_{\bm z,V^c,\bm z}$ from $V_{\varepsilon/2}(\bm z)$ to $V^c$ and then back to $V_{\varepsilon/2}(\bm z)$, if there exists a long edge connecting  $V_{\varepsilon/2}(\bm z)$ and $V^c$, $P_{\bm z,V^c,\bm z}$ can use that edge at most once. Combined with Axiom I and Proposition \ref{SRSL}, this implies that $\text{len}(P_{\bm z,V^c,\bm z};D_U)>0$. Combining this with the fact that $\lim_{r\to0}\text{diam}(V_r(\bm z);D_U)=0$, we can see that there is a (random) $0<r(\bm z)<\varepsilon/2$ such that
\begin{equation*}
\text{diam}(V_{r(\bm z)}(\bm z);D_U)<\text{len}(P_{\bm z,V^c,\bm z};D_U).
\end{equation*}
Therefore on the event $A_\varepsilon$, for any two points $\bm x,\bm y\in V_{r(\bm z)}(\bm z)$ there is a $D_U$-geodesic from $\bm x$ and $\bm y$ that is contained in $V$, that is,
\begin{equation}\label{xyVc}D_U(\bm x,\bm y)< \max\{D_U(\bm x, V^c), D_U(\bm y, V^c)\}.\end{equation}
Combining this with the first paragraph in the proof, we arrive at $D_U(\bm x,\bm y;V)=D_V(\bm x,\bm y)$.

Now for any path $P$ in $V$, $\{V_{r(\bm z)}(\bm z)\}_{z\in P}$ is a family of open covers for $P$. Since $P$ is a compact set in $\mathds{R}^d$, according to Heine-Borel-Lebesgue property, there is a finite subcover, denoted by $\{V_{r(\bm z_k)}(\bm z_k)\}_{k\in [1,N]_\mathds{Z}}$, of the open cover $\{V_{r(\bm z)}(\bm z)\}_{z\in P}$. Now for any partition $P(t_0),\cdots, P(t_{N'+1})$ of $P$ such that $(P(t_i),P(t_{i+1}))$ is contained in some $V_{r(\bm z_k)}(\bm z_k)$ for all $i\in [0,N']_\mathds{Z}$, from \eqref{xyVc} and the first paragraph above, we get that $D_U(P(t_i),P(t_{i+1});V)=D_V(P(t_i),P(t_{i+1}))$ for all $i\in [0,N']_\mathds{Z}$. Combining this with the definition of $D_U$-length of $P$, i.e.,
     $$
     \text{len}(P;D_U)=\sup_{T}\sum_{i=1}^{\#T} D_U(P(t_i),P(t_{i+1})),
     $$
     we get that $\text{len}(P;D_U)=\text{len}(P;D_V)$. Furthermore, from this, the definition of $D_U(\cdot,\cdot;V)$ and the fact that $D_U$ and $D_V$ are length metrics, we conclude that $D_U({\bm x},{\bm y};V)=D_V({\bm x},{\bm y})$ for all ${\bm x},{\bm y}\in V$.
\end{proof}



We now present the

\begin{proof}[Proof of Lemma \ref{dis-limitonset}]
    From Propositions \ref{dis-tightness} and \ref{dis-lengthspace} we first see that metrics $\widehat{a}_n^{-1}\widetilde{d}_n$ are tight with respect to the local uniform topology on $\mathds{R}^{2d}$ and any subsequential limit in law is a.s.\ 
    a length metric on $\mathds{R}^d$. Now applying Lemma \ref{dis-tightDW} and the Prokhorov theorem, we get that the joint law of the metrics on the left hand side of \eqref{dis-jointlaw} is tight. Moreover, any subsequential limit of these joint laws is a coupling of a 
    length metric $D$ on $\mathds{R}^d$ and a length metric $D_W$ on $\overline{W}$ for each $W\in\mathcal{W}$. 
    We then apply Proposition \ref{SRSL} and Lemma \ref{internalmetric} to say that $D_W(\cdot,\cdot;W)=D(\cdot,\cdot;W)$ for each $W\in\mathcal{W}$  (note that $D$ satisfies Axiom III because of local uniform convergence).
\end{proof}

\subsection{Weak locality}
    In this subsection, we will prove Axiom II'' (weak locality) for any subsequential limit of $\widehat{a}_n^{-1}\widetilde{d}_n$ with the help of Lemma \ref{dis-limitonset}.

\begin{lemma}\label{dis-weaklocality2}
    Let $(\mathcal{E},D)$ be any subsequential limit of the laws of $(\mathcal{E},\widehat{a}_n^{-1}\widetilde{d}_n)$. Then $D$ satisfies Axiom II'' {\rm(}weak locality{\rm)}.
\end{lemma}
\begin{proof}
Let $\{n_k\}$ be any subsequence such that $(\mathcal{E},\widehat{a}_n^{-1}\widetilde{d}_{n_k})$ converges to $(\mathcal{E},D)$ in law, and
    let $V_1,V_2,\cdots,V_N \subset \mathds{R}^d$ be disjoint open sets for some $N\in\mathds{N}$.
    We now fix dyadic sets $W_1,W_2,\cdots, W_N$ with $W_i\subset \overline{V}_i$ (we will eventually let all $W_i$ increase to all of $V_i$, respectively). By the independence of the Poisson point process, we obtain that when the Euclidean distance between $W_i$ and $\overline{V}_i^c$ is larger than $dn^{-1}$ for all $i\in[1,N]_\mathds{Z}$, we have that
    \begin{equation}\label{dis-independent}
        \left\{\left(\mathcal{E}|_{W_i\times W_i}, \widehat{a}_n^{-1}\widetilde{d}_n(\cdot,\cdot;\overline{W}_i)\right)\right\}_{ i\in[1,N]_\mathds{Z}}\quad\text{and}\quad\mathcal{E}|_{(\cup_{i=1}^N(V_i\times V_i))^c}\quad   \text{are independent}.
    \end{equation}

    In addition, applying Lemma \ref{dis-limitonset} by possibly replacing $\{n'_k\}$ with a deterministic subsequence, there exists
    a coupling $(\mathcal{E},D,D_{W_1},\cdots, D_{W_N})$ of $(\mathcal{E},D)$ with length metrics $D_{W_1}(\cdot,\cdot;\overline{W}_1)$, $\cdots, D_{W_N}(\cdot,\cdot;\overline{W}_N)$
    on $\overline{W}_1,\cdots,\overline{W}_N$, respectively, such that
    \begin{equation}\label{dis-DW=D}
    D_{W_i}(\cdot,\cdot;W_i)=D(\cdot,\cdot;W_i)\quad\text{for all }i\in[1,N]_\mathds{Z}
    \end{equation}
    and that the following holds. We have the convergence of joint laws
    \begin{equation}\label{dis-convergence}
    \left(\mathcal{E},\widehat{a}_n^{-1}\widetilde{d}_n,\widehat{a}_n^{-1}\widetilde{d}_n(\cdot,\cdot;\overline{W}_1),\cdots,\widehat{a}_n^{-1}\widetilde{d}_n(\cdot,\cdot;\overline{W}_N)\right)\to (\mathcal{E},D,D_{W_1},\cdots,D_{W_N})
    \end{equation}
    along $n\in\{n_k\}$ where the last $N$ coordinates are endowed with the topology of uniform convergence on $\overline{W}_1\times \overline{W}_1,\cdots, \overline{W}_N\times \overline{W}_N$, respectively. Since the independence is preserved under convergence in law, from \eqref{dis-independent} and \eqref{dis-convergence} we obtain that $(\mathcal{E}|_{W_i\times W_i},D_{W_i})$, $i\in [1,N]_\mathds{Z}$ are independent.  Combining this with \eqref{dis-DW=D}, we further have
    \begin{equation*}
    \left\{(\mathcal{E}|_{W_i\times W_i},D(\cdot,\cdot;W_i))\right\}_{i\in[1,N]_\mathds{Z}}\quad \text{and}\quad \mathcal{E}|_{(\cup_{i=1}^N(V_i\times V_i))^c}\quad \text{are independent}.
    \end{equation*}
    Letting $W_i$ increases to $V_i$ for all $i\in[1,N]_\mathds{Z}$, one can get that $\{(\mathcal{E}|_{W_i\times W_i},D(\cdot,\cdot;W_i))\}_{i\in[1,N]_\mathds{Z}}$ a.s. converges to $\{(\mathcal{E}|_{V_i\times V_i},D(\cdot,\cdot;V_i))\}_{i\in[1,N]_\mathds{Z}}$ since $D$ is a length metric. Hence,
     we conclude the proof.
\end{proof}

\subsection{Measurability}
In this subsection we prove Proposition \ref{reallocal} below. The proof of Proposition \ref{reallocal} is essentially the same as that of \cite[Corollary 1.8]{GM19c}, and we reproduce its proof here merely for completeness.

\begin{proposition}\label{reallocal}
    If $D$ is a local $\beta$-LRP metric, then $D$ satisfies Axiom II {\rm(}locality{\rm)}. That is, $D$ is a weak $\beta$-LRP metric.
\end{proposition}

Throughout the proof, assume that $D$ is a local $\beta$-LRP metric. Let $D, \widetilde{D}$ be conditionally i.i.d. samples from the conditional law of $D$ given $\mathcal{E}$. Then from Proposition \ref{bilip}, there exists a deterministic constant $C>0$  such that a.s. for each ${\bm x},{\bm y} \in \mathds{R}^d$,
\begin{equation*}
    \widetilde{D}({\bm x},{\bm y})\leq  CD({\bm x},{\bm y}).
\end{equation*}

We use the Efron-Stein inequality \cite{ES81} to prove Proposition \ref{reallocal}. This inequality states that for any measurable function $F = F(X_1,\cdots,X_n)$ of $n$ independent random variables, we have
\begin{equation}\label{ESinequality}
    {\rm Var}[F]\leq\sum_{i=1}^n {\rm Var}[F|\{ X_j\}_{i\neq j}].
\end{equation}
To implement \eqref{ESinequality} in our scenario, we will partition $\mathds{R}^d$ into a grid of $d$ dimensions (shifted randomly for technical reasons; see Lemma \ref{divide}). Using the weak locality of $D$, we will establish that the internal metrics of $D$ on these grid cubes are conditionally independent, given $\mathcal{E}$. Additionally, we will demonstrate that $D$ is almost surely determined by these internal metrics (Lemma \ref{weakmeasurable}).

Next, we will select fixed $\bm x,\bm y\in\mathds{R}^d$ and apply \eqref{ESinequality} to the conditional distribution of the random variable $F=D(x,y)$ given $\mathcal{E}$. To achieve this, we must control the conditional variance when re-sampling the internal metric on one cube $S$. To accomplish this, we will utilize a $D$-geodesic $P$ from $\bm x$ to $\bm y$ and leverage the bi-Lipschitz equivalence from Proposition \ref{bilip}. This equivalence will allow us to bound the difference between the original value of $D(\bm x,\bm y)$ and the new value after re-sampling the internal metric on $S$, ensuring that this difference is no greater than a constant times the $D$-length of $P\cap S$.

As we send the mesh size to zero, the sum of the squared error ${\rm len}(P\cap S)^2$ over all $S$ will almost surely converge to zero. This convergence will demonstrate that ${\rm Var}[D(x,y)|\mathcal{E}]=0$ and, therefore, that $D$ is almost surely determined by $\mathcal{E}$; this implies the desired result finally.

We will now define the grid used in our proof. First, we sample ${\bm \phi}$ uniformly from the Lebesgue measure on $[0,1]^d$, independently from all other random variables.
We then construct the randomly shifted $d$-dimensional grid, denoted as $\mathcal{G}_{\bm \phi}$. Specifically,  $\mathcal{G}_{\bm \phi}=\cup_{{\bm x}\in\mathds{Z}^d}({\bm x}+{\bm \phi}+\partial [0,1]^d)$, where $\partial A$ means the boundary of $A$ for a set $A\subset \mathds{R}^d$.

The reason for introducing the random shift ${\bm \phi}$ is to control the $D$-length of a path on the boundaries of our cubes in the partition.

\begin{lemma}\label{divide}
    Let $P : [0,{\rm len}(P;D)]\to \mathds{R}^d$ {\rm(}parameterized by $D$-length{\rm)} be a random path with finite $D$-length chosen in a manner depending only on $(\mathcal{E},D)$ {\rm(}not on ${\bm \phi}${\rm)}. For each $\varepsilon > 0$, a.s.\ ${\rm len}(P;D) = {\rm len}(P \setminus (\varepsilon \mathcal{G}_{\bm \phi});D)$.
\end{lemma}

We note that $P\setminus (\varepsilon \mathcal{G}_{\bm \phi})$ is a countable union of excursions of $P$ in $\mathds{R}^d\setminus (\varepsilon \mathcal{G}_{\bm \phi})$, so its $D$-length is well-defined.

\begin{proof}[Proof of Lemma \ref{divide}]
    For each fixed $t \in [0, {\rm len}(P;D)]$ (selected based solely on $P$), we have $\mathds{P}[P(t) \in  \varepsilon \mathcal{G}_{\bm \phi} | P] = 0$ since $P$ is independent of ${\bm \phi}$. Therefore the Lebesgue measure of $P^{-1}(\varepsilon \mathcal{G}_{\bm \phi})$ is a.s. equal to zero.
\end{proof}

Let $\varepsilon > 0$. We define $\mathcal{S}_{\bm \phi}^\varepsilon$ as the collection of open cubes of side length $\varepsilon$ which are the connected components of the complement of the scaled grid $(\varepsilon\mathcal{G}_{\bm \phi})$. Based on Lemma \ref{divide}, if $P$ is a path as described in that lemma, then a.s. we have
\begin{equation}\label{dividepath}
    {\rm len}(P;D)=\sum_{S\in \mathcal{S}_{\bm \phi}^\varepsilon}{\rm len}(P\cap S;D).
\end{equation}

In fact, we can a.s. recover $D$ from its internal metrics on the cubes $S \in \mathcal{S}_{\bm \phi}^\varepsilon$ , as demonstrated by the following lemma.
\begin{lemma}\label{weakmeasurable}
    The metric $D$ is a.s.\ determined by $\mathcal{E},{\bm \phi}$, and the set of internal metrics $\{D(\cdot, \cdot; S) : S \in \mathcal{S}_{\bm \phi}^\varepsilon\}.$
\end{lemma}
\begin{proof}
    Conditioning on $\mathcal{E}, {\bm \phi}$ and $\{D(\cdot, \cdot; S) : S \in \mathcal{S}_{\bm \phi}^\varepsilon\}$, we consider two conditionally independent samples $D$ and $D'$ from the conditional law of $D$. This means that a.s. $D(\cdot, \cdot; S) = D'(\cdot, \cdot; S)$ for each $S \in \mathcal{S}_{\bm \phi}^\varepsilon$. To prove the lemma, it suffices to show that a.s.\ $D = D'$.

    We first observe that since $D,D'$ are both local $\beta$-LRP metrics, Proposition \ref{bilip} implies that there exists a constant $C_0$ depending on the laws of $D$ and $D'$ such that a.s. for each ${\bm x},{\bm y} \in \mathds{R}^d$,
    \begin{equation}\label{compare379lem53}
        C_0^{-1}D({\bm x},{\bm y})\leq D'({\bm x},{\bm y})\leq C_0 D({\bm x},{\bm y}).
    \end{equation}
    Next, we fix ${\bm x},{\bm y} \in \mathds{R}^d$ and let $P$ be a geodesic from ${\bm x}$ to ${\bm y}$ chosen based only on $D$.
    Recall that we defined  $V_r(A):=\{{\bm z}\in \mathds{R}^d:\ \text{dist}({\bm z},A;\|\cdot\|_\infty)\leq r/2\}$ for some set $A\subset \mathds{R}^d$ and $r>0$.
    From Lemma \ref{divide}, we get that a.s
    \begin{equation}\label{neargrid}
       \lim_{r\to 0} {\rm len}(P\cap V_r(\varepsilon(\mathcal{G}_{\bm \phi}));D)=0\quad(\varepsilon\text{ is fixed}).
    \end{equation}
    Using \eqref{compare379lem53}, we deduce that \eqref{neargrid} also holds with $D'$-length instead of $D$-length. Consequently, a.s. we have
    \begin{equation}\label{bothdivide}
        {\rm len}(P;D)=\sum_{S\in\mathcal{S}_{\bm \phi}^\varepsilon}{\rm len}(P\cap S;D)\quad \text{and}\quad {\rm len}(P;D')=\sum_{S\in\mathcal{S}_{\bm \phi}^\varepsilon}{\rm len}(P\cap S;D').
    \end{equation}

    Since the internal metrics of $D$ and $D'$ on $S$ for each $S \in \mathcal{S}_{\bm \phi}^\varepsilon$ are equivalent, we have that a.s.\ the $D'$-length of every path contained in some $S \in \mathcal{S}_{\bm \phi}^\varepsilon$ is the same as its $D$-length. Therefore, \eqref{bothdivide} implies that a.s.\ ${\rm len}(P;D) = {\rm len}(P;D')$. Since $D'$ is a length metric and by our choice of $P$, we have $D'({\bm x},{\bm y}) \leq D({\bm x},{\bm y})$.
    By symmetry, we also have a.s.\ $D({\bm x},{\bm y}) \leq D'({\bm x},{\bm y})$. Applying this for all ${\bm x},{\bm y}\in \mathds{Q}^d$, we conclude that a.s.\ $D = D'$, which completes the proof of the lemma.
\end{proof}

Using Lemma \ref{weakmeasurable} and the following lemma, we can represent $D$ as a function of a set of random variables which are conditionally independent given $(\mathcal{E}, {\bm \phi})$. This representation will enable us to apply the Efron-Stein inequality.

\begin{lemma}\label{weakindependence}
    Fix $\varepsilon > 0$. Under the conditional law given $(\mathcal{E}, {\bm \phi})$, a.s.\ the internal metrics $\{D(\cdot,\cdot; S) : S \in \mathcal{S}_{\bm \phi}^\varepsilon\}$ are conditionally independent.
\end{lemma}
\begin{proof}
    First, we condition on ${\bm \phi}$, which determines $\mathcal{S}_{\bm \phi}^\varepsilon$. Then we apply Axiom II'' (weak locality) to the collection of disjoint open sets $\{ S  : S \in \mathcal{S}_{\bm \phi}^\varepsilon\}$ and get the desired conditional independence.
\end{proof}

Now we present the

\begin{proof}[Proof of Proposition \ref{reallocal}]
    To simplify the analysis, we fix a large bounded, connected open set $V$, and define $\mathcal{S}_{\bm \phi}^\varepsilon(V)$ as the set of cubes in $\mathcal{S}_{\bm \phi}^\varepsilon$ that intersect with $V$. To be precise, let
    \begin{equation*}
        \mathcal{S}_{\bm \phi}^\varepsilon(V):=\{S\in\mathcal{S}_{\bm \phi}^\varepsilon:S\cap V\neq\emptyset\}.
    \end{equation*}
    Furthermore, we consider points ${\bm x},{\bm y} \in V$. We aim to show that the internal distance $D({\bm x},{\bm y}; V )$ is a.s.\ determined by $\mathcal{E}$. By varying ${\bm x},{\bm y}$ over $V \cap \mathds{Q}^d$ and then increasing $V$ to $\mathds{R}^d$, we will conclude the proof.

    Step 1: application of the Efron-Stein inequality. Lemma \ref{weakmeasurable} implies that $D$ can be a.s.\ represented as a measurable function of $\mathcal{E}, {\bm \phi}$, and the set of internal metrics $\{D(\cdot, \cdot; S) : S \in \mathcal{S}_{\bm \phi}^\varepsilon\}$. Furthermore, by Lemma \ref{weakindependence}, these internal metrics are conditionally independent given $(\mathcal{E}, {\bm \phi})$. Therefore, for each $S \in \mathcal{S}_{\bm \phi}^\varepsilon(V)$, we can generate a new random metric $D^S$ by re-sampling $D(\cdot, \cdot; S )$ from its conditional law given $(\mathcal{E}, {\bm \phi})$, while keeping $D(\cdot, \cdot; S')$ unchanged for each $S'\in \mathcal{S}_{\bm \phi}^\varepsilon\setminus\{ S\}$. It follows that $(\mathcal{E}, {\bm \phi}, D) \stackrel{\rm law}{=} (\mathcal{E}, {\bm \phi},D^S)$.

    Applying the Efron-Stein inequality \eqref{ESinequality} under the conditional law given $(\mathcal{E},{\bm \phi})$, we have a.s.
    \begin{equation}\label{variation}
        {\rm Var}[D({\bm x},{\bm y};V)|\mathcal{E},{\bm \phi}]\leq \frac{1}{2}\sum_{S\in\mathcal{S}_{\bm \phi}^\varepsilon(V)}\mathds{E}\left[(D^S({\bm x},{\bm y};V)-D({\bm x},{\bm y};V))^2|\mathcal{E},{\bm \phi} \right].
    \end{equation}
    Since the conditional laws of $(D,D^S)$ and $(D^S,D)$ given $(\mathcal{E}, {\bm \phi})$ are the same, the conditional law of $D({\bm x},{\bm y}; V ) - D^S({\bm x},{\bm y}; V )$ is symmetric around the origin. Thus, each term in \eqref{variation} satisfies
    \begin{equation}\label{positivepart}
        \mathds{E}\left[(D^S({\bm x},{\bm y};V)-D({\bm x},{\bm y};V))^2|\mathcal{E},{\bm \phi} \right]=2\mathds{E}\left[(D^S({\bm x},{\bm y};V)-D({\bm x},{\bm y};V))_+^2|\mathcal{E},{\bm \phi} \right]
    \end{equation}
    where $(h)_+ = h$ if $h \geq 0$ or 0 if $h < 0$. Most of the remainder of the proof is devoted to showing that the right hand side of \eqref{variation} converges to 0 a.s.\ as $\varepsilon\to 0$.

    Step 2: comparison of $D$ and $D^S$. Since $(\mathcal{E},D^S) \stackrel{\rm law}{=} (\mathcal{E},D)$, both $D$ and $D^S$ are local $\beta$-LRP metrics. According to Proposition \ref{bilip}, there exists a constant $C_1>0$ depending only on the law of $D$ such that  a.s.  for all ${\bm x}, {\bm y}\in \mathds{R}^d$ that
    \begin{equation}\label{compareDDS}
        C_1^{-1}D({\bm x},{\bm y};V)\leq D^S({\bm x},{\bm y};V)\leq C_1 D({\bm x},{\bm y};V).
    \end{equation}
    Since $D(\cdot,\cdot; V )$ is a length metric, we can choose a $D(\cdot,\cdot; V )$-geodesic $P$ from ${\bm x}$ to ${\bm y}$ in $V$. Let us fix such a path and recall \eqref{dividepath}. By the definition of $D^S$, we have
    \begin{equation}\label{DDSonSprime}
        {\rm len}(P\cap S';D)={\rm len}(P\cap S';D^S),\quad \forall S'\in\mathcal{S}_{\bm \phi}^\varepsilon (V)\setminus \{S\}.
    \end{equation}
    Using \eqref{compareDDS}, we obtain
    \begin{equation}\label{DDSonS}
        C_1^{-1}{\rm len}(P\cap S;D)\leq {\rm len}(P\cap S;D^S)\leq C_1{\rm len}(P\cap S;D).
    \end{equation}
    According to Lemma \ref{divide}, a.s.\ the contribution to the $D$-length of $P$ from the intersections of $P$ with the boundaries of the cubes in $\mathcal{S}_{\bm \phi}^\varepsilon$ is 0. Since $D$ and $D^S$ are a.s.\ bi-Lipschitz equivalent (by \eqref{compareDDS}), we can apply the same argument as in the proof of Lemma \ref{weakmeasurable} to show that the same is true with the $D^S$-length in place of the $D$-length. Combining this with \eqref{DDSonSprime} and \eqref{DDSonS}, we get that a.s.
    \begin{equation*}
        D^S({\bm x},{\bm y};V)\leq {\rm len}(P;D^S)=\sum_{S'\in \mathcal{S}_{\bm \phi}^\varepsilon(V)}{\rm len}(P\cap S';D^S) \leq D({\bm x},{\bm y};V)+C_1 {\rm len}(P\cap S;D).
    \end{equation*}
    Therefore, a.s.,
    \begin{equation}\label{upperpositive}
        (D^S({\bm x},{\bm y};V)-D({\bm x},{\bm y};V))_+\leq C_1 {\rm len}(P\cap S;D).
    \end{equation}
    By plugging \eqref{upperpositive} into \eqref{positivepart} and then into \eqref{variation} and then applying the Cauchy-Schwarz inequality, we have that
    \begin{equation}\label{boundvariation}
        \begin{aligned}
            {\rm Var}[D({\bm x},{\bm y};V)|\mathcal{E},{\bm \phi}]&\leq\mathds{E}\left[\sum_{S\in\mathcal{S}_{\bm \phi}^\varepsilon(V)}C_1^2({\rm len}(P\cap S;D))^2|\mathcal{E},{\bm \phi}\right]\\
            &\leq C_1^2\mathds{E}\left[\sum_{S\in\mathcal{S}_{\bm \phi}^\varepsilon(V)}{\rm len}(P\cap S;D)(\max_{S\in\mathcal{S}_{\bm \phi}^\varepsilon(V)}\{{\rm len}(P\cap S;D)\})|\mathcal{E},{\bm \phi}\right]\\
            &\leq C_1^2\mathds{E}[D({\bm x},{\bm y};V)^2|\mathcal{E},{\bm \phi}]^{1/2}\mathds{E}\left[\max_{S\in\mathcal{S}_{\bm \phi}^\varepsilon(V)}\{{\rm len}(P\cap S;D)\}^2|\mathcal{E},{\bm \phi}\right]^{1/2}.
        \end{aligned}
    \end{equation}

    Step 3: conclusion. 
    First, we will show that a.s.\ the first expectation in the last line of \eqref{boundvariation} is finite since $D$ satisfies Axiom V1'.

    Next, we will show that the second expectation a.s. tends to 0 as $\varepsilon\to 0$. Note that the path $P$ is a.s.\ contained in $V$, so the range of $P$ is a compact subset of $\mathds{R}^d$.

    Since $P$ is a $D$-geodesic from $\bm x$ to $\bm y$ we have that for any $S \in \mathcal{S}_{\bm \phi}^\varepsilon(V)$, the $D$-length of $P \cap S$ is at most $\sup_{{\bm u},{\bm v}\in S} D({\bm u},{\bm v};V)$ (otherwise, we could find a path from ${\bm x}$ to ${\bm y}$ of $D$-length smaller than $D({\bm x},{\bm y};V)$ by replacing the segment of $P$ between the first and last points of $S$ hit by $P$). Since $D$ can be viewed as a continuous function on $\mathds{R}^d$ and the Euclidean side length of each $S \in\mathcal{S}_{\bm \phi}^\varepsilon(V)$ is $\varepsilon$, it holds a.s.\ that
    \begin{equation}\label{Sconverge0}
        \lim_{\varepsilon\to 0}\max_{S\in\mathcal{S}_{\bm \phi}^\varepsilon(V)}\{{\rm len}(P\cap S;D)\}\leq \lim_{\varepsilon\to 0}\max_{S\in\mathcal{S}_{\bm \phi}^\varepsilon(V),P\cap V\neq\emptyset}\{\sup_{{\bm u},{\bm v}\in S}D({\bm u},{\bm v};V)\}=0\,.
    \end{equation}
    Each of the random variables ${\rm len}(P \cap S;D)$ is bounded above by ${\rm len}(P;D)=D(\bm x,\bm y;V)$.  By \eqref{Sconverge0} and the dominated convergence theorem, the second expectation in the last line of \eqref{boundvariation} a.s.\ converges to 0 as $\varepsilon \to 0$.

    Consequently, a.s.\ ${\rm Var}[D({\bm x},{\bm y}; V ) | \mathcal{E}, {\bm \phi}] \to 0 $ as $\varepsilon \to 0$, which implies that a.s.\ $D({\bm x},{\bm y}; V)$ is determined by $(\mathcal{E}, {\bm \phi})$. Since $D(\cdot, \cdot; V )$ can be viewed as a continuous function on $V\times V$ and this holds for any fixed choice of ${\bm x},{\bm y} \in V$, we conclude that a.s.\ $D(\cdot, \cdot; V)$ is determined by $(\mathcal{E},{\bm \phi})$.

    Furthermore, since $(\mathcal{E},D)$ is independent from ${\bm \phi}$, we further obtain that a.s.\ $D(\cdot,\cdot; V)$ is determined by $\mathcal{E}$. Since $D$ is a length metric, we can let $V$ increase to all of $\mathds{R}^d$ to show that a.s.\ $D$ is determined by $\mathcal{E}$.

    Additionally, for any open set $U\subset \mathds{R}^d$, since $D(\cdot,\cdot;U)$ is independent of $\mathcal{E}|_{(U\times U)^c}$ and determined by $D$ (and thus by $\mathcal{E}$ a.s.), we get that $D(\cdot,\cdot;U)$ is a.s.\ determined by $\mathcal{E}|_{U\times U}$. Thus $D$ satisfies Axiom II (locality). That is, $D$ is a weak $\beta$-LRP metric.
\end{proof}

\subsection{Proof of Proposition \ref{dis-prop-subsequence}}
Before we show the proof of Proposition \ref{dis-prop-subsequence}, we will first present the following elementary probabilistic lemma, which is the reason why we have convergence in probability, instead of convergence in law, in Proposition \ref{dis-prop-subsequence}  (see e.g. \cite[Lemma 4.5]{SS13}).
\begin{lemma}\label{convergePLemma}
    Let $(\Omega_1, d_1)$ and $(\Omega_2, d_2)$ be complete separable metric spaces. Let $X$ be a random variable taking values in $\Omega_1$ and let $\{Y^n\}$ and $Y$ be random variables taking values in $\Omega_2$, all defined on the same probability space, such that $(X, Y^n) \to (X, Y )$ in law. If $Y$ is a.s.\ determined by $X$, then $Y_n \to Y$ in probability.
\end{lemma}

Now we turn to finish
\begin{proof}[Proof of Proposition \ref{dis-prop-subsequence}]
    From Proposition \ref{dis-tightness}, we know that $\{\widehat{a}_n^{-1}\widetilde{d}_n\}$ is tight with respect to the topology of local uniform convergence on $\mathds{R}^{2d}$. Furthermore, any subsequential limiting metric $D$ satisfies the following axioms:
    \begin{itemize}
        \item Axiom I: It is implied by Proposition \ref{dis-lengthspace}.
        \item Axiom II'': It is implied by Lemma \ref{dis-weaklocality2}.
        \item Axiom III: It is immediately implied by local uniform convergence.
        \item Axiom IV': It is implied by Proposition \ref{dis-Axiom4}.
        \item Axioms V1' and V2': It is implied by Proposition \ref{dis-Axiom5}.
    \end{itemize}
    Thus $D$ is a local $\beta$-LRP metric. From Proposition \ref{reallocal}, we get that $D$ is a weak $\beta$-LRP metric.

    Furthermore, convergence in probability comes from Lemma \ref{convergePLemma} with $X=\mathcal{E}$, $Y=D$ and $\{Y_n\}_{n\geq 1}=\{\widehat{a}^{-1}_{n_k} \widehat{d}_{n_k}\}_{k\geq 1}$. Hence, the proof is complete.
\end{proof}

\section{Proof of Theorem \ref{subsequence-exist}}\label{existence}
In this section, we aim to prove Theorem \ref{subsequence-exist}. Specifically, we will prove the tightness of $\{D_n\}_{n\in\mathds{N}}$ in Section \ref{tightness}. Then in Section \ref{lengthspace}, we will provide the proof that the subsequential limiting metric is a length space in an appropriate sense. Section \ref{conv-internal} is devoted to the convergence of the internal metrics of $D_n=a_n^{-1}d_{(1/n,\infty)}$, and Section \ref{weak loc} is devoted to proving Axiom II'' (weak locality) for any subsequential limit of $D_n$. Finally, we will complete the proof of Theorem \ref{subsequence-exist} in Section \ref{proofthm1.9}.

\subsection{Tightness}\label{tightness}
Recall that we defined the distance $d_{(1/n,\infty)}$ in \eqref{dist_delta} and  $D_n=a_n^{-1}d_{(1/n,\infty)}$, where $a_n$ is the median of $d_{(1/n,\infty)}(\bm 0,\bm 1)$. By Lemma \ref{an-bounded} (note that this part of Lemma \ref{an-bounded} follows from \cite[Theorem 1.1]{Baumler22}), $\left\{n^{1-\theta}a_n \right\}$ and $\left\{n^{\theta-1}a_n^{-1}\right\}$ are both uniformly bounded over all $n\in\mathds{N}$. Thus, we only need to prove the tightness of $\{n^{1-\theta}d_{(1/n,\infty)}\}$, which then implies the tightness of $\{D_n\}$.
For convenience of notation, let $\widetilde{D}_n=n^{1-\theta}d_{(1/n,\infty)}$ in the rest of this subsection.

Our primary tool in this proof is the following uniform tail bound on the diameter under $\widetilde{D}_n$, which is an analog of \cite[Theorem 6.1]{Baumler22} for the continuous model.
\begin{proposition}\label{MGFUpperBound}
    For any $\eta\in(0,1/(1-\theta))$, we have the following uniform upper bound on the moment generating function :
    \begin{equation*}
        \sup_{n\in\mathds{N},r>0} \mathds{E}\left[\exp\left\{\left(\frac{{\rm diam}([0,r]^d;\widetilde{D}_n) }{r^\theta}\right)^\eta\right\}\right]<\infty.
    \end{equation*}
\end{proposition}

\begin{proof}
The main idea of the proof is to couple and compare the distances in continuous and discrete models.
    Due to the scaling invariance of the Poisson point process for edges, $r^{-\theta}\widetilde{D}_n(r\cdot,r\cdot)$ has the same law as $(nr)^{-\theta}d_{(1,\infty)}(nr\cdot,nr\cdot)$.
    For $R\in (0,1]$, we have $R^{-\theta}{\rm diam}([0,R]^d;d_{(1,\infty)})\le dR^{1-\theta}\le d$.
    Therefore, it suffices to show that for any $\eta\in(0,1/(1-\theta))$,
    \begin{equation*}
        \sup_{R>1}\mathds{E}\left[\exp\left\{\left(\frac{{\rm diam}([0,R]^d;d_{(1,\infty)})}{R^\theta}\right)^\eta\right\}\right]<\infty.
    \end{equation*}
    This is immediately implied by Proposition \ref{apriori-cont-1}.
\end{proof}

We now move to the proof of the tightness for $\{\widetilde{D}_n\}$, as incorporated in the next proposition.
\begin{theorem}\label{tightmetric}
   The family of random metrics $\{\widetilde{D}_n\}_{n\geq 1}$ {\rm(}viewed as a random continuous function on $C(\mathds{R}^{2d})${\rm)} is tight with respect to the local uniform topology of $C(\mathds{R}^{2d})$. Moreover,
    for any $R>0$, the family of random  metrics $\{\widetilde{D}_n(\cdot,\cdot;[-R,R]^d)\}_{n\geq 1}$ {\rm(}viewed as a random continuous function on $C([-R,R]^{2d})${\rm)} is tight with respect to the uniform topology of $C([-R,R]^{2d})$. 
\end{theorem}

\begin{proof} 
    Note that by Proposition \ref{MGFUpperBound}  with $\eta=1$,
    \begin{equation*}
        C_\beta:=\sup_{n\in\mathds{N},r>0}\mathds{E}\left[\exp\left\{\frac{{\rm diam}([0,r]^d;\widetilde{D}_n)}{r^\theta}\right\}\right]<\infty
    \end{equation*}
    is a constant depending only on $\beta$ and $d$. Therefore, by applying Lemma \ref{HolderLemma} to every $\widetilde{D}_n$, we obtain the tightness of \begin{equation*}
        \left\{ \sup_{{\bm x},{\bm x}'\in[-R,R]^d}\frac{\widetilde{D}_n({\bm x},{\bm x}')}{\|{\bm x}-{\bm x}'\|_\infty^\theta\log\frac{4R}{\|{\bm x}-{\bm x}'\|_\infty}} \right\}_{n\in\mathds{N},R>0}.
    \end{equation*}
    Combined with  Arzela-Ascoli Theorem, this completes the proof of theorem. \end{proof}

Combining Theorem \ref{tightmetric} and Lemma \ref{an-bounded}, we obtain the tightness of $\{D_n\}$. Moreover, combining Proposition \ref{MGFUpperBound} and Lemma \ref{an-bounded} and using a similar argument in the proof of Lemma \ref{dis-Axiom5}, we can show that for any $\eta\in(0,1/(1-\theta))$,
$$
\sup_{n\in\mathds{N},r>0}\mathds{E}\left[\exp\left\{\left(\frac{{\rm diam}([0,r]^d;D_n)}{r^\theta}\right)^\eta\right\}\right]<\infty.
$$ Applying Fatou's Lemma to this implies that any subsequential limiting metric of $\{D_n\}$ satisfies Axiom V2' (tightness across scales for upper bound). To be precise, we summarize these consequences in the following corollary.

\begin{corollary}\label{TightAcrossScale}
 The family of random metrics $\{D_n\}_{n\geq 1}$ {\rm(}viewed as a random continuous function on $C(\mathds{R}^{2d})${\rm)} is tight with respect to the local uniform topology of $C(\mathds{R}^{2d})$.
  Moreover, for any $R>0$, the family of random metric $\{D_n(\cdot,\cdot;[-R,R]^d)\}_{n\geq 1}$ {\rm(}viewed as a random continuous function on $C([-R,R]^{2d})${\rm)} is tight with respect to the uniform topology of $C([-R,R]^{2d})$.
 Thus $\{D_n\}_{n\geq 1}$ has subsequential limiting metrics. Furthermore, for any subsequential limiting metric  $D$, it satisfies Axiom V2', i.e. for any $\eta\in(0,1/(1-\theta))$, 
    \begin{equation*}
        \sup_{r>0}\mathds{E}\left[\exp\left\{\left(\frac{{\rm diam}([0,r]^d;D)}{r^\theta}\right)^\eta\right\}\right]<\infty.
    \end{equation*}
\end{corollary}

In addition, according to \cite[Theorem 1.1]{Baumler22}, we can also show that
\begin{lemma}\label{tightlowerbound}
    For any subsequential limiting metric  $D$, it satisfies Axiom V1', i.e. $\left\{\frac{r^\theta}{D(\bm 0,([-r,r]^d)^c)}\right\}_{r>0}$ is tight.
\end{lemma}
\begin{proof}
    It is immediately implied by the tightness of
    $$\left\{\frac{r^\theta}{n^{1-\theta}d_{(1/n,+\infty)}(\bm 0,([-r,r]^d)^c)}\right\}_{nr>1}$$
    from Proposition \ref{apriori-cont-2} since $\frac{r^{\theta}}{n^{1-\theta}d_{(1/n,+\infty)}(\bm 0,([-r,r]^d)^c)}$ has the same law as $\frac{1}{(nr)^{-\theta}d_{(1,+\infty)}(\bm 0,([-nr,nr]^d)^c)}$ by the scaling invariance of the Poisson point process for edges.
\end{proof}

\subsection{Length space}\label{lengthspace}
In this subsection, we will present a proposition that implies any subsequential limiting metric is, in an appropriate sense, a length space.
\begin{proposition}\label{Prop-lengthspace}
    Assume that $D_{n_k}$ converges to $D$ as $k\to\infty$ with respect to the local uniform topology of $C(\mathds{R}^{2d})$. Then $D$ is a length metric {\rm(}viewed as a metric on the quotient metric $\mathds{R}^d/\sim${\rm)}. Here $\sim$ is the equivalence relation that ${\bm x}\sim {\bm y}$ if and only if $D({\bm x},{\bm y})=0$.
\end{proposition}

Before we start our proof, we will first record the fact that any $D$-bounded set is also Euclidean bounded.
\begin{lemma}\label{equal-bounded}
    Let $D$ be the limit in Proposition \ref{Prop-lengthspace}. Then a.s.,  for every compact set $K\subset\mathds{R}^d$, we have $$\lim_{R\to\infty} D(K, ([-R,R]^d)^c) = \infty.$$ In particular, every $D$-bounded subset of $\mathds{R}^d$ is also Euclidean bounded.
\end{lemma}

The proof of Lemma \ref{equal-bounded} is essentially identical to that of Lemma \ref{dis-bounded} and thus we omit it here.

Next we will show that for every $n$, $D_n$ is a geodesic metric in an appropriate sense. Here a geodesic metric is defined as a metric under which there exists a geodesic between any two points.
\begin{lemma}\label{DnGeodesic}
    For fixed $n\in\mathds{N}$, let $\sim_n$ be the equivalence relationship that ${\bm x}\sim_n {\bm y}$ if and only if $D_n({\bm x},{\bm y})=0$. Then $(\mathds{R}^d/\sim_n, D_n)$ is a geodesic metric.
\end{lemma}
\begin{proof}
    Since $D_n=a_n^{-1}d_{(1/n,+\infty)}$, it suffices to show the lemma for $d_{(1/n,+\infty)}$.
    First, it is clear that ${\rm diam}([-r,r]^d;d_{(1/n,+\infty)})\le 2dr$ and $\lim_{R\to\infty}d_{(1/n,+\infty)}(\bm 0,([-R,R]^d)^c)=\infty.$
    As a result, for any compact $K\subset\mathds{R}^d$,
    $$
    \lim_{R\to\infty}d_{(1/n,+\infty)}(K,([-R,R]^d)^c)=\infty.
    $$
Therefore, for any fixed ${\bm x},{\bm y}\in\mathds{R}^d$, there exists a random $M>0$ such that for any continuous path $P$ from ${\bm x}$ to ${\bm y}$, if ${\rm len}(P;d_{(1/n,+\infty)})\le |\bm x-\bm y|$, then $P\subset [-M,M]^d$. This implies
\begin{equation}\label{d(x,y)}
d_{(1/n,\infty)}({\bm x},{\bm y})=\inf\left\{{\rm len}(P, d_{(1/n,+\infty)}):\ P \text{ is a path from ${\bm x}$ to ${\bm y}$ and } P\subset [-M,M]^d\right\}.
\end{equation}
Additionally, from the property of the Poisson point process for edges, a.s.\ there are only finite edges with both end points in $[-R,R]^d$ and scopes larger than $1/n$ for any $R>0$.
Hence, we observe that the number of paths $P$ in \eqref{d(x,y)} is a.s.\ finite. This implies that there exists a path $P$ from ${\bm x}$ to ${\bm y}$ that achieves the infimum in \eqref{d(x,y)}, that is, there exists a $d_{(1/n,\infty)}$-geodesic from ${\bm x}$ to ${\bm y}$.
\end{proof}


\begin{proof}[Proof of Proposition \ref{Prop-lengthspace}]
Similar with the proof of Proposition \ref{dis-lengthspace}, according to Skorohod representation theorem, we see that there exist a probability space and random variables $D_n',D'$ (equal in distribution to $D_n$ and $D$, respectively) on it such that $D_{n_k}'$ converges a.s.\ to $D'$.
Since $D'$ can be viewed as a H\"older continuous function on $\mathds{R}^{2d}$, it is a complete metric due to the completeness of $\mathds{R}^d$.
Therefore, using the arguments in the proof of Proposition \ref{dis-lengthspace} with replacing $\{{\bm x}_{n_k}\}$ with the midpoints between ${\bm z}$ and ${\bm w}$ under metrics $D_{n_k}$ (note that the existence of midpoints is ensured by Lemma \ref{DnGeodesic}),
and replacing  Lemma \ref{dis-bounded} with Lemma \ref{equal-bounded}, we can obtain the desired result.
\end{proof}

\subsection{Subsequential limits}\label{conv-internal}
In this subsection, we consider the convergence of internal metrics of $D_n=a_n^{-1}d_{(1/n,\infty)}$ on a certain class of sets on $\mathds{R}^d$.

\begin{lemma}\label{limitonset}
Let $\mathcal{W}$ be the set of all dyadic sets {\rm(}recalling Definition \ref{dyadic}{\rm)}. For any sequence $\{n_k\}_{k\geq 1}\subset \mathds{N}$ tending to infinity, there is
a subsequence $\{n'_k\}_{k\geq 1}$ and a 
random length metric $D$ such that the following is true. We have the convergence of joint laws
\begin{equation}\label{jointlaw}
\left(D_{n'_k},\ \{D_{n'_k}(\cdot,\cdot;\overline{W})\}_{W\in\mathcal{W}}\right)\to \left(D,\{D_W\}_{W\in\mathcal{W}}\right),
\end{equation}
where the first coordinate is endowed with the local uniform topology on $\mathds{R}^d\times \mathds{R}^d$ and each element of the collection in the second coordinate is endowed with the uniform topology on $\overline{W}\times \overline{W}$. Furthermore, for each $W\in\mathcal{W}$ we have a.s.\ $D_W(\cdot,\cdot;W)=D(\cdot,\cdot;W)$.
\end{lemma}



We now proceed with the proof of Lemma \ref{limitonset}.  
The following lemma follows from Corollary \ref{TightAcrossScale} and Proposition \ref{Prop-lengthspace} immediately, which gives the tightness of the internal metrics associated to $D_n$. It is also an analog to Lemma \ref{dis-tightDW}.

\begin{lemma}\label{internaltight}
Let $R>0$. The laws of the internal metrics $\{D_n(\cdot,\cdot;[-R,R]^d)\}_{n\geq 1}$ are tight with respect to  the local uniform topology of $C(\mathds{R}^{2d})$ and any subsequential limit of these laws is supported on length metrics which can also be viewed as a continuous function.
\end{lemma}

Recall that in this paper, we use the term ``length metric'' to refer to a metric $D$ that becomes a length metric when considered on the quotient space $\mathds{R}^d/\sim$. Here, the equivalence relation that $\sim $ is defined such that ${\bm x}\sim {\bm y}$ if and only if $D({\bm x},{\bm y})=0$.

We now enhance from internal metrics on cubes to internal metrics on closures of dyadic sets.

\begin{lemma}\label{tightDW}
Let $W\subset \mathds{R}^d$ be a dyadic set. The laws of the internal metrics $\{D_n(\cdot,\cdot;\overline{W})\}_{n\geq 1}$ are tight with respect to the uniform topology of $C(\overline{W}\times \overline{W})$ and any subsequential limit of these laws is supported on length metrics.
\end{lemma}

The proof of the lemma above is similar to that of Corollary \ref{TightAcrossScale} and Proposition \ref{Prop-lengthspace}. Thus we omit the proof here. We also need the strong regularity for the subsequential limits of $\{D_n\}$, which is an  analog to Proposition \ref{SRSL}.

\begin{lemma}\label{SRSL-c}
Any subsequential limit $D$ of $\{D_n\}$ satisfies the strong regularity, i.e., the following holds almost surely. For any $\bm x,\bm y\in \mathds{R}^d$, $D(\bm x,\bm y)=0$ if and only if $\langle \bm x, \bm y\rangle\in \mathcal{E}$.
\end{lemma}

We now present the

\begin{proof}[Proof of Lemma \ref{limitonset}]
From Corollary \ref{TightAcrossScale} and Proposition \ref{Prop-lengthspace} we first see that metrics $D_n$ are tight with respect to the local uniform topology on $\mathds{R}^{2d}$ and any subsequential limit in law is a.s.\ 
a length metric on $\mathds{R}^d$. Now applying Lemma \ref{tightDW} and the Prokhorov theorem, we get that the joint law of the metrics on the left hand side of \eqref{jointlaw} is tight. Moreover, any subsequential limit of these joint laws is a coupling of a 
length metric $D$ on $\mathds{R}^d$ and a length metric $D_W$ on $\overline{W}$ for each $W\in\mathcal{W}$. 
We then apply Lemmas \ref{SRSL-c} and \ref{internalmetric}  and get that $D_W(\cdot,\cdot;W)=D(\cdot,\cdot;W)$ for each $W\in\mathcal{W}$.
\end{proof}

\subsection{Weak locality}\label{weak loc}
In this subsection, we will prove Axiom II'' (weak locality) for any subsequential limit of $D_n$ as follows.

\begin{lemma}\label{weaklocality2}
Let $(\mathcal{E},D)$ be any subsequential limit of the laws of $(\mathcal{E},D_n)$. Then $D$ satisfies Axiom II'' {\rm(}weak locality{\rm)}.
\end{lemma}

\begin{proof}
Using the similar arguments in the proof of Lemma \ref{dis-weaklocality2} with replacing $\widehat{a}_n^{-1}\widetilde{d}_n$ by $D_n$, we finish the proof.
\end{proof}

\subsection{Proof of Theorem \ref{subsequence-exist}}\label{proofthm1.9}

    From Corollary \ref{TightAcrossScale}, we know that $\{D_n\}_{n\geq 1}$ (viewed as a random continuous function on $C(\mathds{R}^{2d})$) is tight with respect to the local uniform topology of $C(\mathds{R}^{2d})$. Furthermore, any subsequential limiting metric $D$ satisfies the following axioms:
    \begin{itemize}
        \item Axiom I: It is implied by Proposition \ref{Prop-lengthspace}.
        \item Axiom II'': It is implied by Lemma \ref{weaklocality2}.
        \item Axiom III: It is immediately implied by local uniform convergence.
        \item Axiom IV': It is immediately implied by the translation invariance of $D_n$.
        \item Axiom V1': It is implied by Lemma \ref{tightlowerbound}.
        \item Axiom V2': It is implied by Corollary \ref{TightAcrossScale}.
    \end{itemize}
    Thus $D$ is a local $\beta$-LRP metric. From Proposition \ref{reallocal}, we get that $D$ is a weak $\beta$-LRP metric.

    Furthermore, convergence in probability comes from Lemma \ref{convergePLemma} with $X=\mathcal{E}$, $Y=D$ and $\{Y_n\}_{n\geq 1}=\{a^{-1}_{n_k} d_{n_k}\}_{k\geq 1}$. Hence, the proof is complete. \hfill $\square$

\section{Quantifying the optimality of the optimal bi-Lipschitz constants}\label{section3}

Let $\beta>0$ be fixed. Consider two weak $\beta$-LRP metrics $D$ and $\widetilde{D}$. 
As we have already observed, there exist deterministic optimal upper and lower bi-Lipschitz constants $c_*$ and $C_*$ such that a.s.\  \eqref{DDtilde} holds.
Recall from Section \ref{section1} that we aim to prove by contradiction that $c_*=C_*$.
With this aim, we assume that $c_*<C_*$ in the following two sections.

From the optimality of $c_*$ and $C_*$, we see that for every $C'<C_*$,
\begin{equation}\label{bad-event}
\mathds{P}[\exists {\bm x},{\bm y}\in \mathds{R}^d\ \text{such that }\widetilde{D}({\bm x},{\bm y})\geq C'D({\bm x},{\bm y})]>0.
\end{equation}
A similar statement holds for every $c'>c_*$. The goal of this section is to prove various quantitative versions of \eqref{bad-event}, which are required to hold uniformly over different Euclidean scales. We remark that the content in this section is in parallel to \cite[Section 3]{GM21} and \cite[Section 3]{DG23}, and the proof also draws inspiration from them.

\subsection{Events for the optimal bi-Lipschitz constants}
In this subsection, we will define some events that are stronger and more complex versions of the event in \eqref{bad-event} (see Definitions \ref{def-G} and \ref{def-H} below). We will then prove some basic facts about these events and state the main estimates we need for them (Propositions \ref{mrinS3} and \ref{mrinS3-tilde}).

The first event is a slightly stronger version of the event in \eqref{bad-event}.

\begin{definition}\label{def-G}
For $r>0$, $\gamma>0$, $q\geq 0$ and $C'>0$, denote by $G_r(\gamma,q,C')$ the event that there exist ${\bm x},{\bm y}\in V_r(\bm 0)$ such that
$$
\widetilde{D}(V_{\gamma r}({\bm x}),V_{\gamma r}(\bm y))>C' D({\bm x},{\bm y})\geq C'q r^\theta.
$$
\end{definition}

Our other event has a more complicated definition, and  includes some regularity conditions on the metric $D$.

\begin{definition}\label{def-H}
For $r>0$, $\alpha\in (0,1)$ and $C'>0$, we let $H_r(\alpha,C')$ be the event that there exist ${\bm x},{\bm y}\in  V_r(\bm 0)$ with $|\bm x-\bm y|\geq \alpha r/3$, such that
\begin{equation}\label{H-ineq}\widetilde{D}({\bm x},{\bm y})>C'D({\bm x},{\bm y}),\end{equation}
and there exists a $D$-geodesic $P$ from $\bm x$ to $\bm y$ satisfying
\begin{itemize}
\item[(1)] $P\subset  V_r(\bm 0)$;

\item[(2)] $ D({\bm x},{\bm y})\geq (b\alpha r)^\theta$ for some $b=b(\alpha)>0$ (here the constant $b$ depends only on $\beta,d,\alpha$ and the law of $D$, and will be chosen finally in Lemma \ref{good-prob});

 \item[(3)] there exists $\gamma_0=\gamma_0(\alpha)>0$ (depending only on $\beta,d,\alpha$ and the law of $D$ and will be chosen in Lemma \ref{good-prob}) such that for each $\gamma\in (0,\gamma_0]$,
$$
\max\left\{{\rm diam}(V_{\gamma r}(\bm x);D),{\rm diam}(V_{\gamma r}(\bm y);D)\right\}\leq \gamma^{\theta/2} D({\bm x},{\bm y}).
$$
\end{itemize}
\end{definition}

In what follows we will denote by Definition $\cdot$ ($\cdot$) the condition ($\cdot$) in Definition $\cdot$. For example, Definition \ref{def-H} (2) means the condition (2) in Definition \ref{def-H}.

The main result of this section, which will be proven in Section \ref{prove-mr}, tells us that if
$$\mathds{P}[G_r(\gamma,q, C'')]\geq \gamma,$$
then there are lots of ``scales'' $r'<r$ for which $\mathds{P}[H_{r'}(\alpha, C')]$ is bounded from below by a constant which does not depend on $r$ or $C'$.

\begin{proposition}\label{mrinS3}
For each small enough $\alpha\in(0,1)$ and $p\in (0,1)$ {\rm(}depending only on $\beta,d$ and the laws of $D$ and $\widetilde{D}${\rm)},
for $b=b(\alpha)>0$ and $\gamma_0=\gamma_0(\alpha)>0$ {\rm(}depending only on $\beta,d,\alpha$ and the law of $D${\rm)},  and for each $C'\in(0,C_*)$, there exists $C''=C''(C' ,\alpha)\in (C',C_*)$ such that for each $\gamma\in (0,1)$ and $q>0$, there exists   $\varepsilon_0=\varepsilon_0(\gamma,q, C')>0$ with the following property. If $r>0$ and $\mathds{P}[G_r(\gamma,q,C'')]\geq \gamma$, then $\mathds{P}[H_{\varepsilon r}(\alpha, C')]\geq p$ for each $\varepsilon\in (0,\varepsilon_0]$.

\end{proposition}

\begin{remark}
We would like to emphasize that Proposition \ref{mrinS3} draws strong inspiration  from  \cite[Proposition 3.3]{DG23} for the LQG metric. The difference lies in the fact that \cite[Proposition 3.3]{DG23} can only control the probability of the $H$-event for a fraction of scales since it relies on the near-independence of GFF at multiple scales to complete the proof. The reason why we obtain a stronger conclusion here is that the independence in the LRP model is available in a straightforward manner and as a result we only need to focus on one scale.
\end{remark}

We emphasize that in Proposition \ref{mrinS3}, the parameters $\alpha$ and $p$ do not depend on $r$ or $C'$. This will be crucial for our arguments in Section \ref{section4}.
In order for Proposition \ref{mrinS3} to have non-trivial content, one needs a lower bound for $\mathds{P}[G_r(\gamma,q,C')]$. It is straightforward to check that we have such a lower bound if $r=1$ and $\gamma, q$ are small.

\begin{lemma}\label{P[G]}
For each $C'<C_*$, there exist $\gamma,q>0$, depending only on $C'$ and the laws of $D$ and $\widetilde{D}$, such that $\mathds{P}[G_1(\gamma,q,C')]>0$.
\end{lemma}

\begin{proof}
We will first show that for each $C'<C_*$, there exists $\gamma>0$ such that
\begin{equation}\label{q=0}
    \mathds{P}[G_1(\gamma,0,C')]>0.
\end{equation}
We will prove its contrapositive by contradiction. To this end, let $C'>0$ and assume that
\begin{equation}\label{G_1}
\mathds{P}[G_1(\gamma,0,C')]=0,\quad \forall \gamma>0.
\end{equation}
We will show that $C'\ge C_*$ (which then arrives at a contradiction on the optimality of $C_*$). The assumption \eqref{G_1} implies that a.s.
\begin{equation}\label{G_1inverse}
\widetilde{D}(V_\gamma(\bm x),V_{\gamma}(\bm y))\leq C' D({\bm x},{\bm y}),\quad \forall {\bm x},{\bm y}\in [0,1)^d,\ \forall \gamma>0.
\end{equation}
By the continuity from Proposition \ref{LimitContinuousTail}, we get that for any ${\bm x},{\bm y}\in [0,1)^d$,
$$
\widetilde{D}({\bm x},{\bm y})=\lim_{\gamma\rightarrow 0^+}\widetilde{D}(V_\gamma(\bm x),V_\gamma(\bm y)).
$$
Thus, combining this with \eqref{G_1inverse} yields that a.s.
\begin{equation}\label{G_1inverse2}
\widetilde{D}({\bm x},{\bm y})\leq C' D({\bm x},{\bm y}),\quad \forall {\bm x},{\bm y}\in [0,1)^d.
\end{equation}
By Axiom IV' (translation invariance) of $D$ and $\widetilde{D}$, \eqref{G_1inverse2} implies that
\begin{equation}\label{G_1inverse3}
\widetilde{D}({\bm x},{\bm y})\leq C' D({\bm x},{\bm y}),\quad \forall {\bm x},{\bm y}\in\mathds{R}^d\ \text{such that }\|\bm x-\bm y\|_\infty< 1.
\end{equation}

For a general pair of points ${\bm x},{\bm y}\in \mathds{R}^d$, denote by $P:[0,D({\bm x},{\bm y})]\rightarrow \mathds{R}^d$ a $D$-geodesic from $\bm x$ to $\bm y$. Here we view $P$ as a path parameterized by time $t\in [0, D(\bm x,\bm y)]$ with $D(0, P(t)) = t$.
Define a sequence of times
$$
0=t_0<t_1<\cdots<t_N<t_{N+1}=D({\bm x},{\bm y})
$$ as follows.
To start with, the property of the Poisson point process ensure that there are only finitely many long edges in $P$ with scopes at least $1$.
Let $\langle {\bm x}_i,{\bm y}_i\rangle$ for $i\in [1,m]_\mathds{Z}$ be such long edges sorted in the order they are traversed by $P$.
We can then partition each portion of the path $P$ from ${\bm y}_i$ to ${\bm x}_{i+1},\ i\in [0,N]_\mathds{Z}$ (with ${\bm y}_0={\bm x}$ and ${\bm x}_{N+1}={\bm y}$) into segments  such that the $\ell^\infty$ distance between the two end points of every segment is less than 1. Sorting all of these segments in the ascending order, we get a sequence of times $\{t_j\}$ such that either $\|P(t_j)-P(t_{j+1})\|_\infty<1$ or $\langle P(t_j),P(t_{j+1})\rangle\in\mathcal{E}$. Since $D({\bm y}_i,{\bm x}_i)=0$ by Axiom III (regularity), we obtain
$$
D({\bm x},{\bm y})=\sum_{j: \|P(t_j)-P(t_{j+1})\|_\infty< 1} D( P(t_j),P(t_{j+1})).
$$
Applying  \eqref{G_1inverse3} to the pairs of points in the above summation, we get that a.s.
\begin{align*}
\widetilde{D}({\bm x},{\bm y})&\leq \sum_{j: \|P(t_j)-P(t_{j+1})\|_\infty< 1} \widetilde{D}(P(t_j),P(t_{j+1}))\\
&\leq C'\sum_{j: \|P(t_j)-P(t_{j+1})\|_\infty< 1} D( P(t_j),P(t_{j+1}))=C'D({\bm x},{\bm y}).
\end{align*}
By the definition of $C_*$, we obtain that $C'\geq C_*$. This implies \eqref{q=0}.

Additionally, for fixed $C'<C_*$ and $\gamma>0$, it is clear that $G_r(\gamma,q,C')$ is decreasing in $q$ and thus we get
$$
G_1(\gamma,0,C')\supseteq \bigcup_{q>0}G_1(\gamma,q,C').
$$
Moreover, we claim that the opposite inclusion relation  also holds. Indeed,
assume that $G_1(\gamma,0,C')$ occurs. Then for the points ${\bm x},{\bm y}\in[0,1)^d$ chosen in Definition \ref{def-G}, it is clear that $\widetilde{D}({\bm x},{\bm y})\geq \widetilde{D}(V_{\gamma r}(\bm x),V_{\gamma r}(\bm y))>0$. Combining this with Proposition \ref{Dxy=0} we can see that $\langle {\bm x},{\bm y}\rangle\notin\mathcal{E}$. Therefore, we get $D({\bm x},{\bm y})>0$ by using Proposition \ref{Dxy=0} again.
This implies that one can find a sufficiently small $q>0$ such that $D({\bm x},{\bm y})\geq q$. Hence,
$$
G_1(\gamma,0,C')\subseteq \bigcup_{q>0}G_1(\gamma,q,C').
$$
As a result, $G_1(\gamma,0,C')=\bigcup_{q>0}G_1(\gamma,q,C').$ 
Combining this with \eqref{q=0} yields the lemma.
\end{proof}

By combining Proposition \ref{mrinS3} and Lemma \ref{P[G]}, we get the following.

\begin{proposition}\label{Hwithr=1}
For each small enough $\alpha\in(0,1)$, $p\in(0,1)$ {\rm(}depending only on $\beta,d$ and the laws of $D$ and $\widetilde{D}${\rm)}, for $b=b(\alpha)>0$ and $\gamma_0=\gamma_0(\alpha)>0$ {\rm(}depending only on $\beta,d,\alpha$ and the law of $D${\rm)},  and for each $C'\in(0,C_*)$, there exists $C''=C''(C' ,\alpha)\in (C',C_*)$ such that the following is true.  There exists $\varepsilon_0=\varepsilon_0(C')>0$ {\rm(}depending on $C'$ and the laws of $D$ and $\widetilde{D}${\rm)} such that $\mathds{P}[H_{\varepsilon }(\alpha, C')]\geq p$ for each $\varepsilon\in (0,\varepsilon_0]$.
\end{proposition}

We will also need an analog of Proposition \ref{Hwithr=1} with the events $G_r(\gamma,q,C')$ in place of the events $H_r(\alpha,C')$.

\begin{proposition}\label{strongP[G]}
For each $C'\in(0,C_*)$, there exist $\gamma,q>0$, depending on $C'$ and the laws of $D$ and $\widetilde{D}$, such that for each small enough $\varepsilon>0$ {\rm(}depending only on $C'$ and the laws of $D$ and $\widetilde{D}${\rm)}, we have $\mathds{P}[G_\varepsilon (\gamma,q,C')]\geq \gamma$.
\end{proposition}

We will show Proposition \ref{strongP[G]} from Proposition \ref{Hwithr=1} and the following elementary relations between the events $H_r(\cdot,\cdot)$ and $G_r(\cdot,\cdot,\cdot)$.

\begin{lemma}\label{HandG}
If $\alpha\in(0,1)$ and $\zeta\in(0,1)$, there exist $\gamma,q>0$, depending only on $\beta,d,\alpha,\zeta$ and the laws of $D$ and $\widetilde{D}$, such that the following is true. For each $r>0$ and each $C'>\zeta$, if $H_r(\alpha,C')$ occurs, then $G_r(\gamma, q,C'-\zeta)$ occurs.
\end{lemma}
\begin{proof}
Assume that $H_r(\alpha,C')$ occurs and let $\bm x$ and $\bm y$ be as in Definition \ref{def-H} of $H_r(\alpha,C')$. By Definition \ref{def-G} of $G_r(\gamma,q,C'-\zeta)$, it suffices to find $\gamma,q>0$ as in the lemma statement such that
\begin{equation}\label{G(C'-zeta)}
\widetilde{D}(V_{\gamma r}(\bm x),V_{\gamma r}(\bm y))>(C'-\zeta)D({\bm x},{\bm y})\geq(C'-\zeta) q r^\theta.
\end{equation}
To this end, we first take $q=(b\alpha)^\theta$. Then from Definition \ref{def-H} (2) of $H_r(\alpha,C')$, it is clear that the second inequality in \eqref{G(C'-zeta)} holds.
Next, let $\rho<\gamma_0$, where $\gamma_0$ is the constant in Definition \ref{def-H} (3). Then by the triangle inequality and Definition \ref{def-H} (3),  we get that
\begin{equation}\label{Dtilde-interval}
\begin{split}
\widetilde{D}({\bm x},{\bm y})&\leq \widetilde{D}(V_{\rho r}(\bm x),V_{\rho r}(\bm y))+{\rm diam}(V_{\rho r}(\bm x);\widetilde{D})+{\rm diam}(V_{\rho r}(\bm y);\widetilde{D})\\
&\leq \widetilde{D}(V_{\rho r}(\bm x),V_{\rho r}(\bm y))
+C_*{\rm diam}(V_{\rho r}(\bm x);D)+C_*{\rm diam}(V_{\rho r}(\bm y);D)\\
&\leq \widetilde{D}(V_{\rho r}(\bm x),V_{\rho r}(\bm y)) +2C_*\rho^{\theta/2} D({\bm x},{\bm y}).
\end{split}
\end{equation}
By combining \eqref{H-ineq} and \eqref{Dtilde-interval}, we arrive at
$$
\widetilde{D}(V_{\rho r}(\bm x),V_{\rho r}(\bm y))\geq (C'-2C_*\rho^{\theta/2})D({\bm x},{\bm y}).
$$
Thus, we now obtain the first inequality in \eqref{G(C'-zeta)} by choosing $\rho<\gamma_0$ to be sufficiently small  (here $\gamma_0$ comes from Definition \ref{def-H} (2) of $H_r(\alpha,C')$ and  depends only on $\zeta,\alpha$ and $C_*$) and setting $\gamma=\rho$.
\end{proof}

We turn to the

\begin{proof}[Proof of Proposition \ref{strongP[G]}]
Let $\alpha,p\in(0,1)$ (depending only on the laws of $D$ and $\widetilde{D}$) be as in Proposition \ref{Hwithr=1}. Also let $C'':=(C'+C_*)/2$. By Proposition \ref{Hwithr=1} (applied with $C''$ instead of $C'$), there exists $\varepsilon_0=\varepsilon_0(C'')>0$ such that $\mathds{P}[H_\varepsilon(\alpha,C'')]\geq p$ for all $\varepsilon \leq \varepsilon_0$. By Lemma \ref{HandG}, applied with $C''$ in place of $C'$ and $\zeta=C''-C'$, we obtain that there exist $\gamma,q$, depending only on $\beta,d,\alpha,C'$ and the laws of $D$ and $\widetilde{D}$, such that if $H_r(\alpha,C'')$ occurs, then $G_r(\gamma,q,C')$ occurs. Combining the preceding two sentences gives the desired statement with $p\wedge \gamma$ in place of $\gamma$.
\end{proof}

Since our assumptions on the metrics $D$ and $\widetilde{D}$ are the same, the results above also hold with the roles of $D$ and $\widetilde{D}$ interchanged. For ease of reference, we will record some of these results here.

\begin{definition}\label{def-tildeG}
For $r>0$, $\gamma>0$, $\widetilde{q}>0$ and $c'>0$, let $\widetilde{G}_r(\gamma,\widetilde{q}, c')$ be the event that there exist ${\bm x},{\bm y}\in V_r(\bm 0)$ such that
$$
\widetilde{q}r^\theta\leq \widetilde{D}({\bm x},{\bm y})< c'D(V_{\gamma r}(\bm x),V_{\gamma r}(\bm y)).
$$
\end{definition}

\begin{definition}\label{def-tildeH}
For $r>0$, $\alpha\in (0,1)$ and $c'>0$, we let $\widetilde{H}_r(\alpha,c')$ be the event that there exist ${\bm x},{\bm y}\in  V_r(\bm 0)$ with $|\bm x-\bm y|\geq \alpha r/3$  such that
$$\widetilde{D}({\bm x},{\bm y})< c'D({\bm x},{\bm y}),$$
and there exists a $\widetilde{D}$-geodesic $\widetilde{P}$ from $\bm x$ to $\bm y$ satisfying
\begin{itemize}
\item[(1)] $\widetilde{P}\subset  V_r(\bm 0)$;

\item[(2)] $\widetilde{D}({\bm x},{\bm y})\geq (\widetilde{b}\alpha r)^\theta$ for some $\widetilde{b}=\widetilde{b}(\alpha)>0$;

\item[(3)] there exists $\widetilde{\gamma}_0=\widetilde{\gamma}_0(\alpha)>0$ such that for each $\widetilde{\gamma}\in(0,\gamma_0]$,
    $$
    \max\left\{{\rm diam}(V_{\widetilde{\gamma} r}(\bm x);\widetilde{D}),{\rm diam}(V_{\widetilde{\gamma} r}(\bm y);\widetilde{D})\right\}\leq \widetilde{\gamma}^{\theta/2} \widetilde{D}({\bm x},{\bm y}).
    $$
\end{itemize}
Here $\widetilde{b}$ and $\widetilde{\gamma}_0$ are chosen in a similar way as $b$ and $\gamma_0$ in Definition \ref{def-H}.
\end{definition}

We have the following analogs of Propositions \ref{mrinS3} and \ref{strongP[G]}, respectively.

\begin{proposition}\label{mrinS3-tilde}
For each small enough $\alpha\in(0,1)$, $p\in (0,1)$ {\rm(}depending only on $\beta,d$, the laws of $D$ and $\widetilde{D}${\rm)},  for $\widetilde{b}=\widetilde{b}(\alpha)>0$  and $\widetilde{\gamma}_0=\widetilde{\gamma}_0(\alpha)>0$ {\rm(}depending only on $\beta,d,\alpha$ and the laws of $D$ and $\widetilde{D}${\rm)}, and for each $c'>c_*$, there exists $c''=c''(c', \alpha)\in (c_*,c')$ such that
for each $\widetilde{\gamma}\in (0,1)$ and $\widetilde{q}>0$, there exists
 $\varepsilon_0=\varepsilon_0(\widetilde{\gamma},\widetilde{q},c')>0$ with the following property. If $r>0$ and $\mathds{P}[\widetilde{G}_r(\widetilde{\gamma},\widetilde{q},c'')]\geq \widetilde{\gamma}$, then $\mathds{P}[\widetilde{H}_{\varepsilon r}(\alpha,c')]\geq p$ for each $\varepsilon\in (0,\varepsilon_0]$.
\end{proposition}

\begin{proposition}\label{P[G]-tilde}
For each $c'>c_*$, there exist $\widetilde{\gamma},\widetilde{q}>0$, depending on $c'$ and the laws of $D$ and $\widetilde{D}$, such that for each small enough $\varepsilon>0$ {\rm(}depending only on $c'$ and the laws of $D$ and $\widetilde{D}${\rm)}, we have $\mathds{P}[\widetilde{G}_\varepsilon (\widetilde{\gamma},\widetilde{q},c')]\geq \widetilde{\gamma}$.
\end{proposition}

\subsection{Proof of Proposition \ref{mrinS3}}\label{prove-mr}

To prove Proposition \ref{mrinS3}, we will prove its contrapositive as stated in the following proposition (note that the contrapositive only applies to the last sentence of the statement).

\begin{proposition}\label{(B)=>G}
For each small enough $\alpha\in(0,1)$ and $p\in(0,1)$ {\rm(}depending only on $\beta,d$ and the laws of $D$ and $\widetilde{D}${\rm)},
 for  $b=b(\alpha)>0$ and $\gamma_0=\gamma_0(\alpha)>0$ {\rm(}depending only $\beta,d,\alpha$  and the laws of $D$ and $\widetilde{D}${\rm)},
 and for each $C'\in(0,C_*)$,
 there exists $C''=C''(C',\alpha)\in(C',C_*)$ such that for each $\gamma\in (0,1)$ and $q>0$, there exists $\varepsilon_0=\varepsilon_0(\gamma,q,C')>0$ with the following property.
If $r>0$ and there exists $\varepsilon\in(0,\varepsilon_0]$ satisfying that $\mathds{P}[H_{\varepsilon r}(\alpha,C')]<p$, then $\mathds{P}[G_r(\gamma,q,C'')]<\gamma$.
\end{proposition}

The basic idea of the proof of Proposition \ref{(B)=>G} is as follows. Assuming that the probability $\mathds{P}[H_{3\varepsilon r}(\alpha, C')]<p$ for some small enough  $p\in(0,1)$ (depending only on $\beta,d$ and the laws of $D$ and $\widetilde{D}$),
we can use a similar renormalization as in Section \ref{bi-lipschitz} to conclude that each path whose $D$-length is not too short should pass through a positive fraction of the small ``nice'' cubes $V_{\varepsilon r}({\bm k})$ for ${\bm k}\in(\varepsilon r)\mathds{Z}^d$.
Here ``nice'' generally means that for each ${\bm x},{\bm y}\in V_{3\varepsilon r}({\bm k})$ with $|\bm x-\bm y|\geq \alpha \varepsilon r$, if they are joined by a geodesic $P$ satisfying the numbered conditions in Definition \ref{def-H}, we have $\widetilde{D}({\bm x},{\bm y})\leq C'D({\bm x},{\bm y})$.

By considering the times when a $D$-geodesic between two fixed points ${\bm x},{\bm y}\in \mathds{R}^d$ passes through such a ``nice'' cube, we can show that $\widetilde{D}(V_{\gamma r}(\bm x), V_{\gamma r}(\bm y))\leq C''D(\bm x, \bm y)$ for a suitable constant $C''\in(C',C_*)$. Applying this to an appropriate collection of pairs of points $({\bm x},{\bm y})$ will show that $\mathds{P}[G_r(\gamma, q,C'')]<\gamma$.

Let us define some events that will be useful in the proof of Proposition \ref{(B)=>G}.

\begin{definition}\label{def-E}
(Converse of $H_r(\alpha,C')$)
For $s>0$, $\alpha\in (0,1)$, $C'>0$ and $\bm z\in\mathds{R}^d$, we let $E_s(\bm z)$ be the event that the following is true.
For each ${\bm x},{\bm y}\in V_s(\bm z)$ with $|\bm x-\bm y|\geq \alpha s/3$, if there is a $D$-geodesic $P$ from $\bm x$ to $\bm y$ such that
\begin{itemize}
\item[(1)] $P\subset V_s(\bm z)$;

\item[(2)] $ D({\bm x},{\bm y})\geq (b\alpha s)^\theta$ for some $b=b(\alpha)>0$ (same as Definition \ref{def-H});

\item[(3)]  there exists $\gamma_0=\gamma_0(\alpha)>0$ (same as Definition \ref{def-H}) such that for each $\gamma\in (0,\gamma_0]$,
$$
\max\left\{{\rm diam}(V_{\gamma r}(\bm x);D),{\rm diam}(V_{\gamma r}(\bm y);D)\right\}\leq \gamma^{\theta/2} D({\bm x},{\bm y}),
$$
\end{itemize}
then
$$
\widetilde{D}({\bm x},{\bm y})\leq C' D({\bm x},{\bm y}).
$$
\end{definition}



Next, we define a slightly stronger version of Definition \ref{h-good} for $(s,\alpha)$-good cube, which involves the occurrence of $E_s(\bm z)$ within the cube. We refer to such cubes as ``nice'' cubes.

\begin{definition}\label{alp-good}
For $s>0$, $\bm z\in \mathds{R}^d$ and $\alpha\in(0,1)$, we say that the cube $V_{3s}(\bm z)$ is {\it$(3s,\alpha)$-nice} if the following is true.
\begin{itemize}

\item[\rm (1)] $E_{3s}(\bm z)$ occurs.

\item[\rm (2)]  
 For any two different edges $\langle {\bm u}_1,{\bm v}_1\rangle$ and $\langle {\bm u}_2,{\bm v}_2\rangle\in\mathcal{E}$,  with ${\bm u}_1\in V_s({\bm z})^c, {\bm v}_1\in V_s({\bm z}), {\bm u}_2\in V_{3s}({\bm z}), {\bm v}_2\in V_{3s}({\bm z})^c$, we have $|{\bm v}_1-{\bm u}_2|\geq \alpha s$. Additionally, there is a constant $b>0$  (the same as that in Definition \ref{def-H}) such that
    $$
      D({\bm v}_1,{\bm u}_2 ;V_{3s}({\bm z}))\geq (b\alpha s)^\theta.
    $$
     Moreover, there exists $\gamma_0=\gamma_0(\alpha)>0$ such that
    $$
    \max\left\{{\rm diam}(V_{\gamma s}({\bm v}_1);D),{\rm diam}(V_{\gamma s}({\bm u}_2);D)\right\}\leq \gamma^{\theta/2} D({\bm u}_2,{\bm v}_1).
    $$

\end{itemize}
\end{definition}

In order to control the probability for Definition \ref{alp-good} (2), we only need to consider the case for ${\bm z} ={\bm 0}$ by Axiom IV' (translation invariance).
This immediately follows from Lemma \ref{h-probgood} and the continuity of $D$ in Proposition \ref{LimitContinuousTail}, and a precise statement of this is incorporated in the next lemma.

\begin{lemma}\label{super good (ii)}
For fixed $s>0$ and sufficiently small $\alpha\in(0,1)$, there exist  constants $b=b(\alpha)>0$  and $\gamma_0=\gamma_0(\alpha)>0$ {\rm(}depending only on $\beta,d,\alpha$ and the law of $D${\rm)}, and $c>0$ {\rm(}depending only on $\beta,d$ and the law of $D${\rm)},  such that Definition \ref{alp-good} {\rm(2)} occurs with probability at least $1-c\alpha(\log(\alpha^{-1}))$.
\end{lemma}

Hence, applying a union bound, we can prove the following lemma.
\begin{lemma}\label{good-prob}
For each $\delta_0\in(0,1)$,
when  $\alpha\in(0,1)$ and $p\in(0,1)$ are sufficiently small (depending only on $\delta_0$, $\beta,d$ and the laws of $D$ and $\widetilde{D}$), we have that for each $C'\in (0,C_*)$, there exist $b=b(\delta_0)>0$  and $\gamma_0=\gamma_0(\delta_0)>0$ satisfying the following property. If $r>0$ and there exists $\varepsilon\in (0,\varepsilon_0]$ satisfying $\mathds{P}[H_{3\varepsilon r}(\alpha,C')]<p$, then for each ${\bm z}\in\mathds{R}^d$, we have:
$$
\mathds{P}[\text{the cube }V_{3\varepsilon r}({\bm z})\ \text{is }(3\varepsilon r,\alpha)\text{-nice}]\geq 1-\delta_0.
$$
\end{lemma}
\begin{proof}
Thanks to Axiom IV'(translation invariance), we only need to prove the desired assertion when ${\bm z}={\bm 0}$.

From definition \ref{def-E}, we obtain that if $\mathds{P}[H_{3\varepsilon r}(\alpha,C')]<p$, then  $\mathds{P}[E_{3\varepsilon r}(\bm 0)]\geq 1-p$.
Additionally, by Lemma \ref{super good (ii)}, the probability of the event in Definition \ref{alp-good} (2) can be arbitrarily close
to 1 as $\alpha\rightarrow 0$. By combining the above analysis, we can conclude that
$$
\mathds{P}[\text{the cube }V_{3\varepsilon r}({\bm z})\ \text{is }(3\varepsilon r,\alpha)\text{-nice}] \ge 1-p-c\alpha(\log(\alpha^{-1}))\rightarrow 1\quad \text{as }p,\alpha\rightarrow 0.
$$
This implies the desired assertion.
\end{proof}

We now prove Proposition \ref{(B)=>G}  with the help of the above lemmas.

Fix $\alpha\in(0,1)$ and $ p\in(0,1)$ which are sufficiently small and will be chosen below. We will show that for each $\gamma>0$ and $q>0$, there exists $\varepsilon_0=\varepsilon_0(\gamma,q,C')>0$ such that if $r>0$, $\varepsilon\in(0,\varepsilon_0]$ and $\mathds{P}[H_{3\varepsilon r}(\alpha,C')]<p$ holds for the above $\alpha,p,r,\varepsilon$, then there exists $C''\in(C',C_*)$ such that
 with probability at least $1-\gamma$, for all ${\bm x},{\bm y}\in  V_r(\bm 0)$  with $D({\bm x},{\bm y})\geq qr^\theta$,
\begin{equation}\label{G^c}
\widetilde{D}(V_{\gamma r}(\bm x),V_{\gamma r}(\bm y))\leq C'' D({\bm x},{\bm y}).
\end{equation}
Then by Definition \ref{def-G}, \eqref{G^c} implies that $\mathds{P}[G_r(\gamma,q,C'')^c]>1-\gamma$, which is the desired result in Proposition \ref{(B)=>G}.

Throughout the proof, we assume that $\varepsilon\ll\gamma$ and $\varepsilon\ll q$. Note that we can make this assumption since $\varepsilon_0$ is allowed to depend on $\gamma$ and $q$. Additionally, without loss of generality, we assume that $1/\varepsilon\in\mathds{Z}$. If this is not true, we could replace $1/\varepsilon$ with $\lfloor1/\varepsilon\rfloor$.

We recall the renormalization of the continuous model in Section \ref{bi-lipschitz}, which will play a crucial role throughout the proof.
We divide $\mathds{R}^d$ into small cubes of side length $\varepsilon r$, i.e., $\mathds{R}^d=\cup_{{\bm k}\in(\varepsilon r)\mathds{Z}^d}V_{\varepsilon r}({\bm k})$.
We identify the cube $V_{\varepsilon r}({\bm k})$ with the vertex ${\bm k}$ and call the resulting graph $G$.
We denote $P$ and $P^G$ for the paths in the continuous model and in $G$, respectively, and denote $d^G$ for the chemical distance of $G$.
For each ${\bm j}\in(\varepsilon r)\mathds{Z}^d$, we say that ${\bm j}$ is {\it nice} in graph $G$ if the cube $V_{3\varepsilon r}({\bm j})$ is $(3\varepsilon r, \alpha)$-nice as in Definition \ref{alp-good}. From Lemma \ref{good-prob}, we see that
\begin{equation}\label{active-pro}
\mathds{P}[{\bm j}\ \text{is nice}]\geq 1-\delta_0(\alpha,p),
\end{equation}
where $\delta_0(\alpha,p)$ can be arbitrarily close to 1 as $\alpha,p\rightarrow 0$.

 For any ${\bm i},{\bm j}\in (\varepsilon r)\mathds{Z}^d$ and $m\in\mathds{N}$, we recall $\mathcal{P}_m({\bm i},{\bm j})$ as the collection of self-avoiding paths $P^G$ from ${\bm i}$ to ${\bm j}$ in $G$ with length $m$ and $\mathcal{P}_{\geq m}({\bm i},{\bm j})=\cup_{n\geq m}\mathcal{P}_n({\bm i},{\bm j})$. Let $F'_{\varepsilon,1}$ be the event that  all paths $P^G$ in $\cup_{{\bm i},{\bm j}\in (\varepsilon r)[0,1/\varepsilon)_{\mathds{Z}}^d}\mathcal{P}_{\geq \varepsilon^{-\theta/4}{-1}}({\bm i},{\bm j})$ pass through at least $|P^G|/(4\cdot 3^d)$ of nice points.

Similar to Lemma \ref{Fr4lemma-2}, we can prove the event $F'_{\varepsilon,1}$ occurs with high probability.
\begin{lemma}\label{Fr4lemma}
For each sufficiently small $\alpha,p\in(0,1)$,  we have that $F'_{\varepsilon,1}$ occurs with probability at least $1-O_\varepsilon(\varepsilon^\mu)$ for all $\mu>0$. Here the implicit constant in the $O_\varepsilon(\cdot)$ depends only on $\beta,d,\alpha, \mu,p$  and the law of $D$.
\end{lemma}

\begin{proof}
Let ${\bm i},{\bm j}\in (\varepsilon r)[0,1/\varepsilon)_{\mathds{Z}}^d$ and let $P_{{\bm i}{\bm j}}^G$ be a path from ${\bm i}$ to ${\bm j}$ in $G$ with $|P_{{\bm i}{\bm j}}^G|\geq \varepsilon^{-\theta/4}{ -1}$.
From \eqref{active-pro} we can apply Lemma \ref{hd-BK} by replacing $\{E_{\bm k}\}_{{\bm k}\in(\varepsilon r)\mathds{Z}^d}$ with $\{E_{3\varepsilon r}({\bm k})\}_{{\bm k}\in(\varepsilon r)\mathds{Z}^d}$ (see Definition \ref{alp-good} (1)) and replacing Definition \ref{hd-PGgood} (2) with Definition \ref{alp-good} (2), and derive the following: for each $m\geq \varepsilon^{-\theta/4} -1$ and each $P_{{\bm i}{\bm j}}^G\in\mathcal{P}_m({\bm i},{\bm j})$,
\begin{equation*}
\begin{split}
&\mathds{P}\left[\text{the number of nice points in $P_{{\bm i}{\bm j}}^G$ is at most }|P_{{\bm i\bm j}}^G|/(4\cdot 3^d) \Big| G\right]\\
&=\mathds{P}\left[ P_{{\bm i \bm j}}^G\text{ is }(\{E_{3\varepsilon r}({\bm k})\}_{{\bm k}\in(\varepsilon r)\mathds{Z}^d},\alpha)\text{-bad}\Big| G\right]\leq \delta_0(\alpha,p)^{m}.
\end{split}
\end{equation*}
Then using the arguments in the proof of Lemma \ref{Fr4lemma-2}, we can obtain the desired statement by taking sufficiently small $\alpha$ and $p$ such that $2C_{dis}\delta_0(\alpha,p)<1$. Here $C_{dis}$ is the constant in Lemma \ref{number-path-k} depending only on $\beta$ and $d$.
\end{proof}


For the same reasons as outlined in Section \ref{bi-lipschitz}, we need the event that each geodesic between any two points ${\bm x},{\bm y}\in  V_r(\bm 0)$ lies entirely within a compact set. To achieve this, we also get from Lemma \ref{geodesic-go-not-far} that for fixed $\gamma>0$, there is a sufficiently large $u=u(\gamma)\in\mathds{N}$ (which depends only on $\beta,d,\gamma$ and the law of $D$) such that with probability at least $1-\gamma/5$, each $D$-geodesic between any two points of $ V_r(\bm 0)$ is contained in $[-ur, ur]^d$. We will refer to this event as $F'_{\gamma,2}$. That is,
\begin{equation}\label{Fr2}
\mathds{P}[F'_{\gamma,2}]\geq 1-\gamma/5.
\end{equation}


For any $\bm x\in V_r(\bm 0)$, let ${\bm k}_{\bm x}\in(\varepsilon r)[-u/\varepsilon,u/\varepsilon)_\mathds{Z}^d$ satisfy $\bm x\in V_{\varepsilon r}(\bm k_{\bm x})$. We next let $F'_{\varepsilon,3}$ be the event that for any ${\bm x},{\bm y}\in V_r(\bm 0)$ with $D({\bm x},{\bm y})\ge qr^\theta\ge 2\varepsilon^{\theta/2}r^\theta$,
    \begin{equation}\label{eq-d-longer-D-3}
        d^G({\bm k}_{\bm x},{\bm k}_{\bm y})+1\ge \max\left\{c_1(\varepsilon r)^{-\theta} D({\bm x},{\bm y}),\varepsilon^{-\theta/4}\right\}.
    \end{equation}
Then from Lemma \ref{d-longer-than-D}, we can see that  there exists a constant $c_1>0$, depending only on $\beta,d$  and the law of $D$, such that
\begin{equation}\label{d-longer-than-D-3}
\mathds{P}[F'_{\gamma,2}\cap (F'_{\varepsilon,3})^c]\leq  O_\varepsilon(\varepsilon^{\mu})\quad \text{for any\ } \mu>0,
\end{equation}
 where the implicit constant depends only on $\beta,d, \gamma,\mu$  and the law of $D$.
Recalling the fact that $\varepsilon \ll \gamma$ and combining \eqref{Fr2}, \eqref{d-longer-than-D-3} with Lemma \ref{Fr4lemma},  we get that
\begin{equation}\label{Fr3}
\mathds{P}[F'_{\varepsilon,1}\cap F'_{\gamma,2}\cap F'_{\varepsilon,3}]\geq 1-\frac{3\gamma}{5}.
\end{equation}
In what follows, we assume that $F'_{\varepsilon,1}\cap F'_{\gamma,2}\cap F'_{\varepsilon,3}$ occurs.

By using \eqref{eq-d-longer-D-3}, we can show that on the event $F'_{\varepsilon,1}\cap F'_{\gamma,2}\cap F'_{\varepsilon,3}$, for all ${\bm x},{\bm y}\in V_r(\bm 0)$ with $D({\bm x},{\bm y})\geq q r^\theta$, we can construct a skeleton path $P^G_{{\bm k}_{\bm x},{\bm k}_{\bm y}}$ in $G$ from a $D$-geodesic from $\bm x$ to $\bm y$ with
\begin{equation}\label{Lsharper}
    |P^G_{{\bm k}_{\bm x},{\bm k}_{\bm y}}|\geq \max\{c_1(\varepsilon r)^{-\theta} D({\bm x},{\bm y}),\varepsilon^{-\theta/4}\}-1\geq\max\{c_1(\varepsilon r)^{-\theta} D({\bm x},{\bm y}),\varepsilon^{-\theta/4}\}/2,
\end{equation}
which implies that the number of cubes $V_{\varepsilon r}({\bm k})$ hit by a $D$-geodesic from $\bm x$ to $\bm y$ has a lower bound given in \eqref{Lsharper}.
This allows us to only consider paths from $\bm x$ to $\bm y$ that pass through at least  $(\varepsilon^{-\theta/4}-1)$ many $V_{\varepsilon r}({\bm k})$'s (or equivalently, the paths from ${\bm k}_{\bm x}$ to ${\bm k}_{\bm y}$ in $G$ with lengths at least $\varepsilon^{-\theta/4}-1$) in the following.
To ease exposition, we denote by $\mathcal{P}_{\geq \varepsilon^{-\theta/4}-1}({\bm k}_{\bm x},{\bm k}_{\bm y})$ the set of such paths in $G$.


We will now prove \eqref{G^c}  on $F'_{\varepsilon,1}\cap F'_{\gamma,2}\cap F'_{\varepsilon,3}$. Let ${\bm x},{\bm y}\in V_r(\bm 0)$  be such that $D({\bm x},{\bm y})\ge qr^\theta$ and let $P:[0,D({\bm x},{\bm y})]\rightarrow \mathds{R}^d$ be a $D$-geodesic from $\bm x$ to $\bm y$. Assume that
\begin{equation*}
|\bm x-\bm y|\geq \gamma r\quad \text{and}\quad V_{\gamma r}(\bm x)\nsim V_{\gamma r}(\bm y).
\end{equation*}
This assumption is also without loss of generality since \eqref{G^c} is trivially satisfied when $|\bm x-\bm y|\leq \gamma r$ or when  a long edge directly connects cubes $V_{\gamma r}(\bm x)$ and  $V_{\gamma r}(\bm y)$.
More specifically, if $|\bm x-\bm y|\leq \gamma r$, it is obvious that $V_{\gamma r}(\bm x)\cap V_{\gamma r}(\bm y)\neq \emptyset$. This means the left hand side of \eqref{G^c} is 0, i.e., $\widetilde{D}(V_{\gamma r}(\bm x),V_{\gamma r}(\bm y))=0$. Thus \eqref{G^c} holds trivially. Similarly, if there is a long edge directly connecting cubes $V_{\gamma r}(\bm x)$ and  $V_{\gamma r}(\bm y)$, we have the same conclusion.

On the event $F'_{\varepsilon,1}\cap F'_{\gamma,2}\cap F'_{\varepsilon,3}$ (by a simple volume consideration), there exist at least  $c_1(\varepsilon r)^{-\theta}D({\bm x},{\bm y})/( 8\cdot 3^d)$ many $(3\varepsilon r,\alpha)$-nice cubes which do not intersect each other such that there exists
a $D$-geodesic $P$ which passes through each such nice cube (say $V_{3\varepsilon r}(\bm k)$) and also hits $V_{\varepsilon r}(\bm k)$ (see Lemma \ref{Fr4lemma} and \eqref{Lsharper}). Let $\Lambda$ be the set of all such cubes.

We next define a sequence of times
$$
0=t_0<s_1<t_1<s_2<t_2<\cdots<s_L<t_L<s_{L+1}=D({\bm x},{\bm y}) 
$$
where $L\geq  c_1(\varepsilon r)^{-\theta}D({\bm x},{\bm y})/(8\cdot 3^d)$ and our definition is by induction as follows. Let $t_0=0$. Let $j\in\mathds{N}$ be arbitrary and assume that $t_{j-1}$ has been defined. Then we let $s_j$ be the first time after $t_{j-1}$ that $P$ enters a new cube $V_{\varepsilon r}({\bm k}_j)$ which is contained in a $(3\varepsilon r,\alpha)$-nice cube $V_{3\varepsilon r}({\bm k}_j)\in\Lambda$, and let $t_j$ be the first time that $P$ exits this $(3\varepsilon r,\alpha)$-nice cube $V_{3\varepsilon r}$ after $s_j$ (see Figure \ref{Proof512}). Thus $P([s_j,t_j])\subset V_{3\varepsilon r}({\bm k}_j)$. Based on Definition \ref{alp-good} (2) for $(3\varepsilon r,\alpha)$-nice cubes, we get
\begin{align}
    \label{Ps-Pt>}
    &|P(s_j)-P(t_j)|\geq \alpha\varepsilon r,\\
    \label{D(PsPt)}
    & (b\alpha \varepsilon r)^\theta \leq D(P(s_j),P(t_j);V_{3\varepsilon r}({\bm k}_j))= D(P(s_j),P(t_j))= t_j-s_j,\\
    & \max\left\{{\rm diam}(V_{\gamma\varepsilon r}(P(s_j));D),{\rm diam}(V_{\gamma\varepsilon r}(P(t_j));D)\right\}
    \leq \gamma^{\theta/2} D(P(s_j),P(t_j))\nonumber
\end{align}
\begin{figure}[htbp]
\centering
\includegraphics[scale=0.55]{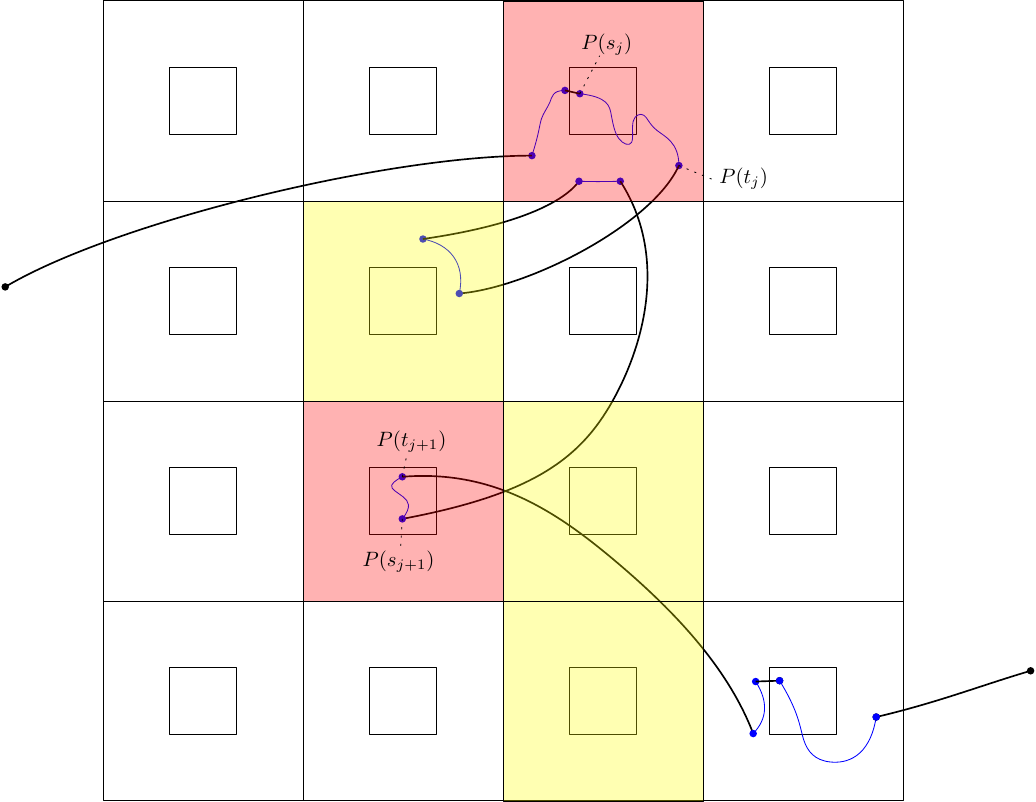}
\caption{Illustration for the definition of the times $t_j$ and $s_j$ in the two-dimensional case. Here the red cubes are  $(3\varepsilon r,\alpha)$-nice contained in $\Lambda$, the yellow cubes are  $(3\varepsilon r,\alpha)$-nice but not contained in  $\Lambda$, while other cubes are not $(3\varepsilon r,\alpha)$-nice.}
\label{Proof512}
\end{figure}
for all $j=1,2,\cdots, L$, where $b$  and $\gamma_0$ are the constants in Definition \ref{alp-good} (2) that depend only on $\beta,d$ and $\alpha$.
Therefore, combining \eqref{Ps-Pt>}, \eqref{D(PsPt)} with the fact that $P([s_j,t_j])\subset V_{3\varepsilon r}({\bm k}_j)$, all conditions in Definition \ref{def-E} of $E_{3\varepsilon r}({\bm k}_j)$ hold. Hence, since $\bm k_j$ is $(3\varepsilon r,\alpha)$-nice, Definition \ref{alp-good} (1) also implies
\begin{equation}\label{goodseg}
\widetilde{D}(P(s_j),P(t_j))\leq C' D(P(s_j),P(t_j))=C'(t_j-s_j)\quad \text{for all }j=1,2,\cdots, L.
\end{equation}

Moreover, Combining the lower bound in \eqref{D(PsPt)} with $L\geq  c_1(\varepsilon r)^{-\theta}D({\bm x},{\bm y})/(8\cdot 3^d)$ yields that
$$
\sum_{j=1}^L (t_j-s_j) \geq c_1(b\alpha)^\theta D({\bm x},{\bm y})/( 8\cdot 3^d).
$$
Hence, we have
\begin{equation}\label{sj-tj-1}
\sum_{j=1}^{L+1}(s_j-t_{j-1})=D({\bm x},{\bm y})-\sum_{j=1}^L (t_j-s_j)\leq
\left(1- \frac{c_1(b\alpha)^\theta}{8\cdot 3^d}
\right)D({\bm x},{\bm y}).
\end{equation}
Therefore, we obtain that

\begin{equation}\label{C''D}
\begin{split}
&\widetilde{D}(V_{\gamma r}(x),V_{\gamma r}(y))\leq \widetilde{D}(V_{3\varepsilon r}({\bm k}_x),V_{3\varepsilon r}({\bm k}_y))
\quad\quad \text{(by $\varepsilon \ll \gamma$)}\\
&\leq \sum_{j=1}^{L+1}\widetilde{D}(P(t_{j-1}),P(s_j))+\sum_{j=1}^L\widetilde{D}(P(s_j),P(t_j))\quad\quad \text{(by the triangle inequality)}\\
&\leq C_*\sum_{j=1}^{L+1}(s_j-t_{j-1})+C'\sum_{j=1}^L(t_j-s_j)\quad\quad \text{(by \eqref{goodseg})}\\
&=C'D({\bm x},{\bm y})+(C_*-C')\sum_{j=1}^{L+1}(s_j-t_{j-1})\quad\quad\text{(since $P$ is a $D$-geodesic from $\bm x$ to $\bm y$)}\\
&\leq \left(C'+\left(1-\frac{c_1(b\alpha)^\theta}{  8\cdot 3^d}
\right)(C_*-C')\right)D({\bm x},{\bm y})\quad\quad\text{(by \eqref{sj-tj-1})}\\
&\leq C''D({\bm x},{\bm y})
\end{split}
\end{equation}
with $C''\in (C'+a(C_*-C'),C_*)$, where $a:=1-\frac{c_1(b\alpha)^\theta}{{8\cdot 3^d}}
\in(0,1)$ does not depend on $r,{\bm x},{\bm y} ,\gamma,q$ or $\varepsilon$. Hence we complete the proof of \eqref{G^c} on the event $F'_{\varepsilon,1}\cap F'_{\gamma,2}\cap F'_{\varepsilon,3}$.

To summarize the subsection (which aims to prove Proposition \ref{(B)=>G}), we recall \eqref{Fr3}.
Therefore, we can choose $\varepsilon_0=\varepsilon_0(\gamma, q,C')>0$ small enough so that $\varepsilon \ll \gamma$ and $\varepsilon \ll q$.
Then by \eqref{C''D} and Definition \ref{def-G} of the event $G_r(\gamma,q, C'')$, we get that for $\varepsilon \in(0,\varepsilon_0]$, the condition $\mathds{P}[H_{\varepsilon r}(\alpha,C')]<p$ implies that $\mathds{P}[G_r(\gamma,q, C'')]<\gamma$. Hence, the proof of Proposition \ref{(B)=>G} is complete.

\section{The core argument}\label{section4}

 Recall that we have assumed that $c_*<C_*$. In this section, we will show that the probability of the event $H_r(\alpha, C_*-\delta)$ goes to zero as $\delta\to 0$ uniformly in $r$, which will yield  a contradiction with Proposition \ref{Hwithr=1} in light of Proposition \ref{mrinS3-tilde}. Therefore we can complete the proof of Theorem \ref{uniqueness}.

The main idea is to count the number of occurrences of events of a certain type (see Definitions \ref{F_Z} and \ref{G-}).
To this end, we will lift the probability $p$ in Proposition \ref{mrinS3-tilde} by using a multi-scale analysis.
This will allow us to generate many ``nice'' cubes in $\mathds{R}^d$, as detailed in Section \ref{def-4nice}.

After that, in order to establish the counting argument, we will introduce a family of candidate sets  $\mathcal{W}_\varepsilon$ (a collection of all ``suitable'' subsets of $[1,\varepsilon^{-d}]_\mathds{Z}$ for some sufficiently small $\varepsilon>0$) 
for the cubes on which we add or delete edges. Then for $Z\in{\mathcal{W}_\varepsilon}$ we define two new metrics with respect to the modified edge sets $D_Z^+$ and $D_Z^-$ which are obtained from adding or deleting a number of edges which concatenate the ``nice'' cubes of small scale together in some large ``nice'' cubes $V_{\varepsilon r}(\bm z_k)$ for $k\in Z$ (see Section \ref{add-delete}).
Using the behavior of $D_Z^+$-paths in these ``nice'' cubes, we can obtain a lower bound on ``shortcuts'' for the metric $D_Z^+$,
which will induce a lower bound on the number of sets $Z$ with $\#Z\leq m$ such that $\mathsf{F}_{Z, \varepsilon}$ (defined in Definition \ref{F_Z}) occurs.
Similarly, by analyzing the behavior of $D_Z^-$-paths in these ``nice'' cubes, we can obtain an upper bound on the number of sets $Z$ with $\#Z\leq m$ such that $\mathsf{G}^-_{Z, \varepsilon}$ (defined in Definitions \ref{G-}) occurs. Combining the above arguments, we could get the desired statement.


\subsection{Definitions of nice and super good cubes}\label{def-4nice}

Before stating the conditions that our events  need to satisfy, we first introduce some notation as follows.

For $\bm z\in \mathds{R}^d$ and $s>0$, recall that $V_s(\bm z)$ denotes the cube with center $\bm z$ and side length $s$.
We begin with a modified version of $(s,\alpha)$-nice cube in Definition \ref{alp-good} as follows.

\begin{definition}\label{interval-good}
 For $c'>0$, $s>0$, $\bm z\in\mathds{R}^d$, sufficiently small $\eta\in(0,1)$ and $\alpha\in(0,1)$, we say that the cube $V_s(\bm z)$ is {\it$(s,\alpha,\eta)$-nice} with respect to $\mathcal{E}$ if there exist two small cubes
$$J^{(1)},\ J^{(2)}\subset V_s(\bm z)$$
 both with side length $\eta s$ satisfying the following conditions. 

\begin{itemize}
\item[(1)] $\text{dist}(J^{(1)},J^{(2)})\geq \alpha s$.

\item[(2)]  There is a  $\widetilde{D}$-geodesic $P$ from $J^{(1)}$ to $J^{(2)}$ such that $P\subset V_s(\bm z)$, and
$$
\widetilde{D}(J^{(1)},J^{(2)})< c'D(J^{(1)},J^{(2)}).
$$

\item[(3)]  There is a constant $b_1>0$ (chosen in Lemma \ref{interval-good1} below) such that
$$
D(J^{(1)},J^{(2)})\geq b_1(\alpha s)^\theta.
$$


\item[(4)]  For any $\bm u,\bm v\in V_s(\bm z)$,
$$
D(\bm u,\bm v ;V_s(\bm z))\le C_1 \|\bm u-\bm v\|_\infty^\theta\log\frac{2s}{\|\bm u-\bm v\|_\infty},
$$
where $C_1=2^{\theta+1}\max\left\{\log(2^{d+3} C_D/p),2^{\theta+d+1}\right\}>0$ with constants {$C_D$} defined in Proposition \ref{LimitContinuousTail} and  $p$ in Proposition \ref{mrinS3}.
\end{itemize}
We refer to the above $(J^{(1)}, J^{(2)})$ as a \textit{great pair of cubes} of $V_s(\bm z)$, and say a path passes through a great pair  $(J^{(1)}, J^{(2)})$ if it hits both $J^{(1)}$ and $J^{(2)}$.
\end{definition}

\medskip
Throughout this section, we set
$$
c'=\frac{c_*+C_*}{2}
$$
which belongs to $(c_*,C_*)$ if $c_*<C_*$.

In the following lemma, we establish a lower bound on the probability for a cube $V_s(\bm z)$ to be $(s, \alpha, \eta)$-nice with suitable choices for $\alpha, \eta, p$ and $s$.

\begin{lemma}\label{interval-good1}
For $\bm z\in\mathds{R}^d$, there exist sufficiently small $\alpha_0,p,\eta_0\in(0,1)$ such that the following property holds for each $\alpha\in(0,\alpha_0)$ and $\eta\in(0,\eta_0)$. Let $\widetilde{\gamma}$, $\widetilde{q}>0$ and $r>0$ be such that $\mathds{P}[\widetilde{G}_r(\widetilde{\gamma},\widetilde{q},c'')]\geq \widetilde{\gamma}$. Here $c''=c''(c',\alpha)$ is the constant defined in Proposition \ref{mrinS3-tilde} with $c'=(c_*+C_*)/2$. Then there exist constants $\varepsilon_0=\varepsilon_0(\widetilde{\gamma},\widetilde{q},c')>0$ {\rm(}depending only on $\beta,d,\alpha,p,\widetilde{\gamma},\widetilde{q},c'$ and the laws of $D$ and $\widetilde{D}${\rm)},
$b_1=b_1(\alpha)>0 $ {\rm(}depending only on $\beta,d,\alpha$  and the laws of $D$ and $\widetilde{D}${\rm)}, and
$C_1:=2^{\theta+1}\max\left\{\log(2^{d+3} C_D/p),2^{\theta+d+1}\right\}>0$ such that for all $\varepsilon\in(0,\varepsilon_0]$,
$$
\mathds{P}[\text{the cube }V_{\varepsilon r}(\bm z)\ \text{is }(\varepsilon r,\alpha,\eta)\text{-nice with respect to $\mathcal{E}$}]\geq 3p/4.
$$
\end{lemma}

\begin{proof}
By Axiom IV' (translation invariance), it suffices to show the desired result in the case when $\bm z=\bm 0$.

Let us start by noting that Definition \ref{interval-good} (1) and (2) correspond to the conditions in Definition \ref{def-tildeH} of $\widetilde{H}_s(\alpha, c')$, with $\bm x,\bm y$ replaced by  $J^{(1)}$ and $J^{(2)}$.
With this in mind, by Propositions \ref{P[G]-tilde} and \ref{mrinS3-tilde},
we see that there exists a sufficiently small $\alpha_0\in (0,1)$ such that for each $\alpha<\alpha_0$ and sufficiently small $\varepsilon$, $\widetilde{H}_{\varepsilon r}(6\alpha, (c'+c_*)/2)$ occurs with probability at least $p$ (the constant in Proposition \ref{mrinS3-tilde}). To be precise (recall definition \ref{def-tildeH}), with probability at least $p$ the following is true. There exist $\widetilde{b}(\alpha)>0$ and $\bm x,\bm y\in V_{\varepsilon r}(\bm 0)$ such that
\begin{itemize}
    \item[(i)] $|\bm x-\bm y|\geq (6\alpha) \varepsilon r/3=2\alpha \varepsilon r$.
    \item[(ii)] There is a $\widetilde{D}$-geodesic from $\bm x$ to $\bm y$ in $V_{\varepsilon r}(\bm 0)$ such that $$\widetilde{D}(\bm x,\bm y)<(c'+c_*)D(\bm x,\bm y)/2.$$
    \item[(iii)] $\widetilde{D}(\bm x,\bm y)\geq\widetilde{b}(\alpha \varepsilon r)^\theta.$
\end{itemize}
It is worth emphasizing that the reason why we chose $6\alpha$ and $(c'+c_*)/2$ in the $\widetilde{H}$-event is that we want to use $2\alpha$ ($>\alpha$) in (i) and the fact that $(c'+c_*)/2\in (c_*,c')$ to control the error term caused by the diameters of $J^{(1)}$ and $J^{(2)}$; see \eqref{def6.1(2)} and \eqref{def6.1(3)} below.

Additionally, let $\widetilde{E}_{\varepsilon r}$ denote the event that Definition \ref{interval-good} (4) holds with $V_s(\bm z)$ replaced by $V_{\varepsilon r}(\bm 0)$. Applying Proposition \ref{LimitContinuousTail} with $C_1=2^{\theta+1}\max\left\{\log(2^{d+3} C_D/p),2^{\theta+d+1}\right\}>0$, we obtain
\begin{equation*} 
\mathds{P}[\widetilde{E}_{\varepsilon r}]\geq 1-p/4
\end{equation*}
for sufficiently small $\varepsilon>0$.

In what follows, we assume that $\widetilde{H}_{\varepsilon r}(6\alpha,(c'+c_*)/2)$ occurs. Let $J^{(1)}$ and $J^{(2)}$  be cubes contained in $V_{\varepsilon r}(\bm 0)$ such that $J^{(1)}$ and $J^{(2)}$ each contains $\bm x$ and $\bm y$, respectively, and has a side length $\eta\varepsilon r$.
Now we check that there exists some $b_1=b_1(\alpha)>0$ such that when $\eta$ is sufficiently small (depending only on $\beta,d,\alpha,p$ and the laws of $D$ and $\widetilde{D}$), $
J^{(1)}$ and $J^{(2)}$ satisfy Definition \ref{interval-good} (1)--(3). It is clear that Definition \ref{interval-good} (1) holds when $\eta<\alpha/2$ from the triangle inequality and the fact that $|\bm x-\bm y|\geq 2\alpha\varepsilon r$. For Definition \ref{interval-good} (2) and (3), we first note that by the definition of $C_*$ and the fact that $\widetilde{D}(\bm x,\bm y)\geq \widetilde{b}(\alpha \varepsilon r)^\theta$,
\begin{equation}\label{lem6.2check1}
    D(\bm x,\bm y)\geq \widetilde{D}(\bm x,\bm y)/C_*\geq \widetilde{b}(\alpha \varepsilon r)^\theta/C_*.
\end{equation}
Additionally, from the triangle inequality, we have
\begin{equation}\label{lem6.2check2}
    D(J^{(1)},J^{(2)})\geq D(\bm x,\bm y)-\sum_{i=1}^2 {\rm diam}(J^{(i)};D).
\end{equation}
Moreover,
from the definition of $\widetilde{E}_{\varepsilon r}$, we see that
\begin{equation}\label{lem6.2check3}
    {\rm diam}(J^{(i)};D)\leq (C_1\sup_{t\in[0,\eta]}t^\theta\log(2/t))(\varepsilon r)^\theta\quad \text{for }i=1,2.
\end{equation}
Thus, combining \eqref{lem6.2check1}, \eqref{lem6.2check2} and \eqref{lem6.2check3} with the definition of $\widetilde{H}_{\varepsilon r}(6\alpha,(c'+c_*)/2)$, we obtain that
\begin{equation}\label{def6.1(2)}
    \begin{split}
    &\widetilde{D}(J^{(1)},J^{(2)})\leq \widetilde{D}(\bm x,\bm y)<(c'+c_*)D(\bm x,\bm y)/2\quad\quad\text{(by (ii))}\\
    &\leq (c'+c_*)D(J^{(1)},J^{(2)})/2+C_*\sum_{i=1}^2 {\rm diam}(J^{(i)};D)\quad\quad\text{(by \eqref{lem6.2check2})}\\
    &\leq c'D(J^{(1)},J^{(2)})+2C_*C_1\sup_{t\in[0,\eta]}\left(t^\theta\log(2/t)\right)(\varepsilon r)^\theta-(c'-c_*)D(\bm x,\bm y)/2
    \quad\quad\text{(by \eqref{lem6.2check3})}\\
    &\leq c'D(J^{(1)},J^{(2)})+\left[2C_*C_1\sup_{t\in[0,\eta]}\left(t^\theta\log(2/t)\right)-\widetilde{b}\alpha^\theta(c'-c_*)/(2C_*)\right](\varepsilon r)^\theta
    \quad\quad\text{(by \eqref{lem6.2check1})}\\
    &\leq c'D(J^{(1)},J^{(2)}),
    \end{split}
\end{equation}
where the last inequality holds when $\eta$ is chosen to satisfy
\begin{equation}\label{eta1}
    2C_*C_1\sup_{t\in[0,\eta]}\left(t^\theta\log(2/t)\right)<\widetilde{b}\alpha^\theta(c'-c_*)/(2C_*).
\end{equation}
Thus, by \eqref{def6.1(2)} we see that $J^{(1)},J^{(2)}$ satisfy Definition \ref{interval-good} (2).  Moreover, from \eqref{lem6.2check1}, \eqref{lem6.2check2} and \eqref{lem6.2check3} again,
we obtain that
\begin{equation}\label{def6.1(3)}
    \begin{split}
    D(J^{(1)},J^{(2)})&\geq D(\bm x,\bm y)-\sum_{i=1}^2 {\rm diam}(J^{(i)};D)\\
    &\geq [\widetilde{b}\alpha^\theta/C_*-2C_1\sup_{t\in[0,\eta]}\left(t^\theta\log(2/t)\right)](\varepsilon r)^\theta\\
    &\geq \widetilde{b}\alpha^\theta(\varepsilon r)^\theta/(2C_*),
    \end{split}
\end{equation}
where the last inequality holds when $\eta$ is chosen to satisfy
\begin{equation}\label{eta2}
    2C_1\sup_{t\in[0,\eta]}\left(t^\theta\log(2/t)\right)<\widetilde{b}\alpha^\theta/(4C_*).
\end{equation}
So by \eqref{def6.1(3)} we see that $J^{(1)},J^{(2)}$ satisfy  Definition \ref{interval-good} (3) with $b_1=\widetilde{b}/(2C_*)$.

Hence, if $\eta<\alpha/2$ is sufficiently small such that \eqref{eta1} and \eqref{eta2} hold, then on the event $\widetilde{H}_{\varepsilon r}(6\alpha,(c'+c_*)/2)\cap \widetilde{E}_{\varepsilon r}$, there exist $J^{(1)},J^{(2)}\subset V_{\varepsilon r}(\bm 0)$ such that Definition \ref{interval-good} (1)--(3) hold. Meanwhile, note that on the event $\widetilde{H}_{\varepsilon r}(6\alpha,(c'+c_*)/2)\cap \widetilde{E}_{\varepsilon r}$, the cube $V_{\varepsilon r}(\bm 0)$ also satisfies  Definition \ref{interval-good} (4). As a result,
\begin{equation*}
    \mathds{P}[V_{\varepsilon r}(\bm 0)\text{ is }(\varepsilon r,\alpha,\eta)\text{-nice}]\geq\mathds{P}[\widetilde{H}_{\varepsilon r}(6\alpha,(c'+c_*)/2)\cap \widetilde{E}_{\varepsilon r}]\geq p+(1-p/4)-1=3p/4.
\end{equation*}
Therefore, we finish the proof.
\end{proof}

As we saw above, the probability of a cube $V_{\varepsilon r}(\bm z)$ being $(\varepsilon r,\alpha,\eta)$-nice is bounded from below but may be very small, which can cause difficulties in our counting arguments. To overcome this, we next enhance the probability by using a multi-scale analysis and taking advantage of independence. To be a bit more precise, we will show that in a relatively large cube, the probability for the existence of a nice cube is close to 1.

 From here on, we let $\alpha\in(0,\alpha_0)$ and $\eta\in(0,\eta_0)$. Fix $r>0$ and $\widetilde{\gamma},\widetilde{q}>0$ such that $\mathds{P}[\widetilde{G}_r(\widetilde{\gamma},\widetilde{q},c'')]\geq \widetilde{\gamma}$. We also let $\varepsilon\in(0,\varepsilon_0]$. For convenience, assume that $1/(3\varepsilon) \in\mathds{Z}$, otherwise we can use $\lfloor 1/(3\varepsilon)\rfloor$ to replace $1/(3\varepsilon)$ in the following.

We first divide $\mathds{R}^d$ into cubes of side length $\varepsilon r$, denoted by $V_{\varepsilon r}(\bm z)$ for $\bm z\in (\varepsilon r)\mathds{Z}^d$. In particular, let $Z_0\subset (\varepsilon r)\mathds{Z}^d$ be the collection of all $\bm z$ such that $V_{\varepsilon r}(\bm z) \cap V_r({\bm 0})  \neq \emptyset$. For convenience, let $\bm z_1,\cdots, \bm z_{\varepsilon^{-d}}$ be the elements of $Z_0$ listed in the dictionary order with respect to the centers of the cubes.

We now choose a sufficiently large integer $K$ ($\gg 1/\alpha$), which will be determined in Proposition \ref{Prop4.5DG21} below.
We then divide each $V_{3\varepsilon r}(\bm z_k)$ into $(3K)^d$ small cubes of side length $\varepsilon r/K$. We denote by $J_{k,i}$ for $i\in[1,(3K)^d]_\mathds{Z}$ the $(3K)^d$ small cubes in the dictionary order with respect to the centers of the cubes. 
Cubes of side length $\varepsilon r/K$ form our first scale (i.e., the smallest scale).
From Lemma \ref{interval-good1} we know that
for each $k\in[1,\varepsilon^{-d}]_{\mathds{Z}}$ and $i\in [1,(3K)^d]_{\mathds{Z}}$,
\begin{equation}\label{prob-good}
\mathds{P}[\text{the cube }J_{k,i}\ \text{is }(\varepsilon r/K,\alpha,\eta)\text{-nice}]\geq 3p/4.
\end{equation}

We now introduce the second scale, whose side length is $(b_2\alpha)^{2.5}\varepsilon r$. Here $b_2$ is a small positive constant depending only on $\beta,d,\alpha$ and the laws of $D$ and $\widetilde{D}$, which will be chosen below.
For convenience, we assume that $(b_2\alpha)^{-0.1}\in \mathds{Z}$; if not,  we can use $\lfloor(b_2\alpha)^{-0.1}\rfloor$ instead.
Specifically, for each $k\in[1,\varepsilon^{-d}]_{\mathds{Z}}$, we divide $V_{3\varepsilon r}(\bm z_k)$ into $3^d(b_2\alpha)^{-2.5d}$ small cubes of side length $(b_2\alpha)^{2.5}\varepsilon r$. We denote by $J'_{k,i}$ for $i\in[1,3^d(b_2\alpha)^{-2.5d}]_\mathds{Z}$ these $3^d(b_2\alpha)^{-2.5d}$ small cubes in the dictionary order with respect to the centers of the cubes. Then we denote $\bm w'_{k,i}$ as the center of $J'_{k,i}$ and we also denote the cube $J'_{k,i}$ as $J'(\bm w'_{k,i})$. 
It is clear that each $J'(\bm w'_{k,i})$ is composed of $((b_2\alpha)^{2.5}K)^d$ small cubes at the first scale.

\begin{definition}\label{very nice}
For each $k\in[1,\varepsilon^{-d}]_{\mathds{Z}}$, $i\in[1,3^d(b_2\alpha)^{-2.5d}]_\mathds{Z}$ and each $\bm w'_{k,i}\in V_{3\varepsilon r}(\bm z_k)$,
we say the cube $J'(\bm w'_{k,i})$ (with center $\bm w'_{k,i}$) is \textit{$(b_2\alpha)^{2.5}\varepsilon r$-very nice} with respect to $\mathcal{E}$ if
\begin{itemize}
\item[(1)] there exists at least one $(\varepsilon r/K,\alpha,\eta)$-nice cube with respect to $\mathcal{E}$ in $V_{(b_2\alpha)^{ 2.6}\varepsilon r/2}(\bm w'_{k,i})$;

\item[(2)]  $D(V_{(b_2\alpha)^{ 2.6}\varepsilon r/2}(\bm w'_{k,i}), (J'(\bm w'_{k,i}))^c)\geq (b_2\alpha)^{2.6\theta}(\varepsilon r)^\theta.$
\end{itemize}
\end{definition}

Note that Definition \ref{very nice} (2) is not a local condition, as there might be a long edge in some $D$-geodesic from $V_{(b_2\alpha)^{ 2.6}\varepsilon r/2}(\bm w'_{k,i})$ to $(J'(\bm w'_{k,i}))^c$, with one of its end points located in $(J'(\bm w'_{k,i}))^c$ (and as a result incurs correlation with the very nice property of the $J'$ cube (i.e., a cube of the form $J'_{k, i}$) that contains this end point).
This will result in some difficulty on calculating the probability of a cube being very nice later on. To overcome this difficulty, we will further consider very nice cubes with respect to different directions. As we will see later, on the one hand, for any fixed direction, the property of being very nice becomes independent over different cubes; on the other hand, a cube is very nice if and only if it is very nice with respect to all directions. To be precise, we start with some new notations.

\begin{definition}\label{e-direction}
For each $\bm e\in \{-1,1\}^d$, $s>0$ and $\bm z\in s\mathds{Z}^d$, we say
$$
\bigcup_{\bm y\in s\mathds{Z}^d\setminus\{\bm z\}:(\bm y-\bm z)^j\bm e^j\geq 0\text{ for all }j}V_s(\bm y)
$$
is the complement of $V_s(\bm z)$ in the direction ${\bm e}$. Here $\bm x^j$ is the $j$-th component of $\bm x\in\mathds{R}^d$.
Additionally, we will write $\mathtt{J}'_{(\bm e)}(\bm w'_{k,i})$ as the complement of $J'(\bm w'_{k,i})$ in the direction $\bm e$ and $\mathtt{V}_{(\bm e)}(\bm z_k)$ as the complement of $V_{\varepsilon r}(\bm z_k)$ in the direction $\bm e$.
\end{definition}

We note that the complements of $J'(\bm w'_{k,i})$ in different directions may intersect each other (see Figure \ref{example-intersect} as an example).
\begin{figure}[htbp]
    \centering
    \includegraphics[scale=0.6]{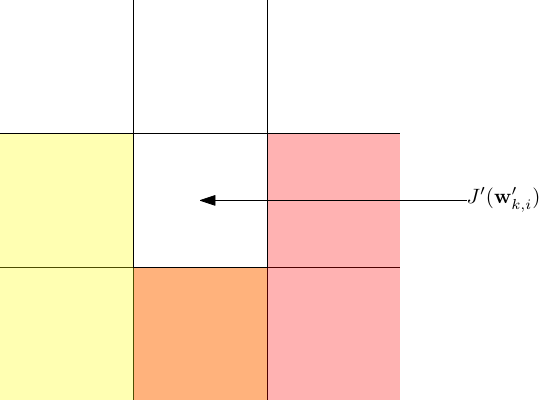}
    \caption{An example in the two-dimensional case where the complements of $V_s(\bm z)$ in two directions intersect each other. The yellow and red regions represent $\mathtt{J}'_{(-1,-1)}(\bm w'_{k,i})$ and $\mathtt{J}'_{(1,-1)}(\bm w'_{k,i})$ respectively. The orange region is the intersection of  $\mathtt{J}'_{(-1,-1)}(\bm w'_{k,i})$ and $\mathtt{J}'_{(1,-1)}(\bm w'_{k,i})$ }
    \label{example-intersect}
\end{figure}

We now can provide the definition of very nice cubes in terms of some direction.

\begin{definition}\label{very nice-e}
For each $\bm e\in \{-1,1\}^d$, $k\in [1,\varepsilon^{-d}]_{\mathds{Z}}$ and $i\in[1,3^d(b_2\alpha)^{-2.5d}]_\mathds{Z}$, we say the cube $J'(\bm w'_{k,i})$ is \textit{$(\bm e,(b_2\alpha)^{2.5}\varepsilon r)$-very nice} with respect to $\mathcal{E}$ if
\begin{itemize}
\item[(1)] there exists at least one $(\varepsilon r/K,\alpha,\eta)$-nice cube with respect to $\mathcal{E}$ in $V_{(b_2\alpha)^{ 2.6}\varepsilon r/2}(\bm w'_{k,i})$;

\item[(2)]  $D(V_{(b_2\alpha)^{ 2.6}\varepsilon r/2}(\bm w'_{k,i}), \mathtt{J}'_{(\bm e)}(\bm w'_{k,i});J'(\bm w'_{k,i})\cup \mathtt{J}'_{(\bm e)}(\bm w'_{k,i}))\geq (b_2\alpha)^{2.6\theta}(\varepsilon r)^\theta.$
\end{itemize}
\end{definition}

As mentioned earlier, we note that $J'(\bm w'_{k,i})$ is $(b_2\alpha)^{2.5}\varepsilon r$-very nice if and only if $J'(\bm w'_{k,i})$ is $(\bm e,(b_2\alpha)^{2.5}\varepsilon r)$-very nice for all $\bm e\in\{-1,1\}^d$ as follows.
\begin{lemma}\label{veryniceequ}
    For each $k\in [1,\varepsilon^{-d}]_{\mathds{Z}}$ and $i\in[1,3^d(b_2\alpha)^{-2.5d}]_\mathds{Z}$, the cube $J'(\bm w'_{k,i})$ is $(b_2\alpha)^{2.5}\varepsilon r$-very nice if and only if $J'(\bm w'_{k,i})$ is $(\bm e,(b_2\alpha)^{2.5}\varepsilon r)$-very nice for all $\bm e\in\{-1,1\}^d$.
\end{lemma}
\begin{proof}
    From Definitions \ref{very nice} and \ref{very nice-e}, it suffices to show that
    \begin{equation}\label{compare-e}
        D(V_{(b_2\alpha)^{ 2.6}\varepsilon r/2}(\bm w'_{k,i}), (J'(\bm w'_{k,i}))^c)=\inf_{\bm e\in\{-1,1\}^d}D(V_{(b_2\alpha)^{ 2.6}\varepsilon r/2}(\bm w'_{k,i}), \mathtt{J}'_{(\bm e)}(\bm w'_{k,i});J'(\bm w'_{k,i})\cup \mathtt{J}'_{(\bm e)}(\bm w'_{k,i})).
    \end{equation}
    Assume that $P$ is a $D$-geodesic from $V_{(b_2\alpha)^{ 2.6}\varepsilon r/2}(\bm w'_{k,i})$ to $(J'(\bm w'_{k,i}))^c$ and uses a long edge $\langle \bm x,\bm y\rangle$ to escape $J'(\bm w'_{k,i})$. Without loss of generality, we let $\bm x\in J'(\bm w'_{k,i})$ and $\bm y \in (J'(\bm w'_{k,i}))^c$. Then there exists $\bm e\in\{-1,1\}^d$ such that $\bm y\in \mathtt{J}'_{(\bm e)}(\bm w'_{k,i})$, which implies that $P\subset J'(\bm w'_{k,i})\cup \mathtt{J}'_{(\bm e)}(\bm w'_{k,i})$. Thus we get that
    \begin{equation}\label{compare-e-1}
        \begin{aligned}
            &D(V_{(b_2\alpha)^{ 2.6}\varepsilon r/2}(\bm w'_{k,i}), (J'(\bm w'_{k,i}))^c)={\rm len}(P;D)\\
            \geq &D(V_{(b_2\alpha)^{2.6}\varepsilon r/2}(\bm w'_{k,i}), \mathtt{J}'_{(\bm e)}(\bm w'_{k,i});J'(\bm w'_{k,i})\cup \mathtt{J}'_{(\bm e)}(\bm w'_{k,i}))\\
            \geq &\inf_{\bm e\in\{-1,1\}^d}D(V_{(b_2\alpha)^{ 2.6}\varepsilon r/2}(\bm w'_{k,i}), \mathtt{J}'_{(\bm e)}(\bm w'_{k,i});J'(\bm w'_{k,i})\cup \mathtt{J}'_{(\bm e)}(\bm w'_{k,i})).
        \end{aligned}
    \end{equation}
    Additionally, since $\mathtt{J}'_{(\bm e)}(\bm w'_{k,i})\subset (J'(\bm w'_{k,i}))^c$ for any $\bm e\in\{-1,1\}^d$, we get that
    \begin{equation}\label{compare-e-2}
        \begin{aligned}
            &\inf_{\bm e\in\{-1,1\}^d}D(V_{(b_2\alpha)^{ 2.6}\varepsilon r/2}(\bm w'_{k,i}), \mathtt{J}'_{(\bm e)}(\bm w'_{k,i});J'(\bm w'_{k,i})\cup \mathtt{J}'_{(\bm e)}(\bm w'_{k,i}))\\
            \geq &D(V_{(b_2\alpha)^{ 2.6}\varepsilon r/2}(\bm w'_{k,i}), (J'(\bm w'_{k,i}))^c).
        \end{aligned}
    \end{equation}
    Combining \eqref{compare-e-1} and \eqref{compare-e-2} yields \eqref{compare-e}, which implies the desired result.
\end{proof}

Additionally, due to the division of directions, for each $\bm e\in\{-1,1\}^d$,
\begin{equation}\label{Jindep}
\left\{J'(\bm w'_{k,i})\ \text{is }(\bm e,(b_2\alpha)^{2.5}\varepsilon r)\text{-very nice}\right\}_{i\in [1,3^d(b_2\alpha)^{-2.5d}]_\mathds{Z}}
\end{equation}
is a collection of independent events and a.s.\ determined by $\mathcal{E}\cap (V_{\varepsilon r}(\bm z_k)\times (V_{\varepsilon r}(\bm z_k)\cup \mathtt{V}_{(\bm e)}(\bm z_k)))$, where $\mathtt{V}_{(\bm e)}(\bm z_k)$ is the  complement of $V_{\varepsilon r}(\bm z_k)$ in the direction $\bm e$ defined as in Definition \ref{e-direction}.

For fixed $\bm e\in \{-1,1\}^d$, $k\in [1,\varepsilon^{-d}]_{\mathds{Z}}$ and $i\in[1,3^d(b_2\alpha)^{-2.5d}]_\mathds{Z}$, we claim that there is a constant $\alpha_1=\alpha_1(\beta,D)>0$ (depending only on $\beta,d$ and the law of $D$) such that for each $\alpha\in(0,\alpha_1]$,
\begin{equation}\label{verynice(3)}
\mathds{P}[\text{Definition \ref{very nice-e} (2) holds}]\geq 1-4^{-d}/2.
\end{equation}
Indeed, from Axiom V1' (tightness across different scales (lower bound)), there exists a constant $\widetilde{c}_1$ (depending only on $\beta,d$ and the law of $D$) such that
\begin{equation}\label{verynice(3)step1}
    \mathds{P}\left[D(\bm w'_{k,i}, \mathtt{J}'_{(\bm e)}(\bm w'_{k.i}))\geq \widetilde{c}_1(b_2\alpha)^{2.5\theta}(\varepsilon r)^\theta\right]\geq 1-4^{-d-1}.
\end{equation}
In addition, from Axiom V2' (tightness across different scales (upper bound)), there exists a constant $\widetilde{c}_2>0$ (depending only on $\beta,d$ and the law of $D$) such that
\begin{equation}\label{verynice(3)step2}
    \mathds{P}\left[{\rm diam}(V_{(b_2\alpha)^{ 2.6}\varepsilon r/2}(\bm w'_{k,i});D)\leq \widetilde{c}_2(b_2\alpha)^{2.6\theta}(\varepsilon r)^\theta\right]\geq 1-4^{-d-1}.
\end{equation}
Thus, combining \eqref{verynice(3)step1}, \eqref{verynice(3)step2} with the following triangle inequality
\begin{equation*}
    \begin{split}
        &D\left(V_{(b_2\alpha)^{ 2.6}\varepsilon r/2}(\bm w'_{k,i}), \mathtt{J}'_{(\bm e)}(\bm w'_{k,i})\right)\\
        &\geq D(\bm w'_{k,i}, \mathtt{J}'_{(\bm e)}(\bm w'_{k,i}))-{\rm diam}(V_{(b_2\alpha)^{ 2.6}\varepsilon r/2}(\bm w'_{k,i});D),
    \end{split}
\end{equation*}
we get \eqref{verynice(3)} when
$$\widetilde{c}_1(b_2\alpha)^{2.5\theta}\geq (\widetilde{c}_2+1)(b_2\alpha)^{2.6\theta}.$$
Thus, \eqref{verynice(3)} holds with the choice that $\alpha_1=(\frac{\widetilde{c}_1}{\widetilde{c}_2+1})^{10/\theta}>0$ .
Meanwhile, from the independence (implied by Axiom II (locality)) and \eqref{prob-good}, one can see that
\begin{equation*} 
    \mathds{P}[\text{Definition \ref{very nice-e} (1) holds}]\geq 1-(1- 3p/4)^{(b_2\alpha)^{2.6d}K^d/2}.
    \end{equation*}
Therefore, we can select a sufficiently large value $K>0$ (depending only on $\beta,d,\alpha$ and the law of $D$) such that
\begin{equation*}
\mathds{P}[\text{Definition \ref{very nice-e} (1) holds}]\geq 1-4^{-d}/2.
\end{equation*}
Combined this with \eqref{verynice(3)}, this yields that
\begin{equation}\label{probverynice}
\mathds{P}[J'(\bm w'_{k,i})\text{ is $(\bm e,(b_2\alpha)^{2.5}\varepsilon r)$-very nice}]\geq  1-4^{-d}.
\end{equation}

We next define the third scale whose side length is $(b_2\alpha)^2\varepsilon r$. For each $k\in[1,\varepsilon^{-d}]_\mathds{Z}$, we divide $V_{3\varepsilon r}(\bm z_k)$ into $3^d(b_2\alpha)^{-2d}$ small cubes of side length $(b_2\alpha)^2\varepsilon r$, denoted by $J''(\bm w''_{k,i}):=V_{(b_2\alpha)^2\varepsilon r}(\bm w''_{k,i})$ for $i\in [1,3^d(b_2\alpha)^{-2d}]_\mathds{Z}$. We can see that each $J''(\bm w''_{k,i})$ is composed of $(b_2\alpha)^{-0.5d}$ small cubes at the second scale.

We also need the following definition of very very nice cube in terms of some direction.

\begin{definition}\label{very very nice-e}
For each $\bm e\in\{-1,1\}^d$, $k\in[1,\varepsilon^{-d}]_\mathds{Z}$ and $i\in[1,3^d(b_2\alpha)^{-2d}]_\mathds{Z}$, we say $J''(\bm w''_{k,i})$ is \textit{$(\bm e,(b_2\alpha)^{2}\varepsilon r)$-very very nice}   with respect to $\mathcal{E}$ if the number of $(\bm e,(b_2\alpha)^{2.5}\varepsilon r)$-very nice cubes  with respect to $\mathcal{E}$ in $J''(\bm w''_{k,i})$ is at least $(1-3^{-d})(b_2\alpha)^{-0.5d}$.
\end{definition}

\begin{definition}\label{very very nice}
For each $k\in[1,\varepsilon^{-d}]_\mathds{Z}$ and $i\in[1,3^d(b_2\alpha)^{-2d}]_\mathds{Z}$, we say $J''(\bm w''_{k,i})$ is \textit{$(b_2\alpha)^{2}\varepsilon r$-very very nice}   with respect to $\mathcal{E}$ if there exists at least one $(b_2\alpha)^{2.5}\varepsilon r$-very nice cube  with respect to $\mathcal{E}$ in $J''(\bm w''_{k,i})$.
\end{definition}

Now we will show that $J''(\bm w''_{k,i})$ is $(\bm e,(b_2\alpha)^{2}\varepsilon r)$-very very nice with very high probability for each $\bm e\in\{-1,1\}^d$.
To be precise, the following lemma holds.
\begin{lemma}\label{Lemmaveryverynice-e}
    There exists a constant $p_\star\in(0,1)$ 
    such that for any sufficiently small $\alpha$ {\rm(}depending only on $\beta,d$ and the law of $D${\rm)}, $\bm e\in\{-1,1\}^d$, $k\in[1,\varepsilon^{-d}]_\mathds{Z}$ and $i\in[1,3^d(b_2\alpha)^{-2d}]_\mathds{Z}$,
    \begin{equation*}
        \mathds{P}[J''(\bm w''_{k,i})\text{ is $(\bm e,(b_2\alpha)^{2}\varepsilon r)$-very very nice  with respect to $\mathcal{E}$}]\geq 1- (1-p_\star)^{(b_2\alpha)^{-0.5d}}.
    \end{equation*}
\end{lemma}

\begin{proof}
By Axiom IV'(translation invariance), we only need to prove the statement in the case that $k=i=1$. For convenience, let $N(\bm e)$ be the number of $(\bm e,(b_2\alpha)^{2.5}\varepsilon r)$-very nice cubes  with respect to $\mathcal{E}$ in $J''(\bm w''_{1,1})$.

As we mentioned in \eqref{Jindep}, the events of the form \{$J'(\bm w'_{1,i})$ is $(\bm e,(b_2\alpha)^{2.5}\varepsilon r)$-very nice\} are independent for all $i\in[1,(b_2\alpha)^{-0.5d}]_\mathds{Z}$ (note that $\bm w'_{1, i} \in J''(\bm w''_{1, 1})$ for all such $i$). Hence, combining this with \eqref{probverynice} and the fact that each $J''(\bm w''_{1,1})$ is composed of $(b_2\alpha)^{-0.5d}$ small cubes at the second scale, we get that for each $\bm e\in\{-1,1\}^d$,
    \begin{equation*}
        \begin{split}
        &\mathds{P}[J''(\bm w''_{1,1})\text{ is $(\bm e,(b_2\alpha)^{2}\varepsilon r)$-very very nice with respect to $\mathcal{E}$}]\\
        &=\mathds{P}[N(\bm e)\geq (1-3^{-d})(b_2\alpha)^{-0.5d}]\\
        &\geq 1-\mathds{P}[{ \mathrm{Bin}}((b_2\alpha)^{-0.5d},1-4^{-d})<(1-3^{-d})(b_2\alpha)^{-0.5d}]\\
        &\geq 1-(1-p_\star)^{(b_2\alpha)^{-0.5d}}
        \end{split}
    \end{equation*}
     for a constant $p_\star\in(0,1)$.
\end{proof}

Recall that from Lemma \ref{veryniceequ} a cube is $(b_2\alpha)^{2.5}\varepsilon r$-very nice if and only if it is $(\bm e,(b_2\alpha)^{2.5}\varepsilon r)$-very nice for all $\bm e\in\{-1,1\}^d$, and
we say $J''(\bm w''_{k,i})$ is \textit{$(b_2\alpha)^{2}\varepsilon r$-very very nice}  with respect to $\mathcal{E}$ if there exists at least one $(b_2\alpha)^{2.5}\varepsilon r$-very nice cube   with respect to $\mathcal{E}$ in $J''(\bm w''_{k,i})$ (see Definition \ref{very very nice}).

\begin{lemma}\label{Lemmaveryverynice}
For any sufficiently small $\alpha$ {\rm(}depending only on $\beta,d$ and the law of $D${\rm)}, $k\in[1,\varepsilon^{-d}]_\mathds{Z}$ and $i\in[1,3^d(b_2\alpha)^{-2d}]_\mathds{Z}$, we have
    \begin{equation*}
        \mathds{P}[J''(\bm w''_{k,i})\text{ is $(b_2\alpha)^{2}\varepsilon r$-very very nice  with respect to $\mathcal{E}$}]\geq 1- 2^d(1-p_\star)^{(b_2\alpha)^{-0.5d}},
    \end{equation*}
    where $p_\star$ is the constant in Lemma \ref{Lemmaveryverynice-e}.
\end{lemma}

\begin{proof}
By Axiom IV' (translation invariance), we also only need to prove the statement in the case that $k=i=1$.
Let $N(\bm e)$ be the number as defined in the proof of Lemma \ref{Lemmaveryverynice-e} for $\bm e\in\{-1,1\}^d$. So the event that $A(\bm e):=\{N(\bm e)\geq (1-3^{-d})(b_2\alpha)^{-0.5d}\}$ is equivalent to the event that $J''(\bm w''_{1,1})$  is $({\bm e},(b_2\alpha)^{2}\varepsilon r)$-very very nice with respect to $\mathcal{E}$. Then from Lemma \ref{Lemmaveryverynice-e} we have that
\begin{equation}\label{PA}
\mathds{P}[A(\bm e)]\geq  1- (1-p_\star)^{(b_2\alpha)^{-0.5d}}.
\end{equation}
Additionally, on the event $\cap_{\bm e\in\{-1,1\}^d}A(\bm e)$, we can see that there must exist at least $(1-2^d\cdot3^{-d})(b_2\alpha)^{-0.5d}\geq 1$ cubes (of side length $(b_2\alpha)^{2.5}\varepsilon r$) in $J''(\bm w''_{1,1})$ which are $(\bm e,(b_2\alpha)^{2.5}\varepsilon r)$-very nice for all ${\bm e}\in\{-1,1\}^d$, i.e., being $(b_2\alpha)^{2.5}\varepsilon r$-very nice. Thus,
  \begin{equation*}
        \begin{split}
        &\mathds{P}[J''(\bm w''_{1,1})\text{ is $(b_2\alpha)^{2}\varepsilon r$-very very nice  with respect to $\mathcal{E}$}]\geq \mathds{P}\left[\cap_{\bm e\in\{-1,1\}^d}A(\bm e)\right]\\
        &\geq 1-2^d(1-p_\star)^{(b_2\alpha)^{-0.5d}},
        \end{split}
    \end{equation*}
    where $p_\star$ is the constant in Lemma \ref{Lemmaveryverynice-e}.
\end{proof}

Next, we define the fourth scale as follows.

\begin{definition}\label{super good}

For any $k\in[1,\varepsilon^{-d}]_{\mathds{Z}}$, we say $V_{3\varepsilon r}(\bm z_k)$ is {\it$(3\varepsilon r,\alpha,\eta)$-super good} (resp. {\it$(\bm e, 3\varepsilon r,\alpha,\eta)$-super good}) with respect to $\mathcal{E}$ if
\begin{enumerate}
    \item  For each $i\in[1,3^d(b_2\alpha)^{-2d}]_\mathds{Z}$, $J''(\bm w''_{k,i})$ is $(b_2\alpha)^{2}\varepsilon r$-very very nice (resp. $(\bm e, (b_2\alpha)^{2}\varepsilon r)$-very very nice) with respect to $\mathcal{E}$;

    \item  There is a constant $\widetilde{C}=\widetilde{C}(\alpha)$ depending only on $\beta,d,\alpha$  and the laws of $D$ and $\widetilde{D}$ such that for all $\bm u,\bm v\in V_{3\varepsilon r}(\bm z_k)$,
         $$D(\bm u,\bm v ;V_{3\varepsilon r}(\bm z_k))\le \widetilde{C}\|\bm u-\bm v\|_\infty^\theta\log\frac{2\varepsilon r}{\|\bm u-\bm v\|_\infty}.$$

\end{enumerate}
\end{definition}


From Proposition  \ref{LimitContinuousTail}, we also have that there exists a constant $\widetilde{C}=\widetilde{C}(\alpha)$ (depending only on $\alpha,\beta,d$ and the law of $D$),
\begin{equation}\label{supergood(2)}
\mathds{P}[\text{Definition \ref{super good} (2) holds}]\geq 1-\alpha.
\end{equation}
In addition, by \eqref{PA} and a union bound (over $i$), one has that
\begin{equation}\label{esupergood-1}
\mathds{P}[\text{Definition \ref{super good} (1) with direction $\bm e$ holds in  $V_{3\varepsilon r}(\bm z_k)$}]
\geq 1-3^d(b_2\alpha)^{-2d} (1-p_\star)^{(b_2\alpha)^{-0.5d}}.
\end{equation}
It is worth noting that, as $\alpha\to 0$, the probability can be arbitrarily close to 1. Let $p_c':=p_c\vee(1-(4C_{dis})^{-8\cdot 6^d})$, where $p_c$ is the constant defined in Lemma \ref{hd-BK} with $\delta=1/(4C_{dis})$
(here $C_{dis}$ is the constant defined in Lemma \ref{number-path-k}).
Combining \eqref{esupergood-1} with \eqref{supergood(2)}, we can find sufficiently small $\alpha_2\in(0,(1-p'_c)/2)$ and sufficiently large $K>0$ such that for each $\bm e\in\{-1,1\}^d$, $\alpha\in(0,\alpha_2]$ and each $k\in [1,\varepsilon^{-d}]_\mathds{Z}$,

\begin{equation}\label{supergoodcriticalprobability}
\begin{split}
    &\mathds{P}[V_{3\varepsilon r}(\bm z_k)\text{ is $(\bm e,3\varepsilon r,\alpha,\eta)$-super good with respect to }\mathcal{E}]\\
    &\geq \mathds{P}[\text{Definition \ref{super good} (1) holds}]-\frac{1-p'_c}{2} >p'_c.
    \end{split}
\end{equation}
Additionally, let $\mathcal{Z}_\varepsilon$ be the collection of all subsets $Z\subset[1,\varepsilon^{-d}]_{\mathds{Z}}$ such that $V_{3\varepsilon r}(\bm z_k)$'s for $k\in Z$ are disjoint.
From \eqref{Jindep} we can see that for each $Z\in\mathcal{Z}_\varepsilon$ and each $\bm e\in\{-1,1\}^d$,
\begin{equation}\label{supergoodindep}
\{V_{3\varepsilon r}(\bm z_k)\text{ is $(\bm e,3\varepsilon r,\alpha,\eta)$-super good} \}_{k\in Z}
\end{equation}
is a collection of independent events. Furthermore, the event that $V_{3\varepsilon r}(\bm z_k)$ is $(\bm e,3\varepsilon r,\alpha,\eta)$-super good is a.s.\ determined by $\mathcal{E}\cap(V_{3\varepsilon r}(\bm z_k)\times (V_{3\varepsilon r}(\bm z_k)\cup  \mathtt{V}_{(\bm e)}(\bm z_k)))$ for each $k\in Z$.

We finally introduce the definition of super super good cubes, which is a slightly stronger version of super good cubes.

\begin{definition}\label{superdouble}
For any $k\in[1,\varepsilon^{-d}]_{\mathds{Z}}$, we say $V_{3\varepsilon r}(\bm z_k)$ is {\it$(3\varepsilon r,\alpha,\eta)$-super super good}  with respect to $\mathcal{E}$ if
\begin{itemize}
\item[(1)] $V_{3\varepsilon r}(\bm z_k)$ is $(3\varepsilon r,\alpha,\eta)$-super good with respect to $\mathcal{E}$;
 \item[(2)]   For any two different edges $\langle {\bm u}_1,{\bm v}_1\rangle$ and $\langle {\bm u}_2,{\bm v}_2\rangle\in\mathcal{E}$, with $\bm u_1\in V_{\varepsilon r}(\bm z_k)^c$, $\bm v_1\in V_{\varepsilon r}(\bm z_k)$, $\bm u_2\in V_{3\varepsilon r}(\bm z_k)$ and $\bm v_2\in  V_{3\varepsilon r}(\bm z_k)^c$, we have that $|\bm v_1-\bm u_2|\geq \alpha \varepsilon r$ and that
    $$
    D(\bm v_1,\bm u_2;V_{3\varepsilon r}(\bm z_k))\geq  (b_2\alpha \varepsilon r)^\theta.
    $$
    \end{itemize}
\end{definition}

As we have shown in Lemma \ref{super good (ii)}, there is a sufficiently small $\alpha_3>0$ such that  for each $\alpha\in(0,\alpha_3]$, there is a constant $C>0$ (depending only on $\beta,d$ and the law of $D$) such that
\begin{equation}\label{superdouble(2)}
\mathds{P}[\text{Definition \ref{superdouble} (2) holds}]\geq 1-C\alpha\log(\alpha^{-1}).
\end{equation}

Thus combining \eqref{superdouble(2)} with \eqref{supergoodcriticalprobability} we can see that for sufficiently small $\alpha>0$,
\begin{equation}\label{porbsupersuper}
\mathds{P}[V_{3\varepsilon r}(\bm z_k)\text{ is }(3\varepsilon r,\alpha,\eta)\text{-super super good with respect to }\mathcal{E}]\geq p_c.
\end{equation}

Intuitively, if we say $\bm z_k$ is super super good when $V_{3\varepsilon r}(\bm z_k)$ is $(3\varepsilon r,\alpha,\eta)$-super super good with respect to $\mathcal{E}$, then in light of Lemma \ref{hd-BK}, it is natural to suspect from \eqref{porbsupersuper} that the skeleton path of any path $P$ in the continuous model (whose length is not too short) hits enough such super super good points. That is, the  path $P$ hits enough super super good cubes.
Furthermore, based on the above definitions, there exist numerous $(\varepsilon r/K,\alpha,\eta)$-nice cubes within each super good cube and thus also within each super super good cube.
We hope that when a path $P$ hits some super super good cubes and spends enough time inside them (guaranteed by Definition \ref{superdouble} (2)), it will pass through some $(\varepsilon r/K,\alpha,\eta)$-great pairs of small cubes within these super super good cubes. This will enable us to utilize Definition \ref{interval-good} (2) to obtain certain estimates.

It is worth emphasizing that all the definitions and estimates mentioned above depend on the realization or on the law for the edge set $\mathcal{E}$.
In the remainder of this section, when we refer a cube as a super super good cube with respect to another edge set $\widetilde{\mathcal{E}}$, it means that we simply replace $\mathcal{E}$ with $\widetilde{\mathcal{E}}$ in Definitions \ref{interval-good}--\ref{superdouble} (here we mean Definitions \ref{interval-good}, \ref{very nice}, \ref{very nice-e}, \ref{very very nice-e}, \ref{very very nice}, \ref{super good} and \ref{superdouble}), and accordingly  the metrics change with the variation of the edge set.

Throughout the rest of this section,  we will also utilize renormalization, similar to the one used in the proof of Proposition \ref{(B)=>G}. Specifically,  we will identify $V_{\varepsilon r}(\bm z_k)$ to the vertex $\bm z_k$, resulting in a graph also denoted as $G$. Thanks to the model's self-similarity,  $G$ can be  viewed as the  critical long-range bond percolation model (see the paragraph before \eqref{connectprob}).

In the remainder of this subsection, we will discuss the selection of parameters used in this section.
Recall that $c'=\frac{c_*+C_*}{2}$ and choose the parameters as follows. We first select $p$, and then choose
 \begin{equation*} 
 \alpha<\min\{\alpha_0,\alpha_1,\alpha_2,\alpha_3\}
  \end{equation*}
  sufficiently small and then choose the corresponding $c''$ (see Proposition \ref{mrinS3-tilde} with $c'=(c_*+C_*)/2$), $b_1,b_2,\widetilde{C}$ (first choose $b_1>0$ and $\widetilde{C}>1$ and then choose a small $b_2>0$) such that \eqref{supergoodcriticalprobability} and the following two inequalities hold:
\begin{align}
    \label{M0new}
    &(b_2\alpha)^\theta> 20\cdot 2^\theta\widetilde{C}^2 (b_2\alpha)^{1.4\theta}(\log (2(b_2\alpha)^{-2}))^2,\\
    \label{b2alpha}
    &(b_2\alpha )^\theta-2^\theta \widetilde{C}(b_2\alpha)^{2\theta}\log(1/2(b_2\alpha)^2)>(b_2\alpha)^{2.6\theta}.
\end{align}
 For convenience, we assume that $(b_2\alpha)^{-0.5}\in\mathds{Z}$ in the rest of paper.

After that, we choose $\eta<\eta_0$ such that $f(t)=t^\theta\log\frac{2}{t}$ is increasing on $[0,\eta]$ and that
\begin{equation}\label{eta0}
    C_1\eta^\theta\log\frac{2}{\eta}<\frac{C_*-c'}{8C_*}b_1\alpha^\theta.
\end{equation}
 Finally we choose a sufficiently large $K>0$ such that \eqref{probverynice} and the following relation holds:
\begin{equation}\label{M0}
    4\cdot3^dK^{-\theta}\log K< \min\left\{(b_2\alpha)^{2\theta+2d}\log((b_2\alpha)^{-2}),\widetilde{C}^{-1}(b_2\alpha)^{2.6\theta}\right\}.
\end{equation}

\subsection{Add and delete edges}\label{add-delete}
In this subsection, we introduce some notations related to adding or deleting edges, which will be useful in our counting arguments.

For simplicity, unless otherwise stated, we classify cubes as \textit{nice, very nice {\rm(resp. \textit{$\bm e$-very nice})}, very very nice {\rm (resp. \textit{$\bm e$-very very nice})}, super good {\rm(resp. \textit{$\bm e$-super good})} or super super good}
if they satisfy the criteria for being  $(\varepsilon r/K,\alpha,\eta)$-nice, $(b_2\alpha)^{2.5}\varepsilon r$-very nice (resp. $(\bm e,b_2\alpha)^{2.5}\varepsilon r$-very nice), $(b_2\alpha)^{2}\varepsilon r$- very very nice (resp. $(\bm e,b_2\alpha)^{2}\varepsilon r$- very very nice),  $(3\varepsilon r,\alpha,\eta)$-super good (resp. $(\bm e,3\varepsilon r,\alpha,\eta)$-super good) or $(3\varepsilon r,\alpha,\eta)$-super super good  with respect to $\mathcal{E}$, respectively.
Recall that for each $k\in [1,\varepsilon^{-d}]_\mathds{Z}$ we have that $V_{3\varepsilon r}(\bm z_k)$ is a union of small cubes of side length $(b_2 \alpha)^2 \varepsilon r$, i.e., $V_{3\varepsilon r}(\bm z_k)  = \cup_{i=1}^{3^d (b_2 \alpha)^{-2d}} J''(\bm w''_{k, i})$. Without loss of generality, we assume that $J''(\bm w''_{k,i})$ and $J''(\bm w''_{k,i+1})$ are adjacent for all $i\in [1,3^d(b_2\alpha)^{-2d})_\mathds{Z}$.

It is worth noting that for $k\in [1,\varepsilon^{-d}]_\mathds{Z}$, if Definition \ref{super good} (1) holds for $V_{3\varepsilon r}(\bm z_k)$, then based on Definitions \ref{interval-good}--\ref{super good}, we can conclude that there are nice cubes of side length $\varepsilon r/K$ in each $J''(\bm w''_{k,i})$ for $i\in[1,3^d(b_2\alpha)^{-2d}]_\mathds{Z}$.
For each $k$ and $i$, we choose a nice cube in an arbitrary but prefixed manner, denoted by $J_{k,q(k,i)}$. Then we have chosen $3^d(b_2\alpha)^{-2d}$ nice cubes in $V_{3\varepsilon r}(\bm z_k)$, and they satisfy the regularity conditions in Definitions \ref{interval-good}--\ref{super good}. For convenience, we list these nice cubes as
\begin{equation}\label{choosenice}
J_{k,q(k,1)},\ \cdots,\ J_{k,q(k, {3^d(b_2\alpha)^{-2d}})}.
\end{equation}

Recall that $\mathcal{Z}_\varepsilon$ is the collection of all subsets $Z\subset[1,\varepsilon^{-d}]_\mathds{Z}$ such that $V_{3\varepsilon r}(\bm z_k)$'s for $k\in Z$ are disjoint. Additionally, let $\mathcal{W}_\varepsilon$ be the collection of all subsets $Z\subset[1,\varepsilon^{-d}]_\mathds{Z}$ such that $\bm z_k-\bm z_l\in(3\varepsilon r)\mathds{Z}^d$ for all $k,l\in Z$. It is clear that $\mathcal{W}_\varepsilon\subset\mathcal{Z}_\varepsilon$.
 Sample a $\widetilde{\beta}$-LRP $\widetilde{\mathcal{E}}$ independent of $\mathcal{E}$. Here $\widetilde{\beta}$ is a sufficiently large positive constant (depending only on $\beta,d,\alpha,K$ and $\eta$), which will be chosen in \eqref{Ek1prob-1} below.

\begin{definition}\label{add}
 (Add edges)
    For any $Z\in \mathcal{W}_\varepsilon$, we define $\mathcal{E}_Z^+$  as follows.
    \begin{enumerate}
        \item 
        For each $k\in Z$, if Definition \ref{super good} (1) of super good cubes holds, we choose the nice cubes defined in \eqref{choosenice}.
        Let $(J^{(1)}_{k,q(k,i)},J^{(2)}_{k,q(k,i)})$ be the great pair of small cubes in $J_{k,q(k,i)}$.
        Then set
         \begin{equation*}\label{D_k}
         \Omega_k=\Omega_k(\mathcal{E})=\bigcup_{i=1}^{3^d(b_2\alpha)^{-2d}-1}\left(J_{k,q(k,i)}^{(2)}\times J_{k,q(k,i+1)}^{(1)}\right).
         \end{equation*}
         Otherwise (if Definition  \ref{super good} (1) fails), we let $\Omega_k=\emptyset$. Finally,
         denote $\Omega_Z=\bigcup_{k\in Z}\Omega_k$ as the potential locations for us to add edges.

        \item We add edges in $\Omega_Z$ using $\widetilde{\mathcal{E}}$. To be precise, define $\mathcal{E}_Z^+=\mathcal{E}\cup(\widetilde{\mathcal{E}}\cap \Omega_Z)$ (see Figure \ref{AddEdge}).
    \end{enumerate}
\end{definition}
\begin{figure}[htbp]
\centering
\includegraphics[scale=0.6]{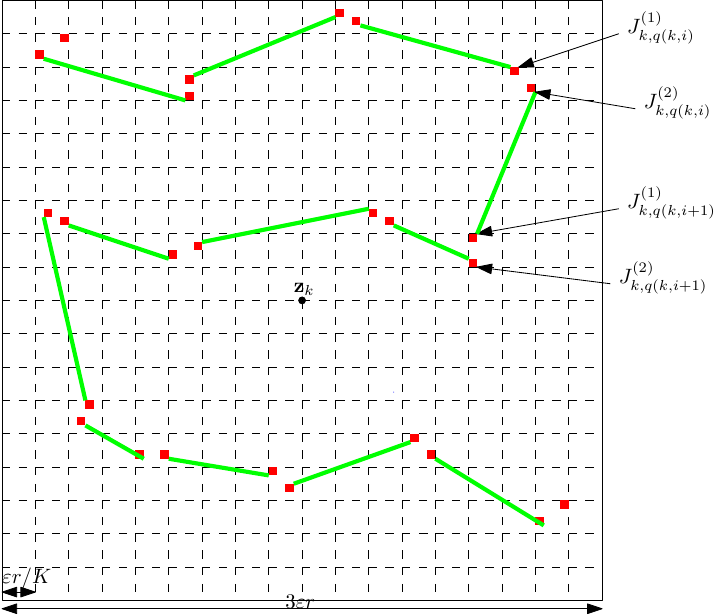}
\caption{Illustration for the definition of adding edges. The red cubes represent great pairs of cubes $(J^{(1)}_{k,q(k,i)},J^{(2)}_{k,q(k,i)})$, and the green lines indicate the edges that we have added.}
\label{AddEdge}
\end{figure}

Because $\mathcal{E}_Z^+$ does not have good local properties, we need  a new set of edges with good local properties to control it. To do this, for $k\in [1,\varepsilon^{-d}]_\mathds{Z} $, let
\begin{equation}\label{detLambda}
\Lambda_k =\left\{(\bm x,\bm y)\in V_{3\varepsilon r}(\bm z_k)\times V_{3\varepsilon r}(\bm z_k):|\bm x-\bm y|\geq (b_2\alpha)^{2.5}\varepsilon r/2\right\}
\end{equation}
be a deterministic region contained in $V_{3\varepsilon r}(\bm z_k)\times V_{3\varepsilon r}(\bm z_k)$. For $Z\in\mathcal{W}_\varepsilon$, define the edge set
\begin{equation}\label{addLambda}
\overline{\mathcal{E}}_Z^+=\mathcal{E}\cup(\widetilde{\mathcal{E}}\cap (\cup_{k\in Z}\Lambda_k)).
\end{equation}
We refer to discussions below \eqref{RNestimateEE+} for the motivation of introducing $\overline{\mathcal{E}}^+_Z$. In particular, we will write $\mathcal{E}_{\{k\}}^+$, $\overline{\mathcal{E}}_{\{k\}}^+$ as $\mathcal{E}_k^+$ and $\overline{\mathcal{E}}_k^+$ for simplicity.

For the laws of $\mathcal{E},\mathcal{E}_Z^+$ and $\overline{\mathcal{E}}_Z^+$, we have the following result.

\begin{lemma}\label{EE+}
For each $Z\in \mathcal{W}_\varepsilon$, the laws of $\mathcal{E}$ and $\mathcal{E}_Z^+$ are mutually absolutely continuous. Furthermore, let $\phi_Z$ be the Radon-Nikodym derivative of $\mathcal{E}$ with respect to $\mathcal{E}_Z^+$. Then there exists $M>0$ (which does not depend on $\varepsilon, r$) such that a.s.
$$
M^{-\#Z} \left(\frac{\beta}{\beta+\widetilde{\beta}}\right)^{|\mathcal{E}_Z^+\cap \Lambda_Z|} \leq \phi_Z(\mathcal{E}_Z^+) \leq  M^{\#Z}.
$$
\end{lemma}
\begin{proof}
For a fixed $Z\in \mathcal{W}_\varepsilon$, by Definition \ref{add} we see that for any positive measurable function $f:\mathcal{D}'\to\mathds{R}$,
\begin{equation}\label{E=E+}
\begin{split}
\mathds{E}[f(\mathcal{E}_Z^+)]&=\mathds{E}[\mathds{E}[f(\mathcal{E}_Z^+)|\mathcal{E}]]
\geq \mathds{E}[\mathds{E}[f(\mathcal{E})\I_{\{\widetilde{\mathcal{E}}\cap\Omega_Z=\emptyset\}}|\mathcal{E}]]
=\mathds{E}[f(\mathcal{E})]\mathds{E}\left[\mathds{P}[\widetilde{\mathcal{E}}\cap\Omega_Z=\emptyset|\mathcal{E}]\right].
\end{split}
\end{equation}
Moreover, note that $(J_{k,q(k,i)}^{(1)},J_{k,q(k,i)}^{(2)})$ lies in the ``center'' of $J'(\bm w'_{k,q(k,i)})$ (see Definition \ref{very nice} (1)). By this and the assumption that $J''(\bm w''_{k,i})$ and $J''(\bm w''_{k,i+1})$ are adjacent for all $i\in[1,3^d(b_2\alpha)^{-2d})_\mathds{Z}$, one can see that given $\mathcal{E}$,
$$
(b_2\alpha)^{2.5}\varepsilon r /2\le {\rm dist}(J_{k,q(k,i)}^{(2)}, J_{k,q(k,i+1)}^{(1)})\le 2d(b_2\alpha)^2 \varepsilon r.
$$
Combining this with the fact that  the side lengths of $J_{k,q(k,i)}^{(1)}$ and  $J_{k,q(k,i)}^{(2)}$ are both $\eta\varepsilon r/K$, we have that
\begin{equation}\label{DkRegular}
        (b_2\alpha)^{-8d}\eta^{2d}/(2dK)^{2d}\le\iint_{\Omega_k}\frac{1}{|\bm x-\bm y|^{2d}}\d \bm x\d \bm y\le 12^{d}(b_2\alpha)^{-9d}\eta^{2d}/K^{2d}.
    \end{equation}
Thus, applying \eqref{DkRegular} to \eqref{E=E+}, we get that
$$
\mathds{E}[f(\mathcal{E}_Z^+)]\geq\exp\left\{-12^{d}\widetilde{\beta}(b_2\alpha)^{-9d}\eta^{2d}\#Z/K^{2d}\right\}\mathds{E}[f(\mathcal{E})].
$$
This implies that the law of $\mathcal{E}$ is absolutely continuous with respect to the law of $\mathcal{E}_Z^+$. Furthermore, let $\phi_Z$ be the Radon-Nikodym derivative of $\mathcal{E}$ with respect to $\mathcal{E}_Z^+$. Then we a.s. have
\begin{equation}\label{RNestimateEE+}
    \phi_Z(\mathcal{E}_Z^+)\leq \exp\left\{12^{d}\widetilde{\beta}(b_2\alpha)^{-9d}\eta^{2d}\#Z/K^{2d}\right\}.
\end{equation}

Now let us turn to show that the other direction is also true. To this end, recall that $\Lambda_k$ is the deterministic region defined in \eqref{detLambda} and $\bar{\mathcal{E}}_Z^+$ is the associated edge set defined \eqref{addLambda}.
For convenience, for $Z\in\mathcal{W}_\varepsilon$, let $\Lambda_Z=\cup_{k\in Z}\Lambda_k$.
By the definition \ref{add} (1), it can be checked that  $\Omega_k\subset\Lambda_k$ for all $k\in Z$ and thus $\mathcal{E}_Z^+\subset \overline{\mathcal{E}}_Z^+$.
We will use $ \overline{\mathcal{E}}_Z^+ $ as an intermediate variable and we will show that (1) the law of $\mathcal{E}_Z^+$ is absolutely continuous with respect to the law of $\overline{\mathcal{E}}_Z^+$; (2) the law of $\overline{\mathcal{E}}_Z^+$ is absolutely continuous with respect to the law of $\mathcal{E}$. Combining the assertions (1) and (2) yields our desired assertion. In addition, we will also give some estimates on the Radon-Nikodym derivatives where the bounds do not depend on $\varepsilon$ and $r$.

We first prove the assertion (1). The method we use is similar to the first part of the proof. For a fixed $Z\in\mathcal{W}_\varepsilon$, we see that for any positive measurable function $f:\mathcal{D}'\to\mathds{R}$,
\begin{equation}\label{E+Ebar}
    \mathds{E}\left[f(\overline{\mathcal{E}}_Z^+)\right]\geq \mathds{E}\left[f(\mathcal{E}_Z^+)\mathds{E}[\I_{\{\widetilde{\mathcal{E}}\cap(\Lambda_Z\setminus\Omega_Z)=\emptyset\}}|\mathcal{E},\Omega_Z,\mathcal{E}_Z^+]\right]
    \geq\mathds{E}\left[f(\mathcal{E}_Z^+)\mathds{P}[\widetilde{\mathcal{E}}\cap \Lambda_Z=\emptyset]\right].
\end{equation}
Since from the definition of $\Lambda_k$, we have
\begin{equation}\label{LambdakNotlarge}
    \iint_{\Lambda_k}\frac{1}{|\bm x-\bm y|^{2d}}\d \bm x\d \bm y\leq 6^{2d}(b_2\alpha)^{-5d},
\end{equation}
thus applying \eqref{LambdakNotlarge} to \eqref{E+Ebar}, we get
$$
\mathds{E}\left[f(\overline{\mathcal{E}}_Z^+)\right]\geq \exp\left\{-6^{2d}\widetilde{\beta}(b_2\alpha)^{-5d}\#Z\right\}\mathds{E}\left[f(\mathcal{E}_Z^+)\right],
$$
which implies that the assertion (1) holds. Moreover, let $\phi_{Z,1}$ be the Radon-Nikodym derivative of $\mathcal{E}_Z^+$ with respect to the law of $\overline{\mathcal{E}}_Z^+$, and then we have that a.s.
\begin{equation}\label{RNestimateE+Ebar}
    \phi_{Z,1}(\overline{\mathcal{E}}_Z^+)\leq \exp\left\{6^{2d}\widetilde{\beta}(b_2\alpha)^{-5d}\#Z\right\}.
\end{equation}

From now on, we turn to prove the assertion (2). In fact, we can directly compute the Radon-Nikodym derivative. Note that $\overline{\mathcal{E}}_Z^+$ is also a Poisson point process with intensity $ \frac{\beta+\widetilde{\beta}\I_{\{(\bm x,\bm y)\in\Lambda_Z\}}}{|\bm x-\bm y|^{2d}}\d \bm x\d \bm y$.
For any $N\in\mathds{N}$ and any $N$ disjoint Borel sets $A_1,\cdots,A_N$ such that  $\cup_{j=1}^{j_0}A_j=\Lambda_Z$ for some $j_0\in[1,N]_\mathds{Z}$, we have that for any $N$ nonnegative integers $n_1,\cdots,n_N$,
\begin{equation*}
    \frac{\mathds{P}[|\mathcal{E}\cap A_j|=n_j,\quad j=1,\cdots,N]}{\mathds{P}[|\overline{\mathcal{E}}_Z^+\cap A_j|=n_j,\quad j=1,\cdots,N]}=\exp\left\{\widetilde{\beta}\iint_{\Lambda_Z}\frac{1}{|\bm x-\bm y|^{2d}}\d \bm x\d \bm y\right\}\left(\frac{\beta}{\beta+\widetilde{\beta}}\right)^{n_1+\cdots+n_{j_0}},
\end{equation*}
which follows from a straightforward computation.
Note that $\sigma(\mathcal{E})$ can be generated by all the events $\{|\mathcal{E}\cap A_j|=n_j\text{ for }j=1,\cdots,N\}$ with $N,A_j,n_j$ as above. Thus for any Borel subset $E\subset\mathcal{D}'$,
$$
\mathds{P}[\mathcal{E}\in E]=\mathds{E}\left[\I_{\{\overline{\mathcal{E}}_Z^+\in E\}}\exp\left\{\widetilde{\beta}\iint_{\Lambda_Z}\frac{1}{|\bm x-\bm y|^{2d}}\d \bm x\d \bm y\right\}\left(\frac{\beta}{\beta+\widetilde{\beta}}\right)^{|\overline{\mathcal{E}}_Z^+\cap\Lambda_Z|}\right].
$$

As a result, we get the assertion (2) and the fact that the laws of $\mathcal{E},\mathcal{E}_Z^+$ and $\overline{\mathcal{E}}_Z^+$ are mutually absolutely continuous.
Furthermore, let $\phi_{Z,2}$ be the Radon-Nikodym derivative of $\overline{\mathcal{E}}_Z^+$ with respect to $\mathcal{E}$. Then we have that a.s.
\begin{equation}\label{RNphi2}
\phi_{Z,2}(\mathcal{E})=\exp\left\{-\widetilde{\beta}\iint_{\Lambda_Z}\frac{1}{|\bm x-\bm y|^{2d}}\d \bm x\d \bm y\right\}\left(\frac{\beta}{\beta+\widetilde{\beta}}\right)^{-|\mathcal{E}\cap\Lambda_Z|},
\end{equation}
whose law  does not depend on $\varepsilon$ or $r$ from the scaling invariance of the Poisson point process.

Finally we will give the estimate in the lemma.  Since $\phi_Z\phi_{Z,1}\phi_{Z,2}=1$ from the mutual absolute continuity, combining with \eqref{RNestimateE+Ebar} and \eqref{RNphi2}, we get that a.s.
$$\phi_Z(\mathcal{E}_Z^+)^{-1}\leq  \exp\left\{6^{2d}\widetilde{\beta}(b_2\alpha)^{-5d}\#Z\right\}\phi_{Z,2}(\mathcal{E}_Z^+)\leq M_1^{\#Z} \exp\left(\frac{\beta+\widetilde{\beta}}{\beta}\right)^{|\mathcal{E}_Z^+\cap \Lambda_Z|}
$$
for some finite constant $M_1>0$, which does not depend on $\varepsilon $ or $r$.
In addition, we can  get the upper bound on $\phi_Z(\mathcal{E}_Z^+)$ from \eqref{RNestimateEE+}. Hence, the proof is complete.
\end{proof}

Now we introduce how we delete edges.
\begin{definition}\label{roughdelete}
    For any $Z\in\mathcal{W}_\varepsilon$, let $\Psi_Z(\mathcal{E}_Z^+)$ be the conditional distribution of $(\mathcal{E},\Omega_Z)$ given  $\mathcal{E}_Z^+$. That is, for a.s. $\nu \in \mathcal{D}'$,
    $\Psi_{Z}(\nu)$ is the conditional law of $(\mathcal{E},\Omega_{Z})$ given $\mathcal{E}_{Z}^+=\nu$.
    Then conditioned on $\mathcal{E}$, we sample $(\mathcal{E}_{Z}^-,\Omega_{Z}^-)$ according to the law $\Psi_{Z}(\mathcal{E})$.
\end{definition}

It is worth emphasizing that, from Definition \ref{roughdelete}, it is not clear that we necessarily have that $\mathcal{E}_{Z_1}^-$ is stochastically dominated by $\mathcal{E}_{Z_2}^-$ when $Z_2\subset Z_1$. In order to clarify the monotonicity, we introduce another definition which starts by deleting edges from some maximal set $Z$ and then restrict the operation of deleting edges to smaller sets which are contained in those maximal sets (such restriction will be implemented via returning some deleted edges, as in Definition \ref{delete} below). Provided with the coincidence of the two definitions (see Lemma \ref{newdeleteedge} below), it then becomes clear that the operation of deleting edges has a natural monotonicity.
For this purpose, we hope that each $Z$ is contained in a unique maximal set.  Therefore, we need to introduce the set $W_{\bm e}$ below and only consider $Z\in\mathcal{W}_\varepsilon $.
Note that there exists a partition of $[1,\varepsilon^{-d}]_\mathds{Z}=\cup_{\bm e\in\{0,1,2\}^d}W_{\bm e}$ with $W_{\bm e}:=\{k\in [1,\varepsilon^{-d}]_\mathds{Z}:\bm z_k-(\varepsilon r)\bm e\in(3\varepsilon r)\mathds{Z}^d\}\in\mathcal{W}_\varepsilon$. Thus for any nonempty $Z\in\mathcal{W}_\varepsilon$, there exists a unique $\bm e\in \{0,1,2\}^d$ such that $Z\subset W_{\bm e}$.

Now we present another definition of $(\mathcal{E}_Z^-,\Omega_Z^-)$.
\begin{definition}\label{delete}
    For each $\bm e\in\{0,1,2\}^d$, let $\Psi_{W_{\bm e}}(\mathcal{E}_{W_{\bm e}}^+)$ be the conditional distribution of $(\mathcal{E},\Omega_{W_{\bm e}})$ given  $\mathcal{E}_{W_{\bm e}}^+$. That is, for a.s. $\nu \in \mathcal{D}'$,
    $\Psi_{W_{\bm e}}(\nu)$ is the conditional law of $(\mathcal{E},\Omega_{W_{\bm e}})$ given $\mathcal{E}_{W_{\bm e}}^+=\nu. $
    Then conditioned on $\mathcal{E}$, we sample $(\mathcal{E}_{W_{\bm e}}^-,\Omega_{W_{\bm e}}^-)$ according to the law $\Psi_{W_{\bm e}}(\mathcal{E})$.
    Additionally, for any general nonempty $Z\in\mathcal{W}_\varepsilon$, choose the unique $W_{\bm e}$ such that $Z\subset W_{\bm e}$, and then let $\mathcal{E}_Z^-=\mathcal{E}_{W_{\bm e}}^-\cup (\cup_{k\in W_e\setminus Z} \mathcal{E} \cap (V_{3\varepsilon r}(\bm z_k)\times V_{3\varepsilon r}(\bm z_k)))$ and $\Omega_Z^-=\Omega_{W_{\bm e}}^-\cap (\cup_{k\in Z}V_{3\varepsilon r}(\bm z_k)\times V_{3\varepsilon r}(\bm z_k))$. When $Z=\emptyset$, we let  $\mathcal{E}_Z^-=\mathcal{E}$ and $\Omega_Z^-=\emptyset$.
\end{definition}

Note that by Definition \ref{delete}, $\mathcal{E}_{Z_1}^-\subset \mathcal{E}_{Z_2}^-$ for all $Z_1,Z_2\in \mathcal{W}_\varepsilon$ with $Z_2\subset Z_1$ (this is the aforementioned monotonicity). Furthermore, we will show that Definitions \ref{roughdelete} and \ref{delete} coincide as follows.
\begin{lemma}\label{newdeleteedge}
    Let $(\mathcal{E}_Z^-,\Omega_Z^-)$ be defined as in Definition \ref{delete}. Then conditioned on $\mathcal{E}$, we have that $(\mathcal{E}_Z^-,\Omega_Z^-)$ follows the law $\Psi_Z(\mathcal{E})$ {\rm(}recall Definition \ref{roughdelete} for the definition of $\Psi_Z${\rm)}.
\end{lemma}

Before we prove Lemma \ref{newdeleteedge}, we will first show the following lemma about the relationship between different $\mathcal{E}_Z^+$.
\begin{lemma}\label{newadd}
    For any $Z_0,Z\in\mathcal{W}_\varepsilon$ such that $Z\subset Z_0 $, 
    the conditional law of $(\mathcal{E}_{Z_0\setminus Z}^+,\Omega_Z)$ conditioned on $\mathcal{E}_{Z_0}^+$ is $\Psi_Z(\mathcal{E}_{Z_0}^+)$.
\end{lemma}
\begin{proof}
    First, by Definition \ref{add},
    $$\mathcal{E}_Z^+\setminus\mathcal{E}=\mathcal{E}_{Z_0}^+\setminus \mathcal{E}_{Z_0\setminus Z}^+=\widetilde{\mathcal{E}}\cap \Omega_Z.$$ Thus it suffices to show that the conditional laws of $(\widetilde{\mathcal{E}}\cap \Omega_Z,\Omega_Z)$ conditioned on $\mathcal{E}_Z^+$ and $\mathcal{E}_{Z_0}^+$ respectively are the same.
    For that, let $\Delta_Z=\cup_{k\in Z}(V_{3\varepsilon r}(\bm z_k)\times \mathds{R}^d)$. From Definitions \ref{add} and \ref{super good}, it is clear that $\Omega_Z$ is determined by $\mathcal{E}\cap \Delta_Z$. Combining this with the independence of the Poisson point process, we can see that the conditional law of $\Omega_Z$ conditioned on $\mathcal{E}_Z^+$ is the same as that conditioned on $\mathcal{E}_Z^+\cap \Delta_Z$. Similarly, we also get that the conditional law of $\Omega_Z$ conditioned on $\mathcal{E}_{Z_0}^+$ is the same as that conditioned on $\mathcal{E}_{Z_0}^+\cap \Delta_Z$. In addition, note that $\mathcal{E}_Z^+\cap \Delta_Z=\mathcal{E}_{Z_0}^+\cap\Delta_Z$. This and the above analysis imply that the conditional law of $\Omega_Z$ conditioned on $\mathcal{E}_Z^+$ is the same as that conditioned on $\mathcal{E}_{Z_0}^+$.

    Finally, since $ \widetilde{\mathcal{E}}\cap \Omega_Z $ is determined by $\Omega_Z$ and $\widetilde{\mathcal{E}}\cap(\cup_{k\in Z} (V_{3\varepsilon r}(\bm z_k)\times V_{3\varepsilon r}(\bm z_k)))$, it suffices to show that the conditional law of $\widetilde{\mathcal{E}}\cap(\cup_{k\in Z} (V_{3\varepsilon r}(\bm z_k)\times V_{3\varepsilon r}(\bm z_k)))$ conditioned on $(\Omega_Z,\mathcal{E}_Z^+)$ is the same as that conditioned on $(\Omega_Z,\mathcal{E}_{Z_0}^+)$.
    Indeed, from  $\widetilde{\mathcal{E}}\cap(\cup_{k\in Z} (V_{3\varepsilon r}(\bm z_k)\times V_{3\varepsilon r}(\bm z_k)))\subset \mathcal{E}_Z^+\cap \Delta_Z$ and the independence of the Poisson point process, we have that the conditional law of $\widetilde{\mathcal{E}}\cap(\cup_{k\in Z} (V_{3\varepsilon r}(\bm z_k)\times V_{3\varepsilon r}(\bm z_k)))$ conditioned on $(\Omega_Z,\mathcal{E}_Z^+)$ is the same as that conditioned on $(\Omega_Z,\mathcal{E}_Z^+\cap \Delta_Z)$. Moreover, as we mentioned in the previous paragraph,  $\mathcal{E}_Z^+\cap \Delta_Z=\mathcal{E}_{Z_0}^+\cap\Delta_Z$, which implies that $(\Omega_Z,\mathcal{E}_{Z}^+\cap\Delta_Z)=(\Omega_Z,\mathcal{E}_{Z_0}^+\cap\Delta_Z)$. So we get that the conditional law of $\widetilde{\mathcal{E}}\cap(\cup_{k\in Z} (V_{3\varepsilon r}(\bm z_k)\times V_{3\varepsilon r}(\bm z_k)))$ conditioned on $(\Omega_Z,\mathcal{E}_Z^+)$ is the same as that conditioned on $(\Omega_Z,\mathcal{E}_{Z_0}^+\cap \Delta_Z)$. By $\widetilde{\mathcal{E}}\cap(\cup_{k\in Z} (V_{3\varepsilon r}(\bm z_k)\times V_{3\varepsilon r}(\bm z_k)))\subset \mathcal{E}_{Z_0}^+\cap \Delta_Z$ and the independence of the Poisson point process again, we obtain the desired statement.
    Thus we complete the proof.
\end{proof}

\begin{proof}[Proof of Lemma \ref{newdeleteedge}]
    For any $Z\in\mathcal{W}_\varepsilon$, let $Z_0\in\{W_{\bm e}\}_{\bm e\in\{0,1,2\}^d}$ such that $Z\subset Z_0$.
    Let $\Phi_Z$ be the conditional law of $(\mathcal{E}_Z^-,\Omega_Z^-)$ conditioned on $\mathcal{E}$, i.e., for a.s. $\nu\in\mathcal{D}'$,
    $\Phi_Z(\nu)$ is the conditional law of $(\mathcal{E}_Z^-,\Omega_Z^-)$ given $\mathcal{E}=\nu$.
    By Definition \ref{delete}, we see that
    \begin{itemize}
    \item[(A1)] the conditional law of $(\mathcal{E},\Omega_{Z_0})$ given $\mathcal{E}_{Z_0}^+=\nu$ and the conditional law of $(\mathcal{E}_{Z_0}^-,\Omega_{Z_0}^-)$ given $\mathcal{E}=\nu$ are equal to $\Psi_{Z_0}(\nu)$ for a.s. $\nu\in\mathcal{D}'$.
        \end{itemize}
    Additionally, note that
\begin{equation}\label{A2}
\mathcal{E}_Z^-=\mathcal{E}_{Z_0}^-\cup (\mathcal{E}\cap (\cup_{k\in Z_0\setminus Z}V_{3\varepsilon r}(\bm z_k)\times V_{3\varepsilon r}(\bm z_k)))\quad \text{and}\quad \Omega_Z^-=\Omega_{Z_0}^-\cap (\cup_{k\in Z_0\setminus Z}V_{3\varepsilon r}(\bm z_k)\times V_{3\varepsilon r}(\bm z_k))\end{equation}
         Combining (A1) and \eqref{A2} implies that
         \begin{itemize}
         \item[(A2)]the conditional law of $(\mathcal{E}\cup (\mathcal{E}_{Z_0}^+\cap (\cup_{k\in Z_0\setminus Z}V_{3\varepsilon r}(\bm z_k)\times V_{3\varepsilon r}(\bm z_k))),\Omega_{Z_0}\cap (\cup_{k\in Z_0\setminus Z}V_{3\varepsilon r}(\bm z_k)\times V_{3\varepsilon r}(\bm z_k)))$ given $\mathcal{E}_{Z_0}^+=\nu$ is also $\Phi_{Z}(\nu)$ for a.s. $\nu\in\mathcal{D}'$.\end{itemize}
    Furthermore, from Definition \ref{add} we have that
 \begin{equation}\label{A3}
    (\mathcal{E}\cup (\mathcal{E}_{Z_0}^+\cap (\cup_{k\in Z_0\setminus Z}V_{3\varepsilon r}(\bm z_k)\times V_{3\varepsilon r}(\bm z_k))),\Omega_{Z_0}\cap (\cup_{k\in Z_0\setminus Z}V_{3\varepsilon r}(\bm z_k)\times V_{3\varepsilon r}(\bm z_k)))=(\mathcal{E}_{Z_0\setminus Z}^+,\Omega_Z).
 \end{equation}
    Then combining (A2) and \eqref{A3} yields that for a.s. $\nu\in\mathcal{D}'$,
    $\Phi_Z(\nu)$ is the conditional law of $(\mathcal{E}_{Z_0\setminus Z}^+,\Omega_Z)$ given $\mathcal{E}_{Z_0}^+=\nu$. Hence the result we desired is implied by Lemma \ref{newadd}.
\end{proof}



With the above definitions and lemmas at hand, we have the following relations for the laws of
$\left(\mathcal{E}_Z^+,\mathcal{E}\right)$ and $\left(\mathcal{E},\mathcal{E}_Z^-\right)$.

\begin{lemma}\label{abscont}
For each $Z\in\mathcal{W}_\varepsilon$, the laws of $\left(\mathcal{E}_Z^+,\mathcal{E}\right)$ and $\left(\mathcal{E},\mathcal{E}_Z^-\right)$ are mutually  absolutely continuous.
Furthermore, let $h_Z$ be the Radon-Nikodym derivative of $(\mathcal{E},\mathcal{E}_Z^-)$ with respect to $(\mathcal{E}_Z^+,\mathcal{E})$. Then $h_Z(\mathcal{E}_Z^+,\mathcal{E})=\phi_Z(\mathcal{E}_Z^+)$. In particular,
there exists $M>0$ {\rm(}which does not depend on $\varepsilon, r${\rm)} such that a.s.
$$
M^{-\#Z} \left(\frac{\beta}{\beta+\widetilde{\beta}}\right)^{|\mathcal{E}_Z^+\cap \Lambda_Z|}  \leq h_Z(\mathcal{E}_Z^+,\mathcal{E}) \leq  M^{\#Z}.
$$
\end{lemma}

\begin{proof}
Let $B\subset \mathcal{D}'\times \mathcal{D}'$ (recall $\mathcal{D}'$ in Definition \ref{strongLRPmetric}) be any Borel subset such that $\mathds{P}[(\mathcal{E}_Z^+,\mathcal{E})\in B]=0$. Then it is clear that a.s.
\begin{equation}\label{condE+}
\mathds{P}[(\mathcal{E}_Z^+,\mathcal{E})\in B|\mathcal{E}_Z^+]=0.
\end{equation}
Additionally, by Definition \ref{delete} and Lemma \ref{newdeleteedge}, we can see that the conditional distribution of $\mathcal{E}$ given $\mathcal{E}_Z^+$ and the conditional distribution of $\mathcal{E}_Z^-$ given $\mathcal{E}$ are the same. Hence, combining this with \eqref{condE+}  and Lemma \ref{newdeleteedge}, we obtain that a.s.\ $\mathds{P}[(\mathcal{E},\mathcal{E}_Z^-)\in B|\mathcal{E}]=0$. Taking expectation, we get that $\mathds{P}[(\mathcal{E},\mathcal{E}_Z^-)\in B]=0$. Thus the law of $(\mathcal{E},\mathcal{E}_Z^-)$ is absolutely continuous with respect to the law of $(\mathcal{E}_Z^+,\mathcal{E})$. Similarly we can get the converse. This implies the desired mutual absolute continuity.

Recall as in Lemma \ref{EE+} $\phi_Z$ is the Radon-Nikodym derivative of $\mathcal{E}$ with respect to $\mathcal{E}_Z^+$.
Let $f(\mathcal{E}_Z^+,\mathcal{E})\in \sigma((\mathcal{E}_Z^+,\mathcal{E}))$ be any positive measurable function. Since $\mathds{E}[f(\mathcal{E}_Z^+,\mathcal{E})|\mathcal{E}_Z^+]$ is a.s. determined by $\mathcal{E}_Z^+$, we can find one measurable function $X:\mathcal{D}'\to \mathds{R}$ such that $\mathds{E}[f(\mathcal{E}_Z^+,\mathcal{E})|\mathcal{E}_Z^+]=X(\mathcal{E}_Z^+)$. Then from Definition \ref{delete} and Lemma \ref{newdeleteedge}, we get that $ \mathds{E}[f(\mathcal{E},\mathcal{E}_Z^-)|\mathcal{E}]=X(\mathcal{E})$. As a result, combining with the definition of $\phi_Z$, we get
\begin{equation*}
    \begin{aligned}
        \mathds{E}\left[f(\mathcal{E}_Z^+,\mathcal{E})\phi_Z(\mathcal{E}_Z^+)\right]
        &=\mathds{E}\left[\mathds{E}\left[f(\mathcal{E}_Z^+,\mathcal{E})|\mathcal{E}_Z^+\right]\phi_Z(\mathcal{E}_Z^+)\right]
        =\mathds{E}\left[X(\mathcal{E}_Z^+)\phi_Z(\mathcal{E}_Z^+)\right]\\
        &=\mathds{E}[X(\mathcal{E})]=\mathds{E}\left[\mathds{E}\left[f(\mathcal{E},\mathcal{E}_Z^-)|\mathcal{E}\right]\right]=\mathds{E}[f(\mathcal{E},\mathcal{E}_Z^-)].
    \end{aligned}
\end{equation*}
Thus we have that a.s. $h_Z(\mathcal{E}_Z^+,\mathcal{E})=\phi_Z(\mathcal{E}_Z^+)$. Then the desired assertion is implied by Lemma \ref{EE+}.
\end{proof}

Recall that a weak $\beta$-LRP metric $D=D(\mathcal{E})$ (resp. $\widetilde{D}=\widetilde{D}(\mathcal{E})$) is a measurable function from $\mathcal{D}'$ to the space of continuous pseudometrics on $\mathds{R}^d$ and is a.s. determined by $\mathcal{E}$.
For each $Z\in \mathcal{W}_\varepsilon$, we define $D_Z^+=D(\mathcal{E}_Z^+)$ (resp. $\widetilde{D}_Z^+=\widetilde{D}(\mathcal{E}_Z^+)$) and $D_Z^-=D(\mathcal{E}_Z^-)$ (resp. $\widetilde{D}_Z^-=\widetilde{D}(\mathcal{E}_Z^-)$). Note that $\mathcal{E}_Z^+,\mathcal{E}_Z^-\in\mathcal{D}'$ are defined in Definitions \ref{add} and \ref{delete}, thus $D_Z^+$ (resp. $\widetilde{D}_Z^+$) and $D_Z^-$ (resp. $\widetilde{D}_Z^-$) are well-defined and they are a.s.\ determined by $\mathcal{E}_Z^+,\mathcal{E}_Z^-$, respectively.

    We refer to numbers $ p,\alpha,b_1,b_2,\widetilde{C},C_1,\eta,K,\widetilde{\beta},M_0$, $\delta_1$ and $M_1$ as the parameters. Here $M_0$ is a sufficiently large number such that the following inequality holds:
\begin{equation}\label{M1-1}
    \mathds{P}[|\overline{\mathcal{E}}_{k}^+\cap \Lambda_{k}|\leq M_0]\geq 1-\frac{1-p_c}{2}\quad \text{for any } k\in [1,\varepsilon^{-d}]_\mathds{Z},
\end{equation}
where $\Lambda_k,\overline{\mathcal{E}}_k^+$ are defined in \eqref{detLambda} and \eqref{addLambda}, respectively,
and $p_c$ is the constant defined in Lemma \ref{hd-BK} with $\delta=1/(4C_{dis})$ (here $C_{dis}$ is the constant defined in Lemma \ref{number-path-k}).
   It is important to note that  $|\overline{\mathcal{E}}_{ k}^+\cap \Lambda_{k}|$ follows a Poisson distribution with parameter $(\beta+\widetilde{\beta})\mu(\Lambda_{ k})$ (here $\mu(\Omega):=\iint_{\Omega}\frac{\d \bm x\d \bm y}{|\bm x-\bm y|^{2d}}$ for $\Omega\subset \mathds{R}^{2d}$) which does not depend on $k,\varepsilon$ or $r$, this ensures that $M_0$ is  well-defined.
Moreover, we recall $M$ in Lemma \ref{abscont} and take $\delta_1$ and $M_1$ as
\begin{equation*}
\delta_1=(\beta/(\beta+\widetilde{\beta}))^{M_0},
\end{equation*}
and
\begin{equation}\label{M0-1}
M_1=\max\{M, M/\delta_1\}.
\end{equation}

         It is worth  emphasizing that none of these parameters depends on $\varepsilon$ or $r$.

\subsection{Compare old and new metrics}\label{oldandnew}
Recall from the paragraphs after the proof of Lemma \ref{interval-good1} that we fixed $r>0$ and $\widetilde{\gamma},\widetilde{q}>0$ such that $\mathds{P}[\widetilde{G}_r(\widetilde{\gamma},\widetilde{q},c'')]\geq \widetilde{\gamma}$. Here $c''$ is the constant defined in Proposition \ref{mrinS3-tilde} applied with $c'=(c_*+C_*)/2$.
For any $Z\in{ \mathcal{W}}_\varepsilon$, let $\mathcal{E}_Z^+$ and  $\mathcal{E}_Z^-$ be defined as in Definitions \ref{add} and \ref{delete}, respectively.

In the rest of paper, for $\bm x,\bm y\in V_r(\bm 0)$,  we select and fix an arbitrary $D$-geodesic from $\bm x$ to $\bm y$, denoted as $P_{\bm x\bm y}$. When there is no ambiguity, we will use the shorthand notation $P$ and refer to it as ``the $D$-geodesic from  $\bm x$ to $\bm y$''.

Our goal is to demonstrate that the probability for a particular ``bad'' event of the $D$-geodesic $P_{\bm x\bm y}$ is small, and we next define this bad event.
\begin{definition}\label{G}
 For $\bm x,\bm y\in V_r(\bm 0)$  and $\varepsilon>0$, we define the event $\mathscr{G}_r^{\varepsilon}:=\mathscr{G}_r^{\varepsilon}(\bm x,\bm y)$ as follows.
    \begin{enumerate}
        \item\label{GCondition1} $\widetilde{D}(\bm x,\bm y)\ge C_* D(\bm x,\bm y)-(\varepsilon  r)^{\theta}.$
         \item\label{GCondition2} The $D$-geodesic $P$ from $\bm x$ to $\bm y$  satisfies that $P\subset V_r(\bm 0)$.
        \item\label{GCondition3} $P$ hits at least $\varepsilon^{-\theta/4}$ cubes of the form $V_{\varepsilon r}(\bm z_k)$ for $k\in [1, \varepsilon^{-d}]_\mathds{Z}$.
        \item\label{GCondition4} There exist at least $\varepsilon^{-\theta/4}/(8\cdot 3^d)$ super super good cubes with respect to $\mathcal{E}$, denoted as $V_{3\varepsilon r}(\bm z_k)$ for $k\in \mathcal{K}$, such that $\mathcal{K}\in\mathcal{W}_\varepsilon$, $P$ hits $V_{\varepsilon r}(\bm z_k)$ for all $k\in\mathcal{K}$ and the following conditions hold.
        \begin{itemize}

        \item[\rm (i)]  For each $k\in \mathcal{K}$, we have $|\overline{\mathcal{E}}_k^+\cap\Lambda_k|\leq M_0$, where $M_0$ is defined in \eqref{M1-1}.



            \item[\rm (ii)] For each $k\in \mathcal{K}$, under $\mathcal{E}_{k}^+$ we have that $J_{k,q(k,i)}^{(2)}$ is connected  to $J_{k,q(k,i+1)}^{(1)}$ by a long edge for all $i\in[1,3^d(b_2\alpha)^{-2d}]_\mathds{Z}$.
             Here $J_{k,q(k,i)}$ is chosen in \eqref{choosenice}, and $(J_{k,q(k,i)}^{(1)},J_{k,q(k,i)}^{(2)})$ is the great pair of small cubes in it.

        \end{itemize}
    \end{enumerate}
\end{definition}

We present the main estimate for $\mathds{P}[\mathscr{G}_r^\varepsilon]$ as follows.

\begin{proposition}\label{Prop4.3DG21}
    Assume that $c_*<C_*$ and  let $r>0$. Then for $\bm x,\bm y\in V_r(\bm 0)$ with $|\bm x-\bm y|\geq \alpha r$,
    \begin{equation}\label{Inequality4.3DG21}
        \mathds{P}\left[\mathscr{G}_r^\varepsilon(\bm x,\bm y)\right]=O_\varepsilon(\varepsilon^\mu)\quad \forall \mu>0
    \end{equation}
    with the implicit constant in the $O_\varepsilon(\cdot)$ depending only on $\mu$, the parameters and the laws of $D$ and $\widetilde{D}$ {\rm(}in particular not depending on $r ,\bm x,\bm y${\rm)}.
\end{proposition}

We will now explain how to prove Proposition \ref{Prop4.3DG21} provided with Propositions \ref{Prop4.5DG21} and \ref{Prop4.6DG21} below, whose proofs will occupy most of the rest of this section. Our strategy for proving this proposition will rely on counting the occurrences of certain events. To do that, let us now define these events.

From here on,  we  fix $\bm x,\bm y\in V_r(\bm 0)$ with $|\bm x-\bm y|\geq \alpha r$.  Recall that $\mathcal{W}_\varepsilon$ is the collection of all subsets $Z$ of $[1,\varepsilon^{-d}]_\mathds{Z}$ such that $\bm z_k-\bm z_l\in(3\varepsilon r)\mathds{Z}^d$ for all $k,l\in Z$. 
 For any $Z\in \mathcal{W}_{\varepsilon}$, we define two metrics $D_Z^{+}=D(\mathcal{E}_Z^+)$ and $D_{Z}^{-}=D(\mathcal{E}_Z^-)$ as above. In what follows, we also select and fix an arbitrary $D_Z^+$-geodesic (resp. $D_Z^-$-geodesic) from $\bm x$ and $\bm y$, denoted as $P_Z^+$ (resp. $P_Z^-$). Similarly, we will refer to $P_Z^+$ (resp. $P_Z^-$) as ``the $D_Z^+$-geodesic (resp. $D_Z^-$-geodesic) from  $\bm x$ to $\bm y$''.


\begin{definition}\label{F_Z}
    For $Z\in  \mathcal{W}_{\varepsilon}$, we let  $\mathsf{F}_{Z,\varepsilon}(\mathcal{E},\mathcal{E}_Z^+):=\mathsf{F}_{Z,\varepsilon}(\mathcal{E},\mathcal{E}_Z^+;\bm x,\bm y)$ be the event that the following conditions hold.
    \begin{enumerate}
        \item $\widetilde{D}(\bm x,\bm y)\ge C_* D(\bm x,\bm y)-(\varepsilon r)^{\theta}.$
        \item For all $k\in Z$, we have that $V_{3\varepsilon r}(\bm z_k)$ is super super good with respect to $\mathcal{E}$. Furthermore, let $J_{k,q(k,1)},\cdots, J_{k,q(k,3^d(b_2\alpha)^{-2d})}$ be the nice cubes chosen in  \eqref{choosenice} and let $(J_{k,q(k,i)}^{(1)},J_{k,q(k,i)}^{(2)})$ be the great pair of small cubes in $J_{k,q(k,i)}$.
        \item For all $k\in Z$, we have that the $D$-geodesic $P$ hits $V_{\varepsilon r}(\bm z_k)$.

        \item  For all $k\in Z$, we have that $|\mathcal{E}_Z^+\cap \Lambda_k|\le M_0$, where $M_0$ is defined in \eqref{M1-1}.

        \item  Under $\mathcal{E}_Z^+$, we have that $J_{k,q(k,i)}^{(2)}$ is connected to $J_{k,q(k,i+1)}^{(1)}$ by a long edge for all $k\in Z$ and  $i\in[1,3^d(b_2\alpha)^{-2d}]_\mathds{Z}$.

        \item For all $k \in Z$, we have that $P_Z^{+}$ passes through a great pair of small cubes $(J_{k,q(k,i)}^{(1)},J_{k,q(k,i)}^{(2)})$ for some $i\in[1,3^d(b_2\alpha)^{-2d}]_\mathds{Z}$.

    \end{enumerate}
\end{definition}

Our first estimate implies that on the event $\mathscr{G}_\varepsilon^r$, there are many choices of $Z$ for which $\mathsf{F}_{Z,\varepsilon}(\mathcal{E},\mathcal{E}_Z^+)$ occurs.
\begin{proposition}\label{Prop4.5DG21}
    There exists $c_3>0$, depending only on the parameters  and the laws of $D$ and $\widetilde{D}$ {\rm(}not on $r,\bm x,\bm y${\rm)}, such that
    for each $m \in \mathds{N}$, there exists $\varepsilon_* > 0$, depending only on $m$, the parameters and the laws of $D$ and $\widetilde{D}$ such that the
    following is true for each $r > 0$ and each  $\varepsilon\in (0, \varepsilon_*]$.
     Assume that $\bm x,\bm y\in V_r(\bm 0)$ with $|\bm x-\bm y|\geq \alpha r$.
    If $\mathscr{G}^\varepsilon_r(\bm x,\bm y)$ occurs, then for all $m\in \mathds{N}$, a.s. 
    \begin{equation}\label{LowerBoundProp4.5DG21}
        \# \left\{Z\in{ \mathcal{W}}_{\varepsilon}:\# Z\le m\text{ and }\mathsf{F}_{Z,\varepsilon}(\mathcal{E},\mathcal{E}_Z^+)\text{ occurs}\right\}\ge\varepsilon^{-c_3 m}.
    \end{equation}

\end{proposition}

We will prove Proposition \ref{Prop4.5DG21} in Section \ref{SectionProp4.5DG21}. Additionally,
our second estimate provides an unconditional upper
bound on the number of $Z$ for which $\mathsf{G}_{Z,\varepsilon}^-(\mathcal{E}_Z^-,\mathcal{E};\bm x,\bm y)$ occurs, where $\mathsf{G}_{Z,\varepsilon}^-(\mathcal{E}_Z^-,\mathcal{E};\bm x,\bm y)$ is defined as follows.
\begin{definition}\label{G-}
    We let $\mathsf{G}_{Z,\varepsilon}^-(\mathcal{E}_Z^-,\mathcal{E}):=\mathsf{G}_{Z,\varepsilon}^-(\mathcal{E}_Z^-,\mathcal{E};\bm x,\bm y)$ be the event that satisfies the following conditions.
    \begin{enumerate}
        \item $\widetilde{D}_Z^-(\bm x,\bm y)\ge C_* D_Z^-(\bm x,\bm y)-(\varepsilon r)^{\theta}.$
        \item For all $k\in Z$, we have that $V_{\varepsilon r}(\bm z_k)$ is super super good  with respect to $\mathcal{E}_Z^-$.
        Furthermore, let  $J^-_{k,q(k,1)},\cdots, J^-_{k,q(k,3^d(b_2\alpha)^{-2d})}$ be the nice cubes chosen as in  \eqref{choosenice} and let $(J_{k,q(k,i)}^{(1)-},J_{k,q(k,i)}^{(2)-})$ be the great pair of small cubes in $J^-_{k,q(k,i)}$.
        \item For all $k\in Z$, we have that the $D_Z^-$-geodesic $P_Z^-$ hits $V_{\varepsilon r}(\bm z_k)$.

        \item For all $k\in Z$, we have that $|\mathcal{E}\cap \Lambda_k|\le M_0$, where $M_0$ is defined in \eqref{M1-1}.


        \item  Under $\mathcal{E}$, we have that $J_{k,q(k,i)}^{(2)-}$ is connected $J_{k,q(k,i+1)}^{(1)-}$ by a long edge for all $k\in Z$ and  $i\in[1,3^d(b_2\alpha)^{-2d}]_\mathds{Z}$.

        \item For all $k \in Z$, we have that $P$ passes through a great pair of small cubes $(J_{k,q(k,i)}^{(1)-},J_{k,q(k,i)}^{(2)-})$ for some $i\in[1,3^d(b_2\alpha)^{-2d}]_\mathds{Z}$.
    \end{enumerate}
\end{definition}
Note that by Definitions \ref{F_Z} and \ref{G-}, for any $Z\in{\mathcal{W}}_\varepsilon$, we have
\begin{equation}\label{EqualityF-=G}
   \mathsf{F}_{Z,\varepsilon}(\mathcal{E},\mathcal{E}_Z^+)=\mathsf{G}^-_{Z,\varepsilon}(\mathcal{E},\mathcal{E}_Z^+).
\end{equation}
Then we can derive the following estimate.
\begin{proposition}\label{Prop4.6DG21}
    There is a constant $C_2 > 0$, depending only on the parameters and the laws of $D$ and $\widetilde{D}$  {\rm(}not on $r,x,y${\rm)}, such that the following is true. For each $m\in\mathds{N}$, a.s.
    \begin{equation}\label{Inequality4.6DG21}
        \# \left\{Z\in{ \mathcal{W}}_{\varepsilon}:\# Z\le m\text{ and }\mathsf{G}_{Z,\varepsilon}^-(\mathcal{E}_Z^-,\mathcal{E})\text{ occurs}\right\}\le C_2^m.
    \end{equation}
\end{proposition}

We will provide the proof of Proposition \ref{Prop4.6DG21} in Section \ref{SectionProp4.6DG21}.
Our third estimate aims to give both upper and lower bounds on the Radon-Nykodym derivative between $(\mathcal{E}_Z^-,\mathcal{E})$ and $(\mathcal{E},\mathcal{E}_Z^+)$.

\begin{lemma}\label{Lemma4.4DG21}
    For any $Z\in{ \mathcal{W}}_\varepsilon$,
    \begin{equation*}
        M_1^{-\# Z}\le\frac{\mathds{P}[\mathsf{G}_{Z,\varepsilon}^-(\mathcal{E}_Z^-,\mathcal{E})]}
        {\mathds{P}[\mathsf{G}_{Z,\varepsilon}^-(\mathcal{E},\mathcal{E}_Z^+)]}\le M_1^{\# Z},
    \end{equation*}
    where $M_1$ is chosen in  \eqref{M0-1}.
\end{lemma}
\begin{proof}
    Recall from Lemma \ref{abscont} that $h_Z$ is the Radon-Nikodym derivative of $(\mathcal{E},\mathcal{E}_Z^-)$ with respect to $(\mathcal{E}_Z^+,\mathcal{E})$. For a fixed  $Z\in{\mathcal{W}}_\varepsilon$, from Definition \ref{G-} (4), Lemma \ref{abscont} and the relationship between $M_0$ and $M_1$ in \eqref{M0-1} we have that on the event $\mathsf{G}_{Z,\varepsilon}^-(\mathcal{E}_Z^-,\mathcal{E}) $,
    \begin{equation}\label{boundh}
        M_1^{-\#Z}\leq h_Z(\mathcal{E},\mathcal{E}_Z^{-})\leq M_1^{\#Z}.
    \end{equation}
    Additionally, by Definition \ref{G-} it is clear that $\mathsf{G}_{Z,\varepsilon}^-(\mathcal{E}_Z^-,\mathcal{E})$ is a.s.\ determined by $(\mathcal{E}_Z^-,\mathcal{E})$ and $\mathsf{G}_{Z,\varepsilon}^-(\mathcal{E},\mathcal{E}_Z^+)$ is a.s.\ determined by $(\mathcal{E},\mathcal{E}_Z^+)$. Combining this measurability with \eqref{boundh}, we complete the proof of the lemma.

\end{proof}

Let us now prove Proposition \ref{Prop4.3DG21} from the above three estimates.
\begin{proof}[Proof of Proposition \ref{Prop4.3DG21}]
    From Propositions \ref{Prop4.5DG21}, \ref{Prop4.6DG21} and Lemma \ref{Lemma4.4DG21}, we get that for each $m\in\mathds{N}$ and each small enough $\varepsilon>0$,

    \begin{equation*}
        \begin{aligned}
            1&=\sum_{Z\in{ \mathcal{W}}_\varepsilon,\# Z\le m}\mathds{E}\left[\frac{\I_{\mathsf{G}_{Z,\varepsilon}^-(\mathcal{E}_Z^-,\mathcal{E})}}{\# \{Z'\in{ \mathcal{W}}_{\varepsilon}:\# Z'\le m\text{ and }\mathsf{G}_{Z',\varepsilon}^-(\mathcal{E}_{Z'}^-,\mathcal{E})\text{ occurs}\}}\right] \\
            &\ge C_2^{-m}\sum_{Z\in{\mathcal{W}}_\varepsilon,\# Z\le m}\mathds{P}[\mathsf{G}_{Z,\varepsilon}^-(\mathcal{E}_Z^-,\mathcal{E})]
            \quad\quad(\text{by Proposition }\ref{Prop4.6DG21}) \\
            &\ge C_2^{-m}M_1^{-m}\sum_{Z\in{\mathcal{W}}_\varepsilon,\# Z\le m}\mathds{P}[\mathsf{G}_{Z,\varepsilon}^-(\mathcal{E},\mathcal{E}_Z^+)]\quad\quad(\text{by Lemma }\ref{Lemma4.4DG21}) \\
            &= C_2^{-m}M_1^{-m}\sum_{Z\in{ \mathcal{W}}_\varepsilon,\# Z\le m}\mathds{P}[\mathsf{F}_{Z,\varepsilon}(\mathcal{E},\mathcal{E}_Z^+)]
            \quad\quad(\text{by \eqref{EqualityF-=G}}) \\
            &= C_2^{-m}M_1^{-m} \mathds{E}\left[\# \{Z\in{\mathcal{W}}_{\varepsilon}:\# Z\le m\text{ and }\mathsf{F}_{Z,\varepsilon}(\mathcal{E},\mathcal{E}_Z^+)\text{ occurs}\}\right] \\
            &\succeq C_2^{-m}M_1^{-m}\varepsilon^{-c_3 m}
            \mathds{P}[\mathscr{G}_r^\varepsilon]  \quad\quad(\text{by Proposition }\ref{Prop4.5DG21}).
        \end{aligned}
    \end{equation*}
    This implies
$$
\mathds{P}[\mathscr{G}_r^\varepsilon]\preceq  C_2^{m}M_1^{m}\varepsilon^{c_3 m}
\quad \text{for all }m\in \mathds{N}.
$$
Thus for any $\mu>0$, choosing $m>\mu/c_3$ yields (\ref{Inequality4.3DG21}).
\end{proof}

\subsection{Proof of Proposition \ref{Prop4.5DG21}}\label{SectionProp4.5DG21}
 It is straightforward to show from the definition of $\mathscr{G}_r^\varepsilon$ that if $\mathscr{G}_r^\varepsilon$ occurs, then there are many $Z\in{ \mathcal{W}}_\varepsilon$ for which all of the conditions in the definition of $\mathsf{F}_{Z,\varepsilon}(\mathcal{E},\mathcal{E}_Z^+)$ occur except possibly condition (6), i.e., the event $\overline{\mathsf{F}}_{Z,\varepsilon}(\mathcal{E},\mathcal{E}_Z^+)$ in the following
definition occurs.

\begin{definition}\label{Fbar}
    For $Z\in{ \mathcal{W}}_\varepsilon$, we let  $\overline{\mathsf{F}}_{Z,\varepsilon}(\mathcal{E},\mathcal{E}_Z^+):=\overline{\mathsf{F}}_{Z,\varepsilon}(\mathcal{E},\mathcal{E}_Z^+;\bm x,\bm y)$ be the event such that the following is true.
    \begin{enumerate}
        \item $\widetilde{D}(\bm x,\bm y)\ge C_* D(\bm x,\bm y)-(\varepsilon r)^{\theta}.$
        \item For all $k\in Z$, we have that $V_{3\varepsilon r}(\bm z_k)$ is super super good with respect to $\mathcal{E}$. Furthermore, let $J_{k,q(k,1)},\cdots, J_{k,q(k,3^d(b_2\alpha)^{-2d})}$ be the nice cubes chosen in  \eqref{choosenice} and let $(J_{k,q(k,i)}^{(1)},J_{k,q(k,i)}^{(2)})$ be the great pair of small cubes in $J_{k,q(k,i)}$.
        \item For all $k\in Z$, we have that the $D$-geodesic $P$ hits $V_{\varepsilon r}(\bm z_k)$.

        \item  For all $k\in Z$, we have that $|\mathcal{E}_Z^+\cap \Lambda_k|\le M_0$, where $M_0$ is defined in \eqref{M1-1}.

        \item  Under $\mathcal{E}_Z^+$, we have that $J_{k,q(k,i)}^{(2)}$ is connected to $J_{k,q(k,i+1)}^{(1)}$ by a long edge for all $k\in Z$ and  $i\in[1,3^d(b_2\alpha)^{-2d}]_\mathds{Z}$.
    \end{enumerate}
\end{definition}

Our intuition behind  Definition \ref{F_Z} (6) of $\mathsf{F}_{Z,\varepsilon}(\mathcal{E},\mathcal{E}_Z^+)$ is that the edges which we add into $V_{3\varepsilon r}(\bm z_k)$ can change the behavior of a geodesic and attract it into $V_{3\varepsilon r}(\bm z_k)$. However, when a $D_Z^+$-geodesic enters $V_{3\varepsilon r}(\bm z_k)$, it may not pass through those great pairs of small cubes.
To address this difficulty, we aim to show that if $Z \in { \mathcal{W}}_\varepsilon$ and $\overline{\mathsf{F}}_{Z,\varepsilon}(\mathcal{E},\mathcal{E}_Z^+)$ occurs, then there exists a subset $Z'\subset Z$ for which $\# Z'$ is at least a constant factor of $\# Z$ and $\mathsf{F}_{Z',\varepsilon}(\mathcal{E},\mathcal{E}_{Z'}^+)$ occurs (see Lemma \ref{Lemma4.13DG21} below).


As preparation for our proof of Proposition \ref{Prop4.5DG21}, we show that the diameter of a super super good cube $V_{3\varepsilon r}(\bm z_k)$ will be much shorter if we add edges $\widetilde{\mathcal{E}}_{k}$ into $V_{3\varepsilon r}(\bm z_k)$. This, in turn, enables us to ``attract'' $D_Z^+$-geodesics to enter $V_{3\varepsilon r}(\bm z_k)$.

\begin{lemma}\label{LemmaNewGeoShort}
    Assume that $\overline{\mathsf{F}}_{Z,\varepsilon}(\mathcal{E},\mathcal{E}_Z^+)$ occurs. Then for any $k\in Z$ and $\bm u,\bm v\in V_{3\varepsilon r}(\bm z_k)$,
         \begin{equation*}\label{InequalityNewGeoShort}
        D_Z^+(\bm u,\bm v)\le5\widetilde{C}(b_2\alpha)^{2\theta}\log((b_2\alpha)^{-2})(\varepsilon r)^\theta.
    \end{equation*}
\end{lemma}
\begin{proof}
Assume that  $\overline{\mathsf{F}}_{Z,\varepsilon}(\mathcal{E},\mathcal{E}_Z^+)$ occurs and fix $k\in Z$.
According to the fact that $V_{3\varepsilon r}(\bm z_k)$ is super super good and Definition \ref{super good} (2) of super good cube (which also holds for a super super cube), we have that for any $i\in[1,3^d(b_2\alpha)^{-2d}]_\mathds{Z}$ and $j\in [1,(3K)^d]_\mathds{Z}$ (recalling $J''(\bm w''_{k,i})$ from Definition \ref{very very nice} and $J_{k,j}$ from \eqref{prob-good}, respectively),
\begin{equation*} 
\begin{split}
{\rm diam}(J''(\bm w''_{k,i});D)
&\le ((b_2\alpha)^2\varepsilon r)^\theta \widetilde{C}\sup_{t\in[0,1]}(t^\theta\log(\frac{2}{(b_2\alpha)^2t}))\\
&\le 2\widetilde{C}(b_2\alpha)^{2\theta}\log(1/(b_2\alpha)^2) (\varepsilon r)^\theta
\end{split}
\end{equation*}
and
\begin{equation*} 
    {\rm diam}(J_{k,j};D)\le K^{-\theta}(\varepsilon r)^\theta \widetilde{C}\sup_{t\in[0,1]}(t^\theta\log(2K/t))\le 2\widetilde{C}K^{-\theta}\log K (\varepsilon r)^\theta
\end{equation*}
when $K$ is sufficiently large (depending only on $\beta,d,\alpha$ and the laws of $D$ and $\widetilde{D}$ here).

From Definition \ref{Fbar} (5), we see that under $\mathcal{E}_Z^+$, $J_{k,q(k,i)}^{(2)}$ is connected to $J_{k,q(k,i+1)}^{(1)}$ by a long edge for all $k\in Z$ and  $i\in[1,3^d(b_2\alpha)^{-2d}]_\mathds{Z}$.
Thus for any $\bm u,\bm v\in V_{3\varepsilon r}(\bm z_k)$, we can construct a path from $\bm u$ to $\bm v$ as follows.
Let $i'$ and $i''$ be the subscripts such that the nice cubes $J_{k,q(k,i')}^{(1)}$ and $J_{k,q(k,i'')}^{(2)}$ satisfy
that ${\rm dist}(\bm u,J_{k,q(k,i')}^{(1)};\|\cdot\|_\infty)\leq 2(b_2\alpha)^2 \varepsilon r$ and ${\rm dist}(\bm v,J_{k,q(k,i'')}^{(2)};\|\cdot\|_\infty)\leq 2(b_2\alpha)^2 \varepsilon r$, respectively.
Without loss of generality we assume $i'\leq i''$.
Then we construct the path
by starting at $\bm u$ and proceeding to  $J_{k,q(k,i')}^{(1)}$ (along a $D$-geodesic).
This path then passes through great pairs of small cubes $(J_{k,q(k,i)}^{(1)}, J_{k,q(k,i)}^{(2)})$ along the long edges connecting $J_{k,q(k,i)}^{(2)}$ to $J_{k,q(k,i+1)}^{(1)}$ for all $i\in[i',i''-1]$, before reaching $\bm v$ at $J_{k,q(k,i'')}^{(2)}$. Finally, the path goes along a $D$-geodesic from $J_{k,q(k,i'')}^{(2)}$ to $\bm v$ (see Figure \ref{PathWithNewEdge}).
\begin{figure}[htbp]
\centering
\includegraphics[scale=0.6]{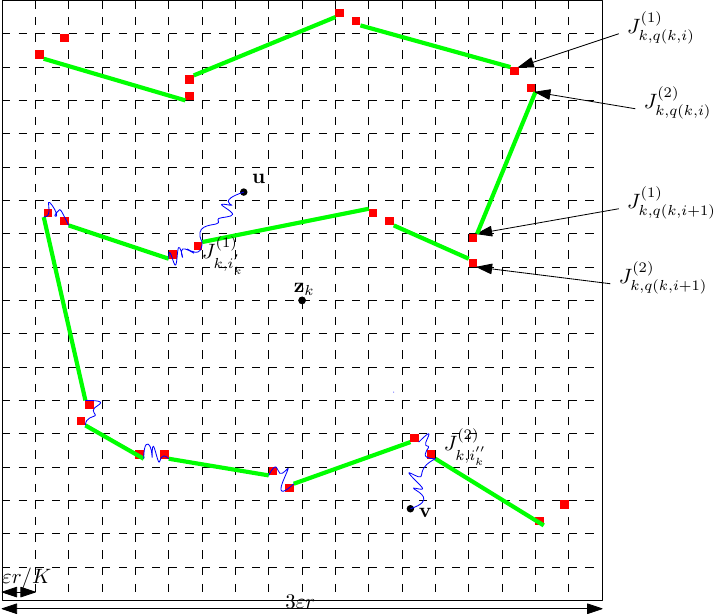}
\caption{Illustration for the proof of Lemma \ref{LemmaNewGeoShort}. The red cubes represent great pairs of cubes $(J^{(1)}_{k,q(k,i)},J^{(2)}_{k,q(k,i)})$, and the green lines indicate the edges that we have added. The path starts at the point $\bm u$ and goes along a $D$-geodesic to find the great pair of cubes $(J_{k,q(k,i')}^{(1)}, J_{k,q(k,i')}^{(2)})$ closest to $\bm u$. Then the path passes through some great pairs of cubes $(J_{k,q(k,i)}^{(1)}, J_{k,q(k,i)}^{(2)})$ and uses a new edge jumping from $J_{k,q(k,i)}^{(2)}$ to $J_{k,q(k,i+1)}^{(1)}$. Finally, the path walks from $(J_{k,q(k,i'')}^{(1)}, J_{k,q(k,i'')}^{(2)})$ (which is the great pair of cubes closest to $\bm v$) to $\bm v$ through a $D$-geodesic. The detailed relationship of inclusion between $J_{k,i}$, $J'(\bm w'_{k,i})$, and $J''(\bm w''_{k,i})$ can be found in Figure \ref{Proof733} (replacing $\bm k$, $J_{\bm k,q(\bm k,i)}$, $J'(\bm w'_{\bm k,q(\bm k,i)})$ and $J''(\bm w''_{\bm k,i})$ with $z_k$, $J_{k,i}$, $J'(\bm w'_{k,i})$, and $J''(\bm w''_{k,i})$, respectively).}
  \label{PathWithNewEdge}
\end{figure}
These facts combined with \eqref{M0} yield that
    \begin{equation*}
    \begin{split}
        D_Z^+(\bm u,\bm v)&\le 4\widetilde{C}(b_2\alpha)^{2\theta}\log(1/(b_2\alpha)^2)(\varepsilon r)^\theta+ 2\cdot 3^d(b_2\alpha)^{-2d}\widetilde{C} K^{-\theta} (\log K)(\varepsilon r)^\theta\\
        &\leq 4\widetilde{C}\left[(b_2\alpha)^{2\theta}\log(1/(b_2\alpha)^2)+ 3^d(b_2\alpha)^{-2d} K^{-\theta} (\log K)\right](\varepsilon r)^\theta\\
        &\leq  5\widetilde{C}(b_2\alpha)^{2\theta}\log((b_2\alpha)^{-2})(\varepsilon r)^\theta\quad \text{by \eqref{M0}}.
        \end{split}
    \end{equation*}
    Hence, the proof is complete.
\end{proof}

By Lemma \ref{LemmaNewGeoShort}, we can then establish an upper bound for $D_Z^+ (\bm x,\bm y)$ in terms of $D(\bm x,\bm y)$ conditioned on the event $\overline{\mathsf{F}}_{Z,\varepsilon}(\mathcal{E},\mathcal{E}_Z^+)$. This bound is valid due to the fact that the $D$-geodesic from  $\bm x$ to $\bm y$ must enter the super super good cubes $V_{3\varepsilon r}(\bm z_k)$ for $k \in Z$ (by Definition \ref{Fbar} (3)). Since a plethora of edges are added to $V_{3\varepsilon r}(\bm z_k)$, their $D_Z^+$ -distance is reduced significantly. This, in turn, enables us to identify many (with the number of $\# Z$) ``shortcuts''
along the $D$-geodesic with small $D_Z^+$ -length.
\begin{lemma}\label{Lemma4.8DG21}
     Let $Z \in { \mathcal{W}}_\varepsilon$ and assume that $\overline{\mathsf{F}}_{Z,\varepsilon}(\mathcal{E},\mathcal{E}_Z^+)$ occurs.
     Then
    \begin{equation}\label{InequalityLemma4.8DG21}
        D_Z^+ (\bm x,\bm y) \le D(\bm x,\bm y) - \left((b_2\alpha)^\theta -5\widetilde{C}(b_2\alpha)^{2\theta}\log((b_2\alpha)^{-2})\right) ( \varepsilon r)^\theta \# Z.
    \end{equation}
\end{lemma}
\begin{proof} Recall that $P$ is the fixed $D$-geodesic from $\bm x$ and $\bm y$. By Definition \ref{Fbar} (3), we see that $P$ hits $V_{\varepsilon r}(\bm z_k)$ for all $k\in Z$. Let $s_k$ be the first time that $P$ first hits $V_{\varepsilon r}(\bm z_k)$, and let $t_k$ be the first time after $s_k$ that $P$ exits  $V_{3\varepsilon r}(\bm z_k)$ (see Figure \ref{Proof725}).
    \begin{figure}[htbp]
\centering
\includegraphics[scale=0.8]{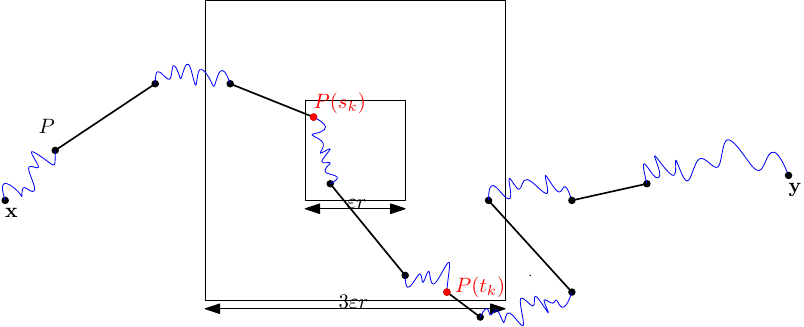}
\caption{Illustration for the time $s_k$ and $t_k$. Specifically, $s_k$ is the first time that $P$ first hits the cube $V_{\varepsilon r}(\bm z_k)$, while $t_k$ is the first time after $s_k$ that  $P$ exits the cubes $V_{3\varepsilon r}(\bm z_k)$.}
\label{Proof725}
\end{figure}
    Then by $D_Z^+\leq D$,

    \begin{equation}\label{InequalityLem4.8Shortcut}
        \begin{split}
            D_Z^+(\bm x,\bm y)
                &\le \mathrm{len}\left(P-\bigcup_{k\in Z}P([s_k,t_k]);D_Z^+\right)+\sum_{k\in Z}D_Z^+(P(s_k),P(t_k)) \\
                &\leq \mathrm{len}\left(P-\bigcup_{k\in Z}P([s_k,t_k]);D\right)+\sum_{k\in Z}D(P(s_k),P(t_k))\\&~~-\sum_{k\in Z}\left[D(P(s_k),P(t_k))-D_Z^+(P(s_k),P(t_k))\right] \\
                &= D(\bm x,\bm y)-\sum_{k\in Z}\left[D(P(s_k),P(t_k))-D_Z^+(P(s_k),P(t_k))\right].
        \end{split}
    \end{equation}
   On the one hand, by  Definition \ref{superdouble} (2) of super super good cube,
    \begin{equation}\label{InequalityLem4.8DLong}
        |P(s_k)-P(t_k)|\geq \alpha\varepsilon r\quad \text{and}\quad D(P(s_k),P(t_k))\ge (b_2\alpha)^\theta (\varepsilon r)^\theta.
    \end{equation}
    On the other hand, by Lemma \ref{LemmaNewGeoShort}, we get that
\begin{equation*}
\begin{split}
        D_Z^+(P(s_k),P(t_k))
        &\leq 5\widetilde{C}(b_2\alpha)^{2\theta}\log((b_2\alpha)^{-2})(\varepsilon r)^\theta.
        \end{split}
\end{equation*}

    Combining this with (\ref{InequalityLem4.8Shortcut}) and (\ref{InequalityLem4.8DLong}), we get (\ref{InequalityLemma4.8DG21}).
    Note that
    $$(b_2\alpha)^\theta- 5\widetilde{C}(b_2\alpha)^{2\theta}\log((b_2\alpha)^{-2})>0
    $$
    due to \eqref{M0new}.
\end{proof}

The following lemma plays an important role in proving Proposition \ref{Prop4.5DG21}, as it enables us to generate configurations $Z$ for which $\mathsf{F}_{Z,\varepsilon}(\mathcal{E},\mathcal{E}_Z^+)$, instead of just $\overline{\mathsf{F}}_{Z,\varepsilon}(\mathcal{E},\mathcal{E}_Z^+)$, occurs.
\begin{lemma}\label{Lemma4.13DG21}
    There is a constant $c_4 > 0$, depending only on the parameters and the laws of $D$ and $\widetilde{D}$, such that the following is true. Let $ Z \in {\mathcal{W}}_\varepsilon$, and assume that $\overline{\mathsf{F}}_{Z,\varepsilon}(\mathcal{E},\mathcal{E}_Z^+)$ occurs. There exists $Z^\prime \subset Z$ such that $\mathsf{F}_{Z^\prime,\varepsilon}(\mathcal{E},\mathcal{E}_{Z'}^+)$ occurs and $\# Z^\prime \ge c_4 \# Z$.
\end{lemma}
\begin{proof}
    Step 1: iteratively removing ``bad'' points. It is immediate from the definition of $\overline{\mathsf{F}}_{Z,\varepsilon}(\mathcal{E},\mathcal{E}_Z^+)$ that if $\overline{\mathsf{F}}_{Z,\varepsilon}(\mathcal{E},\mathcal{E}_Z^+)$ occurs and $Z^\prime \subset Z$ is non-empty, then $\overline{\mathsf{F}}_{Z^\prime,\varepsilon}(\mathcal{E},\mathcal{E}_{Z'}^+)$ occurs. So, we need to produce a set $Z^\prime\subset Z$
    such that $\# Z^\prime$ is at least a constant times $\# Z$ and Definition \ref{F_Z} (6) of $\mathsf{F}_{Z^\prime,\varepsilon}(\mathcal{E},\mathcal{E}_{Z'}^+)$ holds.

    Recall that $P_Z^{+}$ is the fixed $D_Z^{+}$-geodesic from $\bm x$ to $\bm y$. We will construct $Z^\prime$ by iteratively removing the ``bad'' points $k \in Z$ such that $P_{Z}^+$ does not use any edge in $\mathcal{E}_k^+ \setminus\mathcal{E}$.
    To this end, let $Z_0 := Z$. Inductively, suppose that $m \in \mathds{N}$ and $Z_m \subset Z$
    has been defined. Let $P_{Z_m}^+$ be the $D_{Z_m}^+$-geodesic from $\bm x$ to $\bm y$ and let $Z_{m+1}$ be the set of $k \in Z_m$ such that $P_{Z_m}^+$ passes through some great pairs of cubes $(J_{k,q(k,i)}^{(1)},J_{k,q(k,i)}^{(2)})$ in $V_{3\varepsilon r}(\bm z_k)$.

    If $Z_{m+1} = Z_m$, then the event $\mathsf{F}_{Z_m,\varepsilon}(\mathcal{E},\mathcal{E}_{Z_m}^+)$ occurs. So, to prove the lemma, it suffices to show that the above procedure stabilizes before $\# Z_m$ gets too much smaller
    than $\# Z$. More precisely, we will show that there exists $c_4 > 0$ as in the lemma statement such that
    \begin{equation}\label{LowerBoundGoalLemma4.13}
        \# Z_m \ge c_4\# Z,\quad \forall m \in \mathds{N}.
    \end{equation}

    Since $Z_{m+1} \subset Z_m$ for each $m \in \mathds{N}$ and $Z_0$ is finite, it follows that there must be some $m \in \mathds{N}$
    such that $Z_m = Z_{m+1}$. We know that $\mathsf{F}_{Z_m,\varepsilon}(\mathcal{E},\mathcal{E}_{Z_m}^+)$ occurs for any such $m$, so (\ref{LowerBoundGoalLemma4.13}) implies the lemma statement.

    It remains to prove \eqref{LowerBoundGoalLemma4.13}. The idea of the proof is as follows. At each step of our iterative
    procedure, we only remove points $k \in Z_m$ for which $P_{Z_m}^+\cap{V_{3\varepsilon r}(\bm z_k)}$ 
    is small, in a certain sense. Using this, we can show that $D_{Z_{m+1}}^+(\bm x,\bm y)$ is not too much bigger than $D_{Z_m}^+(\bm x,\bm y)$
    (see \eqref{InequalityDmNotTooBiggerDm+1} below). Iterating this leads to an upper bound for $D_{Z_m}^+(\bm x,\bm y)$ in terms of $D_Z^+(\bm x,\bm y)$
    (see \eqref{InequalityDNotTooBiggerDm} below). We then use the fact that $D_{Z}^+(\bm x,\bm y)$ has to be substantially smaller than
    $D(\bm x,\bm y)$ (Lemma \ref{Lemma4.8DG21}) together with our upper bound for the amount of time that $P_{Z_m}^+$ spends
    in each of the $V_{3\varepsilon r}(\bm z_k)$ (Definition \ref{superdouble} (2) of super super good cubes), in order to obtain (\ref{LowerBoundGoalLemma4.13}).

    Step 2: Comparison of $D_{Z_m}^+(\bm x,\bm y)$ and $D(\bm x,\bm y)$. For $k\in Z_m$,  let $\sigma_k$  be the time $P_{Z_m}^+$ first hits $V_{3\varepsilon r}(\bm z_k)$ and let $\tau_k$ be the time $P_{Z_m}^+$ last leaves $V_{3\varepsilon r}(\bm z_k)$.
    Divide $P_{Z_m}^+\cap V_{3\varepsilon r}(\bm z_k)$ into $L_k$ distinct excursions $P_{Z_m}^+([s_l,t_l])$, ensuring that each excursion is contained within $V_{3\varepsilon r}(\bm z_k)$.
    To be precise, we define $s_{k,1}<t_{k,1}<\cdots<s_{k,L_k}<t_{k,L_k}$ iteratively. Let $t_{k,0}=0$. If we have defined $t_{k,l}$ for $l\in\mathds{N}$, then let $s_{k,l+1}$ be the time $P_{Z_m}^{+}$ first exits $V_{3\varepsilon r}(\bm z_k)$ after $t_{k,l}$. In particular, if the time does not exist, we let $s_{k,l+1}=\infty$. Then let $t_{k,l+1}$ be the time $P_{Z_m}^{+}$ first hits $V_{3\varepsilon r}(\bm z_k)$ after $s_{k,l+1}$. Let $L_k$ be the number of finite $s_{k,l}$'s.
    Since $P_{Z_m}^+$ does not pass through any great pair of small cubes $(J_{k,q(k,i)}^{(1)},J_{k,q(k,i)}^{(2)})$ in $V_{3\varepsilon r}(\bm z_k)$ for $k\in Z_m\setminus Z_{m+1}$,
    we get from Lemma \ref{upp-strange-Ik} below that

    \begin{equation}\label{uppD_Z_m+1}
    \begin{split}
         \sum_{l=1}^{L_k} D_{Z_{m+1}}^+(P_{Z_m}^+(s_{k,l}),P_{Z_m}^+(t_{k,l}))&\leq\sum_{l=1}^{L_k}   D(P_{Z_m}^+(s_{k,l}),P_{Z_m}^+(t_{k,l}))\quad \quad \text{(since $D_{Z_{m+1}}^+(\cdot,\cdot)\leq D(\cdot,\cdot)$)}\\
         \leq \sum_{l=1}^{L_k}D_{Z_m}^+(P_{Z_m}^+(s_{k,l}),P_{Z_m}^+(t_{k,l}))& + 10\cdot 2^\theta\widetilde{C}^2 (b_2\alpha)^{1.4\theta}(\log (2(b_2\alpha)^{-2}))^2(\varepsilon r)^\theta\quad  \text{(by Lemma \ref{upp-strange-Ik} below)}
         \end{split}
    \end{equation}
    for any $k\in Z_m\setminus Z_{m+1}$. 
    Then we get
    \begin{equation}\label{InequalityDmNotTooBiggerDm+1}
        \begin{aligned}
            &D_{Z_{m+1}}^+(\bm x,\bm y)\\
            &\le D_{Z_m}^+(\bm x,\bm y)+\sum_{k\in Z_m \setminus Z_{m+1}}\sum_{l=1}^{L_k} \left[D_{Z_{m+1}}^+(P_{Z_m}^+(s_l),P_{Z_m}^+(t_l))-D_{Z_m}^+(P_{Z_m}^+(s_l),P_{Z_m}^+(t_l))\right]\\
            &\le D_{Z_m}^+(\bm x,\bm y)+  10\cdot 2^\theta\widetilde{C}^2 (b_2\alpha)^{1.4\theta}(\log (2(b_2\alpha)^{-2}))^2(\varepsilon r)^\theta(\# Z_m - \# Z_{m+1}),
        \end{aligned}
    \end{equation}
    where the last inequality used \eqref{uppD_Z_m+1}.
    Similarly, for $n=0,1,\cdots,m-1$, we have
    \begin{equation}\label{DmNotTooBigIterate}
        D_{Z_{n+1}}^+(\bm x,\bm y)\le D_{Z_n}^+(\bm x,\bm y)+  10\cdot 2^\theta\widetilde{C}^2 (b_2\alpha)^{1.4\theta}(\log (2(b_2\alpha)^{-2}))^2(\varepsilon r)^\theta(\# Z_n - \# Z_{n+1}).
    \end{equation}
    Summing \eqref{DmNotTooBigIterate} over $n\in[0,m-1]_\mathds{Z}$, we get that
    \begin{equation}\label{DZmDZ}
        \begin{aligned}
            D_{Z_m}^+(\bm x,\bm y)&\le D_Z^+(\bm x,\bm y)+ 10\cdot 2^\theta\widetilde{C}^2 (b_2\alpha)^{1.4\theta}(\log (2(b_2\alpha)^{-2}))^2(\varepsilon r)^\theta(\# Z - \# Z_m)\\
            &\le D_Z^+(\bm x,\bm y)+ 10\cdot 2^\theta\widetilde{C}^2 (b_2\alpha)^{1.4\theta}(\log (2(b_2\alpha)^{-2}))^2(\varepsilon r)^\theta\# Z.
        \end{aligned}
    \end{equation}
    Here the last inequality holds because $\# Z_m\geq 0$. Then applying Lemma \ref{Lemma4.8DG21} to $Z$ gives that
    \begin{equation}\label{applyLem48}
        D_Z^+(\bm x,\bm y)\leq D(\bm x,\bm y)-((b_2\alpha)^\theta-5\cdot 2^\theta\widetilde{C}^2 (b_2\alpha)^{1.4\theta}(\log (2(b_2\alpha)^{-2}))^2) \# Z.
    \end{equation}
    Now we plug \eqref{applyLem48} into \eqref{DZmDZ} and  get that
    \begin{equation}\label{InequalityDNotTooBiggerDm}
            D_{Z_m}^+(\bm x,\bm y)\le  D(\bm x,\bm y)-\left((b_2\alpha)^\theta-  15\cdot 2^\theta\widetilde{C}^2 (b_2\alpha)^{1.4\theta}(\log (2(b_2\alpha)^{-2}))^2\right)(\varepsilon r)^\theta\# Z.
    \end{equation}


    Step 3: conclusion. Note that by Definition \ref{super good} (2) of super good cubes again,
    \begin{equation*}
        \mathrm{diam}(V_{3\varepsilon r}(\bm z_k);D)\le\widetilde{C}(3\varepsilon r)^\theta\sup_{t\in[0,3]}\left(t^\theta\log(2/t)\right)\le C_\theta \widetilde{C}(\varepsilon r)^\theta.
    \end{equation*}
    Here $C_\theta=3^\theta\sup_{t\in[0,3]}(t^\theta\log(2/t))>0$ is a constant only depending on $\theta(\beta,d)$. Then we get
    \begin{equation}\label{Inequality4.5Step3DmNotSmall}
        D(\bm x,\bm y)\le D_{Z_m}^+(\bm x,\bm y)+\sum_{k\in Z_m}\mathrm{diam}(V_{3\varepsilon r}(\bm z_k);D)\le D_{Z_m}^+(\bm x,\bm y)+ C_\theta\widetilde{C}(\varepsilon r)^\theta\# Z_m.
    \end{equation}
    Then combine (\ref{InequalityDNotTooBiggerDm}) and (\ref{Inequality4.5Step3DmNotSmall}) and we get
    \begin{equation*}
    \begin{split}
        &D(\bm x,\bm y)- C_\theta \widetilde{C}(\varepsilon r)^\theta\# Z_m\\
        &\le D(\bm x,\bm y)-\left((b_2\alpha)^\theta-  15\cdot 2^\theta\widetilde{C}^2 (b_2\alpha)^{1.4\theta}(\log (2(b_2\alpha)^{-2}))^2\right)(\varepsilon r)^\theta\# Z.
        \end{split}
    \end{equation*}
     Since $\alpha,b_2$ and $\widetilde{C}$ satisfy \eqref{M0new}, we get
    \begin{equation*}
        \# Z_m\ge c_4\# Z
    \end{equation*}
    for $c_4:=\left((b_2\alpha)^\theta-  15\cdot 2^\theta\widetilde{C}^2 (b_2\alpha)^{1.4\theta}(\log (2(b_2\alpha)^{-2}))^2\right)/(C_\theta\widetilde{C})>0$.
\end{proof}

\begin{lemma}\label{upp-strange-Ik}
      Assume that $\overline{\mathsf{F}}_{Z,\varepsilon}(\mathcal{E},\mathcal{E}_Z^+)$ occurs. For any $k\in Z$, let $\{P_{Z}^+([s_l,t_l])\}_{l=1}^L$ be all the $L$ excursions of $P_Z^+$ in $V_{3\varepsilon r}(\bm z_k)$ {\rm(}the dependence on $k$ is suppressed, see the first paragraph in step 2 in the proof of Lemma \ref{Lemma4.13DG21}{\rm)}.
      If $P_{Z}^{+}$ does not pass through any great pair of $(J^{(1)}_{k,q(k,i)},J^{(2)}_{k,q(k,i)})$ in $V_{3\varepsilon r}(\bm z_k)$, then
        \begin{equation*}\label{upp-strange-Ik-ineq}
            \sum_{l=1}^L D(P_Z^+(s_l),P_Z^+(t_l))\leq \sum_{l=1}^L D_Z^+(P_Z^+(s_l),P_Z^+(t_l))+5\cdot 2^\theta\widetilde{C}^2 (b_2\alpha)^{1.4\theta}(\log (1/(2(b_2\alpha)^2)))^2(\varepsilon r)^\theta.
        \end{equation*}
    \end{lemma}
\begin{proof}
    For each $l\in [1,L]_\mathds{Z}$, let $I_l$ be the set of subscripts $i$ such that there exists at least one newly-added edge in $P_Z^+([s_l, t_l])$ connecting $J_{k,q(k,i)}^{(2)}$ to $J_{k,q(k,i+1)}^{(1)}$. Denote $\kappa_l=\#I_l$ (see Figure \ref{Proof727-0}).
        \begin{figure}[htbp]
\centering
\includegraphics[scale=0.6]{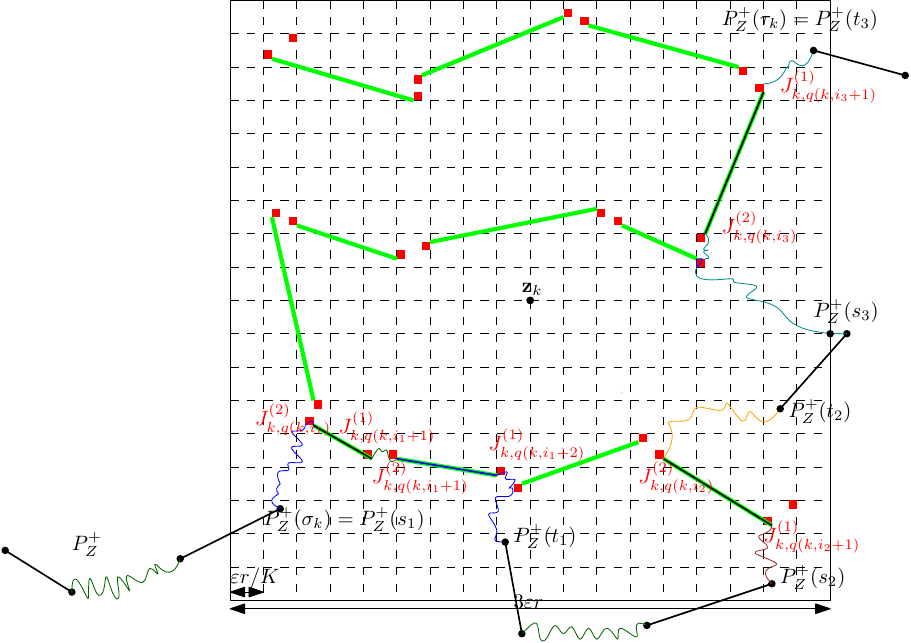}
\caption{Illustration of the definitions of $I_l$ and $\kappa_l$. In this picture, $P_Z^+$ has three distinct excursions in $V_{3\varepsilon r}(\bm z_k)$, i.e., $L=3$. The first one $P_Z^+([s_1,t_1])$ (formed with black lines and blue curves) uses two newly-added edges (green lines) connecting $J_{k,q(k,i_1)}^{(2)}$ to $J_{k,q(k,i_1+1)}^{(1)}$ (red cubes) and $J_{k,q(k,i_1+1)}^{(2)}$ to $J_{k,q(k,i_1+2)}^{(1)}$, respectively. Thus $I_1=\{i_1,i_1+1\}$ and $\kappa_1=2$. Similarly, we can see from the picture that the second excursion is formed with black lines and orange curves, while the third excursion is formed with black lines and darkcyan curves.  Moreover, $I_2=\{i_2\}$, $\kappa_2=1$, $I_3=\{i_3\}$ and $\kappa_3=1$.
}
\label{Proof727-0}
\end{figure}
   Now for each such $i\in I_l$, let $e(i,1),\cdots, e(i,N(i))$ be the sequence of newly-added edges connecting $J_{k,q(k,i)}^{(2)}$ to $J_{k,q(k,i+1)}^{(1)}$ used by $P_Z^+([s_l,t_l])$ in order. Denote $\bm x(i)\in J_{k,q(k,i)}^{(2)}$ by the one end point of $e(i,1)$ and $\bm y(i)\in J_{k,q(k,i+1)}^{(1)}$ by the one end point of $e(i,N(i))$.
         \begin{figure}[htbp]
\centering
\includegraphics[scale=0.6]{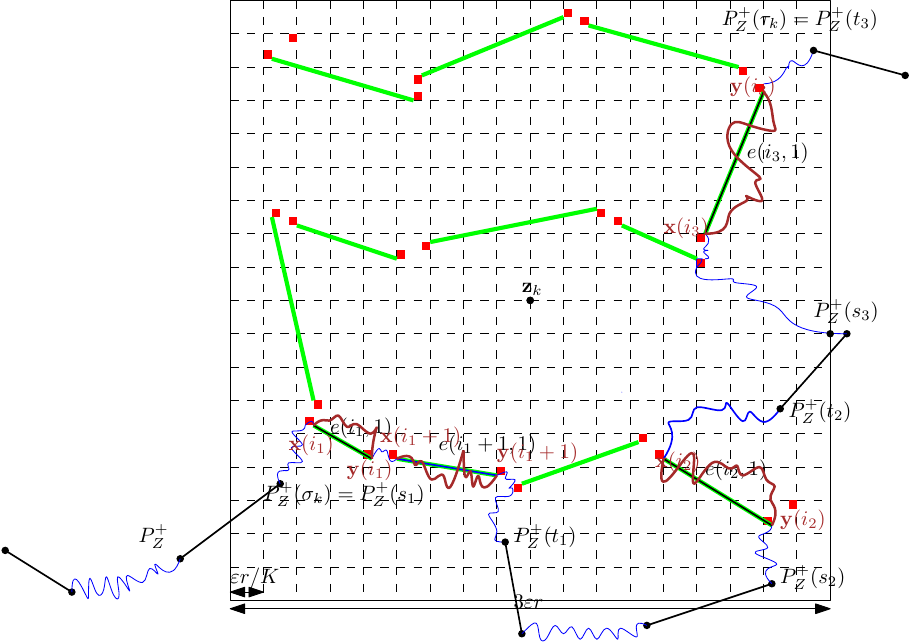}
\caption{Illustration of the definitions of $e(i,\cdot)$, $\bm x(i)$ and $\bm y(i)$. In this picture, $e(i_1,1)=\langle \bm x(i_1),\bm y(i_1) \rangle, e(i_1+1,1)=\langle \bm x(i_1+1),\bm y(i_1+1)\rangle, e(i_2,1)=\langle \bm x(i_2),\bm y(i_2) \rangle$ and $e(i_3,1)=\langle \bm x(i_3),\bm y(i_3) \rangle$ are newly-added edges in $P_Z^+([\sigma_k,\tau_k])$. For each such edge, we use a $D$-geodesic $P_l$ (brown curve) to replace this edge.
}
\label{Proof727-1}
\end{figure}

   Then from $P_Z^+([s_l, t_l])$ we can construct a corresponding path $P_l$ that does not use the newly-added edge as follows: For each $i\in I_l$, the path $P_l$ walks from $\bm x(i)$ to $\bm y(i)$ by a $D$-geodesic (see Figure \ref{Proof727-1}); in all other parts of $P_Z^+([s_l, t_l])$, $P_l$ follows $P_Z^+([s_l, t_l])$. Thus, we can see that
    \begin{equation}\label{originUB}
    \begin{split}
        D(P_Z^+(s_l),P_Z^+(t_l))&\leq \text{len}(P_l;D)\\
        &\leq \kappa_l\sup_i \sup_{\bm u\in J_{k,q(k,i)}^{(2)},\bm v\in J_{k,q(k,i+1)}^{(1)}}D(\bm u,\bm v;V_{3\varepsilon r}(\bm z_k))+D_Z^+(P_Z^+(s_l),P_Z^+(t_l)).
        \end{split}
    \end{equation}
    In addition, since ${\rm dist}(\bm u,\bm v;\|\cdot\|_\infty)\le 2(b_2\alpha)^2\varepsilon r$ for each $\bm u\in J_{k,q(k,i)}^{(2)}$ and $\bm v\in J_{k,q(k,i+1)}^{(1)}$, we get from Definition \ref{super good} (2) that
    \begin{equation}\label{supDJk}
       \sup_{\bm u\in J_{k,q(k,i)}^{(2)},\bm v\in J_{k,q(k,i+1)}^{(1)}}D(\bm u,\bm v;V_{3\varepsilon r}(\bm z_k))\leq 2^\theta \widetilde{C} (b_2\alpha)^{2\theta} \log (1/(2(b_2\alpha)^2))(\varepsilon r)^\theta.
    \end{equation}
    Plugging \eqref{supDJk} into \eqref{originUB} and then summing over $l\in[1,L]_\mathds{Z}$, we get that
        \begin{equation}\label{delete-point-step1}
            \sum_{l=1}^LD(P_Z^+(s_l),P_Z^+(t_l))\le \sum_{l=1}^L\left[\kappa_l 2^\theta \widetilde{C} (b_2\alpha)^{2\theta} \log (1/(2(b_2\alpha)^2))(\varepsilon r)^\theta+D_Z^+(P_Z^+(s_l),P_Z^+(t_l))\right].
        \end{equation}
    Thus it suffices to give an upper bound on $\kappa_l$. Since $P_Z^+$ is a geodesic under $D_Z^+$, we have that
    \begin{equation}\label{delete-point-step2}
        \sum_{l=1}^L D^+_Z(P^+_Z(s_l),P^+_Z(t_l))\le {\rm diam}(V_{3\varepsilon r}(\bm z_k);D_Z^+)\le 5\widetilde{C}(b_2\alpha)^{2\theta}\log(1/(b_2\alpha)^2)(\varepsilon r)^\theta.
    \end{equation}
    Here the last inequality used Lemma \ref{LemmaNewGeoShort}.

Next, we will give an upper bound of $\kappa_l$ in terms of $ D_Z^+(P_Z^+(s_l),P_Z^+(t_l))$ (which, combined with \eqref{delete-point-step2}, gives a lower bound on $\kappa_l$).
For any $l$, let $e_1, \cdots, e_{q_l}$ be the sequence of newly-added edges used by $P^+_Z([s_l, t_l])$ in order. For $1\leq j \leq  q_l$, let $i_j$ be such that $e_j$ connects $J^{(2)}_{k, q(k,i_j)}$ and $J^{(1)}_{k, q(k,i_{j}+1)}$. It is clear that these $i_j$'s form the set $I_l$ defined in the first paragraph in this proof and $\# I_l=\kappa_l$.
Note that for each $i\in I_l$, there might be different $j_1,j_2\in [1,q_l]_\mathds{Z}$ such that $i_{j_1}=i_{j_2}=i$.
For convenience, we denote $\widetilde{e}_{i}$ as the last edge connecting $J^{(2)}_{k, q(k,i)}$ and $J^{(1)}_{k, q(k,i+1)}$ among $e_1,\cdots, e_N$.
By Definition \ref{very nice} (2) and the assumption that $P_Z^+$ does not pass through any great pair in $V_{3\varepsilon r}(\bm z_k)$, we have that for each $i\in I_l$, after using $\widetilde{e}_i$,
 $P_Z^+$ must escape the very nice cube (which contains $ J^{(2)}_{k,q(k,i)}$ (or $J^{(1)}_{k,q(k,i+1)}$) and is contained in a very very nice cube $ J''(\bm w''_{k,i}) $ (or $J''(\bm w''_{k,i+1}$))) before using new edges or leaving $V_{3\varepsilon r}(\bm z_k)$.
    From Definition \ref{very nice} (2), we see that in this case,  $P_{Z}^{+}$ must spend certain amount of time before leaving some very nice cube contained in $J''(\bm w''_{k,i})$ (or $J''(\bm w''_{k,i+1})$) (see Figure \ref{Proof727}).
    \begin{figure}[htbp]
\centering
\includegraphics[scale=0.6]{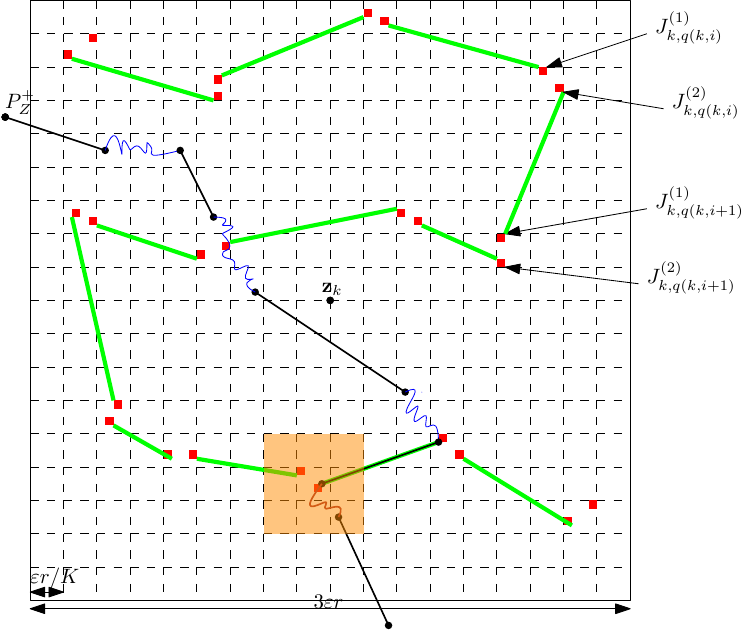}
\caption{Illustration that $P_Z^+$ does not pass through any great pair in $V_{3\varepsilon r}(\bm z_k)$. In this picture, $L=1$. Since $P_Z^+$ does not pass through any great pair in $V_{3\varepsilon r}(\bm z_k)$, although $P_Z^+$ maybe can hit one of cubes in the pair of $(J_{k,q(k,i)}^{(1)}, J_{k,q(k,i)}^{(2)})$ (marked with red) contained in a very nice cube (marked with orange) by a new edge (marked with green), it will not leave the cube through a new edge after using $\widetilde{e}_i$, so it will spend some time to leave the very nice cube (see the path marked with brown).}
\label{Proof727}
\end{figure}
    Hence, summing over such distances, we get that  under the assumption that $P_{Z}^{+}$ does not pass through any great pair of $(J^{(1)}_{k,q(k,i)},J^{(2)}_{k,q(k,i)})$,
    \begin{equation}\label{delete-point-step3}
        D_Z^+(P_Z^+(s_l),P_Z^+(t_l))\geq \kappa_l(b_2\alpha)^{2.6\theta}(\varepsilon r)^\theta.
    \end{equation}
    when $\kappa_l>0$ (note $q_l \geq \kappa_l$). Note that the lower bound on the right hand side is also true when $\kappa_l=0$.
    Thus \eqref{delete-point-step3} holds for any $l$. Summing \eqref{delete-point-step3} over $l$ and combining this with \eqref{delete-point-step2}, we get
    \begin{equation}\label{jUpperbound}
        \sum_{l=1}^L \kappa_l \leq 5\widetilde{C}(b_2\alpha)^{-0.6\theta}\log(1/(b_2\alpha)^2).
    \end{equation}

    Applying \eqref{jUpperbound} to \eqref{delete-point-step1}, we get
    \begin{equation*}
        \begin{split}
            \sum_{l=1}^L &D(P_Z^+(s_l),P_Z^+(t_l))\\
            &\leq \sum_{l=1}^LD_Z^+(P_Z^+(s_l),P_Z^+(t_l))+5\cdot 2^\theta\widetilde{C}^2 (b_2\alpha)^{1.4\theta}(\log (1/(2(b_2\alpha)^2)))^2(\varepsilon r)^\theta.
        \end{split}
    \end{equation*}
    Hence we complete the proof.
\end{proof}

We now present the
\begin{proof}[Proof of Proposition \ref{Prop4.5DG21}]
    By  Definition \ref{G} (4) of $\mathscr{G}_r^\varepsilon$, there exists $\mathcal{S}\in{ \mathcal{W}}_\varepsilon$ such that $\overline{\mathsf{F}}_{Z,\varepsilon}(\mathcal{E},\mathcal{E}_Z^+)$ occurs for every $Z\subset \mathcal{S}$ and $\# \mathcal{S}\ge\varepsilon^{-\theta/4}/(8\cdot 3^d)$.  Note that Definition \ref{Fbar} (4) holds because $\mathcal{E}_Z^+\subset \overline{\mathcal{E}}_Z^+$, which we have shown in the proof of Lemma \ref{EE+} (the paragraph after \eqref{RNestimateEE+}). 



    By Lemma \ref{Lemma4.13DG21}, for each $Z \subset \mathcal{S}$, there exists $Z^\prime \subset Z$ such that $\mathsf{F}_{Z^\prime,\varepsilon}(\mathcal{E},\mathcal{E}_{Z'}^+)$ occurs and $\# Z^\prime \ge c_4 \# Z$.
Fix (in some arbitrary manner) a choice of $Z^\prime$ for each $Z$, so that $Z \mapsto Z^\prime$ is a function from subsets
of $\mathcal{S}$ to subsets of $\mathcal{S}$ for which $\mathsf{F}_{Z^\prime,\varepsilon}$ occurs. For each $m \in \mathds{N}$, the number of sets $Z \subset \mathcal{S}$ such
that $\# Z = m $ is $\binom{\# \mathcal{S}}{m }$. Furthermore, since $Z^\prime \subset Z$, for each $\widetilde{Z} \subset Z$ for which $F_{\widetilde{Z},\varepsilon}(\mathcal{E})$ occurs and
 $\# \widetilde{Z} \in [c_4 m , m ]$, the number of $Z \subset \mathcal{S}$ such that $\# Z =m $ and $Z^\prime =  \widetilde{Z}$ is at most
\begin{equation*}
    \binom{\# \mathcal{S}}{m - \# \widetilde{Z}}\le\binom{\# \mathcal{S}}{\lfloor(1-c_4)m \rfloor}.
\end{equation*}
Consequently, the number of sets $\widetilde{Z} \subset \mathcal{S}$ for which $\mathsf{F}_{\widetilde{Z},\varepsilon}(\mathcal{E},\mathcal{E}_{\widetilde{Z}}^+)$ occurs and $\# \widetilde{Z} \in [c_4 m , m ]$ is at least
\begin{equation*}
    \binom{\# \mathcal{S}}{m}\left(\binom{\# \mathcal{S}}{\lfloor(1-c_4)m \rfloor}\right)^{-1}\succeq(\# \mathcal{S})^{c_4 m }\succeq (\varepsilon^{-\theta /4}/(8\cdot 3^d))^{c_4m}
\end{equation*}
with the implicit constant depending only on the parameters and $m $ and the laws of $D$ and $\widetilde{D}$ (in the last inequality we
used $\# \mathcal{S}\ge\varepsilon^{-\theta/4}/(8\cdot 3^d)$). This gives (\ref{LowerBoundProp4.5DG21}) for $c_3$ slightly smaller than $\theta c_4/(8\cdot 3^d)$.
\end{proof}

\subsection{Proof of Proposition \ref{Prop4.6DG21}}\label{SectionProp4.6DG21}
Recall that we have fixed $\bm x,\bm y\in V_r(\bm 0)$ with $|\bm x-\bm y|\geq \alpha r$ and $P$ is the fixed $D$-geodesic from $\bm x$ to $\bm y$.
To prove Proposition \ref{Prop4.6DG21}, we first identify the number of points $k \in [1,\varepsilon^{-d}]_{\mathds{Z}}$ that are potentially part of a $Z \in {\mathcal{W}}_{\varepsilon}$ for which $\mathsf{G}_{Z,\varepsilon}^-(\mathcal{E}_Z^-,\mathcal{E})$ occurs. To this end, we introduce the following definition for exposition convenience.

\begin{definition}\label{DefeZgood}
    Fix $\varepsilon>0$. 
    We say $k\in [1,\varepsilon^{-d}]_\mathds{Z}$ is {\it$\varepsilon$-excellent} if
    \begin{enumerate}
        \item $V_{3\varepsilon r}(z_k)$ is super super good with respect to $\mathcal{E}_{ k}^-$. Denote by $J^-_{k,q(k,1)},\cdots, J^-_{k,q(k,3^d(b_2\alpha)^{-2d})}$ the nice cubes selected as in  \eqref{choosenice} and let $(J_{k,q(k,i)}^{(1)-},J_{k,q(k,i)}^{(2)-})$ be the great pair of small cubes in $J^-_{k,q(k,i)}$.

        \item $P$ passes through at least one great pair of small cubes $(J^{(1)-}_{k,q(k,i)},J^{(2)-}_{k,q(k,i)})$ in $V_{3\varepsilon r}(\bm z_k)$.
    \end{enumerate}
\end{definition}

\begin{lemma}\label{Lemma4.15DG21}
For any $Z\in { \mathcal{W}}_\varepsilon$ such that $\mathsf{G}_{Z,\varepsilon}^-(\mathcal{E}_Z^-,\mathcal{E})$ occurs, we have that $k$ is $\varepsilon$-excellent whenever $k\in Z$.
\end{lemma}
\begin{proof}
    By the definition of $\varepsilon$-excellent, it suffices to show that for any $k\in Z$, $V_{3\varepsilon r}(\bm z_k)$ is super super good with respect to $\mathcal{E}_k^-$. Note that from Definition \ref{G-} (1), $V_{3\varepsilon r}(\bm z_k)$ is super super good with respect to $\mathcal{E}_Z^-$. Since the event that
    $$\text{$V_{3\varepsilon r}(\bm z_k)$ is super super good with respect to $\mathcal{E}_Z^-$}$$
     is determined by all edges in $\mathcal{E}_Z^- $ that have at least one end point in $V_{3\varepsilon r}(\bm z_k)$, we have that $V_{3\varepsilon r}(\bm z_k)$ is also super super good with respect to $\mathcal{E}_k^-$ since the end points of any edge in $\mathcal{E}_k^-\setminus \mathcal{E}_Z^-$ are not in $V_{3\varepsilon r}(\bm z_k)$.
    Thus the desired statement is true.
\end{proof}

\begin{lemma}\label{Lemma4.16DG21}
    There is a constant $C_3> 0$, depending only on the parameters and the laws of $D$ and $\widetilde{D}$, such that the following is true.
    Let $\varepsilon>0$, and let $Z\in  \mathcal{W}_\varepsilon$  with $Z\neq\emptyset$ and $Z'\in\mathcal{Z}_\varepsilon$. Assume that the event $\mathsf{G}_{Z,\varepsilon}^-(\mathcal{E}_Z^-,\mathcal{E})$ occurs, and that each $k\in Z'$ is
    $\varepsilon$-excellent. Then
    \begin{equation*}\label{Inequality4.16DG21}
        \# Z'\le C_3 \# Z.
    \end{equation*}
\end{lemma}

To prove Lemma \ref{Lemma4.16DG21}, We need the following lemma.
\begin{lemma}\label{Lemma4.17DG21}
    There exists $C_4>0$ depending only on the parameters and the laws of $D$ and $\widetilde{D}$ such that for any $Z\in \mathcal{Z}_\varepsilon$, if $\mathsf{G}_{Z,\varepsilon}^-(\mathcal{E}_Z^-,\mathcal{E};\bm x,\bm y)$ occurs, then the $D$-geodesic $P$ from $\bm x$ to $\bm y$ satisfies
    \begin{equation}\label{Inequality4.17DG21}
        \mathrm{len}(P;D_Z^-)\le D(\bm x,\bm y)+C_4 (\varepsilon r)^\theta\# Z.
    \end{equation}
\end{lemma}
\begin{proof}
For $Z\in\mathcal{Z}_\varepsilon$, assume that $\mathsf{G}_{Z,\varepsilon}^-(\mathcal{E}_Z^-,\mathcal{E};\bm x,\bm y)$ occurs.
Then it is clear that $V_{3\varepsilon r}(\bm z_k)$ is super super good with respect to $\mathcal{E}_Z^-$ for all $k\in Z$. Thus, from the triangle inequality and Definition \ref{super good} (2) of super good cube with respect to $\mathcal{E}_Z^-$,
    \begin{equation}\label{lenP-V}
        \mathrm{len}(P\cap V_{3\varepsilon r}(\bm z_k);D_Z^-)\le \sum_{j=1}^{(3K)^d}\mathrm{diam}(J_{k,j};D_Z^-)\le  2\cdot 3^d\widetilde{C}K^{d-\theta}(\log K)(\varepsilon r)^\theta.
    \end{equation}
Since $D$ and $D^-_Z$ lengths outside the cubes $V_{3\varepsilon r}(\bm z_k)$ for $k\in Z$ are the same, we have
    \begin{equation*}
    \begin{split}
        \mathrm{len}(P;D_Z^-)&\le\sum_{k\in Z}\mathrm{len}(P\cap V_{3\varepsilon r}(\bm z_k);D_Z^-)+\mathrm{len}\left(P\setminus\left(\bigcup_{k\in Z} V_{3\varepsilon r}(\bm z_k)\right);D\right)\\
        \end{split}
    \end{equation*}
We combine this with the inequality \eqref{lenP-V} to derive
$$
\mathrm{len}(P;D_Z^-)\leq 2\cdot 3^d\widetilde{C}K^{d-\theta}(\log K)(\varepsilon r)^\theta\#Z+D(\bm x,\bm y).
$$
Thus the proof is complete by taking $C_4=2\cdot 3^d\widetilde{C}K^{d-\theta}\log K$.
\end{proof}

\begin{proof}[Proof of Lemma \ref{Lemma4.16DG21}]
Assume nonempty $Z\in  \mathcal{W}_\varepsilon$  such that  $\mathsf{G}_{Z,\varepsilon}^-(\mathcal{E}_Z^-,\mathcal{E})$ occurs, and $Z'\in\mathcal{Z}_\varepsilon$ such that each $k\in Z'$ is $\varepsilon$-excellent.
For the sake of convenience, let $N$ be the cardinality of the set $Z'$ and let $Z'=\{k_1,\cdots,k_N\}$.
Based on  Definition \ref{DefeZgood} of $\varepsilon$-excellent, for any $k\in Z'$, there exists a great pair of cubes $(J_{k,q(k,i_k)}^{(1)-},J_{k,q(k,i_k)}^{(2)-})$ within $V_{3\varepsilon r}(\bm z_k)$ for some $i_k\in [1,3^d(b_2\alpha)^{-2d}]_{\mathds{Z}}$, such that $P$ passes through it.
Therefore, we can define $t^{(1)}_k$ as the first time that $P$ hits $J_{k,q(k,i_k)}^{(1)-}$, and $t^{(2)}_k$ as the first time that $P$ hits $J_{k,q(k,i_k)}^{(2)-}$ for each $k\in Z'$.

    Without loss of generality, we assume that $0<t^{(1)}_{k_1}<t^{(2)}_{k_1}<t^{(1)}_{k_2}<t^{(2)}_{k_2}<\cdots<t^{(1)}_{k_N}<t^{(2)}_{k_N}$. Denote  $t^{(2)}_{k_0}=0$ and $t^{(1)}_{k_{N+1}}=D(\bm x,\bm y).$  Then by Definition \ref{interval-good} (2) of nice cube, we see that
     \begin{equation}\label{SegmentShortcut}
        \widetilde{D}_{k_l}^-(J^{(1)-}_{k_l,q(k_l,i_{k_l})},J^{(2)-}_{k_l,q(k_l,i_{k_l})})
        \le c'D_{k_l}^-(J^{(1)-}_{k_l,q(k_l,i_{k_l})},J^{(2)-}_{k_l,q(k_l,i_{k_l})}).
     \end{equation}
     Moreover, by  Definition \ref{interval-good} (2) again, there exists a $\widetilde{D}^-_{k_l}$-geodesic between $J^{(1)-}_{k_l,q(k_l,i_{k_l})}$ and $J^{(2)-}_{k_l,q(k_l,i_{k_l})}$ such that it is contained in $J_{k_l,i_{k_l}}$. Combining this with the fact that the scopes of edges in $\mathcal{E}\setminus \mathcal{E}_Z^-$ are at least $(b_a\alpha)^{2.5}\varepsilon r$,  we get that this $\widetilde{D}^-_{k_l}$-geodesic does not use the edges in $\mathcal{E}\setminus \mathcal{E}_Z^-$. Thus we have
     \begin{equation}\label{DZ=DKL}
        \widetilde{D}_Z^-(J^{(1)-}_{k_l,q(k_l,i_{k_l})},J^{(2)-}_{k_l,q(k_l,i_{k_l})})\leq \widetilde{D}_{k_l}^-(J^{(1)-}_{k_l,q(k_l,i_{k_l})},J^{(2)-}_{k_l,q(k_l,i_{k_l})}).
     \end{equation}
     Recalling Definition \ref{interval-good} (4) of nice cube with respect to $\mathcal{E}_{k_l}^-$ together with the chosen value of $\eta$ in \eqref{eta0}, we get that
     \begin{equation}\label{JkiSmall}
     \begin{split}
        \sum_{j=1}^2{\rm diam}(J_{k_l,q(k_l,i_{k_l})}^{(j)-};D_Z^-)&= \sum_{j=1}^2{\rm diam}(J_{k_l,q(k_l,i_{k_l})}^{(j)-};D_{k_l}^-)\\
        &\le 2K^{-\theta}(\varepsilon r)^\theta C_0\sup_{t\in[0,\eta]}\left(t^\theta\log(2/t)\right)<\frac{C_*-c'}{4C_*}b_1\alpha^\theta K^{-\theta}(\varepsilon r)^\theta.
        \end{split}
     \end{equation}

     Additionally, we will show that any $D$-geodesic between $J^{(1)-}_{k_l,q(k_l,i_{k_l})}$ and $J^{(2)-}_{k_l,q(k_l,i_{k_l})}$ (denoted by $P_{k_l}$) does not use any edge in $\mathcal{E}\setminus \mathcal{E}_{k_l}^-$ by an argument similar to the proof of Lemma \ref{upp-strange-Ik}. To prove this, we employ a proof by contradiction as follows.

     Assume (otherwise) that $P_{k_l}$ uses $e_1,e_2,\cdots,e_N\in\mathcal{E}\setminus \mathcal{E}_{k_l}^-$ in order.
     For convenience, we let $\mathcal{J}^-_{k_l}$ be the collection of $J_{k_l,q(k_l,i)}^{(j)-}$ for $j=1,2$ and $i\in [1,3^d(b_2\alpha)^{-2d}]_{\mathds{Z}}$.
     \begin{claim}\label{claim1}
      There exists some $J \in \mathcal{J}^-_{k_l}$ such that $P_{k_l}$ passes through it and then leaves the very nice cube containing $J$ before using any $e_n\in\{e_1,\cdots,e_N\}$ or leaving $V_{3\varepsilon r}(\bm z_{k_l})$.
     \end{claim}
     Now we finish the proof assuming the claim. From Claim \ref{claim1} and  Definition \ref{very nice} (2), we get that $$D(J^{(1)-}_{k_l,q(k_l,i_{k_l})},J^{(2)-}_{k_l,q(k_l,i_{k_l})})\geq (b_2\alpha)^{2.6\theta}(\varepsilon r)^\theta. $$
     However, from Definition \ref{super good} (2), we get that
     $$
     D(J^{(1)-}_{k_l,q(k_l,i_{k_l})},J^{(2)-}_{k_l,q(k_l,i_{k_l})})\leq D_{k_l}^-(J^{(1)-}_{k_l,q(k_l,i_{k_l})},J^{(2)-}_{k_l,q(k_l,i_{k_l})})\leq 2\widetilde{C}K^{-\theta}\log K (\varepsilon r)^\theta.
     $$
     Thus we get a contradiction recalling \eqref{M0}. Hence, one can see that any $D$-geodesic between $J^{(1)-}_{k_l,q(k_l,i_{k_l})}$ and $J^{(2)-}_{k_l,q(k_l,i_{k_l})}$ (denoted by $P_{k_l}$) does not use any edge in $\mathcal{E}\setminus \mathcal{E}_{k_l}^-$.

     As a result, we get that $P_{k_l}$ is also a path on $\mathcal{E}_{k_l}^-$ and
     \begin{equation}\label{DKLleDZ-}
         D_{k_l}^-(J^{(1)-}_{k_l,q(k_l,i_{k_l})},J^{(2)-}_{k_l,q(k_l,i_{k_l})})
         =D(J^{(1)-}_{k_l,q(k_l,i_{k_l})},J^{(2)-}_{k_l,q(k_l,i_{k_l})})\leq D_Z^-(J^{(1)-}_{k_l,q(k_l,i_{k_l})},J^{(2)-}_{k_l,q(k_l,i_{k_l})}).
     \end{equation}
     Plugging \eqref{DZ=DKL} and \eqref{DKLleDZ-} into \eqref{SegmentShortcut}, we get that
     \begin{equation}\label{newshortcut}
        \widetilde{D}_Z^-(J^{(1)-}_{k_l,q(k_l,i_{k_l})},J^{(2)-}_{k_l,q(k_l,i_{k_l})})
        \le c'D_Z^-(J^{(1)-}_{k_l,q(k_l,i_{k_l})},J^{(2)-}_{k_l,q(k_l,i_{k_l})}).
     \end{equation}

     Combining \eqref{newshortcut}, \eqref{JkiSmall} with the triangle inequality, we obtain that for any $l\in[1,N]_{\mathds{Z}}$,
     \begin{equation}\label{PointShortcut}
        \begin{split}
        &\widetilde{D}_{Z}^-(P(t^{(1)}_{k_l}),P(t^{(2)}_{k_l}))\\
        &\le\widetilde{D}_{Z}^-(J^{(1)-}_{k_l,q(k_l,i_{k_l})},J^{(2)-}_{k_l,q(k_l,i_{k_l})})
        +\sum_{j=1}^2{\rm diam}(J_{k_l,q(k_l,i_{k_l})}^{(j)-};\widetilde{D}_Z^{-})\quad
        (\text{by the triangle inequality})\\
        &\le c'D_Z^-(J^{(1)-}_{k_l,q(k_l,i_{k_l})},J^{(2)-}_{k_l,q(k_l,i_{k_l})})+\sum_{j=1}^2{\rm diam}(J_{k_l,q(k_l,i_{k_l})}^{(j)-};\widetilde{D})\quad (\text{by } \eqref{newshortcut})\\
        &\le c'D_Z^-(J^{(1)-}_{k_l,q(k_l,i_{k_l})},J^{(2)-}_{k_l,q(k_l,i_{k_l})})+C_*\sum_{j=1}^2{\rm diam}(J_{k_l,q(k_l,i_{k_l})}^{(j)-};D)\quad (\text{by the definition of }C_*)\\
        &\le c'D_Z^-(P(t^{(1)}_{k_l}),P(t^{(2)}_{k_l}))+C_*\sum_{j=1}^2{\rm diam}(J_{k_l,q(k_l,i_{k_l})}^{(j)-};D)\quad (\text{by the triangle inequality}) \\
        &\le C_* D_Z^-(P(t^{(1)}_{k_l}),P(t^{(2)}_{k_l}))-(C_*-c')b_1\alpha^\theta K^{-\theta}(\varepsilon r)^\theta+C_*\sum_{j=1}^2{\rm diam}(J_{k_l,q(k_l,i_{k_l})}^{(j)-};D)\\
        &\quad\quad  \text{(by Definition \ref{interval-good} (3) of nice cube)}\\
        &\le C_*D_Z^-(P(t^{(1)}_{k_l}),P(t^{(2)}_{k_l}))-\frac{(C_*-c')b_1\alpha^\theta}{2} K^{-\theta}(\varepsilon r)^\theta. \quad (\text{by }\eqref{JkiSmall})
        \end{split}
     \end{equation}


  Moreover, it follows from the definition of $C_*$ that
    \begin{equation}\label{LipschitzDeleteEdge}
        \widetilde{D}_{Z}^-(P(t^{(2)}_{k_l}),P(t^{(1)}_{k_{l+1}}))\le C_* D_{Z}^-(P(t_{k_l}),P(s_{k_{l+1}})).
    \end{equation}
    Therefore,  we get that

     \begin{equation}\label{Inequality4.16Compare}
        \begin{split}
            \widetilde{D}_{Z}^-(\bm x,\bm y)&\le\sum_{l=1}^N \widetilde{D}_{Z}^-(P(t^{(1)}_{k_l}),P(t^{(2)}_{k_l}))+\sum_{l=0}^N \widetilde{D}_{Z}^-(P(t^{(2)}_{k_l}),P(t^{(1)}_{k_{l+1}}))\quad \quad
            \text{(by the triangle inequality)}\\
            &\le C_*\sum_{l=1}^N D^-_Z(P(t^{(1)}_{k_l}),P(t^{(2)}_{k_l}))+C_*\sum_{i=0}^N D_{Z}^-(P(t^{(2)}_{k_l}),P(t^{(1)}_{k_{l+1}}))\\
            &\quad\quad -\frac{(C_*-c')b_1\alpha^\theta}{2} K^{-\theta}(\varepsilon r)^\theta\# Z'\quad \text{(by \eqref{PointShortcut} and \eqref{LipschitzDeleteEdge})}\\
            &= C_*\mathrm{len}(P;D_Z^-)-\frac{(C_*-c')b_1\alpha^\theta}{2} K^{-\theta}(\varepsilon r)^\theta\# Z'\\
            &\le C_* D(\bm x,\bm y)+C_* C_4 (\varepsilon r)^\theta \# Z-\frac{(C_*-c')b_1\alpha^\theta}{2} K^{-\theta}(\varepsilon r)^\theta\# Z'
            \quad\quad \text{(by \eqref{Inequality4.17DG21})}\\
            &\le C_* D_Z^-(\bm x,\bm y)+C_* C_4 (\varepsilon r)^\theta \# Z-\frac{(C_*-c')b_1\alpha^\theta}{2} K^{-\theta}(\varepsilon r)^\theta\# Z'\quad \quad \text{(by $\mathcal{E}_Z^-\subset\mathcal{E}$)}.
        \end{split}
    \end{equation}
    Additionally, by Definition \ref{G-} (1) of $\mathsf{G}_{Z,\varepsilon}^-(\mathcal{E})$ and $Z\neq\emptyset$, we have
    \begin{equation*}\label{Inequality4.16EventG1}
        \widetilde{D}_{Z}^-(\bm x,\bm y)\ge C_* D_Z^-(\bm x,\bm y)-(\varepsilon r)^\theta\ge  C_*D_Z^-(\bm x,\bm y)-(\varepsilon r)^\theta\# Z.
    \end{equation*}
    Combining this with (\ref{Inequality4.16Compare}), we get
    \begin{equation*}
        -(\varepsilon r)^\theta\# Z\le C_* C_4 (\varepsilon r)^\theta \# Z-(C_*-c')b_1 K^{-\theta}(\alpha\varepsilon r)^\theta \# Z'/2
    \end{equation*}
    which implies
    \begin{equation*}
        \# Z'\le C_3\# Z\quad\text{for}\quad C_3=\frac{2(1+C_4C_*)}{b_1\alpha^\theta K^{-\theta}(C_*-c')}.
    \end{equation*}

     All that now remains is to give a proof of the claim.

\noindent \textbf{Proof of Claim \ref{claim1}.}
Recall that $\mathcal{J}^-_{k_l}$ is the collection of $J_{k_l,q(k_l,i)}^{(j)-}$ for $j=1,2$ and $i\in [1,3^d(b_2\alpha)^{-2d}]_{\mathds{Z}}$. Without loss of generality, assume that the end point $\bm x$ of $e_1=\langle \bm x,\bm y\rangle $ is in $J_{k_l,q(k_l,i_{k_l})}^{(1)-}=:J(\bm x)$. The reason why we can make such assumption is that otherwise, $P_{k_l}$ passes through $J_{k_l,q(k_l,i_{k_l})}^{(1)-}$ and then leaves the nice cube containing $J_{k_l,q(k_l,i_{k_l})}^{(1)-}$ before using any $e_n\in \{e_1,\cdots,e_N\}$, which implies Claim \ref{claim1} holds.

    From the assumption, we can see that $\bm x$ is the start point of $P_{k_l}$. Note that after walking along  $e_1$, $P_{k_l}$ must pass through the nice cube $J(\bm y):=J_{k_l,q(k_l,i_{k_l}-1)}^{(2)-}$ containing $\bm y$.

    We now prove Claim \ref{claim1} by contradiction. Suppose that Claim \ref{claim1} does not hold. We then can see that $P_{k_l}$ must be a union of some $e_n$'s and some $D_{k_l}^-$-geodesics such that each of these geodesics only intersects one very nice cube. Recall that $P_{k_l}$ is a $D$-geodesic between $J(\bm x)=J_{k_l,q(k_l,i_{k_l})}^{(1)-}$ and $J_{k_l,q(k_l,i_{k_l})}^{(2)-}$. From Definitions \ref{add} and \ref{delete}, we only remove edges in the region
      $$
      \Omega_{k_l}^-=\bigcup_{i=1}^{3^d(b_2\alpha)^{-2d}-1}\left(J_{k_l,q(k_l,i)}^{(2)-}\times J_{k_l,q(k_l,i+1)}^{(1)-}\right).
      $$
        Recall that for any $i\in[1,3^d(b_2\alpha)^{-2d}]_\mathds{Z}$, $J_{k_l,i}$ is the very nice cube containing $J_{k_l,q(k_l,i)}^{(1)-} $ and $J_{k_l,q(k_l,i)}^{(2)-} $. We will show that for any $n\in[1,N]_\mathds{N}$, after walking along $e_n$, $P_{k_l}$ enters $J_{k_l,i_n}$ for some $i_n<i_{k_l}$ by induction on $n$. The case when $n=1$ has been shown since after walking along $e_1$, $P_{k_l}$ enters $J(\bm y)\subset J_{k_l,i_{k_l-1}}$. Assume that after walking along $e_n$, $P_{k_l}$ enters $J_{k_l,i_n}$ for some $n<N$ and $i_n<i_{k_l}$, and we next consider the case for $n+1$. By the assumption that Claim \ref{claim1} is not true, after walking along $e_n$, $P_{k_l}$ must use an edge in $\mathcal{E}\setminus \mathcal{E}_{k_l}^-$ to escape $J_{k_l,i_n}$. By the definition of $e_1,e_2,\cdots,e_N$ before Claim \ref{claim1}, $P_{k_l}$ must use $e_{n+1}$ to escape $J_{k_l,i_n}$. Thus after walking along $e_{n+1}$, $P_{k_l}$ enters either $J^{(1)-}_{k_l,q(k_l,i_n+1)}$ or $J^{(2)-}_{k_l,q(k_l,i_n-1)}$. However, since $P_{k_l}$ is a $D$-geodesic between $J(\bm x)=J_{k_l,q(k_l,i_{k_l})}^{(1)-}$ and $J_{k_l,q(k_l,i_{k_l})}^{(2)-}$, it cannot enter $J_{k_l,q(k_l,i_{k_l})}^{(1)-}$ after walking along $e_{n+1}$. As a result, recalling that $i_n<i_{k_l}$, we get that after walking along $e_{n+1}$, $P_{k_l}$ enters $J_{k_l,i_{n+1}}$ for some $i_{n+1}<i_{k_l}$.

        From the preceding paragraph, we get that after walking along $e_N$, $P_{k_l}$ enters a very nice cube $J_{k_l,i}$ for some $i<i_{k_l}$. Since $P_{k_l}$ is a $D$-geodesic between $J(\bm x)=J_{k_l,q(k_l,i_{k_l})}^{(1)-}$ and $J_{k_l,q(k_l,i_{k_l})}^{(2)-}$, it must escape $J_{k_l,i}$ after walking along $e_N$. Then from the assumption that Claim \ref{claim1} is not true, $P_{k_l}$ must use an edge in $\mathcal{E}\setminus\mathcal{E}_{k_l}^-$ to escape $J_{k_l,i}$ after walking along $e_N$. Then this leads to a contradiction since we have assumed that $e_N$ is the last edge in $\mathcal{E}\setminus\mathcal{E}_{k_l}^-$ used by $P_{k_l}$.
    \begin{figure}[htbp]
        \centering
        \includegraphics[scale=0.6]{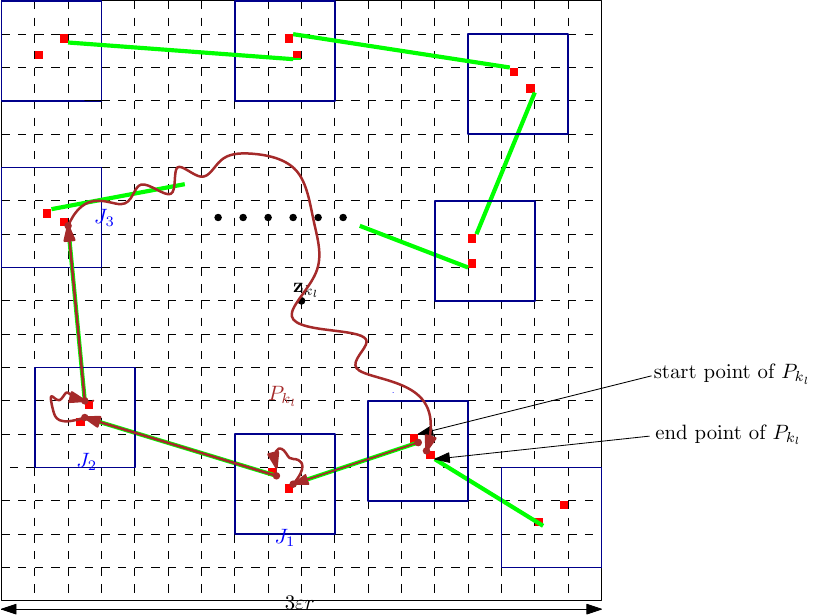}
        \caption{The illustration for Claim \ref{claim1}. The brown path $P_{k_l}$ uses  edges in $\mathcal{E}\setminus\mathcal{E}_{k_l}^-$ to arrive at $J_1, J_2, J_3\in \mathcal{J}^-_{k_l}$ respectively. 
        Then, there exists at least one very nice cube (i.e. blue cubes) containing some $J\in\mathcal{J}_{k_l}^-$ (i.e. $J_3$ in the figure) such that $P$ leaves it without using any edge in $\mathcal{E}\setminus\mathcal{E}_{k_l}^-$ (green lines) at least once.}
        \label{claim1T}
    \end{figure}
    Hence we complete the proof of Claim \ref{claim1}.
\end{proof}

\begin{proof}[Proof of Proposition \ref{Prop4.6DG21}]
Let us assume the existence of a nonempty set $Z_0 \in \mathcal{W}_\varepsilon$ with $\# Z_0 \le m$ such that $\mathsf{G}_{Z_0,\varepsilon}^-(\mathcal{E}_{Z_0}^-,\mathcal{E})$ occurs (otherwise, (\ref{Inequality4.6DG21}) holds vacuously).
We define $Z_1 \in \mathcal{Z}_\varepsilon$ to be a set such that every $k \in Z_1$ is $\varepsilon$-excellent (Definition \ref{DefeZgood}). We further assume that $\# Z_1$ is maximal among all subsets of $\mathcal{Z}_\varepsilon$ with this property. Applying Lemma \ref{Lemma4.16DG21}, we find that $\# Z_1 \le C_3 m$.

Now suppose that $Z\in\mathcal{Z}_\varepsilon$ such that $\mathsf{G}_{Z,\varepsilon}^-(\mathcal{E}_Z^-,\mathcal{E})$ occurs. We claim that for any $k\in Z$, there exists $k'\in  Z_1$ such that $V_{3\varepsilon r}(\bm z_k)\cap V_{3\varepsilon r}(\bm z_{k'})\neq \emptyset$. Otherwise, If there exists some $k\in Z$  such that  $k$ is $\varepsilon$-excellent and $V_{3\varepsilon r}(\bm z_k)$ does not intersect with $ V_{3\varepsilon r}(\bm z_{k'})$ for all $k'\in Z_1$, then $Z_1\cup\{k\}$ satisfies all requirements for $Z_1$, which contradicts the maximality of $\# Z_1$. Since the number of the cubes of the form $V_{3\varepsilon r}(\bm z_k)$ that intersect $V_{3\varepsilon r}(\bm z_{k_0})$ is at most $3^d$ for any $k_0\in  Z_1$, thus we conclude that
    \begin{equation*}
        \# \{ Z\in\mathcal{Z}_\varepsilon: \#Z\leq m\text{ and }\mathsf{G}_{Z,\varepsilon}^-(\mathcal{E}_Z^-,\mathcal{E})\text{ occurs} \}\le 2^{3^d\# Z_1}\le 2^{3^dC_3m}.
    \end{equation*}
    This gives (\ref{Inequality4.6DG21}) with $C_2=2^{3^dC_3}$.
\end{proof}

\subsection{Proof of uniqueness}


Before proceeding, we verify that auxiliary conditions (3) and (4) in Definition \ref{G} of the event $\mathscr{G}_r^\varepsilon(\bm x,\bm y)$ hold with high probability for small values of $\varepsilon$. Combined with  Proposition \ref{Prop4.3DG21}, this allows us to place an upper bound on the probability of the main condition (1).

\begin{lemma}\label{PG(2)(3)}
Let $r>0$ and let $\widetilde{\gamma},\widetilde{q}>0$ such that $\mathds{P}[\widetilde{G}_r(\widetilde{\gamma},\widetilde{q},c'')]\geq \widetilde{\gamma}$. It holds with probability tending to 1 as $\varepsilon\rightarrow 0$ {\rm(}at a rate depending only on $\beta,d$  and the laws of $D$ and $\widetilde{D}$, not on $r${\rm)} that Definition \ref{G} {\rm(3)} and {\rm (4)} of $\mathscr{G}_r^\varepsilon(\bm x,\bm y)$
hold for any $\bm x,\bm y\in V_r(\bm 0)$ with $|\bm x-\bm y|\geq \alpha r$ and $D(\bm x,\bm y)\geq (b\alpha r)^\theta/2$ {\rm(}recall $b=b(\alpha)$ is the constant in Definition \ref{def-H}{\rm)}.
\end{lemma}

The main challenge in proving Lemma  \ref{PG(2)(3)} lies in the fact that the event described in Definition \ref{G} (4) is nonlocal, as condition (2) in Definition \ref{very nice} undermines the local property.
To overcome this challenge, similar to the proof of Lemma \ref{EE+}, we will introduce some deterministic regions (see \eqref{detLambda} and \eqref{det-region-1} below) that allow us to control the (random) regions where edges are added. In order to achieve this, we make some preparations before the proof of Lemma \ref{PG(2)(3)}.

We first apply \eqref{d-longer-than-D-3} and get that
\begin{equation}\label{P(2)}
\begin{split}
&\mathds{P}\left[ \text{Definition \ref{G} (3) holds  for any }  \bm x,\bm y\in V_r(\bm 0) \text{ with }|\bm x-\bm y|\geq\alpha r\text{ and }D(\bm x,\bm y)\geq (b\alpha r)^\theta/2\right]\\
&\to 1\quad \text{as } \varepsilon\rightarrow 0.
\end{split}
\end{equation}
For convenience, we denote the event in \eqref{P(2)} as $A_{r,\varepsilon}$.

We also need the renormalization introduced in Section \ref{bi-lipschitz}.
Specifically,
we divide $\mathds{R}^d$ into small cubes of side length $\varepsilon r$, i.e., $\mathds{R}^d=\cup_{\bm k\in (\varepsilon r)\mathds{Z}^d}V_{\varepsilon r}(\bm k)$. 
Note that under this partition, the cube $V_r(\bm 0)=\cup_{\bm k\in (\varepsilon r)[-1/(2\varepsilon),1/(2\varepsilon)]_\mathds{Z}^d}V_{\varepsilon r}(\bm k)$.
In the following, we write $J_{\bm k,\cdot}$, $J'(\bm w'_{\bm k,\cdot})$, $J''(\bm w''_{\bm k,\cdot})$ as the nice, very nice, very very nice cubes, and write $(J^{(1)}_{\bm k,\cdot},J^{(2)}_{\bm k,\cdot})$ as the great pair of small cubes in $V_{3\varepsilon r}(\bm k)$. If $V_{3\varepsilon r}(\bm k)$ is super super good, we also refer to $J_{\bm k, q(\bm k,\cdot)}$ as the nice cubes in it chosen in \eqref{choosenice}.

We identify the cube $V_{\varepsilon r}(\bm k)$ with the vertex  $\bm k$ and call the resulting graph $G$. Denote $P$ and $P^G$ for the paths in the continuous model and in $G$, respectively.

Now for any $\bm x,\bm y\in V_r(\bm 0)$, let $\bm k_{\bm x},\bm k_{\bm y}\in (\varepsilon r)[-1/(2\varepsilon),1/(2\varepsilon)]_\mathds{Z}^d$ be such that $\bm x\in V_{\varepsilon r}(\bm k_x)$ and $\bm y\in V_{\varepsilon r}(\bm k_y)$. Then using the renormalization, we can see that on the event $A_{r,\varepsilon}$, the skeleton path (recall Definition \ref{pathc-d}) $P^G_{\bm k_{\bm x} \bm k_{\bm y}}$ of any $D$-geodesic $P$ from $\bm x$ to $\bm y$ satisfies
\begin{equation}\label{PG>}
|P^G_{\bm k_x\bm k_y}|\geq \varepsilon^{-\theta/4}.
\end{equation}
This implies that we only need to consider the paths in $G$ that satisfy \eqref{PG>}. To ease exposition, recall that $\mathcal{P}_m(\bm i,\bm j)$ is the collection of self-avoiding paths from $\bm i$ to $\bm j$ in $G$ with length $m$ and that $\mathcal{P}_{\geq m}(\bm i,\bm j)=\cup_{n\geq m}\mathcal{P}_n(\bm i,\bm j)$. For the sake of convenience, for each $\bm k\in (\varepsilon r)\mathds{Z}^d$ if $V_{3\varepsilon r}(\bm k)$ is super good (resp. $\bm e$-super good) (see Definition \ref{super good}), we say $\bm k$ is super good (resp. $\bm e$-super good).

We also note that for a path $P_{\bm i\bm j}^G$ from $\bm i$ to $\bm j$ with length $m$ for some $m\geq \varepsilon^{-\theta/4}$ and $\bm i,\bm j\in (\varepsilon r)[-1/(2\varepsilon),1/(2\varepsilon)]_\mathds{Z}^d$, it must pass through at least $\varepsilon^{-\theta/4}m/3^d$ points $\bm k\in (\varepsilon r)\mathds{Z}^d$ such that $V_{3\varepsilon r}(\bm k)$'s are disjoint. Denote the set of those points as $\mathcal{K}(P_{\bm i\bm j}^G)$.

Using \eqref{supergoodcriticalprobability}, we can show that with high probability, $P_{\bm i\bm j}^G$ passes through lots of super good sites (here a site is super good if the cube centered at it is super good) as follows.
\begin{lemma}\label{supersuper1}
For $m\geq \varepsilon^{-\theta/4}$ and $\bm i,\bm j\in (\varepsilon r)[-1/(2\varepsilon),1/(2\varepsilon)]_\mathds{Z}^d$, let $P_{\bm i\bm j}^G$ be a path from $\bm i$ to $\bm j$ with length $m$. Then we have
$$
\mathds{P}[P^G_{\bm i\bm j} \text{ passes through at most $ 7m/(8\cdot3^d)$ super good sites}|G]\leq 2^d(1-p'_c)^{m/(8\cdot6^d)},
$$
where $p_c':=p_c\vee (1-(4C_{dis})^{-8\cdot 6^d})$ and $p_c$ is the  constant defined in Lemma \ref{hd-BK} with $\delta=1/(4C_{dis})$ {\rm(}here $C_{dis}$ is the constant defined in Lemma \ref{number-path-k}{\rm)}.
\end{lemma}
\begin{proof}
For each direction $\bm e\in\{-1,1\}^d$, as we have shown in \eqref{supergoodcriticalprobability}, there exist $\alpha,\eta$ and $K$, depending only on $\beta,d$ and the laws of $D$ and $\widetilde{D}$ (not on $r$), such that for each $\bm k\in (\varepsilon r)\mathds{Z}^d$,
\begin{equation}\label{e-supergoodprob}
 \mathds{P}[V_{3\varepsilon r}(\bm k)\text{ is $\bm e$-super good}]
 \geq  p'_c.
\end{equation}
 Moreover, as we mentioned in \eqref{supergoodindep}, for each $Z\subset (\varepsilon r)\mathds{Z}^d$ such that $V_{3\varepsilon r}(\bm k)$'s for $\bm k\in Z$ are disjoint,
\begin{equation}\label{e-super-indep}
\left\{V_{3\varepsilon r}(\bm k)\ \text{is $\bm e$-super good}\right\}_{\bm k\in Z}
\end{equation}
is a collection of independent events.
Therefore, combining the definition of $\mathcal{K}(P_{\bm i\bm j}^G)$   with \eqref{e-supergoodprob} and \eqref{e-super-indep}, we can see that
$$
\mathds{P}[\text{there are at most $(1-2^{-d-3})m/3^d$ $\bm e$-super good sites in }\mathcal{K}(P_{\bm i\bm j}^G) |G ]\leq (1-p'_c)^{m/(8\cdot6^d)}.
$$
Taking a union bound over all directions $\bm e$, we get that
\begin{equation*}
\begin{split}
&\mathds{P}[P^G_{\bm i\bm j} \text{ pass through at most $ 7m/(8\cdot3^d)$ super good sites}|G]\\
&\leq \mathds{P}[\text{there are at most $7 m/(8\cdot3^d)$ super good sites in }\mathcal{K}(P_{\bm i\bm j}^G) |G]\\
&\leq 2^d(1-p'_c)^{m/(8\cdot6^d)}.
\end{split}
\end{equation*}
Hence the proof is complete.
\end{proof}

We next introduce a deterministic region, which will contain the region where edges are added (see Definition \ref{add}). Let $\bm k\in (\varepsilon r)\mathds{Z}^d$. We divide $V_{3\varepsilon r}(\bm k)$ into $(6K/\eta)^d$ small cubes of side length $\eta\varepsilon r/(2K)$. Denote those small cubes as $U_{\bm k,i}$ for $i\in[1,(6K/\eta)^d]_\mathds{Z}$.
We observe that for any cube $V$ with side length $\eta\varepsilon r/K$ (which is also the side length of the cube in the great pair) that is contained in $V_{3\varepsilon r}({\bm k})$ (especially for $J_{\bm k,q(\bm k,i)}^{(l)}$ for $l=1,2$ if $V_{3\varepsilon r}(\bm k)$ is super good), there must exist an $i\in[1,(6K/\eta)^d]_{\mathds{Z}}$ such that $U_{{\bm k},i}\subset V$ (see Figure \ref{Proof733}).
Additionally, by \eqref{choosenice}, we see that if $V_{3\varepsilon r}(\bm k)$ is super good, then the great pairs of cubes in $V_{3\varepsilon r}({\bm k})$ all satisfy ${\rm dist}(J_{\bm k,q(\bm k,i)}^{(2)},J_{\bm k,q(\bm k,i+1)}^{(1)};\|\cdot\|_\infty)>(b_2\alpha)^{2.5}\varepsilon r$ (see Definition \ref{very nice} (2) which implies that $J_{\bm k,q(\bm k,i)}$ lies in the ``center'' of the associated very nice cube with side length $(b_2\alpha)^{2.5}\varepsilon r$).
    \begin{figure}[htbp]
        \centering
        \includegraphics[scale=0.7]{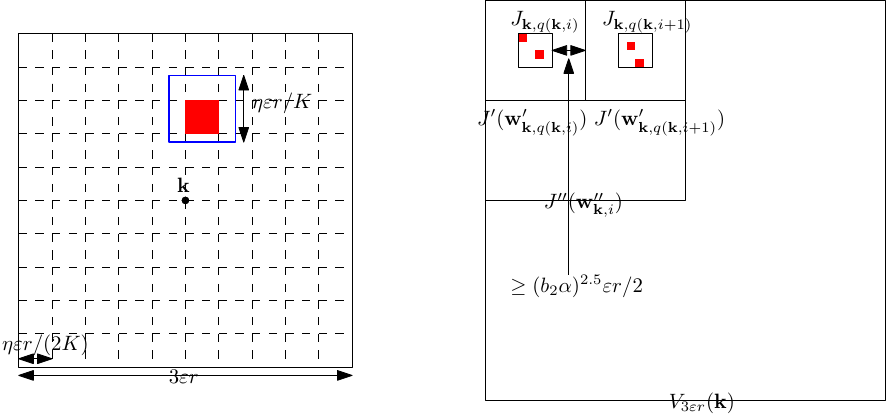}
        \caption{The left figure provides an explanation for the fact that the cube with side length $\eta\varepsilon r/(2K)$ (red cube) must be contained in a cube with side length $\eta\varepsilon r/K$ (blue box). The right figure provides an explanation for the $\ell^\infty$-distance between $J_{\bm k,q(\bm k,i)}^{(2)}$ and $J_{\bm k,q(\bm k,i+1)}^{(1)}$.}
        \label{Proof733}
    \end{figure}
Therefore, combining the above analysis with Definition \ref{add}, one has
\begin{equation}\label{det-region-1}
\Omega_{\bm k}\subset \bigcup_{i,j\in[1,(6K/\eta)^d]_\mathds{Z}:{\rm dist}(U_{\bm k,i},U_{\bm k,j})>(b_2\alpha)^{2.5}\varepsilon r} (U_{\bm k,i}\times U_{\bm k,j}).
\end{equation}

Now let $E_{\bm k,1}$ be the event that $U_{{\bm k},i}$ and $U_{{\bm k},j}$ are connected by $\widetilde{\mathcal{E}}$ whenever ${\rm dist}(U_{{\bm k},i},U_{{\bm k},j};\|\cdot\|_\infty)> (b_2\alpha)^{2.5}\varepsilon r $,
and let $E'_{\bm k, 1}$ be the event that $J_{\bm k,q(\bm k,i)}^{(2)}$ and $J_{\bm k,q(\bm k,i+1)}^{(1)}$ are connected by $\mathcal{E}^+_{\bm k}$
 for all $i\in[1,3^d(b_2\alpha)^{-2d}]_\mathds{Z}$ when $\bm k$ is super good. Then it is clear that $E_{\bm k,1}\subset E'_{\bm k,1}$ when $\bm k$ is super good. By the independence of the Poisson point process $\widetilde{\mathcal{E}}$, we also see that for each $Z\in (\varepsilon r)\mathds{Z}^d$ such that $V_{3\varepsilon r}(\bm k)$'s for $\bm k\in Z$ are disjoint,
\begin{equation}\label{E1-indep}
\{E_{\bm k,1}\}_{\bm k\in Z} \quad \text{is a collection of independent events}.
\end{equation}
Furthermore, from the locality of the Poisson point process $\widetilde{\mathcal{E}}$, we have that $E_{{\bm k,1}}$ is a.s.\ determined by $\widetilde{\mathcal{E}}|_{V_{3\varepsilon r}({\bm k})\times V_{3\varepsilon r}({\bm k})}$. It is worth emphasizing that $\{E'_{\bm k,1}\}_{\bm k\in Z}$ are not independent since Definition \ref{very nice} yields that the random regions $(J^{(2)}_{k,q(k,i)},J^{(1)}_{k,q(k,i+1)})$ in $\Omega_k$ are not independent. The introduction of the deterministic sets $U_{\bm k,i}\times U_{\bm k,j}$ allows us to define a collection of independent events $\{E_{\bm k,1}\}$ that are contained in $\{E'_{\bm k,1}\}$. Therefore, we can apply Lemma \ref{hd-BK} with $\{E_{\bm k,1}\}$ instead of $\{E'_{\bm k,1}\}$.

The following lemma provides a lower bound on $\mathds{P}[E_{\bm k,1}]$, which is also a lower bound on $\mathds{P}[E'_{\bm k,1}]$ when $\bm k$ is super good.

\begin{lemma}\label{Ek1porb}
For each $\bm k\in (\varepsilon r)\mathds{Z}^d$,
$$
\mathds{P}[E_{\bm k,1}]\geq  1-(6K/\eta)^{2d}\exp\{-\widetilde{\beta}(\eta/(2K))^{2d}\}.
$$
\end{lemma}

\begin{proof}
By the definitions of small cubes $U_{\bm k,i}$ and the event $E_{\bm k,1}$, we obtain
\begin{equation*}
    \begin{split}
        \mathds{P}[E_{{\bm k},1}]&\geq 1- \sum_{i,j\in[1,(6K/\eta)^d]_\mathds{Z}:i\neq j}\exp\left\{-\int_{U_{{\bm k},i}}\int_{U_{{\bm k},j}}\frac{\widetilde{\beta}\d \bm u\d \bm v}{|\bm u-\bm v|^{2d}}\right\}\\
        &\geq 1-(6K/\eta)^{2d}\exp\{-\widetilde{\beta}(\eta/(2K))^{2d}\}.
    \end{split}
\end{equation*}
Thus we complete the proof.
\end{proof}

From Lemma \ref{Ek1porb} one can find that the probability of $E_{\bm k,1}$ can be arbitrarily close to 1 as $\widetilde{\beta}$ increases to infinity. Thus,  we choose a sufficiently large $\widetilde{\beta}>0$ such that for each $\bm k\in (\varepsilon r)\mathds{Z}^d$,
\begin{equation}\label{Ek1prob-1}
\mathds{P}[E_{\bm k,1}]\geq 1-\frac{1-p_c}{2},
\end{equation}
where $p_c$ is the  constant defined in Lemma \ref{hd-BK} with $\delta=1/(4C_{dis})$ (here $C_{dis}$ is the constant defined in Lemma \ref{number-path-k}).

We now recall another deterministic region, which is defined in \eqref{detLambda}. More precisely,
\begin{equation*} 
\Lambda_{\bm k} =\left\{(\bm x,\bm y)\in V_{3\varepsilon r}(\bm k)\times V_{3\varepsilon r}(\bm k):|\bm x-\bm y|\geq (b_2\alpha)^{2.5}\varepsilon r/2\right\}
\end{equation*}
is a deterministic region contained in $V_{3\varepsilon r}(\bm k)\times V_{3\varepsilon r}(\bm k)$. For each $Z\in (\varepsilon r)\mathds{Z}^d$ such that $V_{3\varepsilon r}(\bm k)$'s for $\bm k\in Z$ are disjoint, set $\Lambda_Z=\cup_{\bm k\in Z}\Lambda_{\bm k}$. Recall that $\overline{\mathcal{E}}_Z^+=\mathcal{E}\cup (\widetilde{\mathcal{E}}\cap (\cup_{\bm k\in Z}\Lambda_{\bm k}))$ and  we write $\overline{\mathcal{E}}_{\{\bm k\}}^+$ as $\overline{\mathcal{E}}_{\bm k}^+$.
As we mentioned in the proof of Lemma \ref{EE+} (the paragraph below \eqref{RNestimateEE+}), we have that $\Omega_{\bm k}\subset \Lambda_{\bm k}$ by Definition \ref{add}, and thus we have
$\mathcal{E}_Z^+\subset \overline{\mathcal{E}}_Z^+$.

Since $\{\overline{\mathcal{E}}_{\bm k}^+\cap\Lambda_{\bm k}\}_{\bm k\in Z}$ are i.i.d., to estimate the term associated to $\overline{\mathcal{E}}_Z^+$ in Definition \ref{G} (4)(i), we only need to consider that of each $\bm k$ as follows.
Recall that $M_0$ is a sufficiently large number chosen in \eqref{M1-1}, that is,
\begin{equation}\label{E2}
    \mathds{P}[|\overline{\mathcal{E}}_{\bm k}^+\cap \Lambda_{\bm k}|\leq M_0]\geq 1-\frac{1-p_c}{2}.
\end{equation}
%
We will denote the event on the left hand side of \eqref{E2} as $E_{\bm k,2}$. For each $Z\subset (\varepsilon r)\mathds{Z}^d$ such that $V_{3\varepsilon r}(\bm k)$'s for $\bm k\in Z$ are disjoint,
from the independence of $\{\overline{\mathcal{E}}_{\bm k}^+\}_{\bm k\in Z}$ we get that
\begin{equation}\label{E2-indep}
\{E_{\bm k,2}\}_{\bm k\in Z}\quad \text{is a collection of independent events.}
\end{equation}

With \eqref{E2} at hand, we can present the
\begin{proof}[Proof of Lemma \ref{PG(2)(3)}]
For any $\bm k\in (\varepsilon r)\mathds{Z}^d$, let $E_{\bm k}=E_{\bm k,1}\cap E_{\bm k,2}$.
Then by \eqref{E1-indep} and \eqref{E2-indep}, for each $Z\subset (\varepsilon r)\mathds{Z}^d$ such that $V_{3\varepsilon r}(\bm k)$'s for $\bm k\in Z$ are disjoint,  we have that $\{E_{\bm k}\}_{\bm k\in Z}$ is also a collection of independent events. Moreover, by \eqref{Ek1prob-1} and \eqref{E2}, we get that $\mathds{P}[E_{\bm k}]\geq p_c$ for each $\bm k\in (\varepsilon r)\mathds{Z}^d$.

For $m\geq \varepsilon^{-\theta/4}$ and $\bm i,\bm j\in (\varepsilon r)[-1/(2\varepsilon),1/(2\varepsilon)]_\mathds{Z}^d$, let $P_{\bm i\bm j}^G$ be a path from $\bm i$ to $\bm j$ with length $m$. Recall that $\mathcal{K}(P_{\bm i\bm j}^G)$ is the set of points $\bm k$ in $P_{\bm i\bm j}^G$ such that all $V_{3\varepsilon r}(\bm k)$'s are disjoint.
In addition, recall that Definition \ref{superdouble} (2) is Definition \ref{hd-PGgood} (2), wherein $R$ and $b_0$ are replaced with $r$ and $b_2$, respectively. Therefore, we can apply Lemma \ref{hd-BK} with the above event $\{E_{\bm k}\}_{\bm k\in \mathcal{K}(P^G_{\bm i\bm j})}$ and $\delta=1/(4C_{dis})$ to obtain
\begin{equation*}
\mathds{P}[P^G_{\bm i\bm j}\ \text{is } (\{E_{\bm k}\}_{\bm k\in \mathcal{K}(P^G_{\bm i\bm j})},\alpha)\text{-good}]\geq 1-(4C_{dis})^{-m}.
\end{equation*}
That is to say, similar to Lemma \ref{hd-BK}, we can get that with probability at least $1-(4C_{dis})^{-m}$ there exist at least $m/(4\cdot 3^d)$ $\bm k$'s in $\mathcal{K}(P_{\bm i\bm j}^G)$ such that $E_{\bm k}$ happens and Definition \ref{superdouble} (2) holds in $V_{3\varepsilon r}({\bm k})$. Additionally, from Definition \ref{hd-PGgood} (3) we can see that for any two different such subscripts ${\bm k_1}$ and ${\bm k_2}$, ${\bm k_1}-{\bm k_2}\in(3\varepsilon r)\mathds{Z}^d$.

Note that if $P_{\bm i\bm j}^G$ is $(\{E_{\bm k}\}_{\bm k\in \mathcal{K}(P^G_{\bm i\bm j})},\alpha)$-good and passes through at least $7m/(8\cdot 3^d)$ super good sites, then there are at least $(1/(4\cdot 3^d)-1/(8\cdot 3^d))m=m/(8\cdot 3^d) $ many $\bm k$'s such that (1) for two such $\bm k_1$ and $\bm k_2$, we have $\bm k_1-\bm k_2\in (3\varepsilon r)\mathds{Z}^d$;
(2) $V_{3\varepsilon r}({\bm k})$ is super super good; (3) $E_{\bm k} $ occurs. Then from the analysis before \eqref{E1-indep} and the definitions of $E_{{\bm k},1}$ and $E_{{\bm k},2}$ for ${\bm k}\in(\varepsilon r)\mathds{Z}^d$, 
we can get that $P_{\bm i\bm j}^G$ satisfies Definition \ref{G} (4).
Combining this with Lemma \ref{supersuper1}, we obtain that
\begin{equation}\label{G(4)}
\begin{split}
\mathds{P}[P_{\bm i\bm j}^G\text{ does not satisfy Definition \ref{G} (4)}] &\leq 2^d(1-p'_c)^{m/(8\cdot6^d)}+(4C_{dis})^{-m}\\
&\leq (2^d+1)(4C_{dis})^{-m}.
\end{split}
\end{equation}

Now let $B_\varepsilon$ be the event that for any $m\geq \varepsilon^{-\theta/4}$ and  $\bm i,\bm j\in (\varepsilon r)[-1/(2\varepsilon),1/(\varepsilon r)]_\mathds{Z}^d$ we have that $|\mathcal{P}_m(\bm i,\bm j)|\leq (2C_{dis})^m$, and let  $H_\varepsilon$ be the event that Definition \ref{G} (3) holds but Definition \ref{G} (4) does not hold for some $\bm x,\bm y\in V_r(\bm 0)$  with $|\bm x-\bm y|\geq\alpha r\text{ and }D(\bm x,\bm y)\geq (b\alpha r)^\theta/2$. Using the similar arguments in the proof of Lemma \ref{Fr4lemma-2}, we can see that
\begin{equation}\label{Beps}
\begin{split}
\mathds{P}[B^c_\varepsilon]&\leq \sum_{\bm i,\bm j\in (\varepsilon r)[-1/(2\varepsilon),1/(2\varepsilon)]_\mathds{Z}^d}\mathds{P}\left[ \exists m\geq \varepsilon^{-\theta/4}\ \text{such that }|\mathcal{P}_m(\bm i ,\bm j)|\geq (2C_{dis})^m\right]\\
&=O_\varepsilon(\varepsilon^{\mu})\quad \forall\mu>0.
\end{split}
\end{equation}
Here the implicit constant in $O_\varepsilon(\cdot)$ depends only on $\beta,d, \mu$ and the law of $D$. In addition, from \eqref{G(4)} we get that
\begin{equation}\label{H-prob}
\begin{split}
\mathds{P}[B_\varepsilon \cap H_\varepsilon]\leq (2^d+1)\sum_{\bm i,\bm j:|\bm i-\bm j|\geq \alpha/\varepsilon }\sum_{m\geq \varepsilon^{-\theta/4}} (2C_{dis})^m(4C_{dis})^{-m} \leq 4^d\varepsilon^{-d}2^{-\varepsilon^{-\theta/4}+1}.
\end{split}
\end{equation}
Hence, combining this with \eqref{P(2)}, \eqref{Beps} and \eqref{H-prob}, we have
\begin{equation*}
\begin{split}
&\mathds{P}[\text{Definition \ref{G} (3) and (4) hold  for any } \bm x,\bm y \text{ with }|\bm x-\bm y|\geq\alpha r\text{ and }D(\bm x,\bm y)\geq (b\alpha r)^\theta/2]\\
&\geq 1-\left(\mathds{P}[\text{Definition \ref{G} (3) does not hold}]+\mathds{P}[B_\varepsilon ^c]+\mathds{P}[B_\varepsilon\cap H_{\varepsilon}]\right)\to 1
\end{split}
\end{equation*}
as $\varepsilon\to 0$.
\end{proof}

Assuming that the parameters $\alpha,\eta, K,\widetilde{\beta}$ and $M$ have been chosen according to Lemma \ref{PG(2)(3)}, we can combine Proposition \ref{Prop4.3DG21} with Lemma \ref{PG(2)(3)} to arrive at the following result.

\begin{lemma}\label{DGLem4.19}
 Let $r>0$ and let $\widetilde{\gamma},\widetilde{q}>0$ be such that $\mathds{P}[\widetilde{G}_r(\widetilde{\gamma},\widetilde{q},c'')]\geq \widetilde{\gamma}$. Also let $\nu\geq 1$. It holds with probability tending to 1 as $\varepsilon\rightarrow 0$ {\rm(}at a rate depending only on $\beta,d,\nu$, the parameters and the laws of $D$ and $\widetilde{D}${\rm)} that
$$
\widetilde{D}(\bm z,\bm w)\leq C_*D(\bm z,\bm w)-(\varepsilon r)^\theta
$$
 for all $\bm z,\bm w\in \varepsilon^\nu r\mathds{Z}^d\cap V_r(\bm 0)$ with $|\bm z-\bm w|\geq \alpha r$, $D(\bm z,\bm w)\geq (b\alpha r)^\theta/2$  and the $D$-geodesic $P$ from $\bm z$ to $\bm w$ satisfying $P\subset V_r(\bm 0)$.
\end{lemma}
\begin{proof}
From Proposition \ref{Prop4.3DG21}, we see that
\begin{equation}\label{G(z,w)}
\mathds{P}[\mathscr{G}_r^\varepsilon(\bm z,\bm w)^c]=1-O_\varepsilon (\varepsilon^\mu)\quad \forall \mu>0
\end{equation}
for any $\bm z,\bm w\in \varepsilon^\nu r\mathds{Z}^d\cap V_r(\bm 0)$ such that $|\bm z-\bm w|\geq \alpha r$  and $D(\bm z,\bm w)\geq (b\alpha r)^\theta/2$. Thus, we sum \eqref{G(z,w)} over all possible $\bm z,\bm w$ to get that with superpolynomially  high probability as $\varepsilon\rightarrow 0$,
$\mathscr{G}_r^\varepsilon(\bm z,\bm w)$ does not occur for any $\bm z,\bm w\in \varepsilon^\nu r\mathds{Z}^d\cap V_r(\bm 0)$ such that $|\bm z-\bm w|\geq \alpha r$  and $D(\bm z,\bm w)\geq (b\alpha r)^\theta/2$.
Thus, the event $\mathscr{G}_r^\varepsilon$ has to fail with probability tending to 1, while by Lemma \ref{PG(2)(3)} we have that with probability tending to 1 Definition \ref{G} (3) and (4) hold.
Therefore,  Definition \ref{G} (1) or (2) has to fail with probability tending to 1.
Hence, we conclude that with probability tending to 1 as $\varepsilon\to 0$, the event in Lemma \ref{DGLem4.19} occurs.
\end{proof}

Recall $H_r(\alpha ,C')$ as defined in Definition \ref{def-H}.

\begin{lemma}\label{delta to 0}
For sufficiently small $\alpha \in(0,1)$ the following holds.
Let $\widetilde{\gamma},\widetilde{q}>0$ and $r>0$ be such that $\mathds{P}[\widetilde{G}_r(\widetilde{\gamma},\widetilde{q},c'')]>\widetilde{\gamma}$. Then we have
$$
\lim_{\delta\to0}\mathds{P}[H_r(\alpha,C_*-\delta)]=0
$$
at a rate depending only on $\beta,d,\widetilde{\gamma}, \widetilde{q}$, the parameters and the laws of $D$ and $\widetilde{D}$.
\end{lemma}

\begin{proof}
Applying Lemma \ref{DGLem4.19} with $\nu=4$, it holds with probability tending to 1 as $\varepsilon \to 0$ that
\begin{equation}\label{veps^2}
\widetilde{D}(\bm z,\bm w)\leq C_*D(\bm z,\bm w)-(\varepsilon r)^\theta
\end{equation}
 for all $\bm z,\bm w\in (\varepsilon^4 r\mathds{Z}^d)\cap V_r(\bm 0)$ with $|\bm z-\bm w|\geq \alpha r$, $D(\bm z,\bm w)\geq (b\alpha r)^\theta/2$,  and the $D$-geodesic $P$ from $\bm z$ to $\bm w$ satisfying $P\subset V_r(\bm 0)$.

Recalling Definition \ref{def-H} of $H_r(\alpha, C')$, we consider points $\bm x,\bm y\in V_r(\bm 0)$ with $|\bm x-\bm y|\geq \alpha r$ such that the $D$-geodesic $P$ from $\bm x$ to $\bm y$ satisfies
\begin{itemize}
\item[\rm (i)] $P\subset V_r(\bm 0)$;

\item[\rm (ii)] $ D(\bm x,\bm y)\geq (b\alpha r)^\theta$ for some $b=b(\alpha)>0$.

\end{itemize}
We will show that
for each small enough  $\delta=\varepsilon^{3\theta/2}/2>0$, depending only on $\beta,d, \alpha, \eta,K,\widetilde{\gamma}, \widetilde{q}$ and the laws of $D$ and $\widetilde{D}$, we have that with probability tending to 1 as $\delta$ tends to 0,
\begin{equation}\label{Dtilde<(C*-)D}
\widetilde{D}(\bm x,\bm y)\leq (C_*-\delta)D(\bm x,\bm y)\quad \forall \bm x,\bm y\text{ satisfying the above conditions}.
\end{equation}
By Definition \ref{def-H} for $H_r(\alpha,C')$, \eqref{Dtilde<(C*-)D} implies that $H_r(\alpha,C_*-\delta)$ does not occur. Since $\varepsilon$ (so does $\delta$) can be made arbitrarily small, this conclusion implies the statement of the lemma.

Next we carry out the proof as outlined above.  Let $\bm z,\bm w\in \varepsilon^4 r\mathds{Z}^d$ be such that $\bm x\in V_{\varepsilon^4 r}(\bm z)$ and $\bm y\in V_{\varepsilon^4 r}(\bm w)$.
First, we need to show that $D(\bm z,\bm w)\geq (b\alpha r)^\theta/2$, which then allows us to utilize \eqref{veps^2}. To do so, by the triangle inequality and the fact that $D(\bm x,\bm y)\geq (b\alpha r)^\theta$,
\begin{equation}\label{lowerboundDzw}
\begin{split}
D(\bm z,\bm w)&\geq D(\bm x,\bm y)-(D(\bm x,\bm z)+D(\bm y,\bm w))\\
&\geq (b\alpha r)^\theta-\left[{\rm diam}\left(V_{\varepsilon^4 r}(\bm z);D\right)+{\rm diam}\left(V_{\varepsilon^4 r}(\bm w);D\right)\right].
\end{split}
\end{equation}
Denoting by ${\rm diam}_{\bm z,\varepsilon^4 r}$ and ${\rm diam}_{\bm w,\varepsilon^4 r}$ the last two terms for diameters above, by Corollary \ref{LimitContinuousTail} we see that
\begin{equation}\label{p[A]-diam}
\mathds{P}[A]:=\mathds{P}\left[{\rm diam}_{\bm z,\varepsilon^4 r}\vee {\rm diam}_{\bm w,\varepsilon^4 r}\geq (\varepsilon^{2} r)^{\theta}\right]= O_\varepsilon(\varepsilon^\mu)\quad \forall \mu>0,
\end{equation}
where the implicit constant in the $O_\varepsilon(\cdot)$ depends only on $\beta,d,\mu$ and the laws of $D$ and $\widetilde{D}$. Combining this with \eqref{lowerboundDzw}, we find that, on the event $A^c$,
\begin{equation}\label{lowerboundDzw2}
D(\bm z,\bm w)\geq (b\alpha r)^\theta-2(\varepsilon^{2} r)^{\theta}\geq (b\alpha r)^\theta/2,
\end{equation}
where the last inequality holds when  $\varepsilon$ is small enough. This implies \eqref{veps^2} holds for $\bm z,\bm w\in \varepsilon^4 r\mathds{Z}^d$  with $\bm x\in V_{\varepsilon^4 r}(\bm z)$ and $\bm y\in V_{\varepsilon^4 r}(\bm w)$.

In addition, let $B$ be the event that $\text{diam}(V_r(\bm 0);D)\geq \varepsilon^{-\theta/2}r^\theta$. By Axiom V2' (tightness across different scales (upper bound)), we see that there is a constant $C>0$ (does not depend on $r$) such that
\begin{equation}\label{p[B]-diam}
\mathds{P}[B]\leq C\e^{-\varepsilon^{-\theta/2}}.
\end{equation}

Thus, we conclude from \eqref{veps^2}, \eqref{lowerboundDzw2} and the bi-Lipschitz equivalence of $D$ and $\widetilde{D}$ that, conditioned on the event $(A\cup B)^c$,
\begin{equation*}\label{Dtilde diam}
\begin{split}
\widetilde{D}(\bm x,\bm y)&\leq \widetilde{D}(\bm z,\bm w)+ \widetilde{D}(\bm x,\bm z)+ \widetilde{D}(\bm y,\bm w)\quad\quad  \text{(by the triangle inequality)}\\
&\leq C_*D(\bm z,\bm w)-(\varepsilon r)^\theta+ \widetilde{D}(\bm x,\bm z)+ \widetilde{D}(\bm y,\bm w)\quad\quad \text{(by $A^c$ which implies \eqref{veps^2})}\\
&\leq C_*D(\bm z,\bm w)-(\varepsilon r)^\theta+ C_*D(\bm x,\bm z)+C_*D(\bm y,\bm w)\quad\quad \text{(by the optimality of $C_*$)}\\
&\leq C_*D(\bm x,\bm y)-(\varepsilon r)^\theta+2C_*[D(\bm x,\bm z)+D(\bm y,\bm w)]\quad\quad \text{(by the triangle inequality)}\\
&\leq C_*D(\bm x,\bm y)-(\varepsilon r)^\theta + 2C_*\left[{\rm diam}_{\bm z,\varepsilon^4 r}+ {\rm diam}_{\bm w,\varepsilon^4 r}\right]\quad\quad\text{(by the definition of diameter)}\\
&\leq  C_*D(\bm x,\bm y)-(\varepsilon r)^\theta +4C_*(\varepsilon^{2} r)^{\theta}\quad \quad \text{(by the definition of the event $A$ and \eqref{p[A]-diam})}\\
&\leq C_*D(\bm x,\bm y)-(\varepsilon r)^\theta/2\quad\quad \text{(if we choose $\varepsilon$ small enough)}\\
&\leq (C_*-\varepsilon^{3\theta/2}/2)D(\bm x,\bm y)\quad\quad \text{(by the definition of the event $B$)}.
\end{split}
\end{equation*}
Consequently, if we choose $\varepsilon$  small enough, then we can infer from \eqref{Dtilde<(C*-)D} on $(A\cup B)^c$ by taking $\delta=\varepsilon^{3\theta/2}/2$. In other words, we can conclude that
\begin{equation}\label{HcapBc}
\lim_{\delta\to 0}\mathds{P}[H_r(\alpha,C_*-\delta)\cap (A\cup B)^c]=0.
\end{equation}

Moreover, by \eqref{p[A]-diam} and \eqref{p[B]-diam} we obtain
\begin{equation}\label{HcapB}
\lim_{\delta\to 0}\mathds{P}[H_r(\alpha,C_*-\delta)\cap (A\cup B)]\leq \lim_{\delta\to 0}(\mathds{P}[A]+\mathds{P}[B])=0.
\end{equation}
Hence, we obtain the desired statement by combining \eqref{HcapBc} and \eqref{HcapB}.
\end{proof}

\begin{proof}[Proof of Theorem \ref{uniqueness} {\rm(}uniqueness{\rm)}.] By Propositions  \ref{Hwithr=1} and  \ref{strongP[G]}, 
there exist sufficiently small  $\alpha,p\in(0,1)$ and $\gamma,q>0$ (depending only on $\beta,d$ and the laws of $D$ and $\widetilde{D}$), such that for each $\delta>0$, there exists $\varepsilon_0=\varepsilon_0(\gamma,q,\delta)>0$ (depending only on $\gamma,q,\delta,\beta,d$ and the laws of $D$ and $\widetilde{D}$) such that the following is true. For all $r\in(0,\varepsilon_0]$,
\begin{equation}\label{P[H]>p}
\mathds{P}[H_{r }(\alpha, C_*-\delta)]\geq p.
\end{equation}

 Moreover, for the above choice of $\alpha,p$ and sufficiently small $\eta>0$, let $c''$ be the constant defined in Proposition \ref{mrinS3-tilde} with  $c'=(c_*+C_*)/2$ (depending only on $\alpha$ and the laws of $D$ and $\widetilde{D}$).
From Proposition \ref{P[G]-tilde},
 there exist $\widetilde{\gamma},\widetilde{q}>0$ and $\widetilde{\varepsilon}_0>0$ (depending only $\beta,d,c''$ and the laws of $D$ and $\widetilde{D}$) such that $\mathds{P}[\widetilde{G}_r(\widetilde{\gamma},\widetilde{q},c'')]>\widetilde{\gamma}$ for each $r\in(0,\widetilde{\varepsilon}_0]$.
 Hence, we get from Lemma \ref{delta to 0} that, there is $\delta_0=\delta_0(\alpha,p,\eta, \widetilde{\gamma},\widetilde{q})>0$ such that for each $\delta\in(0,\delta_0]$ and each $r\in(0,\widetilde{\varepsilon}_0]$,
\begin{equation}\label{P[H]<p/2}
\mathds{P}[H_{r}(\alpha,C_*-\delta)]\leq p/2.
\end{equation}

Now for the above choice of $\alpha,p$ and $\eta$, we first choose $\delta<\delta_0$ such that \eqref{P[H]<p/2} holds uniformly for all $r\in(0,\widetilde{\varepsilon}_0]$.  For this fixed $\delta$, we then choose a $r\in(0,\varepsilon_0(\delta)\wedge \widetilde{\varepsilon}_0]$ such that \eqref{P[H]>p} holds. Therefore, by applying \eqref{P[H]<p/2} with this $r $, we arrive at a contradiction and therefore conclude that $c_*=C_*$.
\end{proof}

\bigskip

\noindent{\bf Acknowledgement.} \rm
We thank Ewain Gwynne and Jian Wang for helpful discussions, and we thank Ewain for helpful comments on an earlier version of the manuscript. L.-J. Huang would like to thank School of Mathematical Sciences, Peking University for their hospitality during her visit. J. Ding is partially supported by NSFC Key Program Project No. 12231002. L.-J. Huang is partially supported by National Key R\&D Program of China No. 2022YFA1006003.

\bibliographystyle{plain}
\bibliography{LRP}

\end{document}